\documentclass[11pt, letterpaper, reqno]{amsart}
\usepackage[utf8]{inputenc} 	
\usepackage{microtype} 			
\usepackage{geometry}           
\usepackage{amsmath}			
\usepackage{perpage}			
\usepackage{amsthm}		 		
\usepackage{amssymb}	 
\usepackage{bm}					
\usepackage{mathrsfs}			
\usepackage{xcolor}				
\definecolor{darkred}{rgb}{0.6, 0.1, 0.1}
\definecolor{darkblue}{rgb}{0.2, 0.2, 0.6}
\definecolor{darkgreen}{rgb}{0.2, 0.4,.1}
\definecolor{bettercyan}{rgb}{0.1, 0.4, 0.7}
\definecolor{bettermagenta}{rgb}{0.6, 0.2, 0.6}
\usepackage{tikz}				
\usepackage{tikz-3dplot}
\usetikzlibrary{patterns}
\usetikzlibrary{arrows}
\usetikzlibrary{decorations.markings,decorations.pathmorphing,shapes}
\usetikzlibrary{arrows.meta,calc}
\usepackage[all]{xy}			
\usepackage{graphicx}			
\usepackage{pgfplots}\pgfplotsset{compat=1.18} 
\usepackage{enumitem} 			
\usepackage{array}
\usepackage{sidecap}	
\usepackage{lmodern}			
\usepackage[T1]{fontenc}		
\usepackage{cancel}
\usepackage{comment}
\numberwithin{equation}{section}
\usepackage[breaklinks=true, colorlinks=true, citecolor=bettercyan, linkcolor=darkred, urlcolor=bettermagenta]{hyperref}
\usepackage{cleveref}

\usepackage{multirow}
\usepackage{caption} 
\usepackage[backend=biber, style=alphabetic, backref=true, giveninits=true, isbn=false, date=year]{biblatex}
\usepackage{aligned-overset}
\usepackage{accents}
\usepackage{blindtext}
\newcommand{\dbtilde}[1]{\accentset{\approx}{#1}}
\newlength{\dhatheight}
\newcommand{\dbhat}[1]{%
    \settoheight{\dhatheight}{\ensuremath{\hat{#1}}}%
    \addtolength{\dhatheight}{-0.2ex}%
    \hat{\vphantom{\rule{1pt}{\dhatheight}}%
    \smash{\hat{#1}}}}

\makeatletter
\renewbibmacro{in:}{}
\DeclareFieldFormat{pages}{#1}
\renewcommand*{\bibnamedash}{%
	\leavevmode\raise +0.6ex\hbox to 5.5ex{\hrulefill}.\space\space}

\InitializeBibliographyStyle{\global\undef\bbx@lasthash}

\newbibmacro*{bbx:savehash}{\savefield{fullhash}{\bbx@lasthash}}

\renewbibmacro*{author}{%
	\ifboolexpr{
		test \ifuseauthor
		and
		not test {\ifnameundef{author}}
	}
	{%
		\iffieldequals{fullhash}{\bbx@lasthash}
		{\bibnamedash\addcomma\space}
		{\printnames{author}}%
		\usebibmacro{bbx:savehash}%
		\iffieldundef{authortype}
		{}
		{%
			\setunit{\addcomma\space}%
			\usebibmacro{authorstrg}%
		}%
	}
	{\global\undef\bbx@lasthash}%
}
\makeatother

\MakePerPage{footnote}

\newtheorem{propositionx}{Proposition}[section]
\newenvironment{proposition}
{\pushQED{\qed}\propositionx}
{\popQED\endpropositionx}
\newenvironment{propositionp}
{\pushQED{\qed}\propositionx}
{\popQED\endpropositionx}

\newtheorem{theorem}{Theorem}

\newtheorem*{theorem*}{Theorem}

\newtheorem{corollaryx}[propositionx]{Corollary}
\newenvironment{corollary}
{\pushQED{\qed}\corollaryx}
{\popQED\endcorollaryx}

\newtheorem{lemmax}[propositionx]{Lemma}
\newenvironment{lemma}
{\pushQED{\qed}\lemmax}
{\popQED\endlemmax}

\theoremstyle{remark}
\newtheorem{remark}[propositionx]{Remark}
\newtheorem*{remark*}{Remark}

\theoremstyle{definition}
\newtheorem*{examplex}{Example}
\newenvironment{example}
{\pushQED{\qed}\examplex}
{\popQED\endexamplex}

\newcommand{\bbC}{\mathbb{C}}

\newcommand{\bbN}{\mathbb{N}}

\newcommand{\bbR}{\mathbb{R}}
\newcommand{\bbS}{\mathbb{S}}

\newcommand{\bbZ}{\mathbb{Z}}

\newcommand{\calA}{\mathcal{A}}

\newcommand{\calC}{\mathcal{C}}
\newcommand{\calD}{\mathcal{D}}
\newcommand{\calE}{\mathcal{E}}
\newcommand{\calF}{\mathcal{F}}

\newcommand{\calI}{\mathcal{I}}

\newcommand{\calK}{\mathcal{K}}

\newcommand{\calR}{\mathcal{R}}
\newcommand{\calS}{\mathcal{S}}

\newcommand{\calV}{\mathcal{V}}

\newcommand{\calX}{\mathcal{X}}
\newcommand{\calY}{\mathcal{Y}}
\newcommand{\calZ}{\mathcal{Z}}

\newcommand{\frakS}{\mathfrak{S}}

\newcommand{\frakW}{\mathfrak{W}}

\newcommand{\bfM}{\mathbf{M}}

\newcommand{\bfc}{\mathbf{c}}

\newcommand{\bfk}{\mathbf{k}}

\newcommand{\bfv}{\mathbf{v}}
\newcommand{\bfw}{\mathbf{w}}
\newcommand{\bfx}{\mathbf{x}}
\newcommand{\bfy}{\mathbf{y}}
\newcommand{\bfz}{\mathbf{z}}

\newcommand{\dd}{\,\mathrm{d}}

\newcommand{\p}{\partial}

\newcommand{\set}[1]{ \left \{ #1 \right  \}}

\usetikzlibrary{arrows.meta, positioning, calc}

\makeatletter
\@namedef{subjclassname@2020}{\textup{2020} Mathematics Subject Classification}
\makeatother

\title{The asymptotic structure of forward scattering}

\date{August 27, 2026\, (Last updated).}

\author{Nicholas Lohr}
\address[Nicholas Lohr]{Department of Mathematics, Purdue University, West Lafayette, Indiana}
\email{nlohr@purdue.edu}
\author{Izak Oltman}
\address[Izak Oltman]{Department of Mathematics, Northwestern University, Evanston, Illinois}
\email{ioltman@northwestern.edu}
\author{Ethan Sussman}
\address[Ethan Sussman]{Department of Mathematics, Northwestern University, Evanston, Illinois}
\email{ethan.sussman@northwestern.edu}
\author{Yuzhou Joey Zou}
\address[Joey Zou]{Department of Mathematics and Statistics, Oakland University, Rochester, Michigan}
\email{yzou@oakland.edu}

\begin{document}
	
\begin{abstract}
    Perturbed plane waves are fundamental objects in scattering theory on Euclidean space and asymptotically Euclidean spaces.
    In this paper, we investigate the structure of perturbed plane waves in the \emph{forward} direction, in which the outgoing spherical wave is typically singular and conjoined to the incoming plane wave.  Melrose \& Zworski provided a microlocal description (in the more general setting of asymptotically conic manifolds) using their notion of Lagrangian distributions associated to pairs of intersecting Legendrian submanifolds, on the way to proving that the S-matrix is an FIO.
    Here, we revisit the problem in the asymptotically Euclidean case, for which the oscillatory integrals used by Melrose--Zworski attain their most complicated form (relative to the more general asymptotically conic case). We seek a more elementary description in terms of physical-space asymptotics. These are specified using a two-faced compactification $X\hookleftarrow \bbR^d$, with one face for each asymptotic regime. We prove full polyhomogeneity.
    A transport equation arises as a model problem at the main face (`bf').
    The quantum inverted harmonic oscillator arises as a model problem at the front face (`ff'). 
\end{abstract}

\maketitle

\tableofcontents

\section{Introduction}\label{sec:intro} 
Perturbed plane waves are certain functions on $\bbR^d$, $d\geq 2$, 
which model the result when a monochromatic wave impinges on some inhomogeneity. Monochromatic waves are governed by the
(Schr\"odinger--)Helmholtz equation 
\begin{equation}
    (\triangle-k^2)u=Qu, 
\end{equation}
where $k>0$ is the frequency of the wave and $Q\in \operatorname{Diff}^2(\bbR^d)$ is some (classical) variable-coefficient second-order differential operator encoding the inhomogeneity. Without loss of generality, we assume $k=1$. 
In potential scattering, $Q$ is a multiplication operator 
\begin{equation} 
    Q:u\mapsto-V u
\end{equation} 
for some sufficiently regular potential $V$. 
In the presence of a magnetic field, $Q$ contains first derivatives.
In the presence of a metric perturbation, second derivatives are present as well. The perturbations we consider here are all short-range, which means that the coefficients in $Q$ are all 
\begin{equation}
    O\Big(\frac{1}{r^2} \Big)\text{ as }r\to\infty, 
\end{equation}
where $r$ is the Euclidean radial coordinate. This introduction discusses only potential scattering. More general situations are considered in the main body of the paper.

\vspace{.25em}
\noindent 
\begin{minipage}{.5\linewidth}
    	\hspace{1em}
    Among all monochromatic waves, the perturbed plane waves are characterized by their asymptotic form. 
    Split 
    \begin{equation} 
    \bbR^d=\bbR_x\times \bbR^{d-1}_{\bfy}, \quad r=\sqrt{x^2+y^2},\quad y=\lvert \bfy \rvert.  
    \end{equation}
    Then, the Laplacian, with our sign convention, can be written 
    \begin{equation} 
    \triangle = - \partial_x^2+\triangle_{\bfy}.
    \end{equation} 
    A perturbed plane wave traveling in the $x$-direction is supposed to be a solution of the Helmholtz equation of the form 
    \begin{equation}
    	u \approx  \overbrace{e^{i x}}^{\mathclap{\text{plane wave}}} + \underbrace{\frac{f(\omega) }{r^{(d-1)/2}} e^{ir}}_{\mathclap{\text{spherical wave}}} ,
    	\label{eq:u_form_rough}
    \end{equation}
    for $r\gg 1$, where $\omega=(x/r,\bfy/r)$ 
    and $f$ is some function defined almost everywhere on the unit sphere $\bbS^{d-1}$.
    We discuss the meaning of `$\approx$' below. 
\end{minipage}
    \hspace{2em}
\begin{minipage}[l]{.4\linewidth}
    \hspace{2em}
		\begin{tikzpicture}[scale=.8]

          \def\bandwidth{0.45}
          \def\ringgap{0.45}
          \def\H{5}

          \pgfmathsetmacro{\halfW}{\bandwidth/2}
          \pgfmathsetmacro{\halfH}{\H/2}

          \foreach \n in {0,1,2} {
            \pgfmathsetmacro{\rin}{0.45 + \n*(\bandwidth + \ringgap)}
            \pgfmathsetmacro{\rout}{\rin + \bandwidth}
            \path[fill=gray!45, fill opacity=0.35, even odd rule]
              (0,0) circle[radius=\rout]
              (0,0) circle[radius=\rin];
          }

          \foreach \x in {-2.10,-1.2,-0.44,0.44,1.2,2.10} {
            \path[fill=gray!45, fill opacity=0.35]
              ({\x - \halfW},-\halfH) rectangle ({\x + \halfW},\halfH);
          }

          \draw[-{Stealth[length=3mm,width=1mm]}, line width=0.45pt]
            (-3,0) -- (3,0);
        
\end{tikzpicture}
        
		\textsc{Figure 1}. A schematic of a perturbed plane wave as consisting of an incoming plane wave superposed with an outgoing spherical wave. \stepcounter{figure}
\end{minipage}
\vspace{.1em}

\noindent Thus, $u$ consists of an incoming plane wave, traveling to the right, and an outgoing spherical wave 
\begin{equation} 
\propto r^{-(d-1)/2}e^{ir}
\end{equation} 
(plus subleading corrections),
and should be the unique Helmholtz solution in this class. 
The function $f$ is of central importance, as it is the ``scattering amplitude'' that physicists measure in scattering experiments, up to a multiplicative constant.

Perturbed plane waves are known to exist for all (sufficiently regular) short-range potentials \cite{Agmon1975, AH, MZ}. However, the precise interpretation of `$\approx$' depends on the decay rate of the potential, $L\in \bbN^{\geq 2}$, defined by: 
\begin{equation}
    V = O\Big( \frac{1}{r^L} \Big)\text{ for }r>1.
\end{equation}
The simplest case is when the potential is decaying at a sufficiently fast rate, namely $L>d$.
This was established in classical papers of Povzner \cite{Pov1, Pov2}. We call $L=d$ the \emph{Povzner threshold}. When the potential decays faster than the Povzner threshold, the precise interpretation of `$\approx$' is the natural one:
\begin{equation}\label{eq:u_form_precise}
    u = e^{ix} + \frac{f(\omega)}{r^{(d-1)/2}} e^{ir} +o\Big( \frac{1}{r^{(d-1)/2}} \Big), \quad f\in C^0(\bbS^{d-1})
\end{equation}
for $r>1$. 
Potentials decaying faster than the Povzner threshold are called \emph{very} short-range.

Unfortunately, the very short-range case excludes typical electromagnetic potentials. In the physical $d=3$ case, the electrostatic potential generated by a charge distribution is decaying at a rate $L>d$ only if its moments up to and including the quadrupole moment all vanish. The Coulomb potential is long-range, while most other electromagnetic potentials are short- but not very short-range.
For such potentials, the description above typically fails: $\not\exists f\in C^0(\bbS^{d-1})$ such that \cref{eq:u_form_precise} holds
\cite{YafaevInverse}.

Our main theorem is a direct description of the perturbed plane wave, partially anticipated by Skriganov. It makes precise the following heuristic:

\vspace{.25em}
\noindent 
\begin{minipage}[l]{.5\linewidth}
    \begin{theorem*}[Heuristic version]
    The perturbed plane wave is a superposition of an incoming plane wave (including decaying corrections) and an outgoing spherical wave conjoined by a forward parabolic wake.
\end{theorem*}
See \Cref{thm:main} for the precise statement. \Cref{thm:main_black-box} contains a more general version for black-box scattering on an asymptotically Euclidean manifold. 
\begin{remark*}
    A modified version of the theory applies for Coulombic potentials (cf.\ \S\ref{sec:Coulomb}), but our focus is on the short-range case. 
\end{remark*}
\end{minipage}
    \hspace{2em}
\begin{minipage}[l]{.4\linewidth}
    \hspace{2em}
    \begin{tikzpicture}[scale=.8]
    
      \def\bandwidth{0.45}
      \def\ringgap{0.45}
      \def\H{5}
      \def\parabcoef{0.22}   
      \def\Ymax{2.35}        
    
      \pgfmathsetmacro{\halfW}{\bandwidth/2}
      \pgfmathsetmacro{\halfH}{\H/2}
    
      \foreach \n in {0,1,2} {
        \pgfmathsetmacro{\rin}{0.45 + \n*(\bandwidth + \ringgap)}
        \pgfmathsetmacro{\rout}{\rin + \bandwidth}
        \path[fill=gray!45, fill opacity=0.35, even odd rule]
          (0,0) circle[radius=\rout]
          (0,0) circle[radius=\rin];
      }
    
      \foreach \x in {-2.10,-1.2,-0.44,0.44,1.2,2.10} {
        \path[fill=gray!45, fill opacity=0.35]
          ({\x - \halfW},-\halfH) rectangle ({\x + \halfW},\halfH);
      }
    
      \foreach \s in {.6,1.40,2.20,3.00} {
        \pgfmathsetmacro{\sin}{\s}
        \pgfmathsetmacro{\sout}{\s + \bandwidth}
        \path[fill=gray!45, fill opacity=0.45]
          plot[domain=-\Ymax:\Ymax, samples=120, variable=\y]
            ({\sout - \parabcoef*\y*\y},{\y})
          --
          plot[domain=\Ymax:-\Ymax, samples=120, variable=\y]
            ({\sin - \parabcoef*\y*\y},{\y})
          -- cycle;
      }
    
      \draw[-{Stealth[length=3mm,width=1mm]}, line width=0.45pt]
        (-3,0) -- (3,0);
    
    \end{tikzpicture}
    
	\textsc{Figure 2}. A schematic amending the previous one by adding \emph{a parabolic wake} in the forward direction. Compare with \Cref{fig:dipole_plane_wave_dif}. \stepcounter{figure}
\end{minipage}

\vspace{.25em}
For example, we will be able to prove: 
\begin{theorem}[Away from forward direction] 
    There exist $f,a\in C^\infty(\bbS^{d-1}\backslash \{\rightarrow\})$, smooth except at the forward direction $\rightarrow \in \bbS^{d-1}$, such that, in any region of the form $\{x<C y\} \subset \bbR^d$, for $C>0$, the approximation
    \begin{equation}
        u = e^{ix} \Big(1+\frac{a(\omega)}{r^{L-1}}  \Big) + e^{ir} \frac{f(\omega)}{r^{(d-1)/2}} + O\Big(\frac{1}{r^{\min\{L,(d+1)/2\}} } \Big)
    \end{equation}
    holds, 
    where $\omega = (x/r,\bfy/r) \in \bbS^{d-1}$.
    
    Thus, away from the forward direction, $u$ is a sum of a plane wave (with a correction $\propto a/r^{L-1}$) and a spherical wave.
\end{theorem}
\begin{theorem}[Structure of parabolic wake]
    Suppose $L\neq d-1,d,d+1$. 
    There exists a function $g\in C^\infty(\bbR^{d-1}_{\hat{\bfy}})$ and a constant $\digamma\in \bbC$ such that, in any neighborhood of the form $\{y^2<Cx\}$, for $C>0$, the approximation
    \begin{equation}
        u = e^{ix} \Big( 1+\frac{g(\hat{\bfy})}{r^{(L-1)/2}} + \frac{ e^{i\hat{y}^2/2} \digamma}{r^{(d-1)/2}} \Big) + O \Big( \frac{1}{r^{\frac{1}{2}\min\{L,d\} } } \Big)
    \end{equation}
    holds, where $\hat{\bfy}=\bfy/\sqrt{x}$ and $\hat y = |\hat{\bfy}|$. 
\end{theorem}

In both parts of the statement, the estimate of the big-$O$ term can depend on the parameter $C$.
\begin{remark*}
    The cases $L=d-1,d,d+1$ are very similar, but with a few extra logarithms.
\end{remark*}

Our main theorem will upgrade these theorems to full (term-by-term differentiable) asymptotic expansions, and will articulate the sense in which the parabolic wake dovetails with the incoming plane wave and outgoing spherical wave.

\subsection{Discussion of theorem}

This theorem emphasizes two features of perturbed plane waves that are missed by the traditional asymptotic description of the wave as consisting of an $O(1)$ plane wave plus outgoing $O(1/r^{(d-1)/2})$ spherical wave. These features are:
\begin{enumerate}[label=(\roman*)] 
    \item polynomially-decaying corrections oscillating like the plane wave, taking the form $e^{ix}  a/r^{L-1}$ for non-oscillatory $a$, 
    \item the parabolic wake left by the passing wave, which we will see propagates along lines of constant $y^2/x$. 
\end{enumerate}

Subleading corrections of this sort
are present for \emph{any} classical potential that decays polynomially but not superpolynomially as $r\to\infty$; see below (cf.\ \cite{YafaevInverse}). (Since we restrict attention to classical potentials, superpolynomial decay is equivalent to being Schwartz.)
For potentials decaying faster than the Povzner threshold $L=d$, the corrections are $o(1/r^{(d-1)/2})$ and therefore faster decaying than the outgoing spherical wave, hence consistent with the traditional description of the perturbed plane wave.   
For short-range potentials that are not very short-range, one or both corrections are asymptotically bigger than $1/r^{(d-1)/2}$ along some rays $\bbR^+ \omega$, thereby invalidating the traditional description. 
Thus, the corrections are not a pathology limited to slowly decaying potentials, though they are most prominent for slowly decaying potentials.

Any ray $\bbR^+\omega$ emanating from the origin -- except the one exactly in the forward direction -- eventually leaves any parabolic neighborhood 
\begin{equation}
    \{(x,\bfy)\in \bbR^d: y^2 < c x\},\quad c>0. 
\end{equation}
Thus, parabolic features are invisible in the large-$r$ asymptotics of $u$ inside any cone disjoint from the forward direction. 
Away from the forward direction, only one correction is relevant, the $o(1)$ corrections to the plane wave. 
We will see below that 
the first plane wave correction has size $O(1/r^{L-1})$ away from the forward ray. 
Thus, the plane wave correction becomes $\Omega(1/r^{(d-1)/2})$, hence the same size or larger than the outgoing spherical wave, exactly at the \emph{Ikebe threshold} $L=(d+1)/2$. Slower decay results in a bigger plane wave correction.

If we follow $u$ exactly along the forward ray, then $u-e^{ix}$ can decay more slowly. 
A sharp statement is
\begin{equation} \label{eq:correction_heuristic}
V=O \Big( \frac{1}{r^L} \Big)\Longrightarrow 
    u-e^{ix} = O\Big( \max \Big\{ \frac{1}{r^{(d-1)/2}}, \frac{1}{r^{(L-1)/2} } \Big\}\Big) \text{ along the forward ray},
\end{equation}
unless $L=d$, in which case an extra log is present. Thus, the Povzner threshold $L=d$ is exactly the threshold for having a large wake. If $L>d$, then the wake is suppressed. Otherwise, it is large. 
In summary, the phenomenologically important thresholds (for $d\geq 3$) are 
\begin{equation}
    2\leq \underbrace{\frac{d+1}{2}}_{\text{Ikebe}} < \underbrace{d}_{\mathclap{\text{Povzner}}}, 
\end{equation}
with the Povzner threshold being that for a large wake and the Ikebe threshold being that for large plane wave corrections. Depending on which of these thresholds $L$ exceeds, different corrections among those described above may be important. 

Throughout this paper, we use the following $L=2$ example to illustrate the main ideas:
\begin{example}[Inverse-square potential]
    Consider 
    \begin{equation}
    V(r) = \frac{\alpha}{r^2},\quad \alpha> - \frac{1}{4}
\end{equation}
on $\bbR^d$, $d=3$. 
We refer to the corresponding perturbed plane wave as the \emph{inverse-square plane wave}. 
Its explicit construction can be found in \S\ref{sec:central}, via separation of variables. This presents the inverse-square plane wave as an infinite series of special functions, 
\begin{equation}
    u= \sqrt{\frac{\pi}{2r}} \sum_{\ell=0}^\infty i^{2\ell-\mu_\ell+\frac{1}{2}} (2\ell+1) J_{\mu_\ell}(r) P_\ell (\cos \theta),\quad \mu_\ell = \sqrt{\alpha+\Big(\ell+\frac{d-2}{2}\Big)^2}.
\end{equation}
Partially summing this series provides an accurate numerical approximation to the plane wave, shown in \Cref{fig:dipole_plane_wave}. For $\alpha>0$ (e.g.\ $\alpha=2$, the case plotted in the figure), the potential is repulsive, leaving a shadow as the wave passes, as seen in the figure. 

\begin{figure}[!htbp]
    \centering
    \includegraphics[scale=1.1]{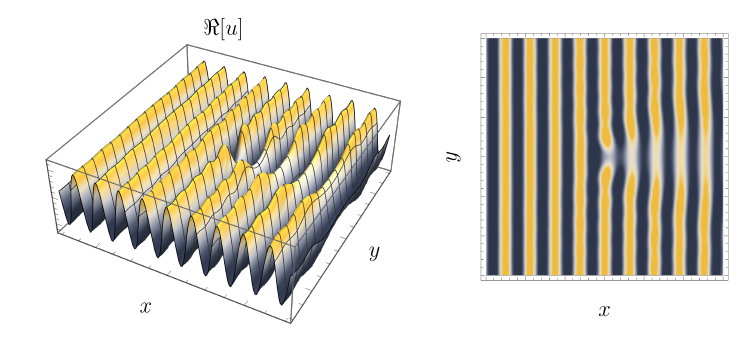}
    \caption{
    The (real part of the) inverse-square plane wave $u$ with parameter $\alpha=2$, calculated using the explicit formula above. The imaginary part looks similar.}
    \label{fig:dipole_plane_wave}
\end{figure}

The difference $u-e^{ix}$ is plotted in \Cref{fig:dipole_plane_wave_dif}. Most prominent among the plot's features is the wake, which is decaying, but only at a $O(1/\sqrt{r})$ rate as $r\to\infty$ (see below). 
It appears to spread along parabolic trajectories $x=c y^2$, $c\geq 0$. In fact, it does --- numerical evidence is provided later. 
A faint outgoing spherical wave can be discerned, in addition to what appears to be a residual plane wave $\sim r^{-1} e^{ix}$. Both these features are $O(1/r)$ within cones away from $\theta=0$ (which is why they are faint relative to the larger wake, being smaller than the latter by a factor of $1/\sqrt{r}$), but they grow as $\theta\to 0$, merging to form the wake. 

\begin{figure}[!htbp]
    \centering
    \includegraphics[scale=1.1]{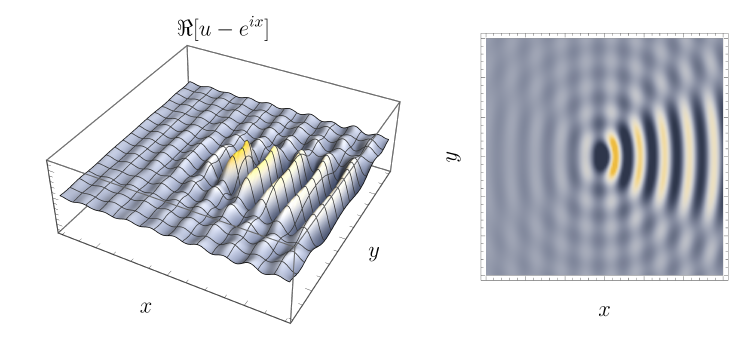}
    \caption{The (real part of the) difference $u-e^{ix}$ between the perturbed plane wave $u$ and the exact plane wave $e^{ix}$. The imaginary part looks similar.}
    \label{fig:dipole_plane_wave_dif}
\end{figure}

The hierarchy of sizes for this particular example is therefore:
\begin{equation}
    \underbrace{\begin{array}{r}
         \text{outgoing spherical wave $\sim  f(\theta) r^{-1} e^{ir}$}  \\
           \text{plane wave correction $\sim a(\theta) r^{-1} e^{ix}$}
    \end{array}}_{O(1/r)\text{ away from }\theta=0} \leq \underbrace{\text{wake}}_{\mathclap{O(1/\sqrt{r})}} \leq \underbrace{\text{main }e^{ix}\text{ term}}_{O(1)}.
\end{equation}
When $d=3$, the Ikebe threshold is exactly $L=2$, so this example is borderline. This is why the first plane wave correction has the same size as the outgoing spherical wave. The occurrence of a large wake of size $\sqrt{r}$ relative to the outgoing spherical wave is a consequence of being one order under the Povzner threshold $d=3$.
\end{example}

\subsection{The S-matrix's forward singularity}
\label{subsec:S-matrix_forward_singularity}
For any short-range potential, the S-matrix coefficient $f$ is defined (see below) and smooth away from the forward direction: 
\begin{equation}
    f\in C^\infty(\bbS^{d-1}\backslash \{\text{forward}\}).
\end{equation}
As long as the potential is not Schwartz, the singularity in the forward direction is genuine:
\begin{equation}
    f\notin C^\infty(\bbS^{d-1}) 
\end{equation}
\cite{YafaevInverse}.
The structure of the singularity has been a topic of some interest \cite{YafaevDiagonal}. 
Weder--Yafaev \cite{YafaevInverse} have even proven that, for sufficiently short-range potentials, the germ of the potential at infinity can be recovered from the singularity. Melrose--Zworski \cite{MZ} and Vasy \cite{VasyLongRange} have proven that the S-matrix is a Fourier integral operator (FIO) associated to the antipodal map on $\bbS^{d-1}$. Such an FIO transports singularities to the opposite side of the sphere, which is why $f$ is smooth away from the forward direction.

Here, we take the perspective that the forward singularity is the price paid for ignoring the parabolic wake when describing the perturbed plane wave as a plane wave + outgoing spherical wave. 
The slower the potential is decaying, the larger the parabolic wake, so the worse $f$'s forward singularity is. The Povzner assumption $V=o(1/r^d)$ is exactly what is required to guarantee that $f$ is continuous. For $V$ decaying faster than the Povzner threshold, 
\begin{equation}
    f\in C^k(\bbS^{d-1}) 
\end{equation}
for some finite $k\in \bbN$.  The exact numerology relating the decay rate of $V$ to the largest such $k$ can be understood in several ways. One is using the Born approximation ---  see \S\ref{sec:Born}. We get one extra order of differentiability for every extra order of decay. 
The conclusion is  
\begin{equation}
    V = O\Big(\frac{1}{r^L} \Big) \Longrightarrow f(\omega) \in
    \begin{cases}
        \theta^{-(d-L)} C^0(\bbS^{d-1}) & (L<d), \\
        (\log \theta )C^0(\bbS^{d-1})  & (L=d).
    \end{cases}
\end{equation}
So, $f(\omega)$ blows up as $\theta\to 0$ like some power of $\theta$, with a logarithmic loss in the threshold case.

A more precise description of $f$'s forward singularity is that it is \emph{polyhomogeneous}. 
Polyhomogeneity \cite{Me93} roughly means smoothness modulo log terms.
So, letting $\theta=\operatorname{arccos}(x/r)$ denote the angle from the forward direction, $f$ admits an asymptotic expansion in the $\theta\to 0$ limit involving terms of the form
\begin{equation}
    a_{j,k} \theta^j (\log \theta)^k,\quad a_{j,k}\in C^\infty(\bbS^{d-2}_{\bfy/y}),\quad (j,k)\in \bbR\times \bbN.
\end{equation}
The pairs $(j,k)$ appearing in this expansion are called ``indices.'' 
The singularity is dominated by the index with the most negative $j$. 

\begin{example}[Inverse-square potential, cont.]
    For the inverse-square plane wave on $\bbR^3$,  
    \begin{equation}
        f = \sum_{\ell=0}^\infty (2\ell+1) e^{i\delta_\ell} \sin(\delta_\ell) P_\ell(\cos \theta), \quad \delta_\ell = \frac{\pi}{4} \big(\nu-\sqrt{4\alpha+\nu^2}\, \big), 
    \end{equation}
    where $\nu=2\ell+1$. 
    It can be shown that 
    \begin{equation}\label{eq:dipole_scat_sing}
        f = - \frac{\pi\alpha}{2|\theta|} - \frac{i\pi^2 \alpha^2}{8} \log |\theta |  + O(1)
    \end{equation}
    as $\theta\to 0$. See \Cref{fig:dipole_scat_sing}.

    \begin{figure}[!htbp]
        \centering
        \includegraphics[scale=.7]{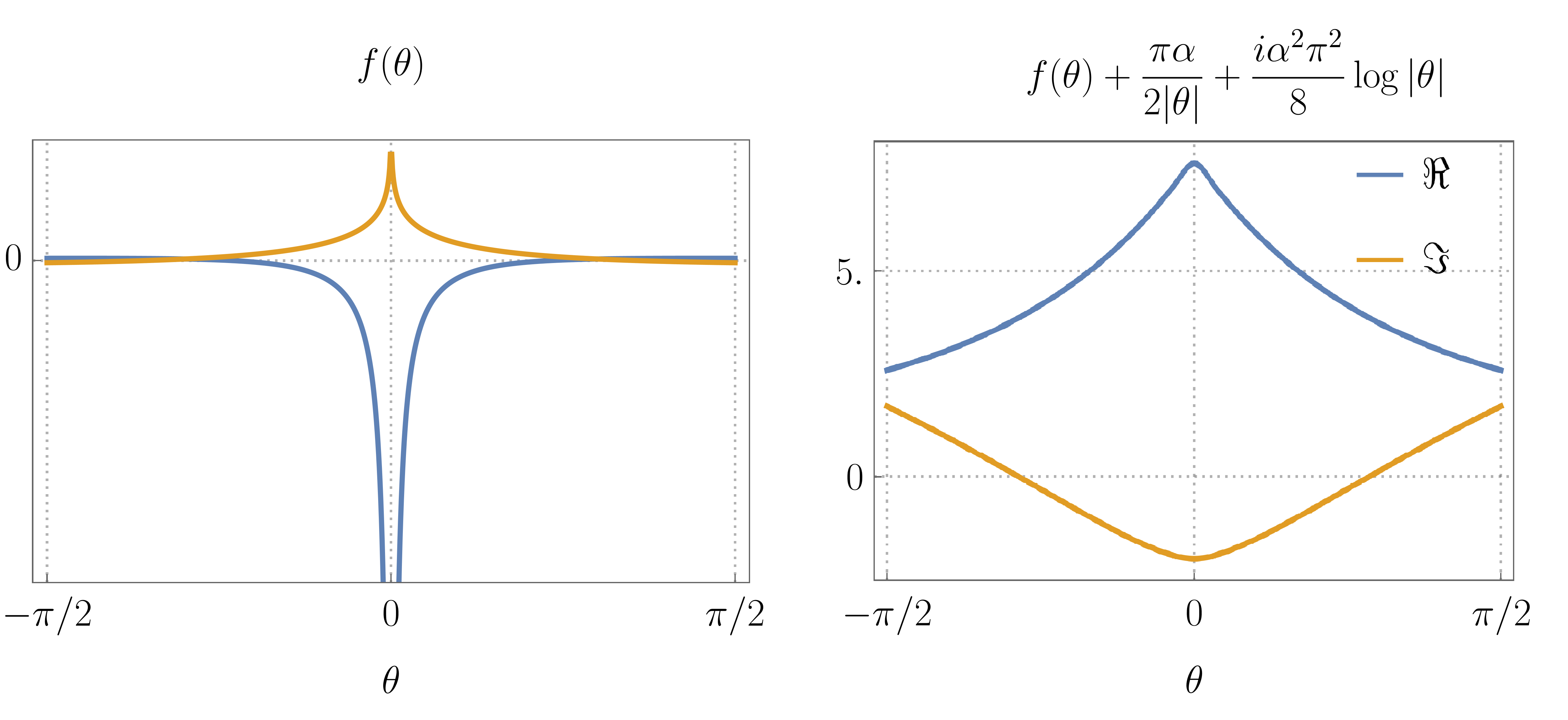}
        \caption{The scattering amplitude $f(\theta)$ for the inverse-square potential $V(r)=2/r^2$.}
        \label{fig:dipole_scat_sing}
    \end{figure}
    
    Thus, $f$ blows up at $\theta=0$, with a $1/\theta$ singularity and a subleading $\log \theta$. 
\end{example}

The notion of polyhomogeneity is needed precisely because of the possibility of log terms, which the previous example shows can actually occur.

\begin{remark*}
	Physicists call $|f(\omega)|^2 $ the ``differential cross-section.'' Its physical interpretation is (roughly) that it measures the scattered flux of particles in the direction $\omega\in \bbS^{d-1}$ per unit of incoming flux. 
    This interpretation is viable for sufficiently short-range potentials, but $f(\omega)$ blows up as $\theta\to 0$ like $\theta^{-(d-L)}$, and 
    \begin{equation}
        \theta^{-(d-L)} \in L^2(\bbS^{d-1}) \iff L>\frac{d+1}{2} . 
    \end{equation}
    So, a sufficiently small sector around the forward directions accounts for negligible outgoing flux \emph{only if the potential is decaying faster than the Ikebe threshold}. Then, physicists are justified in ignoring the parabolic wake, because a scattered particle, in the far future, is unlikely to be found in the wake. Nevertheless, it is still of interest to understand the structure of the wave, including the $L^2$ mass at intermediate times. 
    But if the potential is decaying at the Ikebe threshold rate or slower, then $f\notin L^2$, so the usual interpretation is problematic.

	For a careful introduction to the physics of non-relativistic scattering, see \cite{Taylor}\cite{TaylorAgain}\cite{TaylorAlt}. Making precise the physical interpretation of perturbed plane waves requires relating the Schr\"odinger equation to the Helmholtz equation.
\end{remark*}

\begin{remark*}
    The celebrated \emph{optical theorem} is supposed to relate the imaginary part $\Im f(\theta)$ of the forward scattering amplitude to the ``total cross section'' $\lVert f \rVert_{L^2(\bbS^{d-1})}$. When $d=3$, it states
    \begin{equation}
        \lVert f \rVert_{L^2(\bbS^{d-1})}^2 = 4\pi \Im f(0). 
    \end{equation}
    Physicists rarely spell out the precise hypotheses required. 
    As the inverse-square example shows, $\lVert f \rVert_{L^2}$ can fail to exist. But in this case, $\lim_{\theta\to 0} \Im f(\theta)=\infty$, so the optical theorem can still be read as holding.
    
    Simon \cite[Chp.\ 11, \S6]{ReedSimonIII} proves the optical theorem under the assumption that the potential is $L^1$ and Rollnik class; this excludes the inverse-square example above. 
\end{remark*}

\subsection{Previous work}

This subsection provides an overview of older work on the subject. Readers uninterested in the history may skip directly to the next section of our paper, which presents the main theorem. 

\subsubsection{Classical results}
The work of Povzner noted previously was the first in a series of classical papers establishing the existence and properties of perturbed plane waves $u$ for potential scattering, 
\begin{equation*}\label{eq:Helmholtz_loc}
    (\triangle -1 +V) u =0, \tag{$\bigstar$}
\end{equation*}
in increasing generality. 
The story hinges on the decay rate $L$ of $V=O(1/r^L)$. The upshot is the following trichotomy, as presented by Yafaev \cite{OldYafaev} (take $V\in C^\infty(\bbR^d)$ in each case):
\begin{enumerate}[label=(\Roman*)]
    \item (\cite{Pov1, Pov2}) in the very short-range case, suppose $L >d$, then there exists a unique 
    solution $u \in C^{\infty}(\bbR^d)$ of the Helmholtz equation  
    \cref{eq:Helmholtz_loc}
    satisfying
    \begin{equation}
        u(x,\bfy)=e^{ix}+ \frac{f}{r^{(d-1)/2}} e^{ir} + o\Big(\frac{1}{r^{(d-1)/2}}\Big)
    \end{equation}
    for some $f\in C^0(\bbS^{d-1})$. In fact, letting $k\in \mathbb N$ be the largest nonnegative integer such that $L > d + k$, then $f \in C^k(\bbS^{d-1})\cap C^\infty(\bbS^{d-1}\backslash \{\text{forward}\})$ \cite[Proposition 3.1]{YafMag}. 
    \item If $L \in (\frac{d+1}{2}, d]$, then there exists a unique solution $u\in C^{\infty}(\bbR^d)$ of \cref{eq:Helmholtz_loc}
    satisfying
    \begin{equation}\label{eq:ppwa}
        u(x,\bfy)=e^{ix}+ \frac{f}{r^{(d-1)/2}} e^{ir} + \calE 
    \end{equation}
    for $f \in L^2(\bbS^{d-1})\cap C^\infty(\bbS^{d-1}\backslash \{\text{forward}\})$ and $\calE=o_{\mathrm{av}}(r^{-\frac{d-1}{2}})$, which means 
    \begin{equation}
         \lim_{R \to \infty}\frac{1}{R}\int_{|(x,\bfy)|\leq R}|\calE|^2\dd x\dd \bfy=0.
    \end{equation}
    In this case, the perturbed plane waves were constructed by Ikebe \cite{Ikebe, IkebeErrata} using the Lippmann--Schwinger method (see below). The actual characterization of the error is due to Agmon--H\"ormander \cite{AH}. 
    
    \item If $L \in (1,(d+1)/2]$, then the na\"ive construction is frustrated by the presence of large plane wave corrections. Yafaev \cite{YafAvg} retreats to studying spherical-wave to spherical-wave scattering.
    In the $d=3$ case, for a short-range symbolic potential with decay rate $L\in (1,2]$, Buslaev--Skriganov \cite{buslaev1974coordinate, Sk, SkPhD} announced that there exists a unique solution $u \in C^{\infty}(\bbR^3)$ satisfying \cref{eq:Helmholtz_loc} with $r\to\infty$ asymptotics 
    \begin{equation}\label{eq:sk}
        u(x,\bfy)=e^{ix+(2i)^{-1}\int_{-\infty}^xV(t ,\bfy) \dd t}+\frac{f}{r}e^{ir}+o\Big(\frac{1}{r} \Big),\quad f \in C^{\infty}(\bbS^{2}\backslash \{ \rightarrow \})
    \end{equation}
    away from the forward direction. The authors claimed (without proof) a generalization to $d\geq 4$, but the statement \cref{eq:sk} requires modification. A natural generalization is 
    \begin{equation}\label{eq:sk_mod}
        u(x,\bfy)=e^{ix}\times\text{plane wave corrections}+\frac{f}{r^{(d-1)/2}}e^{ir}+o\Big(\frac{1}{r^{(d-1)/2}} \Big),
    \end{equation}
    but the form of the plane wave corrections does not carry over.
\end{enumerate}

The basic strategy to construct perturbed plane waves goes back to Lippmann--Schwinger \cite{LippmannSchwinger}. 
When $V$ decays faster than the Ikebe threshold $O(1/r^{(d+1)/2})$, then 
\begin{equation}\label{eq:russianwave}
u(x,\bfy)=e^{ix}-R(1 + i0)Ve^{ix},\qquad R(\lambda) \coloneqq (\triangle +V-\lambda)^{-1}.
\end{equation}
The ``limiting resolvent'' $R(1+i0)$ is defined via the usual \textit{limiting absorption principle}: for any short-range $V$, the limit 
\begin{equation} 
    R(1\pm i0) = \lim_{\varepsilon \to 0^+} R(1 \pm i \varepsilon)
    \end{equation}
exists (with respect to the strong operator topology on suitable weighted Sobolev spaces)
and has the mapping property $R(1\pm i0):\langle r \rangle^{-1/2-\delta}L^2 \to \langle r \rangle^{1/2+\delta}L^2$ \cite{Me93}, for any $\delta>0$. The key observation is that
\begin{equation}
    L> \frac{d+1}{2} \Longrightarrow \langle r \rangle^{\frac{1}{2}+\delta} Ve^{ix}\in L^2\text{ for some }\delta>0. 
\end{equation}
Consequently, the right-hand side of \cref{eq:russianwave} is defined as
\begin{equation} 
R(1+i0)[Ve^{ix}] \coloneqq R(1+i0)\big[\langle x \rangle^{-\frac{1}{2}-\delta}\underbrace{\langle r \rangle^{\frac{1}{2}+\delta}Ve^{ix}}_{\in L^2}\big].
\end{equation}

When $1<L\leq (d+1)/2$, the story becomes more complicated, because $\langle r \rangle^{\frac{1}{2}}V e^{ix}  \notin L^2$, so the formula $R(1+i0) Ve^{ix}$ does not make sense, a priori.
Here, the advantages of a microlocal framework, to which we turn, are most clear.

\subsubsection{The Melrose school}\label{sec:melrose}

In \cite{Me94}, Melrose defined the Poisson map for the Helmholtz equation on general asymptotically conic manifolds. We explain the case of Euclidean potential scattering. 
The \emph{Poisson map} $\Pi:b_-\mapsto u$ constructs, for each $ b_-\in C^\infty(\bbS^{d-1})$, a Helmholtz solution $u\in C^\infty(\bbR^{d})$ with ``incoming data'' $b_-$. This means that the Sommerfeld radiation condition 
\begin{equation}\label{eq:poisson}
  \exists \varepsilon>0\text{ s.t. }r^{\varepsilon+(d-1)/2} ( i\partial_r+1 )  \Big[  u - \frac{b_-(\omega)}{r^{(d-1)/2}} e^{-i r }  \Big]  \in L^\infty (\{r>1\})
\end{equation}
is satisfied. Under the assumption that the potential is classical, Melrose proves that $u$ is the unique solution with this property (given fixed $b_-$), and in fact there exist $B_\pm \in C^\infty(\overline{\bbR^d})$ with $B_-|_{\infty \bbS^{d-1}} = b_-$ such that 
\begin{equation}
    u = \frac{B_-}{\langle r \rangle^{(d-1)/2}} e^{-i \langle r \rangle} + \frac{B_+}{\langle r \rangle^{(d-1)/2}} e^{i \langle r \rangle} .
\end{equation}
The scattering matrix is $S:b_-\mapsto b_+ = B_+|_{\infty \bbS^{d-1}}$. 

Here, $\overline{\bbR^d}\hookleftarrow \bbR^d$ is the radial compactification of Euclidean space \cite{HintzMicrolocal}. We write 
\begin{equation}
  \overline{\bbR^d}=  \bbR^d\sqcup \overbrace{\infty \bbS^{d-1}}^{\partial \overline{\bbR^d}}
\end{equation}
to denote the bulk and boundary, the latter of which is the sphere at infinity. 
Thus, we have a full asymptotic expansion in powers of $1/r$, with coefficients that are smooth functions in the angular coordinate $\omega\in \bbS^{d-1}$.

\begin{figure}[!htbp]
    \centering
    \begin{tikzpicture}[scale=1.5, every node/.style={font=\small}]

  \fill[gray!20] (-3.5,-1) rectangle (-1.5,1);  
    \foreach \x in {-3.2,-2.85,-2.5,-2.15,-1.8}
    \draw[black!45]  (\x,-1) -- (\x,1);
    \foreach \x in {-.8,-.4,0,.4,.8}
    \draw[black!45]  (-3.5,\x) -- (-1.5,\x);
    
  \fill[gray!20] (0,0) circle (1);

  \foreach \x in {-0.55,-0.25,0,0.25,0.55}
    \draw[black!45] (0,1) .. controls (\x,0.55) and (\x,-0.55) .. (0,-1);

  \foreach \y in {-0.55,-0.25,0,0.25,0.55}
    \draw[black!45] (-1,0) .. controls (-0.55,\y) and (0.55,\y) .. (1,0);

  \draw[black] (0,0) circle (1);

  \node[right] at (1,0) {$\rightarrow$};
  \node at (-2.5,-1.3) {$\bbR^d$};
  \node[below] at (0,-1.1) {$\overline{\bbR^d}$};
  \node[below] at (-1.2,-1.2) {$\hookrightarrow$};
  \draw[dashed] (-4,-1.5) rectangle (1.5,1.3);
  \node at (-3.8,1.1) {(a)};
\end{tikzpicture}
\hspace{.25em} 
\begin{tikzpicture}[
  scale=1.5,
  every node/.style={font=\small},
  line cap=round,
  line join=round,
]
    
  \draw[dashed] (-1.5,-1.5) rectangle (2,1.3);
  \node at (-1.3,1.1) {(b)};
    \begin{scope}[
  x={(1cm,0cm)},
  y={(0cm,0.28cm)},   
  z={(0cm,1.15cm)},   
  shift = ({0,0,-.1})
]
  \draw[orange!90!black]
    plot[domain=90:270, samples=91, smooth, variable=\t]
      ({cos(\t)}, {sin(\t)}, {.6*cos(\t)});

  \path[fill=gray!20]
    plot[domain=0:360, samples=121, smooth, variable=\t]
      ({cos(\t)}, {sin(\t)}, 0) -- cycle;

  \foreach \xcoord in {-0.55,-0.25,0,0.25,0.55}
    \draw[black!45]
      (0,1,0)
        .. controls (\xcoord,0.55,0) and (\xcoord,-0.55,0)
        .. (0,-1,0);

  \foreach \ycoord in {-0.55,-0.25,0,0.25,0.55}
    \draw[black!45]
      (-1,0,0)
        .. controls (-0.55,\ycoord,0) and (0.55,\ycoord,0)
        .. (1,0,0);

  \draw[black]
    plot[domain=0:360, samples=121, smooth, variable=\t]
      ({cos(\t)}, {sin(\t)}, 0);

  \draw[darkred, thick]
    plot[domain=0:360, samples=121, smooth, variable=\t]
      ({cos(\t)}, {sin(\t)}, .6);

  \draw[darkred, thick, dashed]
    plot[domain=0:360, samples=121, smooth, variable=\t]
      ({cos(\t)}, {sin(\t)}, -.6);

  \draw[orange!90!black, thick]
    plot[domain=-90:90, samples=91, smooth, variable=\t]
      ({cos(\t)}, {sin(\t)}, {.6*cos(\t)});

  \node[left, darkred, dashed]  at (-.95,0,-.6) {$\calR_-$};
  \node[above right, darkred]  at (.9,0,.6) {$\calR_+$};
  \node[right, orange!90!black]  at (.8,0,.25) {$\operatorname{WF}_{\mathrm{sc}} (e^{ix})$};
  \node[right] at (1,0,0) {$\rightarrow$};
  \node[left]  at (-1,0,0) {$\leftarrow$};
    \end{scope} 
\end{tikzpicture}
    \caption{(a) The radial compactification $\bbR^d\hookrightarrow \overline{\bbR^d} $. (b) The subsets $\calR_\pm,\operatorname{WF}_{\mathrm{sc}}(e^{ix})\subset {}^{\mathrm{sc}}T^* \overline{\bbR^d}$, as they sit over the boundary. The vertical direction is the $\nu$-axis, where $\nu$ is the frequency coordinate dual to $r$. The point in $\infty\bbS^{d-1}$ in the backwards direction is labeled `$\leftarrow$,' and the point in the forward direction is labeled `$\rightarrow$.' }
    \label{fig:radial_compactification}
\end{figure}
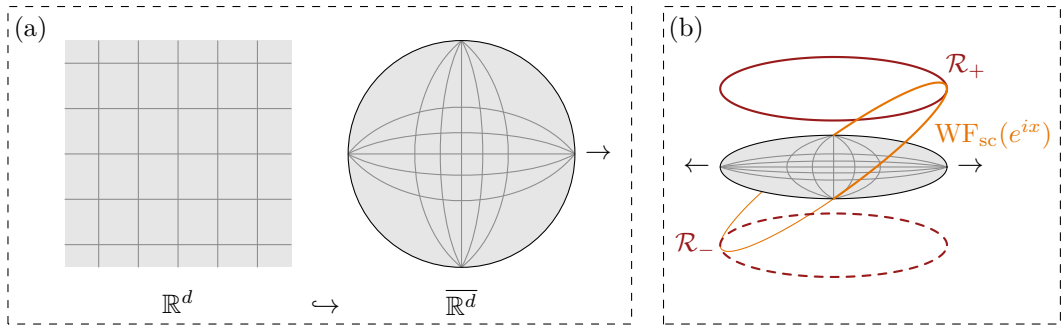

In \cite{MZ}, Melrose--Zworski investigated the structure of the S-matrix \begin{equation}
S\colon b_-\mapsto b_+ \coloneq B_+|_{\infty \bbS^{d-1}}.
\end{equation}
In the process, they provided a general treatment of perturbed plane waves. One upshot is that a perturbed plane wave $u$, incoming from the $\leftarrow\, \in \bbS^{d-1}$ direction, can be defined as 
\begin{equation}
    u \propto \Pi \delta, 
\end{equation}
where $\delta\in \calD'(\bbS^{d-1})$ is a Dirac-$\delta$ in the direction from which the perturbed plane wave is coming, and the constant of proportionality is explicit. Equivalently, 
\begin{equation}
    u \propto K|_{\{\leftarrow \} \times \bbR^d }
\end{equation}
is the restriction of the Schwartz kernel $K\in \calD'(\bbS^{d-1}\times \bbR^d)$ of the Poisson map to the submanifold $\{\leftarrow \}\times \bbR^d$. The fact that this restriction makes sense must be proven. But this definition has the advantage of making sense in full generality. The outgoing spherical wave amplitude $f(\omega)$ can then be defined as the restriction of the Schwartz kernel of the S-matrix to $\{\leftarrow\}\times \bbS^{d-1}$, minus $\delta_{\rightarrow}$. When the potential is very short-range, these definitions agree with those in the classical literature discussed above. 

Melrose and Zworski's argument proceeds in the reverse order.  
First, the perturbed plane wave is constructed; then a map $\Pi$ is defined via its Schwartz kernel, which is defined as having its cross-sections be the specified perturbed plane waves. Then it can be proven that $\Pi$ is the Poisson map. 
This requires understanding the forward singularity.
In \cite{MZ}, this is done using the novel notion of Legendrian distributions associated to a pair of intersecting Legendrian submanifolds with conic points. The meaning of this notion is best explained using Melrose's sc-calculus. 

In \cite{Me94, MelroseGeometric}, Melrose systematized the scattering- (sc-) calculus, a pseudodifferential calculus first introduced in the Euclidean case by Parenti \& Shubin. This calculus is ``purely symbolic,'' in that pseudodifferential operators are controlled by their symbols modulo compact errors. Symbols live on the compactification 
\begin{equation}
    {}^{\mathrm{sc}} \overline{T}^* \bbR^d = (\overline{\bbR^d})_{\text{base}}\times (\overline{\bbR^d})_{\text{fiber}}
\end{equation}
resulting from radially compactifying the base and the fiber individually, after trivializing $\overline{T}^* \bbR^d $ in the usual way. Associated to this calculus is a refined notion of wavefront set, 
\begin{equation}
    \operatorname{WF}_{\mathrm{sc}} : u\mapsto \operatorname{WF}_{\mathrm{sc}}(u) \subseteq \partial {}^{\mathrm{sc}} \overline{T}^* \bbR^d.
\end{equation}
In this paper, we only care about the wavefront set in ${}^{\mathrm{sc}}T^*_{\infty\bbS^{d-1}} \overline{\bbR^d}=(\infty \bbS^{d-1})\times (\bbR^d)_{\text{fiber}}$, i.e.\ over base infinity and outside of fiber infinity. There, the interpretation of the sc-wavefront set is particularly simple: the wavefront set over $\omega\in \bbS^{d-1}$ consists of the frequencies $\zeta\in \bbR^d$ of oscillations $z=(x,\bfy)\mapsto e^{i\zeta\cdot z}$ present (without superpolynomial decay) in the large-$r$ asymptotics in arbitrarily thin cones around the ray $\bbR^+ \omega$.

In Melrose's analysis, the ``radial sets'' 
\begin{equation}
    \calR_\pm = \operatorname{WF}_{\mathrm{sc}}(e^{\pm i \langle r \rangle })
\end{equation}
play a key role; $\calR_-$ is the sc-wavefront set of an incoming spherical wave, and $\calR_+$ is the sc-wavefront set of an outgoing spherical wave. These are disjoint sets, each consisting of a single point in frequency space over each point of base infinity; see \Cref{fig:radial_compactification}(b). They are a prototypical example of Legendrian submanifolds of ${}^{\mathrm{sc}} T^* \overline{\bbR^d}$. 
In the discussion of perturbed plane waves, the other relevant Legendrian submanifold is  
\begin{equation}
    \operatorname{WF}_{\mathrm{sc}}(e^{ix}) =\{ (\omega,\zeta)\in \infty \bbS^{d-1} \times \bbR^d:  \xi=1\text{ and }\eta=0\} 
\end{equation}
where $\zeta=(\xi,\eta)$, with $\xi\in \bbR$ the $x$-frequency and $\eta\in \bbR^{d-1}$ the $y$-frequency.  
This intersects the incoming radial set $\calR_-$ in exactly one point, over the backwards direction, and it intersects $\calR_+$ in exactly one other point, over the forward direction.
Intuitively, $e^{ix}$ agrees with $e^{\pm ir}$ on the backward/forward ray, respectively. 

Given the heuristic picture of a perturbed plane wave $u$ as consisting of an incoming plane wave and the resultant outgoing spherical wave, it is natural to expect 
\begin{equation}
    \operatorname{WF}_{\mathrm{sc}}(e^{ix}) \subseteq 
    \operatorname{WF}_{\mathrm{sc}}(u) \subseteq 
    \operatorname{WF}_{\mathrm{sc}}(e^{ix}) \cup \calR_+.
\end{equation}
This is in fact the case, typically with the second `$\subseteq$' being equality:
\begin{equation}
    \operatorname{WF}_{\mathrm{sc}}(u) =
    \operatorname{WF}_{\mathrm{sc}}(e^{ix}) \cup \calR_+,
\end{equation}
as can be shown in several different ways. The two halves $\operatorname{WF}_{\mathrm{sc}}(e^{ix}), \calR_+$ are each perfectly smooth Legendrian submanifolds of the sc-cotangent bundle, but their union is not, because of how they meet up in the forward direction. This is indicative of a singularity there. 

In order to produce functions with the desired sc-wavefront set, Melrose and Zworski use a complicated oscillatory ansatz. This approach has the advantage that it works in the more general setting of asymptotically conic manifolds (see \cite[\S2]{MZ} for an example rather different from the asymptotically Euclidean case), but it obscures physical-space asymptotics. The purpose of the present paper is to provide a physical-space description of perturbed plane waves in the asymptotically Euclidean setting. 
The basic idea is to resolve the singularity in the forward direction by blowing up $\rightarrow \,\in \!\infty \bbS^{d-1}$. The correct blowup ends up being parabolic. We will explain the main antecedents in the next subsection, but an analogue of the idea can already be found propounded by Melrose for the purpose of splitting Lagrangian distributions associated to intersecting pairs of Lagrangian submanifolds of the cosphere bundle $\bbS^* \bbR^d$. This involves using a ``parabolic neighborhood'' to decompose the cosphere bundle into two sets, one associated with each parabolic neighborhood \cite{Melrose1987MarkedLagrangianDistributions}\cite{Greenleaf}\cite{Joshi1994PreciseCalculus}\cite{UhlmannMarked}\cite{FeleaI, FeleaII}. These parabolic tools do not seem developed enough to be applied here. 

Before moving on, we explain how a microlocal refinement of the limiting absorption principle applies to the construction of perturbed plane waves. First, we recall that there exist refinements $\operatorname{WF}_{\mathrm{sc}}^{m,s}$ of the sc-wavefront set measuring a failure to lie in the Sobolev space $\langle r \rangle^{-s} H^m(\bbR^d)$.  
The limiting absorption principle can be stated for Sobolev spaces in which the decay rate depends on frequency, with different decay rates at the two radial sets. It turns out that a tempered distribution $f$ lies in the domain of $R(1+i0)$ if
\begin{equation}
\operatorname{WF}_{\mathrm{sc}}^{m,1/2+\varepsilon} (f) \cap \calR_- = \varnothing
\end{equation}
for some $m\in \bbR$ and $\varepsilon>0$. That is, $f$ lies ``above threshold'' at the incoming radial set $\calR_-$. The key point is that $f$ can be arbitrarily slowly decaying everywhere else in phase space. 
Because the sc-wavefront set of a function of the form $b e^{ix}$  for $b\in C^\infty(\overline{\bbR^d})$ lies in $\operatorname{WF}_{\mathrm{sc}}(e^{ix})$, the only obstruction to $be^{ix}$ lying in the domain of $R(1+i0)$ is near the backwards direction, in an arbitrarily small backwards cone. This means that we can define a Helmholtz solution by
\begin{equation}
    u =  \overbrace{e^{ix} \Big(1 + \frac{a}{r^{L-1}} \Big)}^{u_1} - R(1+i0) P u_1 ,\qquad P=(\triangle-1+V) 
\end{equation}
if $a \in C^\infty(\overline{\bbR^d})$ has been chosen such that $Pu_1$ has sufficient decay near the backwards direction. This turns out to be elementary, using the technique described below, in \S\ref{sec:bf}. This generalizes the method of Buslaev--Skriganov described above.

\subsubsection{Buslaev--Skriganov's work}
The papers \cite{buslaev1974coordinate, Sk, SkPhD} are relevant to the present paper not because of their construction of a perturbed plane wave, but because of their investigation of the forward asymptotic structure of the whole perturbed plane wave, not just the scattering amplitude. This structure was first conjectured in \cite{buslaev1974coordinate}, where the authors gave a non-rigorous argument based on a formal construction. A rigorous result was announced in \cite{Sk}, but the proof does not appear there, as ``[its] rigorous justification is lengthy.'' For that, Skriganov cites his thesis \cite{SkPhD} (which we were unable to acquire a copy of). 

Unfortunately, the statement \cite[Thm. 1]{Sk} is highly technical. 
It is of interest to us not because of its content (or even its proof), but rather because the formal construction on which the initial paper (with Buslaev) was based made use of parabolic coordinates. In that sense, the authors anticipated the existence of the parabolic wake and thus a key aspect of our formalism.

\section{Main theorem}

We now assemble the ingredients needed to state our main theorem, which makes precise the heuristic statement
given 
in the introduction.
Skip to \Cref{thm:main} in \S\ref{subsec:thm_potential} for the actual statement. 

\subsection{The compactification}

The coexistence in a perturbed plane wave of a parabolic wake and an outgoing spherical wave necessitates a multiscale analysis, which we carry out using the language of compactifications.

In the introduction, we considered the asymptotic behavior of the perturbed plane wave along different rays, and we observed that the asymptotic behavior along a forward ray was different than along any other ray (starting at the origin). More generally, we can follow the wave along any smooth curve of the form 
\begin{equation}
    \Gamma = \operatorname{graph} \gamma =\{(x,\gamma(x)):x\geq 1\} \text{ for }\gamma:[1,\infty)\to \bbR^{d-1}\text{ smooth}.
\end{equation}
Assuming that $\lim_{x\to\infty}\gamma'(x)$ exists, then $\Gamma$ tends to a point on the boundary $\infty \bbS^{d-1} \subset \overline{\bbR^d}$ of the radial compactification.
This probes the structure of scattering in the direction 
\begin{equation} 
	\omega= \Big(\frac{1}{\sqrt{1+|\bfv|^2} },\frac{\bfv}{\sqrt{1+|\bfv|^2}}\Big),\quad \bfv = \lim_{x\to\infty} \gamma'(x).
\end{equation} 
In particular, if $\bfv=\bf0$, then this probes the asymptotic structure of forward scattering.
For $\omega\in \bbS^{d-1}$ in the open forward hemisphere, the ray $\bbR^+\omega$ has the form $\Gamma = \operatorname{graph} \gamma$ for 
\begin{equation} 
    \gamma(x) = \bfv x,
\end{equation} 
where $\bfv$ is the $\bfy$-component of $\omega$ divided by the $x$-component. 
Consider instead the path $\Gamma=\operatorname{graph} \gamma$ for 
\begin{equation} 
\gamma(x) = \bfc x^{1/2}
\end{equation} 
for some $\bfc\in \bbR^{d-1}$. This is one-half of a forward parabola. It turns out that the asymptotic behavior previously observed along the forward ray applies to all forward parabolas (with different constants in the big-$O$). This is why the wake left by a passing wave is described as \emph{parabolic}.

We could state two separate theorems, one about the asymptotic behavior observed along rays and another along forward parabolas. Besides being unsatisfying, this piecemeal treatment would fail to answer what happens for intermediate curves, like 
\begin{equation}\label{eq:intermediate_curve}
    \gamma(x) = \bfc x^{2/3}. 
\end{equation}
Compactification via manifolds-with-corners (mwc) provides a language for keeping track of all possible asymptotic regimes at once, as well as certifying a set of asymptotic behaviors as complete --- equivalently of certifying that a set of asymptotic expansions fit together consistently --- as described in the next subsection. 
To capture the specific multiscale behavior described above, we use a compactification $X\hookleftarrow \bbR^d$ refining the radial compactification, namely 
\begin{equation}
	X = [ \overline{\bbR^d}; \rightarrow ]_{\mathrm{par}}, \label{eq:1205}
\end{equation}
which means the result of performing a parabolic blowup of the forward direction $\rightarrow \in \infty \bbS^{d-1} = \partial \overline{\bbR^d}$. 
See \Cref{fig:X}.
A more explicit definition is by means of a complete atlas of coordinate charts, which show how to identify subsets of $X$ with subsets of $[0,\infty)^2\times \bbR^{d-2}$. Away from the forward direction, the atlas is the same as for the radial compactification. Coordinate charts near the forward direction are specified below. These can be derived from the geometric definition of $X$, or they can be taken as the definition of $X$ (once it is checked that they are compatible).

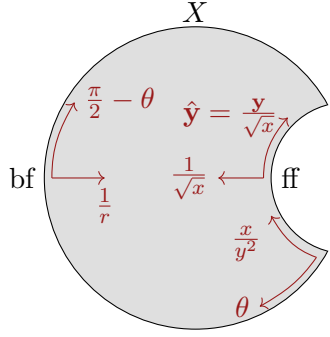
\begin{figure}[!htbp]
	\begin{center}
		\begin{tikzpicture}
			\begin{scope} 
				\filldraw[fill=lightgray!50] (0,0) circle (2);
				\begin{scope}
					\clip (0,0) circle (2.025);
					\filldraw[fill=white] (2,0) circle (1);
				\end{scope}
				\fill[white] (1.85,1) circle (.1);
				\fill[white] (1.85,-1) circle (.1);
			\end{scope} 
			\node () at (0,2.2) {$X$};
			\draw[darkred,->] (.9,0) to[out=90, in=230] (1.22,.8) node[left] {$\hat{\bfy}=\frac{\bfy}{\sqrt{x}}$};
			\draw[darkred,->] (.9,0) -- (.3,0) node[left] {$\frac{1}{\sqrt{x}}$};
			\draw[darkred,->] (-1.9,0) -- (-1.2,0) node[below] {$\frac{1}{r}$};
			\draw[darkred,->] (-1.9,0) to[out=90,in=240] (-1.6,1) node[right] {$\frac{\pi}{2}-\theta$};
			\draw[darkred,->] (1.6,-1.05) to[out=235,in=30] (.85,-1.7) node[left] {$\theta$};
			\draw[darkred,->] (1.6,-1.05) to[out=160, in=-60] (1,-.5) node[below left] {$\frac{x}{y^2}$};
			\node () at (1.25,0) {ff};
			\node () at (-2.3,0) {bf};
		\end{tikzpicture}
	\end{center}
	\caption{The compactification $X$ with an atlas of coordinate charts depicted. 
		Here, $\theta=\operatorname{arctan}(y/x)$ when $x>0$. Not depicted is the azimuthal coordinate $\bfy / y \in \bbS^{d-2}$. Recall that $y=|\bfy|$.}
	\label{fig:X}
\end{figure}

Note that $X$ has two different boundary hypersurfaces, each corresponding to an asymptotic regime:
\begin{itemize}
	\item bf, for ``boundary face,'' consisting of all of the non-forward points $\infty \omega \in \infty \bbS^{d-1}$, 
	\item the \emph{forward face} $\mathrm{ff}$, which can be considered to consist of all of the endpoints ``at infinity'' of the various forward parabolas $\{(x,\bfc x^{1/2}) :x\ge 0\}$, $\bfc\in \bbR^{d-1}$. 
\end{itemize}

An especially useful coordinate is 
\begin{equation}
        \hat{\bfy} \coloneqq  \frac{\bfy}{x^{1/2}} .
\end{equation}
(Warning: the hat does not mean a unit vector.)
The map $\bbR^+\times \bbR^{d-1}\ni (x,\bfy) \mapsto (1/\sqrt{x},\hat{\bfy})$ extends to a diffeomorphism 
\begin{equation}
    X\supset \{ |\theta| < \pi/2 \}\backslash \mathrm{bf} \to [0,\infty)_{1/\sqrt{x}} \times \bbR^{d-1}_{\hat{\bfy}}.
\end{equation}
So, away from $\mathrm{bf}$, the function $\rho_{\mathrm{ff}}=1/\sqrt{x}$ serves as a boundary-defining-function for $\mathrm{ff}$. The coordinate $\hat{\bfy}$ parametrizes $\mathrm{ff}^\circ$.

Near the corner $\mathrm{bf}\cap \mathrm{ff}$, we use the coordinates 
\begin{equation}\label{eq:corner_coords}
    \rho_{\mathrm{bf}} = \frac{1}{\hat{y}^2},\quad \rho_{\mathrm{ff}} = |\theta| ,\quad \phi=\frac{\bfy}{y}.
\end{equation}
These coordinates yield a diffeomorphism 
\begin{equation}
    X \supset \{ |\theta| < \pi/2\} \backslash \operatorname{cl}_X \{r=x\} \to [0,\infty)_{\rho_{\mathrm{bf}}} \times [0,\pi/2)_{\rho_{\mathrm{ff}}} \times \bbS^{d-2}_\phi.    
\end{equation}
From this, we can identify 
\begin{equation} 
    \mathrm{ff} = (\overline{\bbR^{d-1}_{\hat{\bfy}}})_2,
\end{equation} 
where the subscript `$2$' denotes a  change of smooth structure at the boundary making $\rho_{\mathrm{bf}}=1/\hat{y}^2$ a boundary-defining-function instead of $1/\hat{y}$. It is this choice of smooth structure that is compatible with the smooth structure of the radial compactification, where $1/r$ is a boundary-defining-function. 
(Instead of $|\theta|$, we could use $y/x$ as a local boundary-defining-function of $\mathrm{ff}$.)

A projective variant of the coordinates \cref{eq:corner_coords} is useful for explicit computations. Specifically, in the region $\hat{y}_1>0$ (the other sectors are similar),  
\begin{equation}\label{eq:corner_coords_0}
        \rho_{\mathrm{bf}} = \frac{x}{y_1^2}, \quad \rho_{\mathrm{ff}} = \frac{y_1}{x},\quad w_k = \frac{y_k}{y_1}\text{ for }k\geq 2
\end{equation}
    serve as a valid system of coordinates.

One empirical motivation for working on this compactification is the apparently parabolic nature of the wake seen in concrete examples. A more conceptual perspective is the following: the ratio 
\begin{equation}
    e^{ix}/e^{ir} = e^{i(x-r)} 
\end{equation}
measuring the difference between the oscillations characterizing a plane wave and characterizing a spherical wave depends on the difference $r-x = r(1-\cos \theta)$. 
Away from the forward direction, this is related to $r$ by a non-vanishing proportionality factor, but its structure near the forward direction is more interesting: 
\begin{equation}
    r-x = \hat{y}^2 \aleph(\theta),\quad  \aleph(\theta) = \frac{\cos \theta}{1+\cos \theta} \in C^\infty ( (-\pi,\pi)_\theta ).
\end{equation}
Thus, we could equally well consider $r-x$ as a radial coordinate of $\mathrm{ff}^\circ \backslash \mathrm{fd}$, where $\mathrm{fd}\in \mathrm{ff}$ is the midpoint of $\mathrm{ff}$ (where $\theta,\hat{y}$ cease to be smooth coordinates). This means that $e^{i(x-r)}$ is \emph{non-oscillatory} at $\mathrm{ff}$, while oscillating upon approach to $\mathrm{bf}$. Put more intuitively, the oscillations $e^{ix},e^{ir}$ differ by a smooth multiple near any point in $\mathrm{ff}^\circ$, and the parabolic blowup of $\rightarrow \in \infty \bbS^{d-1}$ is the natural one which achieves that aim.

\subsection{Function spaces}\label{s:function_spaces}
To be smooth on a mwc means to extend smoothly to a manifold in which the mwc is embedded.\footnote{This does not depend on the choice of embedding. Throughout this paper, we will conflate smooth functions on a mwc with their restrictions to the interior.}
If 
\begin{equation}
    w\in C^\infty(X) \subsetneq  C^\infty(\bbR^d), 
\end{equation}
then, by Taylor's theorem, $w$ can be expanded in powers of a boundary-defining-function $\rho_{\mathrm{f}}$ of either $\mathrm{f}=\mathrm{bf},\mathrm{ff}$: 
\begin{equation}
     w \sim \sum_{j=0}^\infty w_{j;\mathrm{f}} \rho_{\mathrm{f}}^j,\quad w_{j;\mathrm{f}}\in C^\infty(\mathrm{f}), 
\end{equation}
which, as usual, means that 
\begin{equation}
    \text{for all }K\in \bbN, w-\sum_{j=0}^K w_{j;\mathrm{f}} \rho_{\mathrm{f}}^j \in \rho_{\mathrm{f}}^{K+1}C^\infty(X).
\end{equation}
We have one expansion at $\mathrm{bf}$ and one at $\mathrm{ff}$. Smoothness at the corner means that the two expansions are compatible, and there exists a joint Taylor series in both defining functions. 

To make contact with the discussion in the previous subsection, consider a smooth curve $\Gamma:\bbR^+\to X^\circ$ such that $\Gamma(\infty)=\lim_{t\to\infty} \Gamma(t)$ exists as a point in $\partial X$. Then, 
\begin{equation} 
    \lim_{t\to\infty} w(\Gamma(t))
\end{equation} 
exists and depends only on $\Gamma(\infty)$. 
Thus, $\partial X$ serves as a mnemonic for the different possible limits. Applied to the case at hand: we have one limit for each non-forward outgoing direction, but in the forward direction we have a whole continuum of possible limits, parametrized by the points in $\mathrm{ff}^\circ$. Note how this framework addresses the problem of intermediate curves. Any intermediate curve hits the corner $\mathrm{bf}\cap \mathrm{ff}$, so the limit depends only on the azimuthal coordinate $\bfy/|\bfy| \in \bbS^{d-2}$.

As already mentioned, we need to make use of Melrose's notion of polyhomogeneity. Recall that an index set $\calE$ is a subset $\calE\subset \bbC\times \bbN$ such that 
\begin{enumerate}[label=(\roman*)] 
    \item for any $\alpha\in \bbR$, only finitely many $(j,k)\in \calE$ have $\Re j<\alpha$, 
    \item $(j,k)\in \calE\Longrightarrow (j+1,k)\in \calE$, 
    \item $(j,k+1)\in \calE \Longrightarrow (j,k)\in \calE$ for all $k\in \bbN$. 
\end{enumerate}
For any two index sets $\calE,\calF$, we have a function space 
\begin{equation}
    \calA^{\calE,\calF}(X)\subset C^\infty(\bbR^d) 
\end{equation}
consisting of polyhomogeneous functions on $X$ with index set $\calE$ at $\mathrm{bf}$ and $\calF$ at $\mathrm{ff}$. Thus, we have two asymptotic expansions, one at $\mathrm{bf}$ and one at $\mathrm{ff}$:
\begin{align}
\begin{split} 
    \exists a_{j,k} \in \calA^\calF(\mathrm{bf})\text{ s.t. } &w\sim \sum_{(j,k)\in \calE} a_{j,k} \rho_{\mathrm{bf}}^j (\log \rho_{\mathrm{bf}})^k , \\
    \exists b_{j,k} \in \calA^\calE(\mathrm{ff})\text{ s.t. } & w\sim \sum_{(j,k)\in \calF} b_{j,k} \rho_{\mathrm{ff}}^j (\log \rho_{\mathrm{ff}})^k, 
\end{split} 
\end{align}
and the two are compatible, so 
\begin{align}
\begin{split} 
    \exists a_{j,k,j',k'}=b_{j',k',j,k} \in C^\infty(\bbS^{d-2}_\phi)\text{ s.t. } &a_{j,k} \sim \sum_{(j',k')\in \calF} a_{j,k,j',k'}(\phi) \rho_{\mathrm{ff}}^{j'} (\log \rho_{\mathrm{ff}})^{k'}, \\ 
    &b_{j',k'} \sim \sum_{(j,k)\in \calE} b_{j',k',j,k}(\phi) \rho_{\mathrm{bf}}^j (\log \rho_{\mathrm{bf}})^k.
    \end{split} 
\end{align}

The precise meaning of `$\sim$' is the usual one, with respect to the conormal function spaces 
\begin{equation} \label{eq:1349}
    \calA^{\alpha,\beta}(X) = \rho_{\mathrm{bf}}^\alpha \rho_{\mathrm{ff}}^\beta \calA^{0,0}(X),
\end{equation} 
\begin{equation}\label{eq:1352}
    \calA^{0,0}(X) = \{g\in C^\infty(\bbR^d): \operatorname{Diff}_{\mathrm{b}}(X) g\in L^\infty \}, 
\end{equation}
where 
$\operatorname{Diff}_{\mathrm{b}}(X)$  is the algebra of differential operators generated over $C^\infty(X)$ by the smooth vector fields on $X$ tangent to both $\mathrm{bf},\mathrm{ff}$.

We refer to \cite{Me93} for a fuller introduction to polyhomogeneity and conormality. 

\subsection{Theorem statement}
\label{subsec:thm_potential}
Our main theorem, as it applies to potential scattering, states the following.  Let 
\begin{equation} 
V \in \langle r \rangle^{-L}C^{\infty}(\overline{\bbR^d}\backslash \{0\};\bbR)
\end{equation} 
denote a short range potential, with decay rate $L\in \bbN^{\geq 2}$. 
Note that $V$ is required to be smooth on the radial compactification (which is what it means for a potential to be classical), except possibly at the origin. We assume the following about the singularity at the origin:
\begin{equation}
    \exists \alpha> - \frac{(d-2)^2}{4}\text{ and } \mathsf{Z}\in \bbR \text{ s.t. }V-\frac{\alpha}{r^2} + \frac{\mathsf{Z}}{r \langle r \rangle} \in \langle r \rangle^{-2}C^{\infty}(\overline{\bbR^d};\bbR),
\end{equation}
so that the perturbed plane wave  $u \in C^{\infty}(\bbR^d\backslash \{0\})$ corresponding to $V$ is well-defined, where we use the Friedrichs extension for $\bigtriangleup+V-1$ (see \Cref{examp:singularities} for details on the functional analysis).
Then:
\begin{theorem*}Away from the origin, the perturbed plane wave has the form  
    \begin{equation}
        u(x,\bfy) = e^{ix} \Big( 1 + \frac{a}{r^{L-1}} \Big) + e^{ir} \frac{w}{r^{(d-1)/2}}
    \end{equation}
    for some $a,w \in C^\infty(\bbR^{d})$ polyhomogeneous on $X$ and smooth at the boundary hypersurface $\mathrm{bf}\subset X$. More precisely, 
    \begin{equation}
        a \in \rho_{\mathrm{ff}}^{-(L-1)} C^\infty(X),\quad w\in  \calA^{(0,0),\calF}(X)  
    \end{equation}
    for some index set $\calF\subset (\bbN\times \{0\})\cup (\bbZ^{\geq L-d}\times \bbN )$.   
\end{theorem*}
\begin{remark*}
    The scattering amplitude $f:\bbS^{d-1}\backslash\{\rightarrow\}\to \bbC$ is the restriction $w|_{\infty \bbS^{d-1}}$ of $w$ to the sphere at infinity. This makes literal sense away from the forward direction. We will see below, based on the particular index set $\calF$, that \begin{equation} 
    w|_{\infty \bbS^{d-1}\backslash \{\rightarrow\}}\in L^1(\infty \bbS^{d-1}),
    \end{equation} 
    so $w|_{\infty \bbS^{d-1}}$ is well-defined as a distribution.
\end{remark*}
    
    Here, $(0,0)=\bbN\times \{0\}$ is (standard) shorthand for the index set describing smoothness. 
    Thus, $u$ has the form claimed in the introduction:
    \begin{equation}
        u = e^{ix}  + \underbrace{\frac{ e^{ix} a}{r^{L-1}}}_{\mathclap{\text{plane wave correction}}}+ \overbrace{ \frac{ e^{ir}w}{r^{(d-1)/2}}}^{\mathclap{\text{spherical wave}}},
    \end{equation}
    with the plane wave corrections and parabolic wake (split among $a,w$) of the claimed sizes.

    Because $e^{i(x-r)}$ is smooth at $\mathrm{ff}\backslash \mathrm{bf}$, we can trade elements of $C_{\mathrm{c}}^\infty(X\backslash \mathrm{bf})$ between the plane wave and spherical wave parts of $u$. The two parts are therefore indistinguishable away from $\mathrm{bf}$. On the other hand, the germs of $a,w$ at $\mathrm{bf}$ can be shown to be unique.

\begin{remark*}
    We assume that $V$ is smooth at $\infty \bbS^{d-1}$, but one can allow controlled polyhomogeneity instead. Then, the corrections above would be merely polyhomogeneous at $\mathrm{bf}$, instead of smooth.  
\end{remark*}

\subsection{Example: inverse-square potential, cont.}\label{subsec:inverse_square_example}
In order to illustrate the main theorem, we return to the example of the inverse-square plane wave $u\in C^\infty(\bbR^3\setminus \{0\})$ (\cref{eq:special_dipole_plane_wave}).
Our starting point is the following: for any $r_0>0$, in the region $\{r>r_0\}$, the approximation  
\begin{equation}\label{eq:dipole_good}
        u = e^{ix}  +\frac{e^{ir}}{\sqrt{r}} a_0 e^{-iz} J_0(z)  + \frac{e^{ir}}{r} a_1  \log \bigg(\theta^2 + \frac{1}{r} \bigg) + O \Big(\frac{1}{r} \Big),\qquad z=\frac{r-x}{2}
\end{equation}
holds, 
where 
\begin{equation}\label{eq:dipole_constants}
    \quad a_0 = - \alpha e^{\frac{i\pi}{4}}  \Big(\frac{\pi}{2} \Big)^{\frac{3}{2}},\quad a_1 = -\frac{i\pi^2\alpha^2}{16}
\end{equation}
are some numerical constants. See \S\ref{sec:dipole_some_more} for a derivation.

The constant in the big-$O$ is uniform, meaning that it does not depend on $\theta \in [0,\pi]$. In particular, it holds uniformly near the forward direction $\theta=0$. As noted below, all of the terms on the right-hand side of \cref{eq:dipole_good}, except $e^{ix}$, are $O(1/r)$ outside of any small sector 
\begin{equation}
\{ |\theta|<\varepsilon, r>r_0>0 \},\quad \varepsilon>0
\end{equation}
around the forward direction. But they are \emph{not} $O(1/r)$ globally in any angular sector containing the forward direction; the first correction term is only $O(1/\sqrt{r})$ globally, and the second is only $O( (\log r)/r)$:
\begin{equation}
     u = \underbrace{e^{ix}}_{O(1)}  +\underbrace{\frac{e^{ir}}{\sqrt{r}} a_0 e^{-iz} J_0(z)}_{u^{[1]}=O(1/\sqrt{r})}  + \underbrace{\frac{e^{ir}}{r} a_1  \log \bigg(\theta^2 + \frac{1}{r} \bigg)}_{u^{[2,0]}=O((\log r)/r)}  +  O\Big(\frac{1}{r} \Big).
\end{equation}
This substantiates our claims about this example made in the introduction.

\begin{figure}[!htbp]
    \centering
    \includegraphics[scale=1.1]{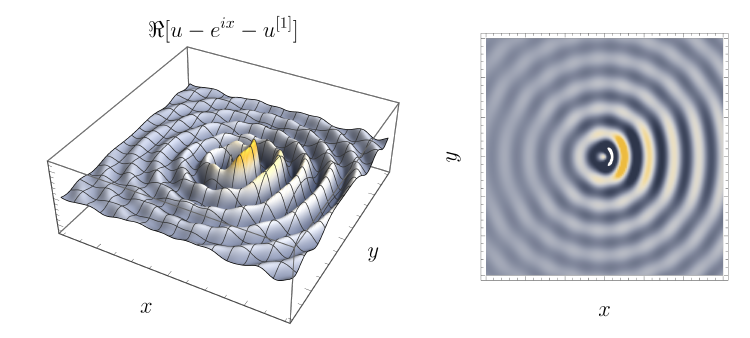}
    \caption{The (real part of the) difference $u-e^{ix}-u^{[1]}$, where $u$ is the inverse-square plane wave, and $u^{[1]}$ is \cref{eq:u_brak1}. The vertical scale is the same as in \Cref{fig:dipole_plane_wave}, \Cref{fig:dipole_plane_wave_dif}.}
    \label{fig:dipole_plane_wave_res}
\end{figure}

Notice that the primary correction 
\begin{equation}\label{eq:u_brak1}
    u^{[1]} = \frac{e^{ir}}{\sqrt{r}}  a_0 e^{-iz} J_0(z) 
\end{equation}
depends prominently on the difference $r-x=2z$.
Since this coordinate parametrizes the interior of the front face $\mathrm{ff}\subset X$, $u^{[1]}$ is polyhomogeneous on $X\backslash \mathrm{bf}$ with one order of decay (meaning $O(1/\sqrt{r})$) at $\mathrm{ff}^\circ$. This is why the forward wake left by the wave spreads out parabolically as it propagates.
To analyze the situation at the remaining face $\mathrm{bf}\subset X$, we use the large-argument approximation of the Bessel function:
\begin{equation}
    u^{[1]} \approx  - \alpha \Big(\frac{\pi}{2} \Big)^{\frac{3}{2}} \frac{1 }{r} \frac{e^{ir} + ie^{ix} }{\sqrt{ \pi (1-\cos \theta ) }}  = - \frac{\pi \alpha}{4} \frac{1}{r} \frac{e^{ir} + ie^{ix} }{ \sin(\theta/2) } \text{ when }z\gg 1.
\end{equation}
Thus, $u^{[1]}$ consists of an outgoing spherical wave and a correction to the incoming plane wave. These are both $O(1/r)$ with a singular amplitude $\propto (1-\cos \theta)^{-1/2}$. 
The large-argument approximation of the Bessel function is the just the leading term in the Hankel expansion. The Hankel expansion (which can be differentiated term-by-term) gives a (noncanonical) splitting 
\begin{equation} 
    J_0(z) = \frac{1}{\langle z \rangle^{1/2}} (e^{iz}f_1+e^{-iz}f_2)
\end{equation} 
for $f_1,f_2\in C^\infty([0,\infty]_{z})$. (This splitting is noncanonical because we can trade an arbitrary Schwartz function between the two $f_j$'s.) 
Thus, after redefining $f_1$ and $f_2$,
\begin{equation}
    u^{[1]}= \frac{e^{ir}}{r} \underbrace{\frac{f_1}{\sqrt{1-\cos \theta }}}_{w^{[1]} } + \frac{e^{ix}}{r} \underbrace{\frac{f_2}{\sqrt{1-\cos \theta} }}_{a^{[1]} }   = \frac{e^{ix}}{r} a^{[1]} + \frac{e^{ir}}{r} w^{[1]}  
\end{equation}
for $a^{[1]},w^{[1]}$ defined as indicated. 
These are polyhomogeneous functions on $X$. Specifically, $a^{[1]}, w^{[1]} \in \rho_{\mathrm{ff}}^{-1} C^\infty(X)$.
Note the singularity in the forward direction.
Thus, the correction $u^{[1]}$ is exactly of the form captured by our main theorem. 

Regarding the correction 
\begin{equation}\label{eq:u_brak20}
    u^{[2,0]}=\frac{e^{ir}}{r} a_1  \log \bigg(\underbrace{\theta^2 + \frac{1}{r} }_{\sim \rho_{\mathrm{ff}}^2}\bigg), 
\end{equation}
we remark only that it contributes a logarithmic term at the front face $\mathrm{ff}$, reiterating the need for working with the notion of polyhomogeneity.

In order to numerically test the form of $u^{[2,0]}$, it is prudent to use the following refinement:
\begin{equation}
    u^{[2]} = \frac{e^{ir}}{r} \bigg[ R(\theta) + \frac{i\pi^2\alpha^2}{8} \bigg( K \bigg( \cos \frac{\theta}{2} \bigg) -\frac{1}{2} E_1(i(r-x)) \bigg) \bigg],
\end{equation}
where 
\begin{itemize}
    \item $R(\theta)\in C^0(\bbS^1_\theta)$ is defined by 
    \begin{equation}
        R(\theta) = \sum_{\ell=0}^\infty \bigg[a_\ell +\frac{\pi \alpha}{2} - \frac{i\pi^2\alpha^2}{8(\ell+1/2)} \bigg] P_\ell(\cos \theta) ,
    \end{equation}
    for 
    \begin{equation}
         a_\ell = (2\ell+1) e^{i\delta_\ell} \sin \delta_\ell,\quad \delta_\ell = \frac{\pi}{2} \bigg( \ell+\frac{1}{2} - \sqrt{\alpha+\Big(\ell+\frac{1}{2}\Big)^2 }\, \bigg) 
    \end{equation}
    \item 
        $K$ is the complete elliptic integral of the first kind:
        \begin{equation}\label{eq:K}
            K (s) =\int_0^{\pi/2} \frac{\dd \phi}{\sqrt{1-s^2 (\sin \phi)^2}}
        \end{equation}
        \item and 
        \begin{equation}
            E_1(z) = \int_z^\infty \frac{e^{-t}}{t} \dd t 
        \end{equation} 
        is the exponential integral, whose principal branch is defined on $\bbC \setminus (-\infty,0]$.
\end{itemize}
The difference $u^{[2]}-u^{[2,0]}$ is uniformly $O(1/r)$ as $r\to\infty$. Indeed, $R(\theta)$ is a continuous function of $\theta$, and: 

\begin{proposition}
    The quantity $\Lambda = K (\cos (\theta/2)) -\frac{1}{2} E_1(i(r-x))$ is given by 
    \begin{equation}
         \Lambda = - \frac{1}{2} \log \bigg( \theta^2+\frac{1}{r} \bigg) + O(1) \text{ in }r>1.
    \end{equation}
\end{proposition}
\begin{proof}
    First, note that 
    \begin{equation} 
        K \Big( \cos \frac{\theta}{2} \Big) = -\log \theta + O(1) 
    \end{equation} 
    as $\theta\to 0$,  while, for $r-x>0$, 
    \begin{equation*}
        E_1(i(r-x))=\int_{i(r-x)}^{i} \frac{e^{-t}}{t} \dd t +  O(1) = \int_{r-x}^1 \frac{e^{-it}}{ t} \dd t+O(1)= \begin{cases} -\log(r-x)+O(1) & ( r-x<1),\\
        O(1) &(r-x\geq 1).\end{cases}
    \end{equation*}
    So, away from the ray $\operatorname{cl}_X\{r=x\}$ (hitting the front face $\mathrm{ff}$ at its midpoint), we have $\Lambda = - \log \theta+O(1)$. For comparison, $-2^{-1} \log(\theta^2+1/r) = - \log \theta + O(1)$ there, since 
    \begin{equation} 
    1/r \in \theta^2 C^\infty(X\backslash \mathrm{cl}_X\{r=x\}).
    \end{equation} 
    The two agree.  
    On the other hand, away from $\mathrm{bf}$, we have 
    \begin{equation}
        \Lambda = - \log \theta  + \frac{1}{2} \log (r-x) + O(1) = \frac{1}{2} \log r +   \underbrace{\frac{1}{2} \log \Big( \frac{1-\cos \theta}{\theta^2} \Big) }_{=O(1)}+ O(1), 
    \end{equation}
    whereas, since $\theta^2 \in r^{-1} C^\infty(X\backslash \mathrm{bf})$,
    \begin{equation}
        -\frac{1}{2} \log \Big(\theta^2 +\frac{1}{r} \Big) = \frac{1}{2} \log r+ O(1).
    \end{equation}
    Thus, we have agreement here as well.
\end{proof}

Although $u^{[2]}-u^{[2,0]}=O(1/r)$ as $r\to\infty$, the constant is big, so the difference is numerically significant for our plots.  
A plot of 
\begin{equation} 
u-e^{ix}-u^{[1]}- u^{[2]} = O(1/r)
\end{equation} 
is shown in \Cref{fig:dipole_plane_wave_Out}. 

\begin{figure}[!htbp]
    \centering
    \includegraphics[scale=1.1]{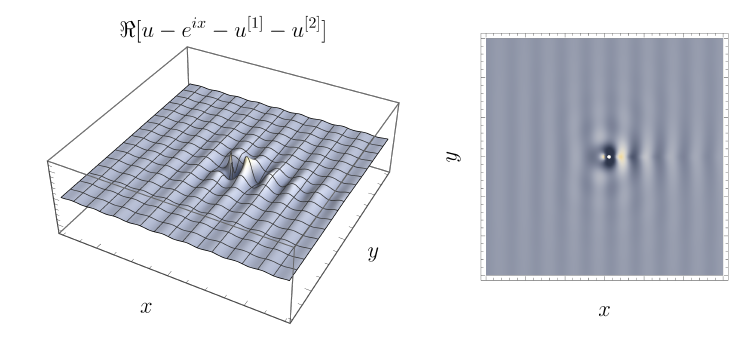}
    \caption{The (real part of the) difference $u-e^{ix}-u^{[1]}-u^{[2]}$, for the example above.}
    \label{fig:dipole_plane_wave_Out}
\end{figure}

\subsection{Asymptotically Euclidean data}
We now state our theorem for general asymptotically Euclidean settings.
\label{subsec:asymptotically_Euclidean}

Consider $P\in \operatorname{Diff}^2(\bbR^d)$ given by 
\begin{align}\label{eq:P}
\begin{split} 
    P &= P_0 + Q,\qquad P_0=\triangle-1, \\
    Q&= \sum_{i,j=1}^d \partial_i ( a^{ij} \partial_j)  + i\sum_{j=1}^d (b^j \partial_j + \partial_j b^j ) + V 
\end{split} 
\end{align}
where $\triangle = -\partial_x^2+\triangle_{\bfy}$ is the positive semidefinite Laplacian, as usual, and 
\begin{equation}
a^{ij} ,b^j, V \in \langle r \rangle^{-2} C^\infty(\overline{\bbR^d};\bbR)\quad \text{such that}\quad a^{ij}=a^{ji}
\end{equation}
are classical short-range perturbations of the coefficients. 
In addition, assume that the form $g^{ij}=\{\delta^{ij} - a^{ij}\}_{i,j=1}^d$ is positive definite: 
\begin{equation}
    |v|^2 - \sum_{i,j=1}^d v_i a^{ij} v_j >0\text{ for all nonzero }v\in \bbR^d,  
\end{equation}
so that $P$ is elliptic in the ordinary sense. This is equivalent to saying that $P$ is a short-range perturbation of $\bigtriangleup_g-1$ for some asymptotically Euclidean metric $g$.

Let $L\in \bbN^{\geq 2}$ be such that $Q$ has the form above at large-$r$, with $L-2$ extra orders of decay:
\begin{equation}\label{eq:L_def}
a^{ij} ,b^j, V \in \langle r \rangle^{-L} C^\infty(\overline{\bbR^d};\bbR).
\end{equation}
\begin{theorem} 
    \label{thm:main}
Away from the origin, the perturbed plane wave has the form  
    \begin{equation}
        u(x,\bfy) = e^{ix} \Big( 1 + \frac{a}{r^{L-1}} \Big) + e^{ir} \frac{w}{r^{(d-1)/2}}
    \end{equation}
    for some $a,w \in C^\infty(\bbR^{d})$ polyhomogeneous on $X$ and smooth at the boundary hypersurface $\mathrm{bf}\subset X$. More precisely, 
    \begin{equation}
        a \in \rho_{\mathrm{ff}}^{-(L-1)} C^\infty(X),\quad w\in  \calA^{(0,0),\calF}(X)  
    \end{equation}
    for some index set $\calF\subset (\bbN\times \{0\})\cup (\bbZ^{\geq L-d}\times \bbN )$.  
\end{theorem}

The proof is spread over the rest of the body of this paper.

\begin{remark*}
    The potential scattering version of our theorem allowed a mild singularity at the origin and is therefore technically not covered by the version of the theorem just stated.  In \S\ref{sec:black-box}, we will state a version (\Cref{thm:main_black-box}) of our main theorem allowing for general sorts of singularities, obstacles, etc.
\end{remark*}

\section{Outline of proof}

\label{sec:outline}
Let us restate our goal. 
Consider the Helmholtz--Schr\"odinger operator 
\begin{equation}\label{eq:general_P}
    P=\triangle -1+Q,\text{ where }
    Q = \sum_{i,j=1}^d \partial_{i}( a^{ij} \partial_j) + i\sum_{j=1}^d (b^j \partial_j +\partial_j b^j) + V 
\end{equation}
for $a^{ij},b^j, V\in \langle r \rangle^{-L} C^\infty(\overline{\bbR^d})$, $L\in \bbN^{\geq 2}$.
Our goal is to produce a solution $u$ of the Helmholtz--Schr\"odinger equation $Pu=0$,
possibly satisfying some black-box admissibility criterion, of the form 
\begin{equation}
   u = e^{i  x} \Big(1+\frac{a}{r^{L-1}}\Big) + e^{ir} \frac{1}{r^{(d-1)/2}} w, 
   \label{eq:form}
\end{equation}
for $a,w$ polyhomogeneous functions on $X$ with some sufficiently mild singularity in the forward direction. 
First, we algorithmically construct a quasimode that satisfies the desired equation modulo a Schwartz error. This is defined on $\bbR^d$, outside of some large ball.
Second, we solve away the remaining Schwartz error using the limiting resolvent $R(1+i0)$, leveraging our knowledge about its mapping properties when applied to Schwartz functions.

\subsection*{Structure of the algorithm} 
We divide the steps of the algorithm into \emph{main} steps and \emph{sub}steps. 

The algorithm begins with the crude approximation 
\begin{equation}
    u_0  =  e^{i x}
\end{equation}
to $u$. Of course, this will not solve 
the PDE exactly. It does solve $P_0 u_0=0$, where $P_0=\triangle-1$ is the free Helmholtz operator. Because $P$ differs from $P_0$ by decaying terms, $u_0$ is an \emph{approximate} solution (quasimode) of $Pu=0$. 
The goal of each step in the algorithm is to improve our approximate solution to a better approximate solution (by adding small terms), until we arrive at a quasimode with a Schwartz error, at which point we solve that Schwartz error away via the limiting absorption principle. 
The algorithm consists of three main steps, each of which consists of infinitely many substeps, which are followed by a Borel-type construction to tie the infinitely many sub-steps together. 
After the third step, we apply the limiting absorption principle to the Schwartz remainder.

We will denote the output of the $k$th little step within the $j$th main step as $u_{j,k}$. Thus, 
\begin{equation}
\boxed{ 
     u \coloneq u_{4}
     }
\end{equation}
will be the final output of the algorithm. We use $u_{j}$ to denote the final output of the $j$th main step, for $j=1,2$. 
Thus, we have an ordinal sequence 
\begin{align}
\begin{split} 
     u_0=&u_{1,0}, \\
     &u_{1,1},\\ 
     &\;\; \vdots\\ &u_{1}=u_{2,0}, \\
    &\hspace{3.3em} u_{2,1},\\
    &\hspace{4em} \vdots\\
    &\hspace{3.3em} u_{2}=u_{3,0},\\ 
    &\hspace{6.6em} u_{3,1}, \\
    &\hspace{7.3em} \vdots\\
    &\hspace{7em} u_{3} ,
    \end{split} 
\end{align}
of better approximate solutions of the PDE. 
For each $j=1,2,3$ and substep label $k\in \bbN^+$, let
\begin{equation}
    f_{j,k} \coloneq -Pu_{j,k-1}
\end{equation}
denote the error by which $u_{j,k-1}$ fails to solve the Helmholtz--Schr\"odinger equation. Then, we will construct 
\begin{equation}
    u_{j,k} \coloneq u_{j,k-1} + v_{j,k} 
\end{equation}
by adding to $u_{j,k-1}$ a correction $v_{j,k}$, which will be chosen such that 
\begin{equation} \label{eq:approx_goal}
P v_{j,k}\approx  f_{j,k},
\end{equation} 
where the ``$\approx$'' means modulo terms which are considered negligible at this point in the algorithm, usually in the sense of decay rate, either at bf or ff.
Given this choice of $v_{j,k}$, our $u_{j,k}$ satisfies $P u_{j,k} \approx 0$. 
So, $f_{j,k+1} = -Pu_{j,k}$ will have better decay rate than $f_{j,k}$, in the relevant sense. Now repeat.

For $j=1,2,3$, we will use 
\begin{equation} 
    v_j \coloneq \mathsf{Borel}\bigg(\sum_{k} v_{j,k}\bigg)
\end{equation} 
to denote the result of asymptotically summing the various $v_{j,k}$'s in the $j$th step. More accurately, we asymptotically sum the polyhomogeneous part of the $v_{j,k}$'s, and then multiply back in the oscillations.

Exactly solving 
\begin{equation} \label{eq:exact_goal}
    P v_{j,k}= f_{j,k}
\end{equation} 
would be equivalent in difficulty to exactly solving the original PDE $Pu=0$. If that were understood, then we would have no need for the algorithm above. Consequently, the algorithm's utility derives from the fact that solving the approximation problem \cref{eq:approx_goal} is easier than solving the exact analog
\cref{eq:exact_goal}.
Notwithstanding that, we \emph{will} solve the PDE 
\begin{equation}
    P v_{3,1} =  \underbrace{f_{3,1}}_{= -P u_{3,0}}
\end{equation}
modulo Schwartz error in the third main step of the algorithm, so that $u_3=u_{3,0}+v_{3,1}$ satisfies $Pu_3\in\mathcal S(\bbR^d)$. 
Subsequently, we apply the limiting absorption principle to the remaining
Schwartz error, using the limiting resolvent
\begin{equation}
    R(1+i0) = \lim_{\varepsilon \to 0^+} (P-i\varepsilon)^{-1}.
\end{equation}
This operator produces, when applied to $f$ in its domain, a solution $v=R(1+i0)f$ of the PDE $Pv=f$. 

We could apply $R(1+i0)$ to $f_{\bullet}$ in some earlier step in the algorithm to produce an exact solution to $Pv_{\bullet}=f_{\bullet}$, thus yielding an exact solution to $Pu=0$---in fact, the same one as above. The reason we do not take this shortcut is because we understand the mapping properties of 
\begin{equation} 
f\mapsto R(1+i0)f
\end{equation} 
in full detail \emph{only 
when $f$ has the form of a spherical wave}, meaning that $e^{-ir} f$ is polyhomogeneous on the radial compactification $\smash{\overline{\bbR^d}}\hookleftarrow \bbR^d$. The $f_{j,k}$ arising in the $j=1,2$ steps of the algorithm are either oscillating like plane waves or are well-behaved only on the parabolic compactification $X\hookleftarrow \bbR^{d}$.
The purpose of the algorithm above is to improve $f_{\bullet}$ so
that the resolvent can be applied with understood output. At the end
of the third step, we have constructed $u_3$ such that
\begin{equation}
    f_4\coloneqq -Pu_3\in\mathcal{S}(\bbR^d).
\end{equation}
We then define
\begin{equation}\label{eq:v31_res_out}
    v_4\coloneqq R(1+i0)f_4
    \in
    \langle r\rangle^{-(d-1)/2}e^{ir}
    C^\infty(\overline{\bbR^d}).
\end{equation}
Thus, the final correction is a pure outgoing spherical wave.

On the other hand, the $v_{j,k}$ in steps $j=1,2,3$ are all constructed explicitly, so we know which function spaces they lie in. 
Consequently, all contributions to $u$ are understood. Our main theorem, \Cref{thm:main}, will be read off this. In the formula \cref{eq:form} for $u$, the plane wave contribution $a$ is produced by the first step of the algorithm:
\begin{equation}
    a = \langle r \rangle^{L-1} e^{-ix}(\underbrace{u_{1} - e^{ix}}_{v_1})\sim \langle r \rangle^{L-1} e^{-ix} \sum_{k=1}^\infty v_{1,k}. 
\end{equation}

The contribution $w$ is a combination of the outputs from the second
and third steps, together with the final limiting-absorption
correction:
\begin{equation}\label{eq:w_def}
\begin{aligned}
    w
    &= \langle r\rangle^{(d-1)/2}e^{-ir}(u_4-u_1) = \langle r\rangle^{(d-1)/2}e^{-ir}
       \bigl(v_2+v_3+\underbrace{R(1+i0)f_4}_{v_4}\bigr).
\end{aligned}
\end{equation}
The only contribution here that is not constructed explicitly is
$v_4=R(1+i0)f_4$. Consequently:

\vspace{.5em}
\noindent \fbox{%
		\parbox{.98\textwidth}{%
            The construction produces the expansion of $a$ at $\mathrm{bf}$,
as well as $w\bmod C^\infty(\overline{\bbR^d})$, using only the
germs of the PDE coefficients at infinity.
		}%
	}
	
	\vspace{.5em} 
	\noindent

\begin{remark*}
	When a Euclidean plane wave scatters off  a compact obstacle, $a=0$, but $w \in C^\infty(\overline{\bbR^d})$ is usually nonzero. Consequently, $w$ is not determined by the germs of the PDE coefficients at infinity. Only the singular part $w\bmod C^\infty(\overline{\bbR^d})$ is so-determined. 
\end{remark*}

\subsection{Structure of an individual substep}
The content of the algorithm is specifying how each $v_\bullet$ is constructed. Like many constructions in asymptotic analysis, we use a variant of the Frobenius method. The basic idea takes many forms, such as the Liouville--Green/WKB expansion. We look for $v_\bullet$ of the form 
\begin{equation}\label{eq:ind_form}
    v_\bullet = s \rho_{\mathrm{f}}^k (\log \rho_{\mathrm{f}} )^\ell a(\theta), \quad s\in \{e^{ix},e^{ir}\}
\end{equation}
where $\rho_{\mathrm{f}}$ is a boundary-defining function of one of the boundary hypersurfaces $\mathrm{f}\in \{ \mathrm{bf},\mathrm{ff}\}$, and $a\in C^\infty(\mathrm{f}^\circ)$. 
In order to interpret $v_\bullet$ as a function on some open neighborhood $U\supset \mathrm{f}$, we identify 
\begin{equation} 
U \cong [0,\delta)\times \mathrm{f}
\end{equation} 
using the coordinates $(\rho_{\mathrm{f}},\theta)$, for some $\delta>0$ and (smooth) $\theta: U\to \mathrm{f}$. An explicit atlas was presented above.

\begin{remark*}
As written, $v_\bullet$ may only make sense near $\mathrm{f}$; in order to interpret $v_\bullet$ as a function on all of $\bbR^d$, we may need to insert some cutoffs.    
\end{remark*}

So, at each step in the algorithm, $v_\bullet$ will oscillate either
\begin{enumerate}[label=(\roman*)]
	\item like the plane wave $e^{ix}$ or
	\item like the spherical wave $e^{ir}$.
\end{enumerate}
These two possibilities are essentially mutually exclusive. If $f_{\bullet}$ is oscillating like $s\in \{e^{ix},e^{ir}\}$, then it is reasonable to expect that 
\begin{equation}
    \exists v_\bullet\text{ of form above s.t. }Pv_\bullet \approx f_\bullet. 
\end{equation}
This expectation is borne out below. 
Key to the algorithm's success is that when $v_{j,1},\cdots ,v_{j,k}$ have the form above, so will $f_{j,k+1}$, with a better rate of decay.

The question before us is how to find $v_\bullet$, given $f_\bullet$ of the form \cref{eq:ind_form} (possibly with different weights $j,k$). The method is to throw out from $P$ unimportant terms, resulting in another (simpler!) operator $N$ such that the PDE 
\begin{equation}
    N v_\bullet = f_\bullet 
\end{equation}
can be solved exactly. The operator $N$ is called the \emph{normal operator}; there exists a systematic framework due to Melrose for finding such operators.\footnote{We refer to \cite{Me93} for the systematic exposition. Our use of the b-calculus is light, so an imprecise description suffices.} They depend on the face $\mathrm{f}$ and the oscillation $s$, \emph{but not on $f_\bullet$}. First, one forms the ``conjugated operator''
\begin{equation}\label{eq:883}
    \bar{P}\coloneq s^{-1}Ps =
    \begin{cases} 
    \tilde{P}\coloneq e^{-ix} P e^{ix} & (s=e^{ix}),\\ 
    \widehat{P}\coloneq  e^{-ir} P e^{ir}& (s=e^{ir}).
    \end{cases} 
\end{equation}
Here, we mean the result of conjugating $P$ by the multiplication operator $f\mapsto s f$. Thus, $\bar{P} g = s^{-1} P (s g)$.
Next, take the \emph{b-normal operator} 
\begin{equation} 
N=N_{\mathrm{f}}(\bar{P}). 
\end{equation}
Concretely, this means assigning to each term in $\bar{P}$ a ``b-growth order'' $\in \bbZ$ by saying that $\smash{\rho_{\mathrm{f}}\partial_{\rho_{\mathrm{f}}}},\partial_\theta$ have weight zero, and each additional $\smash{\rho_{\mathrm{f}}^{-1}}$ gives one extra order of b-growth. 
Then, the b-normal operator is defined as
\begin{equation} 
N_{\mathrm{f}}(\bar{P})=\text{terms in $\bar{P}$ with least b-decay at $\mathrm{f}$}.
\end{equation} 
We will demonstrate this below for the cases of interest. 

\begin{remark*}
    More precisely, the normal operator should be defined as the operator modulo the space of operators with greater b-decay. However, it is standard practice to conflate this family with a representative thereof.
\end{remark*}

For $v_\bullet$ of the form \cref{eq:ind_form},  $P v_\bullet =  s \bar{P} ( \bar{v}_\bullet)$,
where 
\begin{equation}\label{eq:ind_form_bar}
     \bar{v}_\bullet = s^{-1} v_\bullet = \rho_{\mathrm{f}}^k (\log \rho_{\mathrm{f}} )^\ell a(\theta).
\end{equation}
The significance of the b-decay of an operator is that it counts how many extra $\rho_{\mathrm{f}}$ are generated when applying the operator to a function of the form \cref{eq:ind_form_bar}. Thus, compared to the normal operator $N_{\mathrm{f}}(\bar{P})$, the other terms in $\bar{P}$, when applied to $\bar{v}_\bullet$, generate terms with strictly more decay at $\mathrm{f}$. Such terms can be neglected when trying to approximately solve $Pv_\bullet\approx f_\bullet$, and will presumably contribute to \emph{later} $f_\bullet$.

A key property of $N=N_{\mathrm{f}}(\bar{P})$ is that it does not see the metric perturbation $g$ or the potential $V$. Thus, 
\begin{equation}
    N=N_{\mathrm{f}}(\bar{P}_0)
\end{equation}
is the same as for the free Helmholtz operator $P_0=\triangle-1$, which we can invert on appropriate function spaces explicitly. Actually, $N$ can be solved via the Mellin transform, in the same way that constant-coefficient operators on Euclidean space can be solved via the Fourier transform. 
Indeed, letting $w\in \bbZ$ be the b-growth order of $N$, the weighted operator $\rho_{\mathrm{f}}^w N$ is a linear combination over $C^\infty(\mathrm{f}^\circ_\theta)$ of $\rho_{\mathrm{f}}\partial_{\rho_{\mathrm{f}}},\partial_\theta$, which are invariant under the dilations 
\begin{equation}
    (\rho_{\mathrm{f}},\theta ) \mapsto 
    (\lambda \rho_{\mathrm{f}},\theta ),\qquad \lambda>0 .
\end{equation}
Consequently, $\rho_{\mathrm{f}}^w N$ is invariant under dilations.  
The Mellin transform is used to solve dilation-invariant PDE on a cone in the same way that the Fourier transform is used to solve translation-invariant PDE on Euclidean space.  
Either way, $N$ is an operator whose solvability is understood.

See \cref{eq:Ns} for the normal operators used below.
    
Altogether, the result of the substeps within a single step of the algorithm will be a sequence of $\bar{v}_{j,k}$ with $\rho_{\mathrm{f}}^k$ orders of decay at $\mathrm{f}$. These can be combined in a formal series 
\begin{equation}
    \sum_{k} \bar{v}_{j,k} = \sum_{k,\ell} \rho_{\mathrm{f}}^k (\log \rho_{\mathrm{f}})^\ell a_{k,\ell}(\theta)  .
\end{equation}
Finally, these infinitely many contributions must be asymptotically summed. That is, we must find a $u_{j,\infty}$ such that the asymptotic expansion of $u_{j,\infty}-u_{j,0}$ at $\mathrm{f}$ is the prescribed series. This is a standard construction, but, owing to its status as a folklore lemma, we provide a full proof in the appendix \S\ref{sec:Borel}.

\subsection{Computation of normal operators}
Explicitly, the conjugated free Helmholtz operators $\bar{P}_0$ are
	\begin{align}
		\begin{split} 
			\tilde{P}_0 &= -\partial_x^2+\triangle_{\bfy}-2i \partial_x \in \operatorname{Diff}_{\mathrm{b}}^{2,-1,-2}(X) \\ 
			\hat{P}_0&=  -\partial_x^2+\triangle_{\bfy} -2i \partial_r -i(d-1)/r \in \operatorname{Diff}_{\mathrm{b}}^{2,-1,-2}(X).
		\end{split} \label{eq:conj_P0}
	\end{align}
    (In the latter case, we are working away from the origin.) The partial-$r$ derivative can be rewritten 
    \begin{equation}
        \partial_r = r^{-1} (x\partial_x + \bfy\cdot \nabla_{\bfy}) .
    \end{equation}
    
    The b-decay orders of the terms above are given in \Cref{tab:b_orders}. The operator $-\partial_x^2$ is always lower-order (in the sense of decay---not regularity), but the other terms may show up in the normal operator. The four possibilities are 
    \begin{align}\label{eq:Ns}
    \begin{split}
        N_{\mathrm{bf}}(\tilde{P}) &= N_{\mathrm{bf}}(\tilde{P}_0) = -2i\partial_x, \\
        N_{\mathrm{bf}}(\hat{P}) &=N_{\mathrm{bf}}(\hat{P}_0) = -2i\partial_r -  i(d-1) /r,\\
        N_{\mathrm{ff}}(\tilde{P}) &= N_{\mathrm{ff}}(\tilde{P}_0) = \triangle_{\bfy}-2i\partial_x ,\\
        N_{\mathrm{ff}}(\hat{P}) &= N_{\mathrm{ff}}(\hat{P}_0) =  \triangle_{\bfy}-2i\partial_r-i(d-1)/r.
        \end{split} 
    \end{align}
    Our argument makes use of three of these four, the exception being $N_{\mathrm{ff}}(\tilde{P})$ (the paraxial Helmholtz operator).

    See \Cref{lem:M_to_X_b_vec} for the lemma guaranteeing that $N_\bullet(\bar{P})=N_\bullet(\bar{P}_0)$. 
     \begin{table}[t!]
        \centering
        \begin{tabular}{c|cc}
             & bf & ff \\
             \hline\hline 
             $-\partial_x^2$ & $-2$ & $-4$ \\
             $\triangle_{\bfy}$ & $-2$ & $-2$ \\
             $-2i\partial_x$, $-2i\partial_r$ & $-1$ & $-2$  \\
             $-r^{-1} i(d-1)$ & $-1$ & $-2$
        \end{tabular}
        \caption{The b-growth orders of the terms in \cref{eq:conj_P0}. Thus, negative means decay. The values in this table are calculated using \Cref{lem:M_to_X_b_vec}. }
        \label{tab:b_orders}
    \end{table}
    
\subsection{Specifics of steps}
\subsubsection*{Step one}

In step $j=1$ of the algorithm, we start with the plane wave $u_{1,0}=u_0=e^{ix}$, and then we proceed to solve away the error at $\mathrm{bf}$, allowing an error at $\mathrm{ff}$. This means that, in this step, errors of the form 
\begin{equation}\label{eq:I_goal}
    e^{i x} \calA^{\infty,*}(X) 
\end{equation}
are considered completely negligible; `$*$' refers to a to-be-determined index set that is unimportant as far as this outline is concerned. Such errors are Schwartz as $r\to\infty$, except in the forward direction, where more subtle behavior is allowed. After each individual substep, the error $f_{1,k}$ will lie in $e^{i x} \calA^{k,*}(X)$. This gets better at $\mathrm{bf}$ as $k\to\infty$, tending to \cref{eq:I_goal}.

The quasimodes produced in this step have the form 
\begin{equation}
    v\in  e^{ix} \calA^{*,*}(X)
\end{equation}
as well. These oscillate like plane waves; they are the plane wave corrections alluded to in the introduction.
According to the b-philosophy above, the normal operator $N$ which must be inverted is 
\begin{equation}
		N_{\mathrm{bf}}(\tilde{P}) = - 2i \partial_x,  \label{eq:937}
\end{equation}
because we are trying to solve away a plane wave-like error at $\mathrm{bf}$. 
Indeed, if $\tilde{v} = e^{-ix} v$, $\tilde{f} = e^{-ix} f $
    are the result of peeling the oscillation off $v,f$, 
    then 
    \begin{align}
    \begin{split} 
        N_{\mathrm{bf}}(\tilde{P})\tilde{v} = \tilde{f} &\iff \tilde{P}\tilde{v} \approx  \tilde{f} \\ 
        &\iff Pv\approx f,
        \end{split}
    \end{align}
    where `$\approx$' means modulo terms with better decay at $\mathrm{bf}$. 
The operator $\partial_x$ is inverted just by integrating. The constant of integration should be chosen so that the result is decaying as $x\to-\infty$.  This means that the lower limit of integration should be $-\infty$. Forward integration produces a singularity in the forward direction, which we will grapple with later. Our goal at present is to produce a good solution away from the forward direction.

The first plane wave correction has the form  
\begin{equation}
    v=e^{ix} \tilde{v}\text{ for }
    \tilde{v}=-N_{\mathrm{bf}}(\tilde{P})^{-1} V  =  \frac{1}{2i} \int_{-\infty}^x V(s,\bfy) \dd s ,
\end{equation}
so $u\approx e^{ix} + \frac{e^{ix}}{2i}\int_{-\infty}^x V(s,\bfy)\dd s $.

\begin{example}[Inverse-square potential, cont.]
    For the inverse-square potential, 
    \begin{equation}
        \tilde{v} =   \frac{\alpha }{2i} \int_{-\infty}^x \frac{ \dd s}{s^2+y^2} = \frac{\alpha }{2iy}\Big[\frac{\pi}{2} + \operatorname{arctan} \frac{x}{y}\Big]  =  \frac{\alpha(\pi-\theta)}{2iy} = \frac{\alpha }{2i r} \frac{\pi-\theta}{ \sin \theta}. 
    \end{equation}
    This is evidently singular on the forward ray $\{r=x\}$. 
    At first glance, it also appears singular near the \emph{backward} ray $\{r=-x\}$. However, since $\operatorname{arctan}(t)= -\pi/2-1/t+t^{-2} C^\infty([0,1)_{-1/t})$ as $t\to-\infty$, 
    \begin{equation}
       2i  \tilde{v} = - \frac{\alpha}{x} + \frac{y}{x^2} C^\infty([0,\infty)_{-y/x} )
    \end{equation}
    in the half-plane $\{x<0\}$. This is finite as $y\to 0^+$ for $x<0$ fixed.  

    Thus, the first plane wave correction to our eventual perturbed plane wave has the form 
    \begin{equation}\label{eq:misc_143}
        v= e^{ix} \tilde{v} = \frac{e^{ix} \alpha }{r} \underbrace{\Big( \frac{\pi-\theta}{2i\sin \theta} \Big)}_{a_1 }.
    \end{equation}
    Part of this correction is in what we termed $u^{[1]}$ in \cref{eq:u_brak1}. That correction embedded the most forward-singular part of $v$ into a term with nontrivial front face behavior and thereby regularized the singularity. 
The remainder term is
\begin{align}
      \frac{\alpha e^{ix}}{r}q(\theta),
\end{align}
where
\begin{equation}\label{eq:q_theta}
    q(\theta)
    \coloneqq
    \frac{\pi-\theta}{2i\sin\theta}
    +\frac{\pi i}{4\sin(\theta/2)}.
\end{equation}
The angular coefficient $q(\theta)$ is not singular at $\theta=0$.

    This remainder can be seen in \Cref{fig:dipole_plane_wave_Out}. Subtracting this last term leaves a remainder that is $O(1/r^2)$ away from the forward direction; see \Cref{fig:dipole_plane_wave_whole}. 
\end{example}

\begin{figure}[t!]
    \centering
    \includegraphics[scale=1.1]{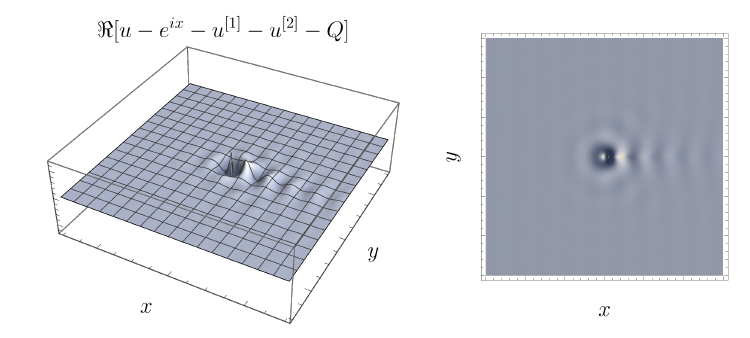}
    \caption{The real part of $u-e^{ix}-u^{[1]}
-u^{[2]}-\alpha e^{ix}r^{-1}q(\theta)$,
with $q$ as in \cref{eq:q_theta}. Notice the absence of the plane wave term (which, in contrast, is visible in \Cref{fig:dipole_plane_wave_Out}).}
    \label{fig:dipole_plane_wave_whole}
\end{figure}

This step corresponds to the transport equation in \cite{MZ}.  

Some details are outsourced to the appendices \S\ref{sec:forward}, \S\ref{sec:logs}.

\subsubsection*{Step two}
Next, we solve away, at $\mathrm{ff}$, the error term at the end of the previous step. By construction, this looks like a unidirectional Gaussian beam 
\begin{equation}
    e^{ix}f(\hat{\bfy})
    =
    e^{ix}f\Big(\frac{\bfy}{x^{1/2}}\Big),
    \quad f\in\calS(\bbR^{d-1}),\quad x\gg1.
    \label{eq:misc_010}
\end{equation}
	pointed in the forward direction.
    An important observation, making good on the intuition that a spherical wave and a plane wave cannot be distinguished in the forward direction, is that 
	\begin{equation}
		e^{ir} \calA^{\infty,*}(X) = e^{ix} \calA^{\infty,*}(X);
		\label{eq:null_magic}
	\end{equation}
	see \Cref{lem:gaussian_conversion}. A consequence is that the error from the previous step could equally well have been written with $e^{ir}$ in place of $e^{ix}$.

    At this point in the argument we switch from working with plane waves to working with spherical waves. 
    The reason why this is necessary is that, were we to solve away the error from the previous step by inverting 
    \begin{equation} 
    N_{\mathrm{ff}}(\tilde{P})=\triangle_{\bfy}-2i\partial_x.
    \end{equation}
    the result would be undoing the work in the previous step. After all, cutting off the exact plane wave $u_0=e^{ix}$ near the forward direction produces an error which is as well behaved as possible near $\mathrm{ff}$. 
    Thus, it is essential that the allowed errors produced in the present step are oscillating differently, like a spherical wave instead of like a plane wave.

    According to the b-philosophy above, 
    since we are trying to solve away an error at $\mathrm{ff}$, while allowing errors that oscillate like a spherical wave, we must invert the operator  
    \begin{equation}\label{eq:Nff}
        N_{\mathrm{ff}}(\hat{P}) = \triangle_{\bfy}-2i\partial_r - i(d-1)/r .
    \end{equation} 
    Indeed, letting $\hat{v}=e^{-ir}v$, $\hat{f}=e^{-ir} f$, then 
    \begin{align}
        \begin{split} 
        N_{\mathrm{ff}}(\hat{P}) \hat{v} = \hat{f} &\iff \hat{P}\hat{v}\approx \hat{f} \\ 
        &\iff Pv\approx f, 
        \end{split} 
    \end{align}
    where `$\approx$' means modulo terms with better decay at $\mathrm{ff}$. 
    The operator $N_{\mathrm{ff}}(\hat{P})$ is closely related to the paraxial Helmholtz operator, which is used in the study of Gaussian beams. Evidently, \cref{eq:Nff} is a spherical wave analog.  Later, it will be shown to be related to the Hamiltonian of the quantum inverted harmonic oscillator (QIHO).

    This step accomplishes the same task that Melrose--Zworski solve in \cite{MZ} using their distributions associated to pairs of intersecting Legendrian submanifolds.  
    The errors allowed at the end are of the form
    \begin{equation}\label{eq:II_allowed_errors}
            e^{ir} \langle r \rangle^{-(d+3)/2}\calA^{(0,0),\infty}(X)\subset e^{ir} \langle r \rangle^{-(d+3)/2}C^\infty(\overline{\bbR^d} ).
    \end{equation}

    Appendix \S\ref{sec:QHO} contains the theory of the QIHO required here.
    We were unable to locate the desired results in the literature, so have provided our own proofs. A computation via separation of variables serves as a cross-check.    
\subsubsection*{Step three}

While the mapping properties of the resolvent on the output of Step 2 can be understood (cf. \cite[\S 3]{phgFull}), Step 3 further improves the remaining error to a Schwartz function. This reduces the final correction to the most classical form of the limiting absorption principle.

Concretely, in view of \eqref{eq:II_allowed_errors}, Step 3 shows that for
\begin{align}
      f\in e^{ir}\langle r\rangle^{-(d+3)/2}
    C^\infty(\overline{\mathbb R^d}),
\end{align}
one can construct
\begin{align}
     v\in e^{ir}\langle r\rangle^{-(d+1)/2}
    C^\infty(\overline{\mathbb R^d})
\end{align}
such that
\begin{align}
     Pv-f\in \mathcal S(\mathbb R^d).
\end{align}
This is done by a Frobenius-type construction at the boundary face $\mathrm{bf}$: after conjugating by \(e^{ir}\), one recursively inverts the normal operator
\begin{align}
     N_{\mathrm{bf}}(\widehat P)
    =
    -2i\partial_r-\frac{i(d-1)}{r},
\end{align}
thereby determining the coefficients of a formal asymptotic expansion for \(v\). Borel's lemma is then used to realize this formal expansion by an actual smooth function with the stated asymptotics.

After step 3, the remaining Schwartz error is then removed by applying \(R(1+i0)\), whose action on Schwartz functions produces a pure outgoing spherical wave, and hence preserves the final asymptotic form claimed in the theorem.

The full algorithm is summarized in \Cref{fig:Izak}.

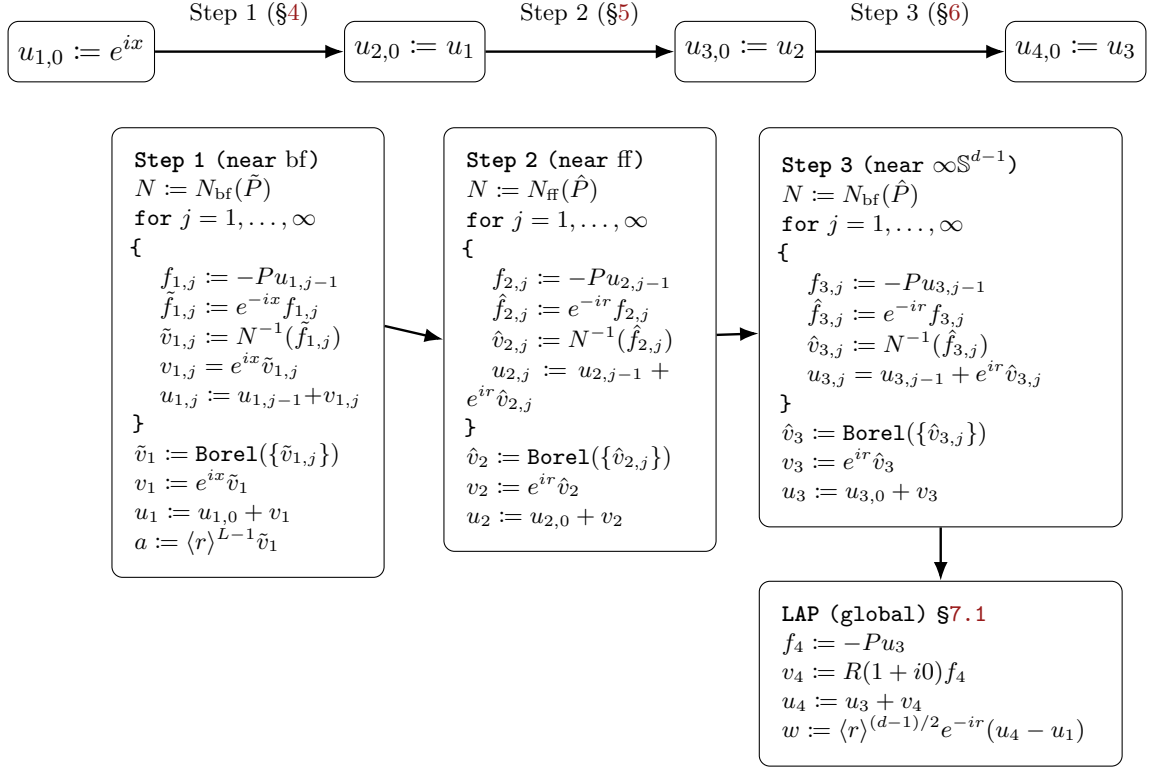
\begin{figure}[!htbp]
\centering
\begin{tikzpicture}[
  >=Latex,
  node distance=25mm,
  every node/.style={font=\normalsize},
  state/.style={draw, rounded corners, minimum width=18mm, minimum height=8mm, align=center},
  step-label/.style={midway, above, yshift=2mm, font=\footnotesize},
  codebox/.style={
    draw,
    rounded corners,
    align=left,
    inner xsep=3mm,
    inner ysep=3mm,
    font=\footnotesize\ttfamily,
    text width=30mm
  }
]

\node[state] (u0) {$u_{1,0}\coloneq e^{ix}$};
\node[state, right=of u0] (u1inf) {$ u_{2,0}\coloneq u_{1} $};
\node[state, right=of u1inf] (u2inf) {$u_{3,0}\coloneq u_{2}$};
\node[state, right=of u2inf] (u3) {$u_{4,0}\coloneqq u_{3} $};

\draw[->, line width=0.9pt] (u0) -- node[step-label]{Step 1 (\S\ref{sec:bf})} (u1inf);
\draw[->, line width=0.9pt] (u1inf) -- node[step-label]{Step 2 (\S\ref{sec:ff})} (u2inf);
\draw[->, line width=0.9pt] (u2inf) -- node[step-label]{Step 3 (\S\ref{sec:outgoing})} (u3);

\coordinate (m01) at ($ (u0)!0.5!(u1inf) $);
\coordinate (m12) at ($ (u1inf)!0.5!(u2inf) $);
\coordinate (m23) at ($ (u2inf)!0.5!(u3) $);

\node[codebox, below=10mm of m01] (code1) {%
  \textbf{Step 1} (near $\mathrm{bf}$)\\
  $N\coloneq N_{\mathrm{bf} }(\tilde P)$\\
  for $j = 1,\dots,\infty$\\
  \{\\
  \quad $f_{1,j} \coloneq - Pu_{1,j-1}$ \\
  \quad $\tilde f_{1,j} \coloneq  e^{-ix} f_{1,j} $\\
  \quad $\tilde v_{1,j} \coloneq  N^{-1}(\tilde f_{1,j})$\\
  \quad $v_{1,j}=e^{ix} \tilde v_{1,j}$\\
 \quad  {$u_{1,j} \coloneq u_{1,j-1}+ v_{1,j}$}\\
  \}\\
  $\tilde v_{1} \coloneq \text{Borel}(\set{\tilde v_{1,j}})$\\
  $v_{1}\coloneq e^{ix}\tilde v_{1}$\\
  $u_{1} \coloneq u_{1,0} + v_{1} $ \\
  $a \coloneq \langle r \rangle^{L-1}\tilde v_1$
  
};

\node[codebox, below=10mm of m12] (code2) {%
  \textbf{Step 2} (near $\mathrm{ff}$)\\
 $N\coloneq N_{\mathrm{ff} }(\hat P)$ \\
for $j= 1,\dots,\infty$\\
\{\\
\quad $f_{2,j} \coloneq -Pu_{2,j-1}$\\
\quad $\hat f_{2,j} \coloneq e^{-ir}  f_{2,j}$\\
\quad $\hat v_{2,j} \coloneq  N^{-1}(\hat f_{2,j})$ \\
\quad $u_{2,j} \coloneq u_{2,j-1}+ e^{ir} \hat v_{2,j}$ \\
\}\\
$\hat v_{2} \coloneq \text{Borel}(\set{\hat v_{2,j}})$\\
$v_{2} \coloneq e^{ir}\hat v_{2}$\\
$u_{2} \coloneq u_{2,0}+v_2$ 
};

\node[codebox, text width=42mm, below=10mm of m23, xshift=4mm] (code3) {%
  \textbf{Step 3} (near $\infty \mathbb S^{d-1}$)\\
  $N \coloneq N_{\mathrm{bf}}(\hat P)$\\
  for $j=1,\dots,\infty$\\
  \{ \\
  \quad $f_{3,j} \coloneq - Pu _{3,j-1}$\\
  \quad $\hat f_{3,j} \coloneq e^{-ir}f_{3,j}$\\
  \quad $\hat v_{3,j} \coloneq N^{-1} (\hat f_{3,j})$\\
  \quad $u_{3,j} = u_{3,j-1} + e^{ir} \hat v_{3,j}$\\
  \}\\
  $\hat v_3 \coloneq \texttt{Borel}(\set{\hat v_{3,j}})$ \\
  $v_3 \coloneq e^{ir} \hat v_3$\\
  $u_{3} \coloneq u_{3,0} + v_3$\\
};
\node[codebox, text width=42mm, below=70mm of m23, xshift=4mm] (code4) {%
\textbf{LAP} (global) \S \ref{s:complete_proof}\\
 $f_{4} \coloneq -P u_{3}$\\
 $v_{4} \coloneq R(1+i0)f_{4}$\\
 $u_{4} \coloneq u_{3} +v_{4}$ \\
 $w\coloneq \langle r \rangle^{(d-1)/2}e^{-ir} ( u_4- u_{1})$ 
};
\draw[-{Latex}, line width=0.9pt]
  ([yshift=10pt]code1.east) -- (code2.west);
\draw[->, line width = 0.9pt] (code2) -- (code3);
\draw[->, line width = 0.9pt] (code3) -- (code4);
\end{tikzpicture}
\caption{Pseudo-code for the three-step construction of perturbed plane waves (along with an application of the limiting absorption principle).
The initial input is $u_0 = e^{ix}$ (the exact plane wave). 
The operators $\tilde P$ and $\hat P$ are the conjugations of $P$ given in \cref{eq:883}. 
The normal operators are specified above. 
By $\texttt{Borel}$, we mean an application of \Cref{lem:Borel}.
The functions $a$ and $w$ are exactly the ones appearing in \cref{eq:form}.
} \label{fig:Izak}
\end{figure}

\subsection{The geometry of the parabolic compactification}
\label{subsec:geo}

\subsubsection*{b-operators}
Here, we record the conversion rate between b-operators on $M=\overline{\bbR^d}$ and b-operators on $X$.  Recall that $\operatorname{Diff}_{\mathrm{b}}^{1,0,0}(X)=\calV_{\mathrm{b}}(X)$ consists of all smooth vector fields on $X$ tangent to every boundary hypersurface, and 
\begin{equation}
    \operatorname{Diff}_{\mathrm{b}}^{1,s,\ell}(X) = \rho_{\mathrm{bf}}^{-s}\rho_{\mathrm{ff}}^{-\ell} \operatorname{Diff}_{\mathrm{b}}^{1,0,0}(X)
\end{equation}
for $s,\ell\in \bbR$.

\begin{lemma}
	\label{lem:M_to_X_b_vec}
    $x\partial_x \in \operatorname{Diff}_{\mathrm{b}}^{1,0,0}(X)$, $x\partial_{y_j}\in \operatorname{Diff}_{\mathrm{b}}^{1,0,1}(X)$. 
\end{lemma}
\begin{proof}
    Only the situation near $\mathrm{ff}$ is new. (The content of the lemma away from $\mathrm{ff}$ is contained in \cite[Lemma 1]{Me94}.)

    Near $\mathrm{ff}$ but outside $\mathrm{bf}$, we can use $\rho=x^{-1/2}$ as a bdf and $\hat{\bfy} = \bfy/x^{1/2}$ for the remaining $d-1$  coordinates. Then, 
	\begin{align}
	\begin{split} 
		x\frac{\partial}{\partial x} &= \frac{1}{\rho^2}\Big(\frac{\partial \rho}{\partial x} \frac{\partial}{\partial \rho} + \sum_{j=1}^{d-1} \frac{\partial \hat{y}_j}{\partial x} \frac{\partial}{\partial \hat{y}_j}\Big) = - \frac{\rho}{2} \frac{\partial}{\partial \rho} - \sum_{j=1}^{d-1}\frac{\hat{y}_j}{2}  \frac{\partial}{\partial \hat{y}_j}, \\ 
		x\frac{\partial}{\partial y_j}
&=
x\frac{\partial\hat y_j}{\partial y_j}
\frac{\partial}{\partial\hat y_j}
=
x^{1/2}\frac{\partial}{\partial\hat y_j}
=
\frac{1}{\rho}
\frac{\partial}{\partial\hat y_j},
	\end{split} 
	\label{eq:jacoby}
	\end{align}
	which evidently lie in the desired spaces,  locally. 
    
    In order to understand the situation near the corner $\mathrm{bf}\cap \mathrm{ff}$, we can use the coordinates in \cref{eq:corner_coords_0}. Executing this coordinate change, 
    \begin{align}
        x\frac{\partial}{\partial x} &= x \Big( \frac{\partial \rho_{\mathrm{bf}}} {\partial x}\frac{\partial}{\partial \rho_{\mathrm{bf}}} + \frac{\partial \rho_{\mathrm{ff}}} {\partial x}\frac{\partial}{\partial \rho_{\mathrm{ff}}}   \Big) =  \Big( \rho_{\mathrm{bf}} \frac{\partial}{\partial \rho_{\mathrm{bf}}} - \rho_{\mathrm{ff}} \frac{\partial}{\partial \rho_{\mathrm{ff}}}   \Big) \in \calV_{\mathrm{b}}, \\ 
        \begin{split} 
        x\frac{\partial}{\partial y_1} &= x  \Big( \frac{\partial \rho_{\mathrm{bf}}} {\partial y_1}\frac{\partial}{\partial \rho_{\mathrm{bf}}} + \frac{\partial \rho_{\mathrm{ff}}} {\partial y_1}\frac{\partial}{\partial \rho_{\mathrm{ff}}} + \sum_{k=2}^{d-1} \frac{\partial w_k} {\partial y_1}\frac{\partial}{\partial w_k }      \Big)    \\ 
        &=\Big( -\frac{2\rho_{\mathrm{bf}}}{\rho_{\mathrm{ff}}}\frac{\partial}{\partial \rho_{\mathrm{bf}}} + \ \frac{\partial}{\partial \rho_{\mathrm{ff}}} - \sum_{k=2}^{d-1}  \frac{w_k}{\rho_{\mathrm{ff}}} \frac{\partial}{\partial w_k }      \Big) \in \rho_{\mathrm{ff}}^{-1} \calV_{\mathrm{b}} , 
        \end{split} 
        \intertext{and, for $k=2,\dots, d-1$,}
        x \frac{\partial}{\partial y_k} &= x \frac{\partial w_k}{\partial y_k}\frac{\partial}{\partial w_k} = \frac{1}{\rho_{\mathrm{ff}}}\frac{\partial}{\partial w_k} \in \rho_{\mathrm{ff}}^{-1} \calV_{\mathrm{b}} . 
    \end{align}
    So, the desired conclusion holds near the corner as well. 
\end{proof}

\begin{corollary}\label{cor:basicb}
$\operatorname{Diff}^{m,s}_{\mathrm{b}}(M)\subseteq \operatorname{Diff}_{\mathrm{b}}^{m,s,2s+m}(X)$.
	\label{lem:M_to_X_b}
\end{corollary}
This follows from \Cref{lem:M_to_X_b_vec}, since differential orders
and boundary weights add under composition, while the pullback of the
weight $s$ on $M$ contributes weights $s$ and $2s$ at
$\mathrm{bf}$ and $\mathrm{ff}$, respectively.

The next proposition is stated in terms of 
\begin{equation}
    \operatorname{Diff}_{\mathrm{sc}}^{m,-2}(M) = \sum_{k=0}^m \operatorname{Diff}_{\mathrm{b}}^{k,-2-k}(M) .
\end{equation}

\begin{lemma} 
For any classical asymptotically Euclidean metric $g$ on $\bbR^d$ and $Q\in \operatorname{Diff}^{1,-2}_{\mathrm{sc}}(M)$, the operator $P=\triangle_g + Q-1$ satisfies, for each choice of oscillation $s\in \{e^{ix}, e^{ir} \}$, 
	 \begin{equation} 
     s^{-1}(P-P_0)s\in \operatorname{Diff}_{\mathrm{b}}^{2,-2,-4}(X),
    \end{equation} 
    where $P_0=\triangle-1$ is the free Helmholtz operator.
	 \label{lem:rem} 
\end{lemma} 
\begin{proof}
	It suffices to note that $\triangle_g + Q - \triangle \in \operatorname{Diff}^{2,-2}_{\mathrm{sc}}(M)$,
	and conjugating an element of $\operatorname{Diff}^{2,-2}_{\mathrm{sc}}(M)$ by $s$ produces another element of the same space. Consequently, 
	\begin{align}
		\begin{split} 
			s^{-1}(P-P_0)s \in \operatorname{Diff}^{2,-2}_{\mathrm{sc}}(M) &\subseteq \operatorname{Diff}_{\mathrm{b}}^{2,-4}(M)+\operatorname{Diff}_{\mathrm{b}}^{1,-3}(M) + \langle r \rangle^{-2} C^\infty(M) \\
			&\subseteq \operatorname{Diff}_{\mathrm{b}}^{2,-4,-6}(X)+\operatorname{Diff}_{\mathrm{b}}^{1,-3,-5}(X) + \operatorname{Diff}_{\mathrm{b}}^{0,-2,-4}(X),
		\end{split} 
	\end{align}
	using \Cref{lem:M_to_X_b}.
\end{proof}

\subsubsection*{A Gaussian beam lemma}

Recall that $\calA^{(\calE,\alpha)}=\calA^\calE+\calA^\alpha$ denotes partial polyhomogeneity with a conormal error of order $\alpha\in \bbR$. 
\begin{lemma}\label{lem:gaussian_conversion}
	For any $\alpha\in \bbR$ and index set $\calE$, we have $e^{i r  } \calA^{\infty,(\calE,\alpha)}(X) = e^{i  x} \calA^{\infty,(\calE,\alpha)}(X)$.
\end{lemma}
\begin{proof}
	Away from $\mathrm{ff}$, the two function spaces just consist of Schwartz functions. 
	
	For $x\geq \lVert \bfy \rVert+1$, we can expand 
	\begin{equation}
	r = \sqrt{x^2 + \lVert \bfy \rVert^2} = x \sqrt{1+ \frac{\lVert \bfy \rVert^2}{x^2} } = x + \frac{\lVert \bfy \rVert^2}{ x} f(\theta), \qquad \theta = \operatorname{tan}^{-1}(\lVert \bfy \rVert/x ) 
	\end{equation}
	for some function $f\in C^\infty(\bbR)$. So, in terms of the coordinate $\hat{y}=\lVert \bfy \rVert/x^{1/2}$, 
    \begin{equation}\label{eq:trans_lem}
    e^{i  r} = e^{i x} \cdot e^{i \hat{y}^2 f(\theta)}.
    \end{equation} 

    Next, we verify 
    that 
    \begin{equation}\label{eq:trans_goal}
        e^{i \hat{y}^2 f(\theta)} \calA^{\infty,(\calE,\alpha)} =     \calA^{\infty,(\calE,\alpha)}.
    \end{equation} 
    holds near $\mathrm{ff}$. The function $f(\theta)$ is smooth on $\smash{\overline{\bbR^d}}$ therefore $X$. Since $\hat{y}^2$ is smooth near any point in $\mathrm{ff}^\circ=\mathrm{ff}\backslash \mathrm{bf}$, $e^{\smash{i \hat{y}^2 f(\theta)}}$ is smooth there. 
    This establishes \cref{eq:trans_goal}, except at the corner $\mathrm{bf}\cap\mathrm{ff}$. 
    Near the corner, 
    \begin{equation}
    \calA^{\infty,(\calE,\alpha)}(X)  = \calA^{(\calE,\alpha)}([0, 1)_\theta; \calS(\bbR^{d-1}_{\hat{\bfy}}) ),
    \end{equation}
    and so
    \begin{align}
    \begin{split} 
        e^{i \hat{y}^2 f(\theta)} \calA^{\infty,(\calE,\alpha)} &= e^{i  \hat{y}^2 f(\theta)} \calA^{(\calE,\alpha)}([0, 1)_\theta;    \calS(\bbR^{d-1}_{\hat{\bfy}}) ) \\
        &=\calA^{(\calE,\alpha)}([0, 1)_\theta;\calS(\bbR^{d-1}_{\hat{\bfy}}) )  = \calA^{\infty,(\calE,\alpha)},
        \end{split}
    \end{align}
    as claimed. The key equality is the second, where we make crucial use of the Schwartz decay at bf.

   Finally, \cref{eq:trans_goal}, combined with \cref{eq:trans_lem}, gives 
   \begin{equation} 
    e^{i r  } \calA^{\infty,(\calE,\alpha)}(X) =e^{i  x}e^{i \hat{y}^2 f(\theta)}  \calA^{\infty,(\calE,\alpha)}(X) = e^{i  x} \calA^{\infty,(\calE,\alpha)}(X).
    \end{equation} 
\end{proof}

\section{Step one: the incoming plane wave (the transport equation)}
\label{sec:bf}
This section contains step one in the construction of the perturbed plane wave. We have already summarized the main aspects of the argument in \S\ref{sec:outline}. 
Our goal is to solve, for suitable polyhomogeneous forcing $\tilde{f}$, 
\begin{equation}
    \underbrace{\tilde{P}}_{\mathclap{\tilde{P} = e^{-i x} P e^{i  x}}}\tilde{v}\approx \tilde{f} 
    \label{eq:misc_037}
\end{equation}
where the $\approx$ means that we are allowing a $\calA^{\infty,*}(X)$ error (so Schwartz at $\mathrm{bf}$); the `$*$' means the index set at $\mathrm{ff}$ is not specified as part of the problem statement. The specific forcing $\tilde{f}$ of interest is
\begin{equation} 
    \tilde{f}\coloneq -  e^{-ix} P [e^{i x}] = -  e^{-ix} (P-P_0) [e^{i x}] \in \rho_{\mathrm{bf}}^LC^{\infty}(\overline{\bbR^d}),
    \label{eq:f_init}
\end{equation} 
where $P_0=\triangle-1$ is the free Helmholtz operator. Here, $L\in \bbN^{\geq 2}$ is the decay order of the coefficients of $P-P_0$. For the special case of potential scattering, 
\begin{equation} 
    \tilde{f}=-  V.
\end{equation} 

The main proposition of this section is:
\begin{proposition}\label{prop:I_main}
For $\tilde f$ defined by \cref{eq:f_init}, there exists $\tilde{v} \in \rho_{\mathrm{bf}}^{L-1}\rho_{\mathrm{ff}}^{L-1} C^{\infty}(X)$ satisfying
\begin{equation}
\tilde P\tilde{v} -\tilde f\in \rho_{\mathrm{bf}}^{\infty}\rho_{\mathrm{ff}}^{L+1}C^{\infty}(X),
\end{equation}
where $X=[\overline{\bbR^d};\rightarrow]_{\mathrm{par}}$.
\end{proposition}
\begin{remark}\label{rem:I_main_tilde}
    The conclusion of the proposition also holds with $\tilde{f}$ defined by $\tilde{f}=-e^{-i x} P[\chi e^{i  x}]$ for $\chi \in C^\infty(X)$ such that $\chi=1$ near $\partial X$. Indeed, the same $\tilde{v}$ works, since the terms that result from $\chi\neq 1$ are all in $C_{\mathrm{c}}^\infty(\bbR^d)$.
\end{remark}

The proof of \Cref{prop:I_main} is postponed to \S \ref{subsec:stepI}.
The relevant ``b-normal operator'' $N$ at this stage is 
\begin{equation}\label{eq:nbf}
    N_{\mathrm{bf}}(\tilde{P})= -2i \partial_x\in \operatorname{Diff}_{\mathrm{b}}^{1,-1,-2}(X),
\end{equation}
as recorded in \cref{eq:Ns}. Inverting $N\coloneq N_{\mathrm{bf}}(\tilde{P})$ just means integration:
\begin{equation}
    Nu=f \Rightarrow u(x,\bfy) = C(\bfy) -\frac{1}{2i}\int_0^x f(s,\bfy) \dd s.
\end{equation}
Each choice of constant of integration $\bfy\mapsto C(\bfy)$ gives a different one-sided inverse. The one we are interested in is 
integration from $x=-\infty$:
\begin{equation}\label{eq:ninvdef}
    (N^{-1} f)(x,\bfy) = -\frac{1}{2i}\int_{-\infty}^x f(s,\bfy) \dd s. 
\end{equation}
(We only consider $f$ such that the integral above is absolutely convergent.) We will call $N^{-1}$ \emph{forward integration}.

That $N$ is the b-normal operator of $\tilde{P}$ means that the difference $E=\tilde{P}-N$ produces better decay at $\mathrm{bf}$ compared to $\tilde{P}$, $N$ (`$E$' for ``error''). Consider the formal series 
\begin{equation}\label{eq:formu}
    \tilde{\mathsf{v}}_{\circ} = \sum_{j=0}^\infty N^{-1} ( -E N^{-1})^j  \underbrace{\tilde{f}}_{\mathclap{= -E [1] }} = \sum_{j=0}^\infty  (-N^{-1} E)^{j+1} [1].
\end{equation}
This formally satisfies $\tilde{P} \tilde{\mathsf{v}}_{\circ} = (N+E) \tilde{\mathsf{v}}_{\circ} = \tilde{f}$. 
To \emph{asymptotically sum} $\tilde{\mathsf{v}}_{\circ}$ on $\overline{\bbR^d}$ would mean constructing some function $\tilde{v}_{1,\circ}$ such that, for each $J$, the difference 
\begin{equation}
    \tilde{v}_{\circ} - \sum_{j=0}^{J-1} (-N^{-1} E)^{j+1} [1]
\end{equation}
has $J$ orders of decay (relative to, say, the conormal function spaces). That is, 
\begin{equation}
     \tilde{v}_{\circ} - \sum_{j=0}^{J-1} N^{-1} (-EN^{-1})^j \tilde f = O\Big( \frac{1}{r^J} \Big)\text{ as }r\to\infty. 
\end{equation}
Then, one writes $\tilde{v}_{\circ}\sim \tilde{\mathsf{v}}_{\circ}$, and $\tilde{v}_{\circ}$ would be an actual solution to the PDE $\tilde{P}\tilde{v}_{\circ}=\tilde{f}$ modulo a negligible error. 

Unfortunately, this does not work as described, because the terms in $\tilde{\mathsf{v}}_{\circ}$ are not increasingly decaying. Already the first term, $N^{-1} \tilde{f}$, is not decaying in the forward direction. Instead, 
\begin{equation}
    N^{-1}\tilde{f}(x,\bfy) = - \frac{1}{2i} \int_{-\infty}^x \tilde{f}(s,\bfy)\dd s = o(1)-\frac{1}{2i} \int_{-\infty}^\infty \tilde{f}(s,\bfy) \dd s
\end{equation}
as $x\to\infty$, and the integral on the right-hand side is a typically nonzero function of $\bfy$, $\bfy\mapsto \int_{-\infty}^\infty \tilde{f}(s,\bfy)\dd s$. 
\begin{example}[Inverse-square potential, cont.]
    For the inverse-square potential with strength $\alpha\neq 0$, 
    \begin{equation}
        \tilde{f} =-V = -\alpha/r^2, \quad E = \triangle+\alpha/r^2.
    \end{equation}
    So, 
    \begin{align}
        N^{-1} \tilde{f} &= \frac{\alpha}{2iy} \Big[ \frac{\pi}{2}+\operatorname{arctan} \Big( \frac{x}{y} \Big) \Big].
    \end{align}
    The singularity when $y=0$ and $x>0$ is due to the potential's singularity at the origin and is therefore irrelevant for the present demonstration. 
    The important observation is that $N^{-1} \tilde{f}$ has more decay away from the forward direction, when $y\to\infty$, than in the forward direction, when $x\to\infty$ with $\bfy\in \bbR^{d-1}$ fixed. Indeed, 
    \begin{equation}
        \lim_{x\to\infty} N^{-1} \tilde{f} = \frac{\alpha \pi}{2iy} \neq 0.
    \end{equation}
\end{example}

If this were the only issue, one fix would be to perform a polar blowup of the forward direction and then weaken the notion of asymptotic summation to only stipulate that truncations approximate the sum modulo terms with better decay at $\mathrm{bf}$, not the front face $\mathrm{ff}$. In other words, the asymptotic summation should take place on the compactification
\begin{equation} 
    Y=[\overline{\bbR^d};\rightarrow] \hookleftarrow \bbR^d,
\end{equation}
not $\overline{\bbR^d}\hookleftarrow \bbR^d$. The lack of subscript on the ``[$\cdots\!\,$]'' indicates that the blowup of $\{\rightarrow\}\subset \infty \bbS^{d-1}$ used to construct $Y$ is the polar blowup (as opposed to the parabolic blowup used to construct $X$). 
Like $X$, the manifold-with-corners $Y$ has two boundary hypersurfaces. 

The mapping properties of forward integration on $Y$ are collected in \S\ref{sec:forward}.

\begin{remark*}
As a manifold-with-corners, $Y$ is in fact diffeomorphic to $X$. But the category relevant to us is not that of manifolds-with-corners, but compactifications of Euclidean space via manifolds-with-corners. 

The identity map $\bbR^d\to\bbR^d$ does not extend to a diffeomorphism $X\to Y$, so $X,Y$ are inequivalent as compactifications of $\bbR^d$. 
\end{remark*}

Unfortunately, there is another issue: successive terms in $\tilde{\mathsf{v}}_{\circ}$ blow up. The $j$th term should blow up like $x^j$ as $x\to\infty$.
This is because $E$ has $\triangle_{\bfy}$ in it, and \begin{equation} 
    \triangle_{\bfy}\in \operatorname{Diff}^{2,*,0}(Y\backslash \mathrm{bf}),
\end{equation} 
so when $E$ is applied to $ (N^{-1}E)^{j+1} [1]$ we do not necessarily get any additional decay in the forward direction. Then, applying $N^{-1}$ worsens the decay by one order.

The solution turns out to be working on the parabolic blowup $X=[\overline{\bbR^d},\rightarrow]_{\mathrm{par}}$. 
The notion of asymptotic summation is correspondingly weakened to only stipulate that truncations approximate the sum modulo terms with better decay at the boundary hypersurface $\mathrm{bf}$ \emph{of $X$}. Then, we try to asymptotically sum 
\begin{equation}  \tilde{\mathsf{v}}\coloneq \chi\Big(\frac{1}{\hat{y}^2}\Big)\tilde{\mathsf{v}}_{\circ},\quad \hat{y} = \frac{y}{x^{1/2}}, 
\end{equation} 
where $\chi \in C_{\mathrm{c}}^\infty(\bbR)$ is $=1$ identically near the origin. This cuts off $\tilde{\mathsf{v}}_{\circ}$ before reaching the problematic center of $\mathrm{ff}$. Crucially, successive terms in $\tilde{\mathsf{v}}$ are \emph{not} getting worse at $\mathrm{ff}$. The strategy above thus goes through --- asymptotically sum $\tilde{\mathsf{v}}$ on $X$, and the result $\tilde{v}$ solves $\tilde{P}\tilde{v}= \tilde{f}$ modulo the unavoidable error in the asymptotic summation procedure.

\begin{figure}[!htbp]
	\begin{center}
		\begin{tikzpicture}[ decoration={
				markings,
				mark=at position 0.6 with {\arrow[scale=1.5,>=latex]{>}},
			}
			]
			\begin{scope} 
				\filldraw[fill=lightgray!50] (0,0) circle (2);
				\begin{scope}
					\clip (0,0) circle (2.025);
					\filldraw[fill=white] (2,0) circle (1);
				\end{scope}
				\fill[white] (1.85,1) circle (.1);
				\fill[white] (1.85,-1) circle (.1);
			\end{scope} 
			\draw[darkgray, postaction={decorate},dashed] (-2,0) -- (1,0);
			\draw[darkgray, postaction={decorate},dashed] (-2,0) to[out=50,in=140] (1,0);
			\draw[darkgray, postaction={decorate}, dashed] (-2,0) to[out=-50,in=220] (1,0);
			\draw[darkgray, postaction={decorate},dashed ] (-2,0) to[out=-70,in=180] (0,-1.3) to[out=0,in=240] (1,0);
			\draw[darkgray, postaction={decorate}, dashed] (-2,0) to[out=70,in=180] (0,1.3) to[out=0,in=-240] (1,0);
			\draw[darkgray, postaction={decorate}, dashed] (-2,0) to[out=-85, in=180] (0,-1.8) to[out=0,in=224] (1.35,-1.2) to[out=44,in=-30 ] (1.45,-1) to[out=150,in=260] (1,0);
			\draw[darkgray, postaction={decorate}, dashed] (-2,0) to[out=85, in=180] (0,1.8) to[out=0,in=-224] (1.35,1.2) to[out=-44,in=30 ] (1.45,1) to[out=-150,in=-260] (1,0);
			\node () at (1.25,0) {ff};
			\node () at (-2.3,0) {bf};
		\end{tikzpicture}
	\end{center}
	\caption{Level sets of constant ${\bfy}$, in $X$, along which we are integrating in this step of the overall argument.}
\end{figure}
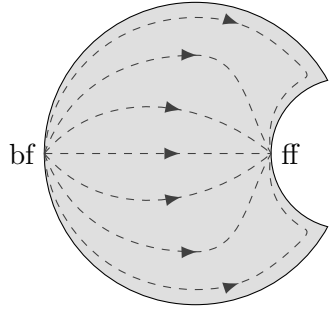

\subsection{Why $X$?}

Above, we have asserted that the parabolic blowup is the choice of quasihomogeneous blowup of the forward direction allowing the argument above. 
We now explain why. 

The normal operators of $\tilde{P}$ at the front face $\mathrm{ff}$ in $X,Y$ are 
\begin{align}
    N_{\mathrm{ff}}(\tilde{P}) &= -2i\partial_x +\triangle_{\bfy} \\ 
    N_{\mathrm{ff}}(\tilde{P}) &= \triangle_{\bfy}, 
\end{align}
respectively. Note that $N=N_{\mathrm{bf}}(\tilde{P})=-2i\partial_x$ is contained in the former, but not the latter.

So, on $X$, the operators $N$ and $E=\tilde{P}-N$ have the same $b$-decay order at $\mathrm{ff}$:
\begin{equation}
    N \in \operatorname{Diff}_{\mathrm{b}}^{1,-1,-2}(X),\quad E \in \operatorname{Diff}_{\mathrm{b}}^{2,-1,-2}(X) ,
\end{equation}
where the key order is the final one, $-2$. 
Consequently, the product $N^{-1} E$ does not worsen decay. Applying $E$ produces two orders of decay, but $N^{-1}$ loses these. (Here, we are restricting attention to a small neighborhood of $\mathrm{bf}$.)

In contrast, on $Y$,  the operator $N$
produces one more order of decay at $\mathrm{ff}$ than $E$, because 
\begin{equation}
    N \in \operatorname{Diff}_{\mathrm{b}}^{1,-1,-1}(Y),\quad E\in \operatorname{Diff}_{\mathrm{b}}^{2,-1,0}(Y); 
\end{equation}
the key order is the final one, and $-1\neq 0$. 
Consequently, $N^{-1} E$ should worsen decay at $\mathrm{ff}$ each time it is applied, because $E$ fails to produce any decay, and then $N^{-1}$ loses one order. 
This applies arbitrarily close to $\mathrm{bf}$, so inserting cutoffs does not help.

Another way of getting at the same conceptual point is to compare the behavior of successive terms in $\tilde{\mathsf{v}}_{\circ}$ at the front faces of $X,Y$. 
In order to see this, we use the fact that the front face of $X$ arises from blowing up the corner of $Y$. More precisely, let 
\begin{equation}
    \tilde{Y} = [Y; \mathrm{bf}\cap \mathrm{ff}];
\end{equation}
see \Cref{fig:Y}. Then, the identity map $\bbR^d\to \bbR^d$ extends to a diffeomorphism $\tilde{Y} \backslash \mathrm{ff}_0 \to X\backslash \mathrm{fd}$,
where $\mathrm{ff}_0$ is the closure in $\tilde{Y}$ of the interior of the original front face $\mathrm{ff}\subset Y$, and $\mathrm{fd}\in X$ is the midpoint of $X$'s front face. For example, the interior of the new front face in $\tilde{Y}$ is parametrized by the ratio of boundary-defining-functions on $Y$,
\begin{equation}
    \frac{\overbrace{y/x}^{\text{old }\rho_{\mathrm{ff}}}}{\underbrace{1/y}_{\text{old }\rho_{\mathrm{bf}} }} =\frac{y^2}{x} = \hat{y}^2,
\end{equation}
in addition to the longitudinal coordinate $\phi = \bfy/y$. These are combined into the coordinate $\hat{y}\phi = \hat{\bfy}\in \bbR^{d-1}\backslash \{0\}$, which happens to parametrize the front face of $X$. A more careful analysis proves the diffeomorphism claim above.

\begin{figure}[!htbp]
    \centering
         \begin{minipage}{.33\textwidth}
         \centering
     	\begin{tikzpicture}
			\begin{scope} 
				\filldraw[fill=lightgray!50] (0,0) circle (2);
				\begin{scope}
					\clip (0,0) circle (2.025);
					\filldraw[fill=white] (2,0) circle (1);
				\end{scope}
				\fill[white] (1.85,1) circle (.1);
				\fill[white] (1.85,-1) circle (.1);
			\end{scope} 
			\node () at (0,2.2) {$X$};
			\draw[darkred,->] (.9,0) to[out=90, in=230] (1.22,.8) node[left] {$\frac{\bfy}{\sqrt{x}}$};
			\draw[darkred,->] (.9,0) -- (.3,0) node[left] {$\frac{1}{\sqrt{x}}$};
			\draw[darkred,->] (-1.9,0) -- (-1.2,0) node[below] {$\frac{1}{r}$};
			\draw[darkred,->] (-1.9,0) to[out=90,in=240] (-1.6,1) node[right] {$\frac{\pi}{2}-\theta$};
			\draw[darkred,->] (1.6,-1.05) to[out=235,in=30] (.85,-1.7) node[left] {$\theta$};
			\draw[darkred,->] (1.6,-1.05) to[out=160, in=-60] (1,-.5) node[below left] {$\frac{x}{y^2}$};
			\node () at (1.25,0) {ff};
			\node () at (-2.3,0) {bf};
		\end{tikzpicture}
        \end{minipage}%
      \begin{minipage}{.33\textwidth}
         \centering
            		\begin{tikzpicture}
			\begin{scope} 
				\filldraw[fill=lightgray!50] (0,0) circle (2);
				\begin{scope}
					\clip (0,0) circle (2.025);
					\filldraw[fill=white] (2,0) circle (1);
				\end{scope}
				\fill[white] (1.85,1) circle (.1);
				\fill[white] (1.85,-1) circle (.1);
			\end{scope} 
			\node () at (0,2.2) {$Y$};
			\draw[darkred,->] (.9,0) to[out=90, in=230] (1.22,.8) node[left] {$\bfy$};
			\draw[darkred,->] (.9,0) -- (.3,0) node[left] {$\frac{1}{x}$};
			\draw[darkred,->] (-1.9,0) -- (-1.2,0) node[below] {$-\frac{1}{x}$};
			\draw[darkred,->] (-1.9,0) to[out=90,in=240] (-1.6,1) node[right] {$-\frac{\bfy}{x}$};
			\draw[darkred,->] (1.6,-1.05) to[out=235,in=30] (.85,-1.7) node[above] {$\frac{y}{x}$};
			\draw[darkred,->] (1.6,-1.05) to[out=160, in=-60] (1,-.5) node[left] {$\frac{1}{y}$};
			\node () at (1.25,0) {ff};
			\node () at (-2.3,0) {bf};
		\end{tikzpicture}
        \end{minipage}%
    \begin{minipage}{.33\textwidth}
     \centering
    \begin{tikzpicture}
    \tikzset{
    partial ellipse/.style args={#1:#2:#3}{
        insert path={+ (#1:#3) arc (#1:#2:#3)}
    }
}
			\begin{scope} 
				\filldraw[fill=lightgray!50] (0,0) circle (2);
				\begin{scope}
                    \clip (0,0) circle (2.025);
                    \filldraw[fill=white, overlay] (2.2,0) [partial ellipse=90:270:2cm and 1.3cm];
                    \fill[white] (1.44,1.35) circle (0.7);
                    \fill[white] (1.44,-1.35) circle (0.7);
                    \draw (.888,1.79) arc (141.2:214.2:0.7cm);
                    \draw (.888,-1.79) arc (218.8:145.8:0.7cm);
				\end{scope}
			\end{scope} 
			\node () at (0,2.2) {$\vphantom{Y}\smash{\tilde{Y}}$};
			\draw[darkred,->] (.1,0) to[out=90, in=230] (.3,.6) node[left] {$\bfy$};
			\draw[darkred,->] (.1,0) -- (-.5,0) node[below] {$\frac{1}{x}$};
			\draw[darkred,->] (-1.9,0) -- (-1.2,0) node[below] {$-\frac{1}{x}$};
			\draw[darkred,->] (-1.9,0) to[out=90,in=240] (-1.6,1) node[right] {$-\frac{\bfy}{x}$};
			\draw[darkred,->] (.7,-1.75) to[out=205,in=0] (0,-1.9) node[above] {$\frac{y}{x}$};
			\draw[darkred,->] (.7,-1.75) to[out=115, in=250] (0.65,-1.2) node[right] {$\;\frac{x}{y^2}$};
            \draw[darkred,->] (.7,1) to[out=115,in=240] (.7,1.6) node[right] {$\;\hat{y}^2=\frac{y^2}{x}$};
			\draw[darkred,->] (.7,1) to[out=205, in=55] (0.35,0.67) node[above] {$\frac{1}{y}\,$};
			\node () at (0.5,0) {ff};
			\node () at (-2.3,0) {bf};
		\end{tikzpicture}
        \end{minipage}%
        \caption{The compactifications $X,Y,\tilde Y$ with an atlas of coordinate charts for each depicted. Not depicted is the azimuthal coordinate $\bfy / y \in \bbS^{d-2}$. Recall that $y=|\bfy|$.}
    \label{fig:Y}
\end{figure}
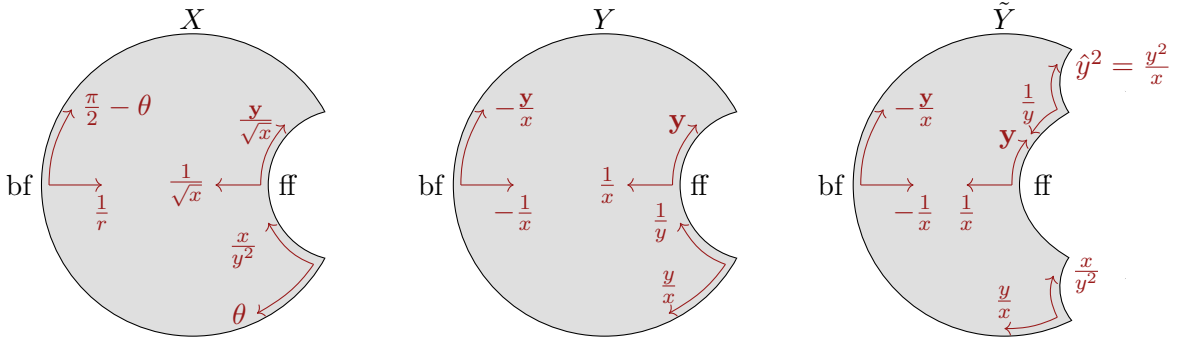

Consequently, a function such as
\begin{equation} 
\rho_{\mathrm{bf}}^{j}\rho_{\mathrm{ff}}^k \in C^\infty(Y) 
\end{equation} 
(where the boundary-defining-functions are those for $Y$)
has $j+k$ orders of decay at the front face of $X$ (away from the midpoint). This is analogous to how, on the quadrant $[0,\infty)^2\subset \smash{\bbR^2_{x,y}}$, the monomial $x^jy^k$ has $r^{j+k}$ orders of decay as $r\to 0^+$. 
So, since $N^{-1} E$ induces one order of decay at $\mathrm{bf}\subset Y$ and one order of growth at $\mathrm{ff}\subset Y$, it breaks even at the corner, hence at the front face of $X$.

\subsection{Mapping properties of forward integration on $X$}
\label{subsec:int}

Let $E=\tilde{P}-N$, as above.

\begin{proposition}\label{prop:mainI_init0}
    For $\tilde f$ defined in \cref{eq:f_init} and $j \geq 0$, 
\begin{equation}\label{eq:misc_eskimo}
    N^{-1}(EN^{-1})^j\tilde f \in \mathcal{A}^{(L-1+j,0),\mathcal{E}^{+\cdots +}}(Y\backslash \mathrm{fd}), 
\end{equation}
for some index set $\calE^{+\cdots +}$, 
where $\mathrm{fd}$ denotes the midpoint of $\mathrm{ff}$.  Specifically, 
    \begin{equation} \label{eq:E++_form}
    \mathcal{E}^{+\cdots++}=\{(n-j,\ell):n \in \bbN,0\leq \ell\leq j\}. 
    \end{equation} 
\end{proposition}
\begin{proof}
    The mapping properties of $N^{-1}$ vis-\`a-vis polyhomogeneity are described in \Cref{prop:intmain}, in the appendix. 
    Using that proposition and the fact that $E \in \operatorname{Diff}_{\mathrm{b}}^{2,-2,0}(Y)$ (and thus induces two orders of decay at $\mathrm{bf}$), 
    \begin{equation}\label{eq:en}
		EN^{-1} :  \mathcal{A}^{(\alpha,0),\mathcal{E}}(Y\backslash \mathrm{fd}) \to \mathcal{A}^{(\alpha+1,0),\calE^+}( Y\backslash \mathrm{fd}).
	\end{equation}
    Now $\tilde f \in \rho_{\mathrm{bf}}^L\rho_{\mathrm{ff}}^LC^{\infty}(Y) = \calA^{(L,0),\calE}$, for $\calE = (L,0)$. Thus, 
    \begin{equation}
        (EN^{-1})^j \tilde{f} \in \calA^{(L+j),\calE^{+\cdots +}}(Y\backslash \mathrm{fd}), 
    \end{equation}
    and thus $N^{-1}(EN^{-1})^j\tilde f \in \mathcal{A}^{(L-1+j,0),\mathcal{E}^{+\cdots ++}}(Y\backslash \mathrm{fd})$. There are $j+1$ plus symbols in the final superscript. An induction yields \cref{eq:E++_form}. 
\end{proof}

It turns out that the log terms vanish. This (unexpected) fact, which depends on particular algebraic cancellations, is recorded in an appendix, \S\ref{sec:logs}.\footnote{
Otherwise, could only conclude polyhomogeneity. Polyhomogeneity suffices for the rest of our main argument --- our main theorem would remain unchanged, except it would have to allow more log terms. Thus, we consider the absence of logs an inessential point. }
We can therefore conclude:
\begin{propositionp}\label{prop:mainI_init}
    For $\tilde{f}$ as above, $N^{-1} (EN^{-1})^j \tilde{f} \in \rho_{\mathrm{bf}}^{L-1+j} \rho_{\mathrm{ff}}^{-j} C^\infty(Y\backslash \mathrm{fd})$.
\end{propositionp}

Finally:
\begin{proposition}\label{cor:mainI} For $\tilde f$ defined in \cref{eq:f_init} and $j \geq 0$, 
\begin{equation}\label{eq:misc_eskimo2}
    N^{-1}(EN^{-1})^j\tilde f \in \rho_{\mathrm{bf}}^{L-1+j}\rho_{\mathrm{ff}}^{L-1}C^{\infty}(X \backslash \mathrm{fd}), 
\end{equation}
where $\mathrm{fd}$ denotes the midpoint of $\mathrm{ff}$.  
\end{proposition}
\begin{proof}
    Beginning with \Cref{prop:mainI_init},  we then pass to $X$, via $\tilde{Y}$. It suffices to analyze the behavior of $N^{-1}(EN^{-1})^j\tilde{f}$ on $\tilde{Y}$ away from the old front face of $Y$.
    Index sets add at the new face formed from the blow-up, so the order at the new front face is the sum $(L-1+j)-j=L-1$ of the old orders. This gives \cref{eq:misc_eskimo2}.
\end{proof}

\subsection{Proof of Proposition \ref{prop:I_main}}\label{subsec:stepI} 
Define the formal sum
\begin{equation}
    \tilde{\mathsf{v}}\coloneqq \sum_{j=0}^{\infty} N^{-1}(-EN^{-1})^j\tilde f = \sum_{j=0}^{\infty}(-N^{-1}E)^{j+1} [1] .
\end{equation}
This formally satisfies $\tilde P  \tilde{\mathsf{v}}=\tilde f$.

By \Cref{cor:mainI}, the $j$th term in the summand belongs to $\rho_{\mathrm{bf}}^{L-1+j}\rho_{\mathrm{ff}}^{L-1}C^{\infty}(X)$. Asymptotically summing $\tilde{\mathsf{v}}$ (say, using \Cref{lem:Borel}), and localizing near $\mathrm{bf}$ we obtain a function 
\begin{equation} 
    \tilde{v}\in \rho_{\mathrm{bf}}^{L-1}\rho_{\mathrm{ff}}^{L-1}C^{\infty}(X)
\end{equation} 
supported away from the midpoint of $\mathrm{ff}$. 
By construction, $\tilde P \tilde{v}-\tilde f \in \rho_{\mathrm{bf}}^{\infty}\rho_{\mathrm{ff}}^{L+1}C^{\infty}(X)$, 
which is the desired result. 

\begin{remark*}
    For the reader not familiar with these sorts of Frobenius constructions, we expand on the final step in the argument above. 
Recall that $\tilde P \in \operatorname{Diff}_{\mathrm{b}}^{2,-1,-2}(X)$. Also, $\tilde f \in \rho_{\mathrm{bf}}^L\rho_{\mathrm{ff}}^{2L}C^{\infty}(X)$, because 
\begin{equation} 
    \rho_{\mathrm{bf}}^LC^{\infty}(\overline{\bbR^d}) \subseteq \rho_{\mathrm{bf}}^L\rho_{\mathrm{ff}}^{2L}C^{\infty}(X).
\end{equation}     
It follows that
\begin{equation} 
    \tilde P \tilde{v} -\tilde f \in \rho_{\mathrm{bf}}^{L}\rho_{\mathrm{ff}}^{L+1}C^{\infty}(X).
\end{equation} 
Thus, it makes sense to speak of the Taylor expansion of $\smash{\rho_{\mathrm{bf}}^{-L} (\tilde{P}\tilde{v}-\tilde{f})}$ at $\mathrm{bf}$. 

The algebraic construction of $\tilde{\mathsf{v}}$ guarantees that this Taylor expansion vanishes term-by-term. 
Indeed, for any $J\in \bbN$, 
\begin{equation}
    \tilde{v} - \sum_{j=0}^J N^{-1}(-EN^{-1})^j \tilde{f} \in \rho_{\mathrm{bf}}^{L+J} \rho_{\mathrm{ff}}^{L-1} C^\infty(X) , 
\end{equation}
by construction. Applying $\tilde{P}$ to both sides yields
\begin{equation}
    \tilde{P} \tilde{v} -  \sum_{j=0}^J \tilde{P} N^{-1}(-EN^{-1})^j  \tilde{f} \in \rho_{\mathrm{bf}}^{L+J+1} \rho_{\mathrm{ff}}^{L+1} C^\infty(X). 
\end{equation}
The sum telescopes:
\begin{align}
\begin{split} 
    \sum_{j=0}^J \tilde{P} N^{-1}(-EN^{-1})^j \tilde{f} &= \sum_{j=0}^J (N+E) N^{-1}(-EN^{-1})^j \tilde{f} \\
    &= \sum_{j=0}^J ((-EN^{-1})^j  - (-EN^{-1})^{j+1}) \tilde{f} =  \tilde{f} - \underbrace{(-E N^{-1})^{J+1} \tilde{f}}_{\in \rho_{\mathrm{bf}}^{L+1+J} \rho_{\mathrm{ff}}^{L+1} C^\infty(X)  }.\end{split} 
\end{align}
The claim that the last term lies in $\rho_{\mathrm{bf}}^{L+1+J} \rho_{\mathrm{ff}}^{L+1} C^\infty(X)$ follows from \Cref{cor:mainI} and the fact that $E\in \operatorname{Diff}_{\mathrm{b}}^{2,-2,-2}(X)$, hence induces two orders of decay at each boundary hypersurface. 

So, we conclude that 
\begin{equation}
    \tilde{P}\tilde{v} - \tilde{f} \in \rho_{\mathrm{bf}}^{L+J+1} \rho_{\mathrm{ff}}^{L+1} C^\infty(X). 
\end{equation}
Since $J$ can be taken arbitrarily large, we are done.

\end{remark*}

\section{Step two: the parabolic juncture (front face model problem)}
\label{sec:ff}

Next, we turn to the second step of the construction.
The PDE we would like to solve is 
\begin{equation}
    \hat{P} \hat{v}_{2} \approx \hat{f},
\end{equation} 
where $\hat{f}\in  \calA^{\infty,(L+1,0)}(X) = \rho_{\mathrm{bf}}^\infty \rho_{\mathrm{ff}}^{L+1} C^\infty(X)$, for $L\geq 2$, and $\hat{P} = e^{-ir}Pe^{ir}$.
The $\approx$ in ``$\hat{P} v \approx f$'' means we are allowing an error of the form $f_1 \in \langle r \rangle^{-(d+3)/2} C^\infty(\overline{\bbR^d})$. Thus, $\hat{P}v=f+f_1$. We will prove:

\begin{proposition}
	\label{prop:II_main}
	Let $L\in \bbN^{\geq 2}$. There exists an index set $\calF\subset \bbN^2$ such that, given $\hat{f}\in \rho_{\mathrm{bf}}^{(d+3)/2 } \rho_{\mathrm{ff}}^{L+1} C^\infty(X)$, there exists a function 
    \begin{equation}
        \hat{v} \in \rho_{\mathrm{bf}}^{(d-1)/2}\rho_{\mathrm{ff}}^{L-1} \calA^{ (0,0) ,\calF}(X)  
    \end{equation}
    such that $\hat{P} \hat{v}-\hat{f} \in \langle r \rangle^{-(d+3)/2} \calA^{(0,0),\infty}(X)\subseteq 
    \langle r \rangle^{-(d+3)/2} C^\infty(\overline{\bbR^d}).$
\end{proposition}

To do so, it will be convenient to consider ``the'' b-normal operator $N=N_{\mathrm{ff}}(\hat{P})$ of $\hat{P}$ at $\mathrm{ff}$, defined by
\begin{equation}
			N  :=\triangle_\bfy- \frac{2i}{r} \Big(x\frac{\partial}{\partial x} + \bfy\cdot \nabla_\bfy \Big)  -\frac{i(d-1)}{r} \in \operatorname{Diff}_{\mathrm{b}}^{2,-1,-2}(X).
			\label{eq:NffhatP}
\end{equation} 
This being ``the'' b-normal operator of $\hat{P}$ at $\mathrm{ff}$ means that $N-\hat{P}$ has one more order of b-decay at $\mathrm{ff}$ than $\hat{P} \in \operatorname{Diff}_{\mathrm{b}}^{2,-1,-2}(X)$
itself, i.e.
$N-\hat{P} \in \operatorname{Diff}_{\mathrm{b}}^{2,-1,-3}(X) $. In this case, more is true: 
\begin{equation}\label{eq:N_tot}
    N -\hat{P} = \underbrace{(N-\hat{P}_0)}_{\mathclap{=\partial_x^2 \in \operatorname{Diff}_{\mathrm{b}}^{2,-2,-4}}} - \overbrace{(\hat{P}-\hat{P}_0 )}^{\mathclap{\in \operatorname{Diff}_{\mathrm{b}}^{2,-2,-4}\text{ by \Cref{lem:rem}.}}} \in \operatorname{Diff}_{\mathrm{b}}^{2,-2,-4}(X) .
\end{equation}
Thus, $N$ captures $\hat{P}$ modulo subsubprincipal errors at $\mathrm{ff}$ (it turns out the observation 
\begin{equation} 
    N-\hat{P} \in \operatorname{Diff}_{\mathrm{b}}^{2,-2,-3}(X)
\end{equation}
would suffice for our purposes).

We build $v$ by asymptotically summing a formal series in $\rho_{\mathrm{ff}}$ (which we will abbreviate $\rho$ below). Rather than directly inverting $N$ as we did in the previous step, we use the ``Frobenius method.'' This means that the terms $v_\bullet$ in the series are \emph{functions of $\hat{\bfy}=\bfy/x^{1/2}$ alone} and given by solving an inhomogeneous problem 
of the form 
\begin{equation} 
        N  (\rho^\bullet v_\bullet(\hat{\bfy})) = F_\bullet,
\end{equation}
for  $F_\bullet$ some particular function given in terms of the previous $v_\bullet$'s and the terms $f_\bullet$ in $f$'s polyhomogeneous expansion at $\mathrm{ff}$.
We can use 
\begin{equation} 
    \rho=\frac{\sqrt{x+y^2}}{x}
\end{equation} 
as a boundary-defining-function of $\mathrm{ff}$ valid near all of $\mathrm{ff}$, including the midpoint. 
In addition, we will use $\hat{\bfy}\in \bbR^{d-1}$
to parametrize $\mathrm{ff}^\circ$ itself, as well as the other level sets of $\rho$. 
As a locally valid boundary-defining-function for $\mathrm{bf}$, we can take 
\begin{equation} 
    \varrho = \frac{x}{ x+y^2}=\frac{1}{1+\hat{y}^2},
\end{equation} 
where $\hat{y}=|\hat{\bfy}|$. Keep in mind that $\rho_{\mathrm{bf}}\sim 1/\hat{y}^2$ when $\hat{y}\gg 1$, not $\rho_{\mathrm{bf}} \sim 1/\hat{y}$ (since $\rho_{\mathrm{bf}}\sim 1/r$ away from the forward direction).
So, 
\begin{equation}
		X\cap \operatorname{cl}_X\{\lVert \bfy \rVert \leq x\} =  [0, 1)_\rho \times  (\overline{\bbR^{d-1}_{\hat{\bfy}}})_{2} 
\end{equation}
canonically, where the `$2$' in  $\smash{(\overline{\bbR^{d-1}_{\hat{\bfy}}})_{2}}$ just means that we are radially compactifying $\smash{\bbR^{d-1}_{\hat{\bfy}}}$ and then changing the smooth structure at infinity so that, instead of $\varrho^{1/2}$ being a boundary-defining-function, $\varrho$ becomes a boundary-defining-function (of $\mathrm{ff}$). This change of smooth structure is an inessential point conceptually speaking, but must be remembered when bookkeeping.

Central to carrying out the strategy above is a certain family $\{I(j)\}_{j\in \bbR}$ of operators, 
\begin{equation} 
	I(j)\in \operatorname{Diff}^2(\bbR^{d-1}_{\hat{\bfy}}) =  \operatorname{Diff}^2(\mathrm{ff}^\circ ),
\end{equation} 
\emph{on $\mathrm{ff}$}, the ``indicial family,'' which captures how $N$ acts on asymptotic series in $\rho_{\mathrm{ff}}=\rho$ valued in $C^\infty(\mathrm{ff}^\circ)$. Its definition is simply
\begin{equation}
	I(j) g(\hat{\bfy}) = M_{\rho^{-j}} (x N) M_{\rho^j} g(\hat{\bfy}) = \rho^{-j} (xN) (\rho^j g(\hat{\bfy}))
	\label{eq:I_def}
\end{equation}
So, $I(j)$ is formed by writing $xN$ in terms of the coordinate system $\rho,\hat{\bfy}$ and then replacing $\rho\partial_\rho$ by $j$; the factor of $x$ in \cref{eq:I_def} guarantees that the result has coefficients \emph{without any $\rho$'s}. In this way, we trade one independent variable, $\rho$, for a parameter, $j\in \bbR$, getting a one-parameter family of differential operators with fewer independent variables. Then,  $I(j)$ can be considered a differential operator on \smash{$\bbR^{d-1}_{\hat{\bfy}}$}, or, equivalently, on $\mathrm{ff}^\circ$.\footnote{We will not distinguish between $\bbR^{d-1}_{\hat{\bfy}},\mathrm{ff}^\circ$ below.}

In order to accomplish the goal of this section, we need a right inverse $I(j)^{-1}$ (on a sufficiently large domain).
Key to our analysis is that $I$ is a conjugated form of the spectral family of the \emph{quantum inverted harmonic oscillator} (QIHO): 
\begin{equation}
    L(j)=\triangle_{\hat{\bfy}} - \frac{\hat{y}^2}{4}  + i\Big(j-\frac{d-1}{2}\Big).
\end{equation}
This is worked out in \S\ref{subsec:I_comp}, first in a computationally simplified setting where $\rho$ is replaced in the definition \cref{eq:I_def} of $I$ by $x^{-1/2}$. This simplifies the algebra, but the resultant proxy $\tilde{I}$ for $I$ is only satisfactory away from $\mathrm{bf}$.
The actual $I$ is treated afterward. It can either be handled from scratch, or otherwise related to $\tilde{I}$. The two families are related by conjugation by a power of $\varrho$, due to the identity $x^{-1}=\varrho \rho^2$.
In \S\ref{subsec:indicial}, we deduce mapping properties of $I(j)$ from those of $L(j)$.
Invertibility between appropriate function spaces, modulo a finite-dimensional obstruction, will come from a variant of the limiting absorption principle in which the \emph{spring constant}, rather than the spectral parameter, is given a small imaginary part:
\begin{equation}
    L(j)\rightsquigarrow \triangle_{\hat{\bfy}} - (1+i\varepsilon) \frac{\hat{y}^2}{4}  + i\Big(j-\frac{d-1}{2}\Big).
\end{equation}
This procedure produces a (partial) right-inverse to $L(j)$.\footnote{For most $j$ of interest to us, this is \emph{not} the $L^2$-bounded right-inverse which the essential self-adjointness of the QIHO guarantees exists for almost all $j\in \bbC$. The latter produces problematic oscillations at large-$\hat{y}$ which prevent it from being used in the construction of our perturbed plane waves. We will comment more on this below.} 
Undoing the conjugation, we get the desired (partial) right-inverse to $I(j)$.  
The already mentioned finite-dimensional obstruction is discussed in \S\ref{subsec:obstruction}.
This is the source of $\log \rho_{\mathrm{ff}}$ terms in our main theorem (which we have already seen may be present).

These ingredients are assembled in \S\ref{subsec:v2_construction}  to prove \Cref{prop:II_main}.

\subsection{Computation of $I(j)$}
\label{subsec:I_comp}
\subsubsection{Preliminary calculation}
\label{subsubsec:QHO_prelim}
First, we compute $N= \smash{N_{\mathrm{ff}}(\hat{P})}$ in terms of the coordinates $t=x^{-1/2},\smash{\hat{\bfy}}$. The coordinate $t$ serves as a local boundary-defining-function for $\mathrm{ff}$ away from $\mathrm{bf}$. 
The result of the rewriting is
\begin{align}
\begin{split} 
	rN =  \underbrace{\frac{r}{x}\triangle_{\hat{\bfy}}}_{\approx \triangle_{\hat{\bfy}}}+i  \Big( t \frac{\partial}{\partial t} - \hat{\bfy}\cdot \nabla_{\hat{\bfy}} \Big) - i(d-1)  
	,
	\end{split} 
	\label{eq:misc_yyy}
\end{align}
where 
\begin{equation}
	\triangle_{\hat{\bfy}} = - \sum_{j=1}^{d-1}\frac{\partial^2}{\partial \hat{y}_j^2}, \qquad \nabla_{\hat{\bfy}}=(\partial_{\hat{y}_1},\dots,\partial_{\hat{y}_{d-1}}),
\end{equation}
and where the ``$\approx$'' means that we are throwing away $(r-x)x^{-1} \triangle_{\hat{\bfy}}$, which lies in $\operatorname{Diff}_{\mathrm{b}}^{2,0,-2}(X)$, and is therefore lower-order than 
\begin{equation} 
	rN\in \operatorname{Diff}_{\mathrm{b}}^{2,0,0}
\end{equation} 
at $\mathrm{ff}$. 
Since $rN$ is only defined modulo $\operatorname{Diff}_{\mathrm{b}}^{2,0,-1}(X)$ anyways (for now), we can conflate
\begin{equation}
rN=\triangle_{\hat{\bfy}}+i  \Big( t \frac{\partial}{\partial t} - \hat{\bfy}\cdot \nabla_{\hat{\bfy}} \Big) - i(d-1)  .
\end{equation}

Instead of $I(j) = \rho^{-j}(xN) \rho^j$, let us compute $\tilde{I}(j) = t^{-j} (rN) t^j$.
The effect of the $t\partial_t$ term in  \cref{eq:misc_yyy} will be to contribute to $\tilde{I}$ the multiplication operator $\bullet\mapsto ij\bullet $, which we can group with $i(d-1)$ to get some complex constant $C=ij - i(d-1)$.
Then, the rest of the right-hand side of \cref{eq:misc_yyy} is the \emph{Ornstein--Uhlenbeck operator} in $d-1$ dimensions,  
\begin{align}
	\begin{split}  
	\tilde{I}(j)&=\triangle_{\hat{\bfy}} + \nu \hat{\bfy}\cdot \nabla_{\hat{\bfy}} + C \\
	&: g(\hat{\bfy})\mapsto t^{-j} (r N) (t^j  g(\hat{\bfy}))
	\end{split} 
	\label{eq:tildeIdef}
\end{align}  
with an imaginary friction $\nu=-i$.
The nature of $\tilde{I}$ becomes clearer upon passing to the conjugated operator 
\begin{equation}
	L(j) \coloneqq  M_{e^{i  \hat{y}^2/4}} \tilde{I}(j) M_{e^{ -i  \hat{y}^2/4}}  = \triangle_{\hat{\bfy}} - \frac{\hat{y}^2}{4}+ i\Big(j -\frac{(d-1)}{2}\Big), \quad \hat{y}^2 = |\hat{\bfy}|^2.
	\label{eq:misc_0tt}
\end{equation} 
This is the inverted quantum harmonic oscillator (QIHO) at imaginary ``energy'' $i(j-(d-1)/2)$. 

Recall that we are using the positive semi-definite Laplacian $\triangle_{\hat{\bfy}}$. So, the potential term $-\hat{y}^2/4$ in \cref{eq:misc_0tt} counteracts the ``kinetic'' term $\smash{\triangle_{\hat{\bfy}}}$. This is why $L(j)$ is the \emph{inverted} QHO --- the ordinary QHO would have $+\smash{\hat{y}^2}/4$ instead. 
The coefficient $-1/4$ is the ``spring constant.''

\subsubsection{Full calculation}
\label{subsubsec:I_full}

\begin{lemma}
	\label{lem:coordinate_change}
	Under the coordinate transformation from $(x,\bfy)$ to $(\rho,\hat{\bfy})=(x^{-1} (x+y^2)^{1/2}, \bfy/x^{1/2})$:
	\begin{align}
		\frac{\partial}{\partial x} &=  - \rho^3 \varrho\Big(1-\frac{\varrho}{2}\Big) \frac{\partial}{\partial \rho} -\sum_{j=1}^{d-1}\frac{\hat{y}_j}{2} \rho^2\varrho \frac{\partial}{\partial \hat{y}_j},
		\label{eq:misc_p11}  \\
		\frac{\partial}{\partial y_j} &=  \hat{y}_j \rho^2 \varrho^{3/2}\frac{\partial}{\partial \rho}+ \rho \varrho^{1/2} \frac{\partial}{\partial \hat{y}_j}. \label{eq:misc_p12}
	\end{align}
\end{lemma}
\begin{proof}
	From the Chain Rule, we get
	\begin{align}
		\begin{split} 
			\frac{\partial}{\partial x} &= \frac{\partial \rho^2}{\partial x}\frac{\partial}{\partial \rho^2} + \sum_{j=1}^{d-1}\frac{\partial \hat{y}_j}{\partial x}\frac{\partial}{\partial \hat{y}_j} = - \rho^4 \varrho(2-\varrho) \frac{\partial}{\partial \rho^2} -\sum_{j=1}^{d-1}\frac{\hat{y}_j}{2} \rho^2\varrho \frac{\partial}{\partial \hat{y}_j}, \\
			\frac{\partial}{\partial y_j} &=  \frac{\partial \rho^2}{\partial y_j}\frac{\partial}{\partial \rho^2}+\frac{\partial\hat{y}_j}{\partial y_j}\frac{\partial}{\partial \hat{y}_j}= 2 \hat{y}_j \rho^3 \varrho^{3/2}\frac{\partial}{\partial \rho^2}+ \rho \varrho^{1/2} \frac{\partial}{\partial \hat{y}_j}.
		\end{split} 
	\end{align}
    Equations \eqref{eq:misc_p11} and \eqref{eq:misc_p12} then follow upon noting that $2\rho\frac{\partial}{\partial\rho^2} = \frac{\partial}{\partial\rho}$.
\end{proof}
\begin{lemma} 
	Near $\mathrm{ff}$, the operator $N$ defined by \cref{eq:NffhatP} is given by 
	\begin{multline}
		N= \frac{1}{r}\Big[ - \hat{y}^2 \varrho^2 (\rho\partial_\rho)^2  - 2 \varrho (\rho \partial_\rho ) (\hat{\bfy}\cdot \nabla_{\hat{\bfy}})+ \triangle_{\hat{\bfy}}+( i\varrho   - (d-3)\varrho - 2\varrho^2 ) (\rho \partial_\rho)\\ -i\hat{\bfy} \cdot \nabla_{\hat{\bfy}}  - i(d-1)	\Big]  \bmod \operatorname{Diff}_{\mathrm{b}}^{2,-2,-4}.
		\label{eq:new_rN_0}
	\end{multline}
	Consequently, near $\mathrm{ff}$, 
	\begin{multline}
		N= \varrho \rho^2\Big[ - \hat{y}^2 \varrho^2 (\rho\partial_\rho)^2  - 2 \varrho (\rho \partial_\rho ) (\hat{\bfy}\cdot \nabla_{\hat{\bfy}})+ \triangle_{\hat{\bfy}}+( i\varrho   - (d-3)\varrho - 2\varrho^2 ) (\rho \partial_\rho)\\ -i\hat{\bfy} \cdot \nabla_{\hat{\bfy}}  - i(d-1)	\Big]  \bmod \operatorname{Diff}_{\mathrm{b}}^{2,-1,-4}.
		\label{eq:new_rN}
	\end{multline}
\end{lemma}
\begin{proof}
	We want to rewrite $rN= - 2 i r\partial_r - i(d-1) + x \triangle_{\bfy} + (r-x)\triangle_{\bfy}$
	in terms of $\rho,\hat{\bfy}$, near $\mathrm{ff}$. All computations below are near $\mathrm{ff}$. For example, when we write $f\in C^\infty(X)$, we really mean $f\in C^\infty(U)$ for some open neighborhood $U\supset \mathrm{ff}$. 
	\begin{itemize}
		\item First, we rewrite $r\partial_r = x\partial_x + \bfy\cdot \nabla_{\bfy}$. Using \Cref{lem:coordinate_change}, 
		\begin{equation}
			\bfy \cdot \nabla_{\bfy} =  \frac{\hat{\bfy}}{\rho \varrho^{1/2}}\cdot \nabla_{\bfy} =  \hat{y}^2 \rho \varrho \frac{\partial}{\partial \rho}+  \hat{\bfy}\cdot \nabla_{\hat{\bfy}}.
			\label{eq:misc_uuu}
		\end{equation}
		Inserting this into the definition of $r\partial_r$, we get
		\begin{equation}
			r\partial_r =   -\frac{\rho \varrho}{2} \frac{\partial}{\partial \rho} +\frac{\hat{\bfy}}{2}  \cdot \nabla_{\hat{\bfy}}.
		\end{equation}
		\item Next, we rewrite $x\triangle_{\bfy}=-x\sum_{j=1}^{d-1} \partial_{y_j}^2$. 
		The individual $\partial_{y_j}^2$ are
		\begin{multline}
			\frac{\partial^2}{\partial y_j^2}= \hat{y}_j^2 \rho^4 \varrho^{3}\frac{\partial^2}{\partial \rho^2}+2 \hat{y}_j \rho^3 \varrho^2 \frac{\partial^2}{\partial \rho \partial \hat{y}_j}+ \rho^2 \varrho \frac{\partial^2}{\partial \hat{y}_j^2} \\ +\rho^3 \varrho\Big[ 2\hat{y}_j^2\varrho^2 + \varrho + \frac{3}{2} \hat{y}_j \frac{\partial \varrho}{\partial \hat{y}_j}   \Big]\frac{\partial}{\partial \rho}  + \rho^2  \Big[ \hat{y}_j  \varrho^{2}+\frac{1}{2}  \frac{\partial \varrho}{\partial \hat{y}_j} \Big]\frac{\partial}{\partial \hat{y}_j}.
		\end{multline}
		Because $\partial_{\hat{y}_j}\varrho = - 2 \hat{y}_j (1+\hat{y}^2)^{-2}$, the final term is exactly zero. Summing over $j$:
		\begin{align}
			\begin{split} 
			\triangle_{\bfy}&= -\hat{y}^2 \rho^4 \varrho^{3}\frac{\partial^2}{\partial \rho^2} - 2 \rho^3 \varrho^2  \hat{\bfy}\cdot \nabla_{\hat{\bfy}}\partial_\rho+ \rho^2 \varrho \triangle_{\hat{\bfy}}- \rho^3 \varrho^2 ( (d-2)  +\varrho )\frac{\partial}{\partial \rho} \\
			x\triangle_{\bfy} &=  -\hat{y}^2 \rho^2 \varrho^{2}\frac{\partial^2}{\partial \rho^2} - 2 \rho \varrho \hat{\bfy}\cdot \nabla_{\hat{\bfy}}\partial_\rho+  \triangle_{\hat{\bfy}}- \rho \varrho ( (d-2)  +\varrho )\frac{\partial}{\partial \rho} ,  
			\end{split} 
		\end{align}
		since $x^{-1}=\varrho \rho^{2}$. 
	\end{itemize} 
	Adding up these two terms, the result is $r\times$ the right-hand side of \cref{eq:new_rN}. So far, this has been exact.
	The last thing to note in order to get \cref{eq:new_rN_0} 
	is that \begin{equation} 
		(r-x)\triangle_{\bfy} \in \operatorname{Diff}^{2,-1,-2}_{\mathrm{b}}.
		\label{eq:misc_102071}
	\end{equation} 
	Indeed, since $x=r\cos(\theta)$, we have $x-r \in \rho_{\mathrm{bf}}^{-1} C^\infty(X)$.
	Combining this with $\triangle_{\bfy}\in \operatorname{Diff}_{\mathrm{b}}^{2,-2,-2}$, we get \cref{eq:misc_102071}.
	
	\Cref{eq:new_rN} follows from \cref{eq:new_rN_0} similarly.
\end{proof}

We now pass to the indicial family  
\begin{equation}
	I(j) = M_{\rho^{-j}} (x N) M_{\rho^j}=  \triangle_{\hat{\bfy}} -(i+2j\varrho)\hat{\bfy} \cdot \nabla_{\hat{\bfy}}
	- \hat{y}^2 \varrho^2  j^2 +( i\varrho - (d-3)\varrho - 2\varrho^2 ) j - i(d-1).
	\label{eq:misc_102074}
\end{equation} 
\begin{lemma}
	\label{lem:QHO_final}
     $L(j)= (e^{i \hat{y}^2/4} \varrho^{-j/2}) I(j) (e^{-i\hat{y}^2/4} \varrho^{j/2})$.
\end{lemma}
\begin{proof}
    Immediate from the computation in \S\ref{subsubsec:QHO_prelim} and $x^{-1} = \varrho \rho^2$. See \S\ref{subsubsec:alternative_QHO} for the computation done from scratch. 
\end{proof}

\subsubsection{Alternative derivation of the conjugated front face model problem}
\label{subsubsec:alternative_QHO}

We provide an alternative proof of \Cref{lem:QHO_final}, via a direct computation. This can be skipped on first reading.

We first compute the conjugation $e^{i\hat{y}^2/4}\varrho^{-j/2}\partial_{\hat{y}_k}e^{-i\hat{y}^2/4}\varrho^{j/2}$. Using $\partial_{\hat{y}_k}\varrho = -2\hat{y}_k\varrho^2$, we have
\begin{align}
\begin{split} 
e^{i\hat{y}^2/4}\varrho^{-j/2}\partial_{\hat{y}_k}e^{-i\hat{y}^2/4}\varrho^{j/2} &= \partial_{\hat{y}_k} + e^{i\hat{y}^2/4}\varrho^{-j/2}\partial_{\hat{y}_k}(e^{-i\hat{y}^2/4}\varrho^{j/2}) \\
&= \partial_{\hat{y}_k} -i\hat{y}_k/2+  -j\hat{y}_k\varrho = \partial_{\hat{y}_k} - (i/2+j\varrho)\hat{y}_k.
\end{split} 
\end{align}
It follows that $e^{i\hat{y}^2/4}\varrho^{-j/2}(\hat{\bfy} \cdot \nabla_{\hat{\bfy}})e^{-i\hat{y}^2/4}\varrho^{j/2} = \hat{\bfy} \cdot \nabla_{\hat{\bfy}} - (i/2+j\varrho)\hat{y}^2$. Moreover,
\begin{align}
\begin{split} 
e^{i\hat{y}^2/4}\varrho^{-j/2}\partial_{\hat{y}_k}^2e^{-i\hat{y}^2/4}\varrho^{j/2} &= (\partial_{\hat{y}_k} - (i/2+j\varrho)\hat{y}_k)^2 \\
&= \partial_{\hat{y}_k}^2 - 2(i/2+j\varrho)\hat{y}_k\partial_{\hat{y}_k} - \partial_{\hat{y}_k}((i/2+j\varrho)\hat{y}_k) + ((i/2+j\varrho)\hat{y}_k)^2 \\
&= \partial_{\hat{y}_k}^2 - (i+2j\varrho)\hat{y}_k\partial_{\hat{y}_k} - j(-2\hat{y}_k\varrho^2)\hat{y}_k - (i/2+j\varrho) + \hat{y}_k^2(i/2+j\varrho)^2.
\end{split} 
\end{align}
Hence,
\begin{align}
\begin{split} 
e^{i\hat{y}^2/4}\varrho^{-j/2}\triangle_{\hat{\bfy}}e^{-i\hat{y}^2/4}\varrho^{j/2} &= -\sum_{k=1}^{d-1} e^{i\hat{y}^2/4}\varrho^{-j/2}\partial_{\hat{y}_k}^2e^{-i\hat{y}^2/4}\varrho^{j/2} \\
&= -\sum_{k=1}^{d-1} \left(\partial_{\hat{y}_k}^2 - (i+2j\varrho)\hat{y}_k\partial_{\hat{y}_k} + 2j\hat{y}_k^2\varrho^2 - (i/2+j\varrho) + \hat{y}_k^2(i/2+j\varrho)^2\right),
\end{split} 
\end{align}
which is 
\begin{equation}
    \triangle_{\hat{\bfy}} + (i+2j\varrho)\hat{\bfy} \cdot \nabla_{\hat{\bfy}} - 2j\hat{y}^2\varrho^2 + \left(\frac{i}{2}+j\varrho\right)(d-1) - \hat{y}^2\left(\frac{i}{2}+j\varrho\right)^2.
\end{equation}
Thus, 
\begin{align}
\begin{split} 
L(j) &= e^{i\hat{y}^2/4}\varrho^{-j/2}\triangle_{\hat{\bfy}}e^{-i\hat{y}^2/4}\varrho^{j/2} \\
&\quad -(i+2j\varrho)e^{i\hat{y}^2/4}\varrho^{-j/2}(\hat{\bfy} \cdot \nabla_{\hat{\bfy}})e^{-i\hat{y}^2/4}\varrho^{j/2} - \hat{y}^2\varrho^2 j^2 + (i\varrho - (d-3)\varrho-2\varrho^2)j-i(d-1) \\
&= \triangle_{\hat{\bfy}} + (i+2j\varrho)\hat{\bfy} \cdot \nabla_{\hat{\bfy}} - 2j\hat{y}^2\varrho^2 + (i/2+j\varrho)(d-1) - \hat{y}^2(i/2+j\varrho)^2 \\
&\quad -(i+2j\varrho)(\hat{\bfy} \cdot \nabla_{\hat{\bfy}} - (i/2+j\varrho)\hat{y}^2)- \hat{y}^2\varrho^2 j^2 + (i\varrho - (d-3)\varrho-2\varrho^2)j-i(d-1) \\
&= \triangle_{\hat{\bfy}} - \hat{y}^2/4 + ij\hat{y}^2\varrho + ij\varrho - (d-1)j\varrho - i(d-1)/2 + j\varrho(d-1),
\end{split} 
\end{align}
which simplifies (after using $(\hat{y}^2+1)\varrho=1$) to
\begin{equation}
    L(j) =  \triangle_{\hat{\bfy}} - \frac{\hat{y}^2}{4} + ij - i(d-1)/2 = \triangle_{\hat{\bfy}} - \frac{\hat{y}^2}{4} + i\left(j-\frac{d-1}{2}\right). 
\end{equation}

\subsection{The mapping properties of $L(j)$}\label{subsec:indicial}

The operator $L(j)$ is  in the spectral family $\{H-\lambda\}_{\lambda\in \bbC}$ of the QIHO Hamiltonian
\begin{equation} 
    H=\triangle_{\hat{\bfy}}-\hat{y}^2/4,
\end{equation} 
evaluated at spectral parameter $\lambda=-i(j-(d-1)/2)$.

The QIHO Hamiltonian $H$ is essentially self-adjoint acting on $C_{\mathrm{c}}^\infty(\bbR^{d-1})$, with respect to the $L^2$ inner product; see \cite[Thm. X.38]{ReedSimon2}. 
As a consequence of essential self-adjointness, for all $j\in \bbC$ except those of the form $(d-1)/2+i\bbR$, 
\begin{equation}
	L(j)^{-1} =\bigg[ \triangle_{\hat{\bfy}}-\frac{\hat{y}^2}{4} +i\Big(j-\frac{d-1}{2}\Big)\bigg]^{-1}
	\label{eq:L2_bounded_inverse}
\end{equation}
is well-defined, via the functional calculus, as a bounded operator on $L^2$, mapping $L^2$ into the domain $\calD(H)\subset L^2$ on which $H$ extends to be self-adjoint. This defines a two-sided inverse to $L(j)$.

\begin{figure}[!htbp]
	\includegraphics[scale=.66]{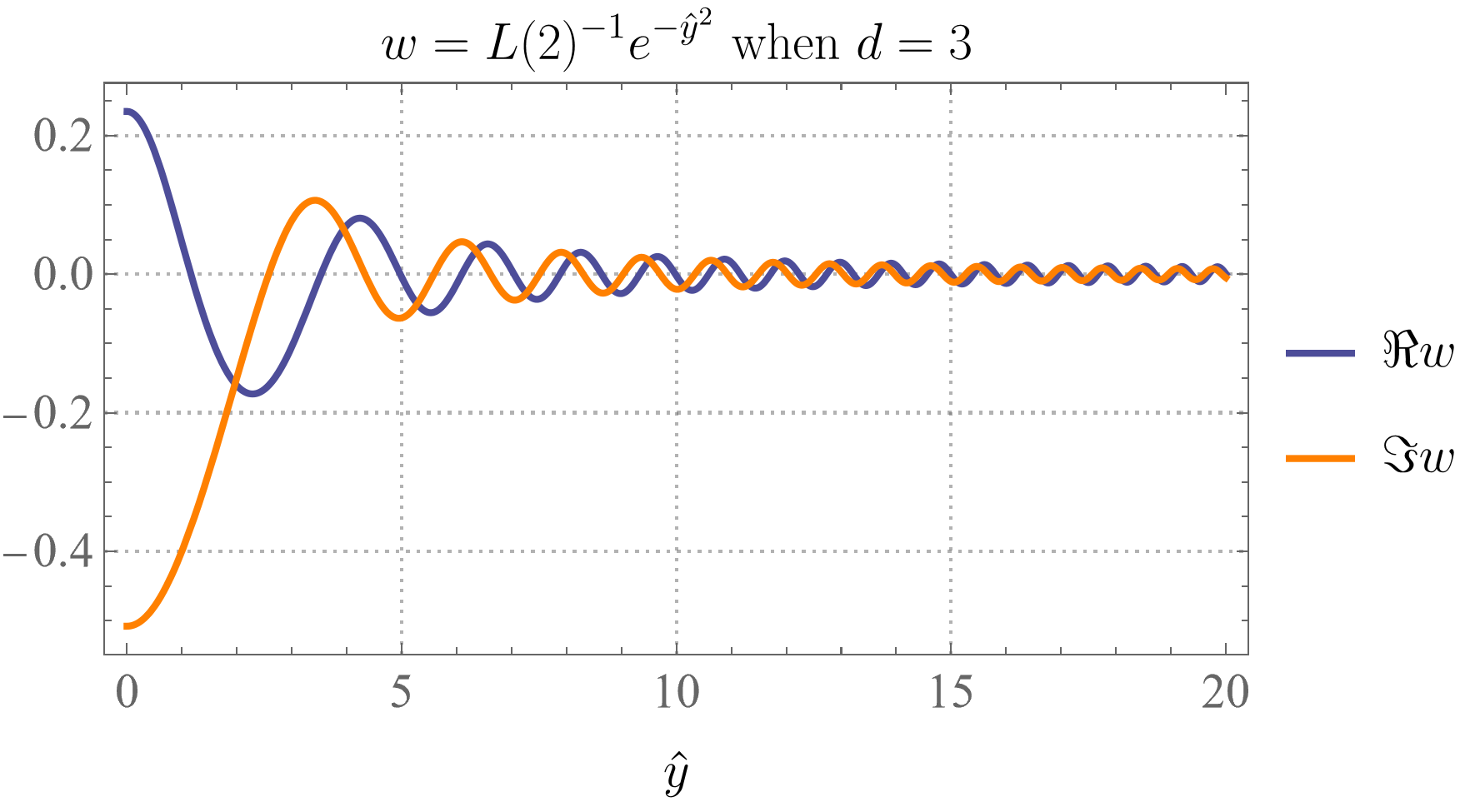}
	\caption{The function $w=L(j)^{-1} f$ (computed numerically) vs. $\hat{y}$ when $d=3$, $j=2$, for $f=e^{-\hat{y}^2/4}$ a Gaussian. }
	\label{fig:L(j)inv}
\end{figure}

Unfortunately, this is not the solution that we want when constructing perturbed plane waves.
The reason, the full justification of which we leave to the appendix (\S\ref{subsec:QHO_separation}), is that $w=L(j)^{-1} f$ oscillates like 
\begin{equation}
    e^{-i\hat{y}^2/4}\text{ as }\hat{y}\to\infty, 
\end{equation}
for most $j$ of interest to us, even if $f$ itself is nice (e.g.\ Schwartz).
These oscillations can be seen in \Cref{fig:L(j)inv}. 

Consequently, since we wanted to solve
\begin{equation}
    I(j)u = f\iff L(j)e^{i\hat{y}^2/4}\varrho^{-j/2}u = e^{i\hat{y}^2/4}\varrho^{-j/2}f,
\end{equation}
naively applying this $L(j)^{-1}$ gives $u = e^{-i\hat{y}^2/4}\varrho^{j/2}w$ (where $w=L(j)^{-1}\left( e^{i\hat{y}^2/4}\varrho^{-j/2}f\right)$), which is highly oscillatory (like $e^{-i\hat{y}^2/2}$), whereas our goal was polyhomogeneity.
An alternative right-inverse is needed. 

As already mentioned in the introduction to this section, we employ a variant of the limiting absorption principle in which the spring constant is given a small imaginary part $\pm i\varepsilon$ and then the limit $\varepsilon \to 0^+$ is taken. 
This results in two possible one-sided inverses,
\begin{equation} 
	R_\pm(j) = \lim_{\varepsilon \to 0^+} \bigg[\triangle_{\hat{\bfy}} +\Big(- \frac{1}{4}\pm i\varepsilon \Big) \hat{y}^2 +  i\Big(j-\frac{d-1}{2}\Big)\bigg]^{-1}
	\label{eq:R_pm_def}
\end{equation} 
where the limit is taken with respect to the strong operator topology between suitable weighted $L^2$-spaces. See  \S\ref{sec:QHO}, which includes a rigorous proof of the existence of the limit \cref{eq:R_pm_def} for most $j$. However, a finite-dimensional space must be excluded from the domain for some $j$, which include those of most interest for us.

The relation of $R_\pm(j)$ to the $L^2$-bounded inverse $L(j)^{-1}$ depends on the sign of $j-(d-1)/2$. Intuitively, one expects 
\begin{equation}
	L(j)^{-1} = 
	\begin{cases}
		R_-(j) & (\Re j<(d-1)/2), \\
		R_+(j) & (\Re j>(d-1)/2),
	\end{cases}
	\label{eq:resolvent_relations}
\end{equation}
because if the signs of the imaginary parts of $-1/4+i\varepsilon$ and $-\lambda=i(j-(d-1)/2)$ agree, then the two should work together. This intuition is correct.

The inverse we will use is $R_-(j)$, regardless of the sign of $j-(d-1)/2$, because it is this one that outputs the desired oscillations. See \S\ref{sec:QHO}.
An alternative heuristic: replacing 
\begin{equation} 
    -1/4 \rightsquigarrow -1/4\pm i \varepsilon
\end{equation} 
yields 
\begin{equation}
e^{i \hat{y}^2/4} \rightsquigarrow e^{-i \hat{y}^2 (-1/4\pm i \varepsilon)} = e^{i\hat{y}^2/4} e^{\pm \varepsilon \hat{y}^2 },
\end{equation}
and the right-hand side is in $L^2$ if and only if we choose the `$-$' case of `$\pm$.'

We need the mapping properties of $R_-(j)$ on certain function spaces. The question is:
\begin{itemize}
	\item[\textbf{Q.}] if $f\in e^{i\hat{y}^2/4}\calA^{\calE}(\mathrm{ff})=e^{i\hat{y}^2/4} \calA^\calE((\overline{\bbR^{d-1}_{\hat{\bfy}}})_2)$ for some (sufficiently fast decaying) index set $\calE$,  then what spaces does $R_-(j) f$ lie in?
\end{itemize}
The answer, for  those $j$ relevant to us, is:
\begin{proposition}
	\label{prop:QHO_mapping}
	On its domain, $R_-(j)$ maps $e^{i\hat{y}^2/4}\calA^{((d+1-j)/2,0) }\big(\mathrm{ff})\to  e^{i\hat{y}^2/4} \calA^{((d-1-j)/2,0) }\big(\mathrm{ff})$.
\end{proposition}
Again, the proof is relegated to \S\ref{sec:QHO}. 
We deduce:
\begin{corollary}
	\label{cor:main_II}
	The operator defined by $I(j)^{-1}=(e^{-i \hat{y}^2/4} \varrho^{j/2}) R_-(j) (e^{i\hat{y}^2/4} \varrho^{-j/2})$
    satisfies  
    \begin{equation} 
        I(j)^{-1} :\underbrace{ \langle \hat{y} \rangle^{-(d+1)} C^\infty((\overline{\bbR^{d-1}_{\hat{y}} })_2 ) }_{\calA^{((d+1)/2,0)}(\mathrm{ff})}\to  \underbrace{ \langle \hat{y} \rangle^{-(d-1)} C^\infty((\overline{\bbR^{d-1}_{\hat{y}} })_2 ) }_{\calA^{((d-1)/2,0)}(\mathrm{ff})}
    \end{equation} 
    on its domain. 
\end{corollary}

Although the proofs of the propositions above require the limiting absorption principle, the two most important numerological aspects, 
\begin{enumerate}[label=(\roman*)]
	\item the $(d-1)/2$ ``threshold,'' indicating the $\langle \hat{y} \rangle^{-(d-1)}\sim \rho_{\mathrm{bf}}^{(d-1)/2} $ large-$\hat{y}$ asymptotic behavior of $R_-(j)$'s output \emph{even on $C_{\mathrm{c}}^\infty$ input},
	\item  the relationship between the index sets on the domain and codomain --- the forcing has one order more decay at $\mathrm{bf}$ 
\end{enumerate}
can be understood using heuristics from the theory of regular singular ODE (i.e.\ b-PDE). The operator $L(j)$ is not regular singular at $\hat{y}=\infty$, where $L(j)^{-1}$ produces oscillations. Neither is $I(j)$. However, the normal operator 
\begin{equation}
	N_{\mathrm{bf}\cap\mathrm{ff}}(I(j)) = -i \hat{\bfy}\cdot \nabla_{\hat{\bfy}}   - i(d-1) \in \operatorname{Diff}^{1,0}_{\mathrm{b}}(\mathrm{ff}) 
	\label{eq:misc_102084a}
 \end{equation}
does yield, after separating variables, a regular singular ODE
\begin{equation} 
    O(j)=-i \hat{y}\partial_{\hat{y}} - i(d-1).\;\;\,
\end{equation}
The indicial root of $O(j)$, which governs thresholds, is the $c\in \bbC$ such that 
\begin{equation} 
    O(j) \hat{y}^{-2c} =0.
\end{equation} 
Then, $\hat{y}^{-2c}$ is the threshold decay rate. This ODE amounts to a linear equation for $c$, whose solution is $c=(d-1)/2$. Hence (i) holds.
It is because $N_{\mathrm{bf}\cap\mathrm{ff}}(I(j))$ has b-decay order 0 that (ii) holds. Indeed, this suggests that $I(j)^{-1}$ maps $\langle \hat{y} \rangle^{-(d-1)} C^\infty \to \langle \hat{y} \rangle^{-(d-1)} C^\infty$. This is not quite true, because being exactly at threshold entails a logarithmic loss. If we demand that the forcing has one more order of decay, then we avoid the logarithmic loss.

\subsection{Kernel and cokernel of $I(j)$}
\label{subsec:obstruction}

\subsubsection{Kernel}

For $j\in\bbR$,
let 
\begin{multline}\label{eq:K_ker}
	\calK= \operatorname{span}_\bbC\Big\{  \frac{\hat{y}^{-(d-3)/2+\alpha}}{(1+\hat{y}^2)^{j/2}} \cdot {}_1F_1\Big( a,1+\alpha,-\frac{i\hat{y}^2}{2}  \Big) Y\Big( \frac{\hat{\bfy}}{\hat{y}} \Big)  \\ :Y\in \calY,\;
	\alpha  = \ell + \frac{d-3}{2},\; a=\frac{d+1}{4} - \frac{j}{2} + \frac{\alpha}{2}\text{ s.t. } a\in -\bbN
	\Big\}  \subset C^\infty(\mathrm{ff}^\circ \backslash \text{origin}) , 
\end{multline}
where $Y$ is a spherical harmonic and $\ell=\ell(Y)\in \bbN$ is its azimuthal quantum number, with $\lambda_Y=\ell(\ell+d-3)\geq 0$ the eigenvalue of $Y$ with respect to $\triangle_{\bbS^{d-2}}$. 
Note that $\calK$ is finite dimensional, because the condition $a\leq 0$ rules out all but finitely many spherical harmonics. 

Above, ${}_1F_1$ denotes the usual confluent hypergeometric function --- see \S\ref{subsec:QHO_separation}; in particular, ${}_1F_1(-,-,z)$ is analytic near $z=0$ (and nonvanishing there). When $a\in - \bbN$, then ${}_1F_1(a,-,z)$ is actually a polynomial in $z$ of degree $|a|$; see \cref{eq:Laguerre_relation}. 

Let $\operatorname{ker} I(j) = \{ u\in \calA^{-\infty}(\mathrm{ff}) : I(j)u=0 \}$, where $\calA^{-\infty}(\mathrm{ff})$ is the space of all conormal functions on $\mathrm{ff}$ of some polynomial growth rate. 
\begin{proposition}
    $\calK=\operatorname{ker} I(j)$. 
\end{proposition}
\begin{proof}
    See \S\ref{subsec:QHO_separation}.
\end{proof}

There is a conceptual explanation for the presence of kernel: 
\begin{remark*}
    An outgoing spherical wave $w=r^{-(d-1)/2} e^{ir} f(\theta)$ on $\bbR^d$, with smooth profile $f\in C^\infty(\bbS^{d-1}_\theta)$, satisfies 
    \begin{equation}
        P w = O \Big( \frac{1}{r^{(d+3)/2} } \Big)\text{ as }r\to\infty .  
    \end{equation}
    This has one more order of decay than we would expect based on the fact that $\hat{P} \in \operatorname{Diff}_{\mathrm{b}}^{2,-1,-2}(X)$. 
    It is therefore a quasimode. Quasimodes correspond to nontrivial kernel in indicial operators, at the relevant weight.

    More precisely, consider the Taylor expansion of $f \sim \sum_{\kappa=0}^\infty f_\kappa$ around the forward direction $\rightarrow \in \!\infty \bbS^{d-1}$. 
    For each $\kappa$, the function $f_\kappa \in C^\infty(\bbS^{d-1})$ vanishes to $\kappa$th order at the forward direction. Thus, 
    \begin{equation}
        r^{-(d-1)/2} e^{ir} f_\kappa(\theta) 
    \end{equation}
    is a quasimode. This has $\kappa+d-1$ orders of decay at $\mathrm{ff}$. So, 
    \begin{equation}
        w_\kappa=\rho_{\mathrm{ff}}^{-(\kappa+d-1)} r^{-(d-1)/2} e^{ir} f_\kappa(\theta) 
    \end{equation}
    can be restricted to $\mathrm{ff}\subset X$. 
    Because $w_\kappa$ is a quasimode, $w_\kappa|_{\mathrm{ff}}$ should be in the kernel of the relevant operator in the indicial family. The $w_\kappa|_{\mathrm{ff}}$ are elements of $\calK$.  
\end{remark*}

\subsubsection{Range}
We need to discuss the range of $I(j)$ in $ \varrho^{(d+1)/2} C^\infty(\mathrm{ff}) = \langle \hat{y} \rangle^{-(d+1)} C^\infty(( \overline{\bbR^{d-1}_{\hat{\bfy}}})_2 )$, 
i.e.,\ 
\begin{equation}
    \operatorname{range} I(j) = \{ f\in \varrho^{(d+1)/2} C^\infty(\mathrm{ff})  : f \in I(j) \calA^{-\infty}(\mathrm{ff})  \}.
\end{equation}

Let 
\begin{equation}\label{eq:I_prime}
    I'(j)=\partial_j I(j) = - 2\varrho\hat{\bfy} \cdot \nabla_{\hat{\bfy}}
	- 2\hat{y}^2 \varrho^2  j +( i\varrho - (d-3)\varrho - 2\varrho^2 ) \in \operatorname{Diff}_{\mathrm{b}}^{1,-1}(\mathrm{ff}) .
\end{equation}
This induces one order of decay at $\mathrm{ff}$, owing to the factors of $\varrho$. 

Let $\operatorname{coker} I(j)$ denote the quotient of $\varrho^{(d+1)/2} C^\infty(\mathrm{ff})$ via $\operatorname{range}I(j)$. Then $I'(j) : \operatorname{ker} I(j) \to  \varrho^{(d+1)/2} C^\infty(\mathrm{ff})$, so we can define 
\begin{align}
\begin{split} 
    [I'(j)] &: \operatorname{ker} I(j) \to \operatorname{coker} I(j) \\ 
    &: k\mapsto I'(j) k \bmod \operatorname{range} I(j). 
\end{split}
\end{align}
\begin{lemma}\label{lem:coker_fund}
    $[I'(j)] : \operatorname{ker} I(j) \to \operatorname{coker} I(j)$ is an onto map of vector spaces.
\end{lemma}
\begin{proof}[Proof idea]
    The reason why $\operatorname{coker} I(j)$ is nonzero for certain values $j_0$ of $j$ is that, for some azimuthal quantum numbers $\ell\in \bbN$,  the Wronskian 
    \begin{equation} 
        \frakW_-=\frakW[v_0,v_{\infty,-}]  \overset{\text{\cref{eq:key_Wronskian}}}{=} 2 \Gamma\bigg( \underbrace{\frac{d-1}{2}-\frac{j}{2}+\frac{\ell}{2}}_{a=a(j)} \bigg)^{-1} \Gamma(\alpha+1) (-i/2)^{-\alpha}, \quad \alpha= \ell+\frac{d-3}{2} \label{eq:restated_Wronskian}
    \end{equation} 
    computed in the appendix
    has a zero at $j=j_0$, namely when $a\in -\bbN$, where $\Gamma(a)=\infty$. The lemma statement will follow from the fact that the vanishing of $\frakW_-$ at $j=j_0$ is simple. (The $\Gamma$-function has only simple poles.)

    We need to know that $\dim \ker I(j)\geq \dim \operatorname{coker} I(j)$, but this is a consequence of analytic Fredholm theory. In fact the two dimensions are equal, but we do not need this.
\end{proof}
Though done below, the reduction of the lemma to the Wronskian encounters technical challenges. For this reason, it is worth observing that a failure of the lemma would be, in some sense, unlucky. Indeed, given that the kernel and cokernel of $I(j)$ have the same nonzero finite dimension, then $[I'(j)]$ is an isomorphism if it is injective (by rank-nullity), which is true unless some element of $I'(j)\ker I(j)$ lands exactly in $\operatorname{range} I(j)$. But this latter space has positive codimension, meaning that such a coincidence ought to be unlikely.

\begin{proof}
We provide an abstract gloss on this argument in \S\ref{sec:Grushin}. In order to fit the present problem into that framework, we need a Fredholm setup with $j$-independent spaces. In \S\ref{subsec:QIHO_Fredholm}, we provide such a setup for $L(j)$, constructing $j$-independent\footnote{See \Cref{rem:j-independence}.} Hilbert function spaces $\calX,\calY \subset \calS'(\bbR^{d-1})$ such that  
\begin{equation}
    L(j) : \calX\to \calY 
\end{equation}
is Fredholm. The precise function spaces allowed depend on $j$, but given one particular value $j_0$, the same spaces work for all $j$ sufficiently close to $j_0$. 
We know that $L(j)$ is invertible for all but a discrete set of points $j$,\footnote{By direct computation, we know that $\ker_\calX L(j)=\ker L(j)$ is zero for all but discrete set of $j$. 
The resolvent construction in \S\ref{subsec:QHO_separation} constructs the resolvent at the level of separation-of-variables for $j$ for which $\ker L(j)$ is empty. Applied to a function consisting of finitely many angular modes times $C_{\mathrm{c}}^\infty(\bbR_{\hat{y}}^+)$ functions, this resolvent outputs something lying in 
\begin{align}
    e^{i\hat{y}^2/4} \langle \hat{y} \rangle^{-(d-1)+j} C^\infty(\mathrm{ff})\subseteq \calX.
\end{align}
Thus, $L(j) \calX$ is dense $\calY$, which means that it cannot have positive codimension. This means that $\operatorname{coker}_\calY L(j)$ has dimension zero for every such nonexceptional $j$.} so the hypotheses of \S\ref{sec:Grushin} are fulfilled. 

A consequence is that $L(j)$ has Fredholm index 0, meaning that 
\begin{equation} 
    \ker_{\calX} L(j), \operatorname{coker}_\calY L(j) =\calY/\operatorname{range}_\calX L(j)
\end{equation} 
have the same finite dimension. The subscripts `$\calX$' and `$\calY$' to emphasize the ambient space. However, $\ker_{\calX} L(j) \subset e^{i\hat{y}^2/4} \langle \hat{y} \rangle^{-(d-1)+j} C^\infty(\mathrm{ff})$ automatically (by \S\ref{subsec:QIHO_mapping}), so 
\begin{equation}
    \ker_{\calX} L(j) = \ker L(j), 
\end{equation}
where
\begin{equation}
    \ker L(j) = \{ u\in e^{i\hat{y}^2/4} \langle \hat{y} \rangle^{-(d-1)+j} C^\infty(\mathrm{ff})  :L(j)u=0 \}.
\end{equation} 
This can also be seen by the explicit computations above, via separation-of-variables (using that $\calX$ is acted on by the group of rotations). 

Now we want to apply \Cref{lem:Grushin_magic}. This relates the invertibility of 
\begin{equation}\label{eq:L_iso}
    [L'(j_0)] : \ker L(j_0) \to \operatorname{coker}_{\calY} L(j_0)  
\end{equation}
to the simple vanishing of a certain Schur complement, which is a map between the finite-dimensional spaces $\ker_{\calX} L(j_0), \operatorname{coker}_{\calY} L(j_0)$. The Schur complement commutes with rotations of $\bbR^{d-1}$, so, separating variables, the Schur complement can be studied mode by mode. The ODE computation reveals that the Schur complement is vanishing at the level of an individual angular mode with the Wronskian $\frakW_-$ above.
Hence, the vanishing is simple in the desired sense and \Cref{lem:Grushin_magic} gives that  \cref{eq:L_iso} is an isomorphism.

Thus, 
\begin{equation}
    \calY = \operatorname{range}_{\calX} L(j_0)\oplus  L'(j_0) \ker L(j_0). 
\end{equation}
In particular, 
\begin{equation}
    \underbrace{e^{i\hat{y}^2/4} \langle \hat{y} \rangle^{-(d+1)+j} C^\infty(\mathrm{ff})}_{\subset \calY}  \subseteq \operatorname{range}_{\calX} L(j_0)\oplus  L'(j_0) \ker L(j_0). 
\end{equation}
This means that any $f$ in the space on the left-hand side can be written $  f = L(j_0) u_{\calX} + L'(j_0)k_{\calX}$
for some $k_{\calX}\in \ker L(j_0)$. 

We would like a statement about $I'(j_0)$, not $L'(j_0)$. Let $T= e^{i\hat{y}^2/4} \varrho^{-j_0/2} I'(j_0) e^{-i\hat{y}^2/4} \varrho^{j_0/2}$. Since
\begin{equation}
    L'(j) = \partial_j ( e^{i\hat{y}^2/4} \varrho^{-j/2} I(j) e^{-i\hat{y}^2/4} \varrho^{j/2} ),
\end{equation}
and the $\partial_j$ can fall on $\varrho^{\pm j/2}$, 
it follows that
\begin{equation}
    L'(j_0)-T = \frac{1}{2}(L(j_0) \log \varrho - (\log \varrho) L(j_0)).
\end{equation}
On $\ker L(j_0)$, the second term vanishes, leaving just the first term. Thus, $L'(j_0) k_\calX = T k_\calX +  \frac{1}{2}L(j_0) (k_\calX \log \varrho)$. Because $k_\calX$ lies in $\calX$ with at least some weight $\varrho^\epsilon$ to spare, the product $k_\calX \log \varrho$ lies in $\calX$. We can therefore combine this with $u_\calX$ to write 
\begin{equation}
    f = L(j_0) v_\calX + T k_\calX, 
\end{equation}
where $v_\calX = u_\calX + \frac{1}{2}k_\calX\log \varrho$. 

As recorded in \cref{eq:I_prime}, the operator $I'(j)$ induces one order of decay at bf (unlike $L'(j)$), so $T$ also induces one order of decay. More precisely, 
\begin{equation} 
    T: e^{i\hat{y}^2/4} \langle \hat{y} \rangle^{-(d-1)+j} C^\infty(\mathrm{ff})\to e^{i\hat{y}^2/4} \langle \hat{y} \rangle^{-(d+1)+j} C^\infty(\mathrm{ff}).
\end{equation} 
(Recall that $\langle \hat{y}\rangle^{-2}$ has one order of decay, according to our conventions.) Thus, $f-T k_\calX$ lies in the same space as $f$ itself. 
We deduce that 
\begin{equation}
     v_{\calX}\in R_-(j_0)  (e^{i\hat{y}^2/4} \langle \hat{y} \rangle^{-(d+1)+j_0} C^\infty(\mathrm{ff})) \subset e^{i\hat{y}^2/4} \langle \hat{y} \rangle^{-(d-1)+j_0} C^\infty(\mathrm{ff})
\end{equation}
(as in \S\ref{subsec:QIHO_Fredholm}). Thus, we have shown that 
\begin{equation}
    e^{i\hat{y}^2/4} \langle \hat{y} \rangle^{-(d+1)+j_0} C^\infty(\mathrm{ff})\subseteq \operatorname{range} L(j_0)+ T \ker L(j_0).
\end{equation}
The reverse inclusion is immediate from the mapping properties of $L(j_0),T$. In summary:
\begin{equation}\label{eq:misc_5aa}
    e^{i\hat{y}^2/4} \langle \hat{y} \rangle^{-(d+1)+j_0} C^\infty(\mathrm{ff})=\operatorname{range} L(j_0)+  T \ker L(j_0).
\end{equation}

This means that $[T]$, with domain $\ker L(j_0)$, maps onto 
\begin{equation}
\operatorname{coker} L(j_0) = e^{i\hat{y}^2/4} \langle \hat{y} \rangle^{-(d+1)+j_0} C^\infty(\mathrm{ff}) / \{ L(j_0)u:u\in e^{i\hat{y}^2/4} \langle \hat{y} \rangle^{-(d-1)+j_0} C^\infty(\mathrm{ff})  \} .   
\end{equation}

Undoing the conjugation, we conclude that $[I'(j_0)] $ is an onto map $\ker I(j_0)\to \operatorname{coker} I(j_0)$.
\end{proof}

Since the kernel and cokernel are both special-function theoretic, the reader may prefer the following more concrete proof:
\begin{proof}[Special-function theoretic proof]
    It suffices to work mode by mode.
    Thus, fix a spherical harmonic $Y$. All of the functions below will have the form 
    \begin{equation}
        \text{special function of $\hat{y}$} \times Y(\hat{\bfy}/\hat{y}).
    \end{equation}
    Let $\ell\in \bbN $ be the azimuthal quantum number of $Y$.  In the following argument, 
    \begin{equation}
        a=\frac{d+1}{4} - \frac{j}{2} + \frac{\alpha}{2}
    \end{equation}
    depends on $j$, but this will be implicit in the notation. Note that $\alpha$ is independent of $j$.
    
    The lemma statement is vacuous for azimuthal quantum numbers such that $a\not\in -\bbN$ (because both the kernel and cokernel are 0), so suppose $a\in -\bbN$. Then, the kernel is one-dimensional at the level of the angular mode at which we are working, and we know that the cokernel is at most one-dimensional. So, it suffices to prove
    \begin{equation} 
        I'(j)k\notin \operatorname{range} I(j)
    \end{equation} 
    for 
    \begin{equation}\label{eq:ethansk}
        k = \frac{\hat{y}^{-(d-3)/2+\alpha}}{(1+\hat{y}^2)^{j/2}} \cdot {}_1F_1\Big( a,1+\alpha,-\frac{i\hat{y}^2}{2}  \Big) Y\Big( \frac{\hat{\bfy}}{\hat{y}} \Big) \in \calK.   
    \end{equation}
    (Which shows that the cokernel is exactly one-dimensional, at the level of this single angular mode.)
    
    Let us write $k=k[j]$ to make the dependence on $j$ explicit. The formula above makes sense for all $j$. Moreover, 
    \begin{equation} 
        I(j)k[j]=0
    \end{equation} 
    for all $j$; the only reason $k[j]$ is not counted as an element of $\ker I(j)$ for generic $j$ is because this function is usually oscillatory at large-$\hat{y}$, with the exception being if $a\in -\bbN$. What we mean by ``the kernel'' requires conormality at $\hat{y}=\infty$.  

    Note that $k[j] \in C^\infty(\bbR^{d-1}_{\hat{\bfy}})$, for all $j\in \bbR$. Indeed, 
    \begin{equation} 
        \hat{y}^{-(d-3)/2+\alpha}Y\Big( \frac{\hat{\bfy}}{\hat{y}} \Big)=\hat{y}^{\ell}Y\Big( \frac{\hat{\bfy}}{\hat{y}} \Big)
    \end{equation} 
    is a harmonic polynomial of degree $\ell$ on $\bbR^{d-1}$ and hence smooth.

     Differentiating $I(j) k[j]=0$ in $j$, we get 
     \begin{equation}
         I'(j) k[j] = - I(j)  k'[j], 
     \end{equation}
     where by $k'[j]$ we mean $\partial_j k[j]$. 
     
     If there were some $v\in \calA^{-\infty}(\mathrm{ff})$ such that 
     \begin{equation} 
        I'(j) k[j]= I(j) v,
    \end{equation} 
    there would be some such $v$ possessing only the same spherical harmonic as $k$ itself. 
     Let $h\coloneqq v+ k'[j]$, which satisfies the PDE $I(j) h=0$. This is an ODE for the radial profile.
    Differentiating \cref{eq:ethansk}, we see that $k'[j]$ is smooth at the origin. Indeed, the $j$-dependence of \cref{eq:ethansk} is only in the Japanese bracket factor $\smash{(1+\hat y^2)^{-j/2}}$ and the variable $a=a(j)$ in the first slot of the ${}_1F_1$ function. The Japanese bracket factor remains smooth after differentiating in $j$. 
    Moreover, ${}_1F_1(-,1+\alpha,-)$ is simultaneously holomorphic in the first and third slots (for generic values of the second argument, namely $1+\alpha\notin -\bbN$; fortunately, $1+\alpha\geq 1/2$). See \cite[\href{http://dlmf.nist.gov/13.2.ii}{\S13.2.(ii)}]{NIST}. We then conclude $k'[j]$ is smooth at the origin, like $k[j]$ itself.

    Our tentative preimage $v$ is also required to be regular at the origin,
    hence the same applies to $h$. And since $h$ satisfies the radial ODE $I(j)=0$, it must be proportional to the regular solution thereof:
    \begin{equation} 
        \exists  c \in\bbC\text{ s.t. }h=  c  k[j].
    \end{equation}  
    But when $a\in -\bbN$, the function $k[j]$ is conormal at $\hat{y}=\infty$. Since $v$ also has this property, $k'[j]$ must as well. Once we rule this out, it can be concluded that no $v$ with the properties required above exists. This is the desired result. Our goal has therefore become 
    \begin{equation}
        a\in -\bbN\Longrightarrow k'[j]\text{ not conormal at }\hat{y}=\infty.
    \end{equation}

    For any $j$, we can write, for some $A,B\in \bbC$,
    \begin{equation} 
        k[j]= A(j) g_-[j] + B(j) g_+[j],
    \end{equation} 
 where $g_\pm$ are the solutions of the homogeneous PDE proportional to $v_{\infty,\pm}(\hat y )Y(\hat \bfy/\hat y)$ of \S\ref{subsec:odesolutions} via the conjugating factors relating $I(j)$ to $L(j)$, found in \Cref{lem:QHO_final}. Explicitly,
    \begin{equation}
    \begin{gathered}\label{eq:thegs}
        g_{-}[j] \coloneqq \frac{\hat y^{\alpha-\frac{d-3}{2}}}{(1+\hat y^2)^{j/2}}U\Big(a,1+\alpha,-\frac{i\hat y^2}{2} \Big)Y\Big( \frac{\hat\bfy}{\hat y}\Big),\\ g_{+}[j] \coloneqq \frac{\hat y^{\alpha-\frac{d-3}{2}}e^{-i\hat y^2/2}}{(1+\hat y^2)^{j/2}}U\Big(a+j-\frac{d-1}{2},1+\alpha,-\frac{i\hat y^2}{2} \Big)^*Y\Big( \frac{\hat\bfy}{\hat y}\Big).
    \end{gathered}
    \end{equation}
    The coefficient $B(j)$ vanishes when $a\in -\bbN$, since then $k[j]$ is purely outgoing. We may differentiate in $j$, getting 
    \begin{align}
        \begin{split} 
        k'[j] &= A'(j) g_-[j] + A(j) g_-'[j] + B'(j) g_+[j]+B(j) g_+'[j] \\
        \overset{a\in -\bbN}&{=}  \underbrace{A'(j) g_-[j] + A(j) g_-'[j]}_{\text{conormal at bf}} +B'(j) \underbrace{g_+
        [j]}_{\mathclap{\text{not conormal at bf}}}.
        \end{split} 
    \end{align}
    We claim $g_{-}'[j]$ is conormal at bf. Indeed, this follows from the fact that $\partial_aU(a,b,-is)$ is conormal at $s =\infty$. Since $U(a,b,-is)$ is conormal at $s=\infty$ for each individual $a$, so the claim is highly plausible. Nevertheless, a rigorous argument is required. One rigorous argument involves directly analyzing an integral formula for $U(a,b,-is)$; the details of this are in \S\ref{sec:tech}.
    
    Now $g_+$ is \textit{not} conormal at bf due to the oscillatory prefactor, so the only way for $k'[j]$ to be conormal at $\mathrm{bf}$, for some particular value of $j$, is if $B'(j)=0$. 
    
    Consider the Wronskian $\frakW[k,g_-](\hat{y})$ of $k[j]$ with $g_-[j]$. 
    (Because $I(j)$ has first-order terms, the Wronskian will not be constant but instead will depend on $\hat{y}$.) We now compute $\frakW[k,g_-]$ in two ways.
    \begin{enumerate}
        \item 
        Because $k[j],g_-[j]$ are proportional to the $v_0,v_{\infty,-}$ of \S\ref{sec:QHO}, $\frakW$ is proportional to their Wronskian $\frakW_-=\frakW[v_0,v_{\infty,-}]$ (where the constant of proportionality is nonvanishing). 
        \item 
        On the other hand, 
    \begin{equation}
        \frakW[k,g_-] = B(j) \underbrace{\frakW[g_+,g_-]}_{\neq 0}.
    \end{equation}
    \end{enumerate}
    Thus, we can write 
    \begin{equation}
        B(j) \propto \frakW_-, 
    \end{equation}
    where the proportionality constant depends on $j$ but is smooth and nonvanishing near the critical $j$.  
    The Wronskian $\frakW_-$ has a simple zero at the critical $j$, 
    so we deduce that $B(j)$ has a simple zero at the critical $j$, hence $B'(j)\neq 0$ there. This completes our proof by contradiction.
\end{proof}

\begin{corollary}
   $\varrho^{(d+1)/2} C^\infty(\mathrm{ff}) = \operatorname{range} I(j) + I'(j) \operatorname{ker} I(j)$.
\end{corollary}

\subsection{Proof of \Cref{prop:II_main}}
\label{subsec:v2_construction}

We are now in a position to prove the main result of this section. 

In this proof, restrict attention to $x>0$.
Let $\hat{g} = (r/x) \hat{f}$ and $E = (r/x) \hat{P}-N$. Then, if $\hat{v}$ is supported near $\mathrm{ff}$, 
\begin{multline}
    \hat{P} \hat{v}-\hat{f} \in \langle r \rangle^{-(d+3)/2}C^\infty(\overline{\bbR^d}) \iff  (r/x)\hat{P} \hat{v}-\hat{g} \in \langle r \rangle^{-(d+3)/2}C^\infty(\overline{\bbR^d})\text{ near ff} \\ 
    \iff  N\hat{v} + E \hat{v} -\hat{g} \in \langle r \rangle^{-(d+3)/2}C^\infty(\overline{\bbR^d})\text{ near ff}. 
\end{multline}
We will construct 
a formal polyhomogeneous series 
\begin{equation}
    \hat{\mathsf{v}} = \sum_{(j,k)\in \calF} \rho^{j+L-1}  (\log \rho)^k \hat{v}_{j,k} ,\quad \hat{v}_{j,k} \in \langle \hat{y} \rangle^{-(d-1)} C^\infty(\mathrm{ff}),
    \label{eq:misc_102084}
\end{equation}
such that $N\hat{\mathsf{v}} = \hat{g} - E\hat{\mathsf{v}}$ holds at the level of formal series. 
Here,  $\calF \subset \bbN^2$ is some index set to be specified. It turns out we will take
\begin{equation} 
\calF = \{(j,k)\in\bbN^2\,:\,k\le j+1\};
\end{equation}
the reason will become clear in the proof. 
Then, asymptotically summing $\hat{\mathsf{v}}$ (and inserting a cutoff localizing near $\mathrm{ff}$) yields, in the desired space, a solution $\hat{v}$ to 
\begin{equation} 
    \hat{P}\hat{v}=\hat{f}
\end{equation} 
modulo an error term lying in 
\begin{equation}
    \langle r \rangle^{-(d+3)/2} \calA^{(0,0),\infty}(X)\subseteq 
    \langle r \rangle^{-(d+3)/2} C^\infty(\overline{\bbR^d}).
\end{equation}
Here we are using that a polyhomogeneous function $w$ on $X$ that is smooth at $\mathrm{bf}$ and Schwartz at $\mathrm{ff}$ is smooth already on the blowdown $X\twoheadrightarrow \smash{\overline{\bbR^d}}$. Indeed, the terms in the expansion of $w$ at $\mathrm{bf}$ are smooth functions on $\bbS^{d-1}$ (that happen to be Schwartz in the forward direction), and the partial Taylor series of $w$ at $\mathrm{bf}$ also serves as a Taylor series at $\infty \bbS^{d-1}$ on the radial compactification.

Consider the result of applying $N$ to a function of the form $\rho^j w(\hat{\bfy})$. This is 
\begin{equation}\label{eq:misc_305}
    N(\rho^j w(\hat{\bfy})) =x^{-1} \rho^j (\rho^{-j} xN (\rho^j w) ) = \varrho \rho^{2+j} I(j) w.
\end{equation}
If we instead apply $N$ to $\rho^j (\log \rho)^k w(\hat{\bfy})$ for some $k\geq 1$, then the result involves the derivatives $I^{(\kappa)} (j) = \partial_j^\kappa I(j)$:
\begin{equation}
    N(\rho^j (\log \rho)^k w(\hat{\bfy})) = \varrho \rho^{2+j} \sum_{\kappa=0}^k \binom{k}{\kappa} (\log \rho)^{k-\kappa} I^{(\kappa)}(j) w(\hat{\bfy}). 
\end{equation}
(This can be proven algebraically, or by differentiating \cref{eq:misc_305} in $j$.)
So, 
\begin{align}
    \begin{split} 
    N \hat{\mathsf{v}} = \sum_{(j,k)\in \calF}  N\rho^{L-1+j} (\log \rho)^k \hat{v}_{j,k}(\hat{\bfy}) &= \sum_{(j,k)\in \calF}\sum_{\kappa=0}^k \binom{k}{\kappa}\varrho\rho^{L+1+j} (\log \rho)^{k-\kappa}I^{(\kappa)}(L-1+j) \hat{v}_{j,k}(\hat{\bfy}) \\
    &= \sum_{(j,k)\in \calF} \varrho\rho^{L+1+j} (\log \rho)^{k} \sum_{\substack{\kappa\in \bbN\\ \mathclap{\text{s.t. } (j,k+\kappa)\in \calF}}} \binom{k+\kappa}{\kappa} I^{(\kappa)}(L-1+j) \hat{v}_{j,k+\kappa}(\hat{\bfy}).
    \end{split} 
\end{align} 

Let 
\begin{equation}
    \hat{g} \sim \hat{\mathsf{g}}= \sum_{j=0}^\infty  \rho^{j+L+1}  \hat{g}_{j}, \quad \hat{g}_j\in \langle \hat{y} \rangle^{-(d+3)} C^\infty(\mathrm{ff})
\end{equation}
denote the expansion of $\hat{g}$ at $\mathrm{ff}$. 

By \cref{eq:N_tot}, $E\in \operatorname{Diff}^{2,-2,-3}_{\mathrm{b}}(X)$, so we can write 
\begin{equation}
    E\sim \mathsf{E} = \sum_{\ell=0}^\infty \sum_{m=0}^2  \rho^{\ell+3} E_\ell^{(m)} (\rho\partial_\rho)^m ,
\end{equation}
for some $E_{\ell}^{(m)} \in \operatorname{Diff}_{\mathrm{b}}^{2,-2}(\mathrm{ff})$.
Thus, 
\begin{multline}
    \mathsf{E} \hat{\mathsf{v}} = \sum_{(j,k)\in \calF} \sum_{\ell=0}^\infty \sum_{m=0}^2 \rho^{\ell+3}(\rho \partial_\rho)^m (\rho^{j+L-1} (\log \rho)^k) E_\ell^{(m)} \hat{v}_{j,k}(\hat{\bfy})\\ 
    = \sum_{(j,k)\in \calF} \sum_{\substack{(j',k')\in \calF\\ j\geq j' }}  \rho^{j+L+2} (\log \rho)^k  E_{j,k,j',k'} \hat{v}_{j',k'} (\hat{\bfy}), 
\end{multline}
where $E_{j,k,j',k'}\in \operatorname{Diff}_{\mathrm{b}}^{2,-2}(\mathrm{ff})$ is the sum of differential operators obtained by collecting the coefficient of $\rho^{j+L+2}(\log\rho)^k$.
The equation $N\hat{\mathsf{v}} = \hat{\mathsf{g}} - \mathsf{E} \hat{\mathsf{v}}$ reads
\begin{multline}
        \sum_{(j,k)\in \calF} \varrho\rho^{j+L+1} (\log \rho)^{k} \sum_{\substack{\kappa\in \bbN\\ \mathclap{\text{s.t. } (j,k+\kappa)\in \calF}}} \binom{k+\kappa}{\kappa} I^{(\kappa)}(L-1+j) \hat{v}_{j,k+\kappa}(\hat{\bfy}) \\ = 
        \sum_{j=0}^\infty \rho^{j+L+1}  \bigg[ \hat{g}_j - \sum_{(j-1,k)\in \calF} (\log \rho)^k\sum_{(j',k')\in \calF\text{ s.t. }j\geq j'+1 }  E_{j-1,k,j',k'} \hat{v}_{j',k'}(\hat{\bfy}) \bigg]. 
\end{multline}
Collecting terms, setting the coefficient of $\rho^{j+L+1} (\log \rho)^k$ equal to $0$ yields 
\begin{equation}
    \varrho \sum_{\substack{\kappa\in \bbN\\ \mathclap{\text{s.t. } (j,k+\kappa)\in \calF}}} \binom{k+\kappa}{\kappa}I^{(\kappa)}(L-1+j) \hat{v}_{j,k+\kappa}(\hat{\bfy}) = \delta_{k0} \hat{g}_j - \sum_{\substack{ (j',k')\in \calF \\ j\geq j'+1} } E_{j-1,k,j',k'} \hat{v}_{j',k'}(\hat{\bfy}) .
\end{equation}
Rearranging:
\begin{equation}\label{eq:II_RHS}
    I(L-1+j) \hat{v}_{j,k}(\hat{\bfy}) = - \sum_{\substack{\kappa\in \bbN^+ \\ (j,k+\kappa)\in \calF}} \binom{k+\kappa}{\kappa} I^{(\kappa)} (L-1+j) \hat{v}_{j,k+\kappa} + \frac{1}{\varrho} \bigg[ \delta_{k0} \hat{g}_j - \sum_{\substack{ (j',k')\in \calF \\ j\geq j'+1} } E_{j-1,k,j',k'} \hat{v}_{j',k'}(\hat{\bfy}) \bigg] 
\end{equation}
Note that the $\hat{v}_{j,k}$'s on the right-hand side either have smaller $j$ or larger $k$, as compared to the left-hand side. This means that, if the right-hand side lies in the image of $I(L-1+j)$ acting on $\langle \hat{y} \rangle^{-(d-1)} C^\infty(\mathrm{ff})$, then 
\begin{multline}
    \hat{v}_{j,k}(\hat{\bfy}) = -   I(L-1+j)^{-1} \bigg( \sum_{\substack{\kappa\in \bbN^+ \\ (j,k+\kappa)\in \calF}} \binom{k+\kappa}{\kappa} I^{(\kappa)} (L-1+j) \hat{v}_{j,k+\kappa} \\ - \frac{1}{\varrho} \bigg[ \delta_{k0} \hat{g}_j - \sum_{\substack{ (j',k')\in \calF \\ j\geq j'+1} } E_{j-1,k,j',k'} \hat{v}_{j',k'}(\hat{\bfy}) \bigg]\bigg) 
\end{multline}
can be used to define the $\hat{v}_{j,k}\in \langle \hat{y} \rangle^{-(d-1)} C^\infty(\mathrm{ff})$ via recursion.
Then, 
\begin{equation}
    \frac{1}{\varrho} \bigg[ \delta_{k0} \hat{g}_j - \sum_{\substack{ (j',k')\in \calF \\ j\geq j'+1} } E_{j-1,k,j',k'} \hat{v}_{j',k'}(\hat{\bfy}) \bigg] \in \langle \hat{y} \rangle^{-(d+1)} C^\infty(\mathrm{ff}).  
\end{equation}
Indeed, applying one of the $E_{\bullet}$'s gains two orders of decay $\varrho\sim \hat{y}^{-2}$, and dividing by $\varrho$ loses one order of decay. We see from the formula for $I'(j)$ (see \cref{eq:I_prime}) that $I^{(\kappa)}(L-1+j)\in \operatorname{Diff}^{1,-1}_{\mathrm{b}}(\mathrm{ff})$ for $\kappa\geq 1$, so 
\begin{equation}
    \sum_{\substack{\kappa\in \bbN^+ \\ (j,k+\kappa)\in \calF}} \binom{k+\kappa}{\kappa}I^{(\kappa)} (L-1+j) \hat{v}_{j,k+\kappa} \in \langle \hat{y} \rangle^{-(d+1)} C^\infty(\mathrm{ff})
\end{equation}
as well. 

We consider $\hat{v}_{j,k}$ as defined for all $k\in \bbN$, just zero for all but finitely many $k$.

Thus, as long as the right-hand side of \cref{eq:II_RHS} lies in the range of $I(L-1+j)$, then $\hat{v}_{j,k}$ can be defined. Note that \cref{eq:II_RHS} is not spoiled if we add to $\hat{v}_{j,k}$ an element of 
\begin{equation} 
    \ker I(L-1+j) \subset \langle \hat{y} \rangle^{-(d-1)} C^\infty(\mathrm{ff}).
\end{equation} 
So the recurrence relation only defines $\hat{v}_{j,k}$ modulo this kernel, and we are free to modify our definition of $\hat{v}_{j,k}$ by adding an element of this kernel. 

The only subtlety is arranging for the parenthetical term  to lie in the domain on which $I(L-1+j)^{-1}$ is defined. Fortunately, 
$I'(L-1+j)\hat{v}_{j,k+1}$ appears on the right-hand side multiplied by a nonzero coefficient.
\Cref{lem:coker_fund} therefore tells us that, given $\hat{v}_{j,k+2},\hat{v}_{j,k+3},\cdots$,
we may modify $\hat{v}_{j,k+1}$ by adding to it an element of $\ker I(L-1+j)$ (which as just noted we are free to do), in order to arrange for the parenthetical term to which $I(L-1+j)^{-1}$ is applied to lie in the range. 
Thus,
$\hat{v}_{j,k}$ is well-defined modulo $\ker I(L-1+j)$. If $k\geq 1$, the component in $\ker I(L-1+j)$ will be determined by ensuring that $\hat{v}_{j,k-1}$ is well-defined. 

The index set $\calF\subset \bbN^2$ is defined at level $j=0$ by having only $(0,0)$ and $(0,1)$. Then, at level $j\geq 1$, it is defined by having as many logarithmic terms at level $j-1$ and then possibly one more. We start our definitions at this level with $\hat{v}_{j,k_{\mathrm{max}}+1}$, where $k_{\mathrm{max}}$ is the largest power of logarithm encountered so far. This is chosen so that $\hat{v}_{j,k_{\mathrm{max}}}$ is well-defined, via \cref{eq:II_RHS}. Then, if $k_{\mathrm{max}}\geq 1$, we modify this definition by adding to it an element of $\ker I(L-1+j)$ so that $\hat{v}_{j,k_{\mathrm{max}}-1}$ is well-defined, and so on.
This leads to the formula for $\calF$ above. 

Thus, we have recursively defined $\{\hat{v}_{j,k}\}_{(j,k)\in\calF}$ to satisfy \Cref{eq:II_RHS} for all $(j,k)\in\calF$. The corresponding formal series $\hat{\mathsf{v}}$ then solves $N\hat{\mathsf{v}} = \hat{g}-E\hat{\mathsf{v}}$; as noted above, this leads to a solution $\hat{v}$ to $\hat{P}\hat{v} = \hat{f}$ modulo $\langle r \rangle^{-(d+3)/2} C^\infty(\overline{\bbR^d})$ after an asymptotic summation.

\begin{remark*}
    The previous proof did not fix the component of $\hat{v}_{j,0}$ in $\ker I(L-1+j)$. Thus, $\hat{v}$ is not unique. However, different choices disagree only modulo a term of the form $\langle r \rangle^{-(d-1)/2} C^\infty(\overline{\bbR^d})$.
\end{remark*}
 
\section{Step three: regular outgoing spherical wave}
\label{sec:outgoing}

In this short section, we give a self-contained exposition of the following classical result: 
\begin{proposition}\label{prop:III_main}
    Let $P$ be as in \cref{eq:P}.
    For any $\hat{f}\in \langle r \rangle^{-(d+3)/2} C^\infty(\overline{\bbR^d})$, there exists a $\hat{v}\in \langle r \rangle^{-(d+1)/2} C^\infty(\overline{\bbR^d})$ such that 
    \begin{equation}
        \hat{P} \hat{v} - \hat{f} \in  \calS(\bbR^d).  
    \end{equation}
\end{proposition}
\begin{proof}
This is a simple example of the Frobenius method. 
After constructing the asymptotic solution near infinity, we multiply by a cutoff which is identically one near infinity and extend smoothly to the interior. The additional error is compactly supported and smooth, hence Schwartz, so it does not affect the conclusion.
    
    Let $N\coloneqq N_{\mathrm{bf}}(\hat{P}) = -2i \partial_r - i(d-1)/r$. Then, $E\coloneqq \hat{P}-N$ satisfies 
    \begin{equation}
        E \in \operatorname{Diff}_{\mathrm{b}}^{2,-2}(\overline{\bbR^{d}} ),
    \end{equation}
    whereas $N\in \operatorname{Diff}_{\mathrm{b}}^{1,-1}(\overline{\bbR^{d}} )$. 
    
    Let 
    \begin{equation}
        \hat{\mathsf{f}} = \sum_{j=0}^\infty r^{-j-(d+3)/2} \hat{f}_j \in r^{-(d+3)/2}C^\infty(\bbS^{d-1})[[r^{-1}]] ,\quad \hat f_j\in C^\infty(\bbS^{d-1})
    \end{equation}
    denote the expansion of $\hat{f} \sim \hat{\mathsf{f}}$ at $\infty \bbS^{d-1}$. We will construct a formal series 
    \begin{equation} 
        \hat{\mathsf{v}} = \sum_{j=0}^\infty r^{-j-(d+1)/2} \hat{v}_j \in r^{-(d+1)/2} C^\infty(\bbS^{d-1})[[r^{-1}]], \quad \hat{v}_j \in C^\infty(\bbS^{d-1}) 
    \end{equation} 
    such that $\hat{\mathsf{P}}\hat{\mathsf{v}} = \hat{\mathsf{f}}$ formally. Then, asymptotically summing $\hat{\mathsf{v}}$ (using \Cref{lem:Borel}) yields a solution $\hat{v}$ to $\hat{P}\hat{v} = \hat{f}\bmod \calS$. 

    We can rewrite $\hat{\mathsf{P}} \hat{\mathsf{v}} = \hat{\mathsf{f}}$ as
    \begin{equation}
        N \hat{\mathsf{v}} = \hat{\mathsf{f}} - \mathsf{E} \hat{\mathsf{v}}, 
    \end{equation}
    where 
    \begin{equation}
        \mathsf{E} = \sum_{k=0}^\infty r^{-2-k} \sum_{m=0}^2  (r\partial_r)^m E_k^{(m)},\quad E_k^{(m)} \in \operatorname{Diff}^{2-m}(\bbS^{d-1}).
    \end{equation}
    For $\ell\in \bbR$, let $E_{k,\ell}\in \operatorname{Diff}^2(\bbS^{d-1})$ be defined by $E_{k,\ell} = \sum_{m=0}^2 (-\ell)^m E_k^{(m)} $. Then
    \begin{align}
        \mathsf{E} r^{-\ell} =  \sum_{k=0}^\infty r^{-2-k-\ell} E_{k,\ell}.
    \end{align}
    So,
    \begin{equation}
        \mathsf{E}\hat{\mathsf{v}} =  \sum_{j,k=0}^\infty  r^{-j-k-(d+5)/2} E_{k,j+(d+1)/2} \hat{v}_j  = \sum_{j=0}^\infty r^{-j-(d+3)/2} \sum_{k=0}^{j-1} E_{k,j-k+(d-1)/2} \hat{v}_{j-1-k}.
    \end{equation}
    Also, 
    \begin{multline} 
        \qquad N\hat{\mathsf{v}} = 2i \sum_{j=0}^\infty r^{-j-(d+3)/2} (j+1) \hat{v}_j \in  r^{-(d+3)/2} C^\infty(\bbS^{d-1})[[r^{-1}]]\\ \subset r^{-(d+1)/2} C^\infty(\bbS^{d-1})[[r^{-1}]].\quad 
    \end{multline} 
    Thus, $N\hat{\mathsf{v}}=\hat{\mathsf{f}} - \mathsf{E}\hat{\mathsf{v}}$ reads 
    \begin{equation}
        0= \sum_{j=0}^\infty r^{-j-(d+3)/2} \Big[ 2i (j+1) \hat{v}_j -\hat{f}_j +\sum_{k=0}^{j-1} E_{k,j-k+(d-1)/2} \hat{v}_{j-1-k} \Big] .
    \end{equation}
    Setting each coefficient equal to zero, we get a recursion relation  
    \begin{equation}
        \hat{v}_j = \frac{1}{2i(j+1)} \Big[ \hat{f}_j - \sum_{k=0}^{j-1} E_{k,j-k+(d-1)/2} \hat{v}_{j-1-k}  \Big]
    \end{equation}
    determining, for $j\in \bbN$, the coefficient $\hat{v}_j$ in terms of the previous ones. 
    So setting $\hat{v}_0 = (2i)^{-1} \hat{f}_0$ and then proceeding recursively, we get $\hat{\mathsf{v}}$ with the desired property. 
\end{proof}

\begin{corollary}\label{cor:outgoing_spherical_wave}
    For any $f\in e^{ir} \langle r \rangle^{-(d+3)/2} C^\infty(\overline{\bbR^d})$ supported away from the origin, there exists a $v\in e^{i r} \langle r \rangle^{-(d+1)/2} C^\infty(\overline{\bbR^d})$ such that 
    \begin{equation}
        Pv-f \in  \calS(\bbR^d),   
    \end{equation}
    away from the origin. 
\end{corollary}

\section{Finishing the main proof}
\label{sec:final}

We can now carry out our proof of \Cref{thm:main}. This follows faithfully the outline presented in \S\ref{sec:outline}.

\subsection{Completing the proof of existence}\label{s:complete_proof}

Fix $\chi\in C^\infty(\bbR^d)$ such that $\chi=0$ identically in some large open ball and $\chi=1$ identically outside of some larger open ball. (This is mostly useful for generalizing the proof to the black-box setting below.)

Let $f_1 = -P [ \chi e^{ix}]$ and $\tilde{f}_1 = e^{-ix} f_1$. 
Then, \Cref{prop:I_main} (or the version with a cutoff, see \Cref{rem:I_main_tilde}) gives a 
\begin{equation} 
    \tilde{v}_1 \in \rho_{\mathrm{bf}}^{L-1} \rho_{\mathrm{ff}}^{L-1}C^\infty(X)
\end{equation} 
(supported outside some big ball) such that $\tilde{P} \tilde{v}_1 - \tilde{f}_1 \in \rho_{\mathrm{bf}}^\infty \rho_{\mathrm{ff}}^{L+1} C^\infty(X)$. 
This means that $u_1\coloneqq e^{ix}+v_1$ satisfies 
\begin{equation}
    f_2 \coloneqq -Pu_1 \in e^{ix} \rho_{\mathrm{bf}}^\infty \rho_{\mathrm{ff}}^{L+1} C^\infty(X) \overset{\text{Lem.\ \ref{lem:gaussian_conversion} }}{=}  e^{ir} \rho_{\mathrm{bf}}^\infty \rho_{\mathrm{ff}}^{L+1} C^\infty(X) .
\end{equation}
Here $v_1 \coloneqq e^{ix} \tilde{v}_1$.
Now let $\hat{f}_2 = e^{-ir} f_2 \in \rho_{\mathrm{bf}}^\infty \rho_{\mathrm{ff}}^{L+1} C^\infty(X)$. \Cref{prop:II_main} then yields a 
\begin{equation} 
    \hat{v}_2 \in \rho_{\mathrm{bf}}^{(d-1)/2} \rho_{\mathrm{ff}}^{L-1} \calA^{(0,0),\calF}(X)
\end{equation} 
(supported outside some big ball), for some index set $\calF\subset \bbN^2$, such that $\hat{P} \hat{v}_2 - \hat{f}_2 \in \langle r \rangle^{-(d+3)/2} C^\infty(\overline{\bbR^d})$. Thus, setting $v_2 \coloneqq e^{ir} \hat{v}_2$ and $u_2 = u_1 + v_2$, 
\begin{equation}
    f_3\coloneqq -Pu_2  \in e^{ir}  r^{-(d+3)/2} C^\infty(\overline{\bbR^d}). 
\end{equation}
Next, \Cref{cor:outgoing_spherical_wave} says that there exists some $v_3\in e^{ir}  r^{-(d+1)/2}C^\infty(\overline{\bbR^d}) $, supported outside of some big ball, such that $Pv_3 - f_3\in \calS(\bbR^d)$. Then, $u_3 \coloneqq u_2+v_3$ satisfies 
\begin{equation}
    f_4\coloneqq -Pu_3  \in \calS(\bbR^d).   
\end{equation}
Finally, the limiting absorption principle gives a $v_4 = \lim_{\varepsilon \to0^+}(P-i\varepsilon)^{-1} f_4 \in e^{i r} r^{-(d-1)/2} C^\infty(\overline{\bbR^d})$ such that $Pv_4 = f_4$. Then, 
\begin{equation} 
    u_4\coloneqq u_3 + v_4 = e^{ix} + v_1 +v_2 + v_3 + v_4 
\end{equation} 
satisfies $Pu_4 = 0$. 

Writing $a = e^{-ix} r^{L-1} v_1$ and $w=e^{-ir} r^{(d-1)/2} (v_2+v_3+v_4)$, we have  
\begin{equation}
        u = e^{ix} \Big( 1 + \frac{a}{r^{L-1}} \Big) + e^{ir} \frac{w}{r^{(d-1)/2}}
\end{equation}
away from the origin, and $a,w$ lie in the desired spaces. 
This completes the construction of the perturbed plane wave.

\subsection{Relation to the Poisson map}
We end by presenting a self-contained proof that the perturbed plane wave constructed above is (up to a constant of proportionality) a slice 
\begin{equation}
    K|_{\omega=\leftarrow} \in C^\infty(\bbR^d)
\end{equation}
of the Schwartz kernel $K\in \calS'(\bbS^{d-1}_\omega\times \bbR^d) $ of the Poisson map $\Pi:C^\infty(\bbS^{d-1})\to C^\infty(\bbR^d)$. This will justify calling our perturbed plane wave \emph{the} perturbed plane wave. A byproduct is a proof that the Schwartz kernel $K$ can actually be restricted to a single incoming direction $\omega\in \bbS^{d-1}$.

\begin{proposition} \label{prop:Poisson}  
Set $C=e^{\pi i(d-1)/4} (2\pi)^{(d-1)/2}$.
    Let $u[\omega]$ denote the perturbed plane wave with incoming direction $\omega\in \bbS^{d-1}$. Then, for any $g\in C^\infty(\bbS^{d-1})$, 
    \begin{equation}
        \Pi g = \frac{1}{C} \int_{\bbS^{d-1}}  u[\omega](x,\bfy) g(\omega) \dd \omega .
    \end{equation}
\end{proposition}
\begin{proof}
    The proof will use only those properties of the perturbed plane wave stated in \Cref{thm:main_black-box}, besides sufficiently smooth dependence on the incoming direction. The latter follows from the proof, since the construction of the coefficients in the transport/model inverse steps, as well as the Borel summation process, are both continuous in the smooth parameter $\omega$, with uniform seminorm estimates.

   In the proof below, we use the space
       \begin{equation}
        X_{\mathrm{ext}} = [\overline{\bbR^d_z}\times \bbS^{d-1}_\omega; \rightarrow]_{\mathrm{par}}.
    \end{equation}
    This is an appropriate quasihomogeneous blowup of the forward submanifold 
    \begin{equation}
    \rightarrow=\{ (\infty\omega',\omega)\in \infty\bbS^{d-1} \times \bbS^{d-1}:\omega'=-\omega\}
    \end{equation}
    (note the minus sign arises since we consider $\omega$ as the \emph{incoming} direction). For a fixed $\omega$, the corresponding $d$-dimensional slice of $X_{\mathrm{ext}}$ is just the space $X$ defined in \eqref{eq:1205}, up to rotating the forward direction to $-\omega$.
    Then \Cref{thm:main_black-box} gives 
    \begin{equation}
        u[\omega]=e^{-i\omega \cdot z}  + e^{-i\omega \cdot z}\rho_{\mathrm{bf}}^{L-1}\rho_{\mathrm{ff}}^{L-1}C^{\infty}(X_{\mathrm{ext}} )  +  \frac{e^{ir}}{r^{(d-1)/2}}\calA^{(0,0),\calF}(X_{\mathrm{ext}}),\quad z=(x,\bfy) 
    \end{equation}
    (away from the origin).

    Let 
    \begin{equation} 
        \psi = \frac{1}{C} \int_{\bbS^{d-1}}  u[\omega](x,\bfy) g(\omega) \dd \omega.
    \end{equation} 
    This is a function of $z\in \bbR^d$ (which we keep implicit in the notation).
    Differentiating under the integral sign, we see that $\psi\in \calS'$ is a solution of the Helmholtz equation $P\psi=0$. To prove that $\psi = \Pi g$, we use the following characterization of $\Pi g$: it is the unique tempered solution of the Helmholtz equation (outside the black-box, and satisfying the black-box condition) such that the difference 
    \begin{equation}
        \Psi\coloneqq \psi - \frac{e^{-ir} }{r^{(d-1)/2}} g \in C^\infty(\bbR^d\backslash \{0\})  
    \end{equation} 
    satisfies the Sommerfeld radiation condition 
    \begin{equation}
        \exists \varepsilon>0\text{ s.t. }
       r^{ \varepsilon+(d-1)/2 }  (i\partial_r + 1) \Psi \in L^\infty(\{r>1\} )  .
    \end{equation}

    Now split $C\psi = I_1+I_2+I_3$,
    where 
    \begin{equation}
        I_1=\int_{\bbS^{d-1}}  e^{-i\omega\cdot z } g(\omega) \dd \omega,\quad  
        I_2=\int_{\bbS^{d-1}}  \frac{e^{-i\omega\cdot z }}{r^{L-1}} a   g(\omega) \dd \omega,\quad 
        I_3= \frac{e^{ir}}{r^{(d-1)/2}} \int_{\bbS^{d-1}}  w[\omega]g(\omega) \dd \omega, 
    \end{equation}
    where $a,w$ are as in \Cref{thm:main_black-box}. 
    We analyze each of $I_1,I_2,I_3$ in turn. 
    First consider $I_1$. This is the solution of the \emph{free} Helmholtz equation with incoming data $g$. Its analysis is standard, proceeding via stationary phase; see e.g.\ \cite[\S1]{MelroseGeometric}. For instance, 
     \begin{equation}
     \int_{\bbS^{d-1}}e^{-i \omega \cdot z}g(\omega)\dd\omega=C r^{-\frac{d-1}{2}}e^{-ir}g\Big( \frac{z}{r} \Big)+ \langle r \rangle^{-(d+1)/2} e^{-i\langle r \rangle} C^\infty(\overline{\bbR^d}) + \langle r \rangle^{-(d-1)/2} e^{i\langle r \rangle} C^\infty(\overline{\bbR^d}).
    \end{equation}
    This shows that $I_1 - Ce^{-ir}r^{-(d-1)/2 }g$ satisfies the Sommerfeld radiation condition. 
    Next, we will show that $I_2,I_3$ each satisfy the radiation condition 
    \begin{equation} \label{eq:Sommerfeld_sub}
        r^{ \varepsilon+(d-1)/2 }  (i\partial_r + 1) I_j \in L^\infty(\{r>1\} ) 
    \end{equation}
    individually. 
    This will complete the proof.

    We begin with $I_3$. Because $(i\partial_r +1)e^{ir}=0$, 
    we have 
\begin{equation}\label{eq:I3_Sommerfeld}
       -i (i\partial_r+1)I_3= - \frac{d-1}{2} \frac{e^{ir}}{r^{(d+1)/2}} \int_{\bbS^{d-1}}  w[\omega]g(\omega) \dd \omega + \frac{e^{ir}}{r^{(d+1)/2}} \int_{\bbS^{d-1}}  (r\partial_r w[\omega])g(\omega) \dd \omega.
    \end{equation}
    Because $w\in \calA^{(0,0),\calF}(X_{\mathrm{ext}})$ for $\calF\subset(\bbN\times\{0\})
\cup(\bbZ^{\ge L-d}\times\bbN)$, 
    \begin{equation}
        w[\omega] \in L^\infty (\bbR^d_z; \theta^{\min\{0,L-d\}-} L^\infty(\bbS^{d-1}_\omega )) \subset  L^\infty (\bbR^d_z;L^1(\bbS^{d-1}_\omega ))
    \end{equation}
    away from the origin. Here, $\theta=\theta(\omega)$ is the polar angle between $z$ and the forward direction. Indeed, since $L\geq 2$, the worst possible singularity occurs for $L=2$, where it is $O(\theta^{-(d-2)-\varepsilon} )$ for any $\varepsilon>0$. In spherical coordinates on $\bbS^{d-1}$, the volume element carries a factor of $\theta^{d-2}$, so 
    \begin{equation} 
        \theta^{-(d-2)-\varepsilon} \in L^1(\bbS^{d-1})
    \end{equation}
    for $\varepsilon$ small enough. 
    From $w[\omega] \in L^\infty (\bbR^d_z;L^1(\bbS^{d-1}_\omega ))$, we conclude 
    \begin{equation}
        \int_{\bbS^{d-1}}  w[\omega]g(\omega) \dd \omega  \in L^\infty(\bbR^d_z) 
    \end{equation}
    (away from the origin). Therefore, the first term on the right-hand side of \cref{eq:I3_Sommerfeld} has the desired size. 
    The other term on the right-hand side is estimated analogously; because $r\partial_r \in \operatorname{Diff}_{\mathrm{b}}^{1,0,0}(X)$,
    \begin{equation} 
        r\partial_r w[\omega] \in \calA^{(0,0),\calF}
    \end{equation} 
    as well. Thus, the estimate used to control the first term on the right-hand side of \cref{eq:I3_Sommerfeld} also works for the second term.

    Finally, consider $I_2$. 
    We show that this will satisfy $ \partial_r I_2 , I_2 \in r^{-(\varepsilon+(d-1)/2)}L^\infty(\{r>1\} )$ for some $\varepsilon>0$.
    This shows that it will satisfy
    the radiation condition without needing to take into account any cancellations between the two summands in $(i\partial_r+1) I_2$. 
    Using a partition of unity, break up 
    \begin{equation} 
        a=a_{\circ}+a_{\rightarrow}
    \end{equation} 
    into a portion $a_\circ \in C^\infty(\overline{\bbR^d_z}\times \bbS^{d-1}_\omega)$ supported away from the forward submanifold $\rightarrow$ and a remainder supported near it. We may assume without loss of generality that $a$ is supported away from the midpoint of $\mathrm{ff}$ as in the construction above\footnote{Note that we may trade any portion of $a$ with $w$ near the midpoint.}, so assume the same of $a_\rightarrow$. Then, $a_\rightarrow$ is supported near the corner $\mathrm{bf}\cap\mathrm{ff}$ of $X_{\mathrm{ext}}$. We have 
    \begin{equation}
        I_2 = \underbrace{\frac{1}{r^{L-1}} \int_{\bbS^{d-1}} e^{-i\omega\cdot z} a_\circ g(\omega) \dd \omega}_{I_{2,\circ}}+\underbrace{\frac{1}{r^{L-1}} \int_{\bbS^{d-1}} e^{-i\omega\cdot z} a_\rightarrow g(\omega) \dd \omega}_{I_{2,\rightarrow}} 
    \end{equation}
    The stationary phase analysis used to control $I_1$ works verbatim for $I_{2,\circ}$, because $a_\circ \in C^\infty(\overline{\bbR^d_z}\times \bbS^{d-1}_\omega)$. 
    In fact, the situation is now better because of the prefactor of $1/r^{L-1}$. Thus, 
    \begin{equation}
        I_{2,\circ} \in e^{ir} r^{-L-(d-3)/2} C^\infty(\overline{\bbR^d}) + e^{-ir} r^{-L-(d-3)/2} C^\infty(\overline{\bbR^d})
    \end{equation}
    away from the origin. Since $L\geq 2$, this piece of $I_2$ satisfies the Sommerfeld condition \cref{eq:Sommerfeld_sub}, with $1-\varepsilon$ orders to spare. 
    
    In order to analyze the remaining piece, we rewrite $a_{\rightarrow}$ in coordinates adapted to the geometric structure of $X_{\mathrm{ext}}$ near its corner. It suffices to prove the Sommerfeld condition for fixed direction $z/r\in \bbS^{d-1}$, as long as the estimates are uniform in this angle (as they will be). For notational simplicity, fix
    the direction $z/r$ to be $(1,\bf0)$, in which case we can write
    $z=(x,\bf0)$, $x>0$.  
    For this fixed angle, the remaining coordinates on $X_{\mathrm{ext}}$ are $x>0$ and the wave's incoming direction $\omega\in \bbS^{d-1}$; note that since $a_{\rightarrow}$ is supported near $\rightarrow$, it suffices to consider 
    \begin{equation} 
    \omega\approx(-1,\bf0).
    \end{equation} 
    We can form the new coordinate $w=-x \omega \in \bbR^d$. Then, the slice $\{z=(x,\bf0)\}$ of $X_{\mathrm{ext}}$ 
    is naturally identified with 
    \begin{equation}
        X\hookleftarrow \bbR^d_{w}.  
    \end{equation}
    Note the change in perspective: in the body of this paper, we considered $\omega$ as fixed and $X$ as the compactification of the space of possible $z=(x,\bfy)$'s. Instead, here we consider the direction $z/r$
    as fixed and consider $X$ as the compactification of the space of possible pairs $(r,\omega)$.

    Write $w=(w',\bfw)$ for $w'>0$ and $\bfw\in \bbR^{d-1}$. 
    Then, 
    \begin{equation}
        \rho=\frac{|\bfw|}{w' }\geq 0,\quad \varrho=\frac{w'}{|\bfw|^2}\geq 0,\quad \phi=\frac{\bfw}{|\bfw|}\in \bbS^{d-2} 
    \end{equation}
    serve as valid coordinates near the corner. This means that, for fixed $z/r$, 
    \begin{equation}
        a_{\rightarrow} = \rho^{-L+1}\alpha(\rho,\varrho,\phi) \text{ for some }\alpha \in C^\infty_{\mathrm{c}}([0,\infty)_\rho \times [0,\infty)_{\varrho}\times \bbS^{d-2}_\phi  ). 
    \end{equation}
    The two coordinates $\rho,\phi$ depend on $\omega$ alone, not on $x$. We can use them to parametrize $\bbS^{d-1}_\omega$ near the forward direction, where $a_\rightarrow$ is supported. Then, for some smooth factor 
    \begin{equation} 
        J\in C^\infty_{\mathrm{c}}([0,\infty)_\rho \times [0,\infty)_{\varrho}\times \bbS^{d-2}_\phi  )
    \end{equation} 
    coming from the Jacobian of the coordinate transformation, $\dd \omega = J \rho^{d-2} \dd \rho \dd \phi$,
    where $\dd \phi$ denotes the usual measure on $\bbS^{d-2}$. Hence, 
    since $e^{-i \omega\cdot z} = e^{i/ (\varrho \rho^2 )}$, we can write 
    \begin{equation}
        I_{2,\rightarrow} = \frac{1}{x^{L-1}} \int_{\bbS^{d-2}} \int_{0}^\infty  e^{i/(\varrho \rho^2) } \alpha(\rho,\varrho,\phi) g(\omega(\rho,\phi)) J \rho^{d-L-1}  \dd \rho  \dd \phi.
    \end{equation}
    We can do the integral over $\phi$ first: letting 
    \begin{equation}
        \beta(\rho,\varrho) = \int_{\bbS^{d-2}} \alpha(\rho,\varrho,\phi) g(\omega(\rho,\phi)) J \dd \phi \in C_{\mathrm{c}}^\infty([0,\infty)_\rho \times [0,\infty)_\varrho), 
    \end{equation}
    we have 
    \begin{equation}
        I_{2,\rightarrow} = \frac{1}{x^{L-1}} \int_0^\infty e^{i/(\varrho \rho^2)} \beta(\rho,\varrho) \rho^{d-L-1} \dd \rho .
    \end{equation}
    Here, $\varrho=\varrho(x,\rho)$ depends on $x$ and the integration variable $\rho$: 
    \begin{equation}
        \varrho = \frac{w'}{|\bfw|^2} = \frac{1}{\rho x} \sqrt{1+ \frac{1}{\rho^2}} = \frac{\sqrt{1+\rho^2}}{\rho^2 x} .
    \end{equation}
    So, 
    \begin{equation}
        I_{2,\rightarrow} = \frac{1}{x^{L-1}} \int_0^\infty e^{\frac{ix}{\sqrt{1+\rho^2}}} \beta\left(\rho, \frac{\sqrt{1+\rho^2}}{\rho^2 x}  \right) \rho^{d-L-1} \dd \rho .
    \end{equation}
    Note that, because $\beta$ has compact support, the integral is finite. 
    Integrals of the form above are investigated in \Cref{lem:Poisson_helper}, which gives a $O(1/x^{(d-2+L)/2-\varepsilon})$ 
    bound for $I_{2,\rightarrow}$ and its $x$-derivative. Here $\varepsilon>0$ is arbitrarily small. So, $I_{2,\rightarrow}$ satisfies the desired radiation condition, with $1/2-\varepsilon$ orders to spare. 
\end{proof}

\begin{lemma}\label{lem:Poisson_helper}
    Consider an integral of the form
    \begin{equation}
        \calI =
        \frac{1}{x^j} \int_0^\infty e^{\frac{ix}{\sqrt{1+\rho^2}}} \beta\left(\rho, \frac{\sqrt{1+\rho^2}}{\rho^2 x}  \right) \rho^{k} \dd \rho
    \end{equation} 
    for $j\in \bbR$, $k\in \bbZ$, and $\beta\in C_{\mathrm{c}}^\infty([0,\infty)^2)$.This obeys the estimate
    \begin{equation}
        |\partial_x^m \calI| = O \Big( \frac{1}{x^{j+ (k+1)/2  -\varepsilon}} \Big) \text{ as }x\to\infty ,
    \end{equation}
    for any $m\in \bbN$ and $\varepsilon>0$. 
\end{lemma}
\begin{proof}
    Note that 
    \begin{multline}
        \partial_x \calI = -\frac{j}{x} \calI+ \frac{i}{x^{j}} \int_0^\infty e^{\frac{ix}{\sqrt{1+\rho^2}}} \frac{1}{\sqrt{1+\rho^2}}  \beta\left(\rho, \frac{\sqrt{1+\rho^2}}{\rho^2 x}  \right) \rho^{k} \dd \rho \\ - \frac{1}{x^{j+2}}  \int_0^\infty e^{\frac{ix}{\sqrt{1+\rho^2}}} \sqrt{1+\rho^2}  (\partial_2\beta)\left(\rho, \frac{\sqrt{1+\rho^2}}{\rho^2 x}  \right) \rho^{k-2} \dd \rho, 
    \end{multline}
    where $\partial_j$ is the derivative in the $j$th slot. 
    The right-hand side is a linear combination of integrals of the same form as $\calI$, with different  $\beta\in C_{\mathrm{c}}^\infty([0,\infty)^2)$.
    The first two terms have the same $j,k$ as $\calI$ itself, and the last has $j+2$ in place of $j$ and $k-2$ in place of $k$: 
    \begin{equation}
        (j,k) \to (j+2,k-2).
    \end{equation}
    Considering that the desired estimate depends on $j+k/2$, this means the last term should have better decay than $\calI$ itself. This sets up an inductive argument where the lemma statement with $m\geq 1$ derivatives follows from the lemma statement with $m-1$ derivatives. Thus, it suffices to prove the $m=0$ case. This is the goal of the rest of the proof. 
    
    We proceed via induction on $k$, taking $k\in \{-1,-2,-3,\cdots\}$ as separate base cases.  
    Choose $C>c>0$ such that $\operatorname{supp} \beta \subset [0,C]\times [0,1/\sqrt{c}]$. Suppose $k\leq -1$. Then, for $x>1$, we have 
    \begin{equation}
        |\calI| \leq \frac{\lVert \beta \rVert_{L^\infty} }{x^j} \int_{c/x^{1/2}}^C \rho^k \dd \rho = 
        O\bigg( \frac{1}{x^{j+\frac{k+1}{2} -\varepsilon}}\bigg)   
    \end{equation}
    for the $k$ above. The $\varepsilon$ shift in the exponent is just to handle the logarithmic loss when $k=-1$. So, the estimate holds in this case. 
    
    Now suppose $k\geq 0$. 
    To exhibit the decay of the integral above in the $x\to\infty$ limit, we integrate by parts, using 
    \begin{equation}
        e^{\frac{ix}{\sqrt{1+\rho^2}}} = -\frac{(1+\rho^2)^{3/2}}{ix \rho }  \frac{\partial}{\partial \rho}  e^{\frac{ix}{\sqrt{1+\rho^2}}} ,
    \end{equation}
    which gives 
\begin{equation}
        \calI = \frac{i}{x^{j+1}} \int_0^\infty \Big( \frac{\partial}{\partial \rho} e^{\frac{ix}{\sqrt{1+\rho^2}}} \Big)   (1+\rho^2)^{3/2} \beta \left(\rho, \frac{\sqrt{1+\rho^2}}{\rho^2 x} \right) \rho^{k-1} \dd \rho .
    \end{equation}
    Now integrate by parts. There are no boundary terms, owing to the compact support of $\beta$ in both slots (note for any $x>0$ that the second argument goes to $\infty$ as $\rho\to 0^+$). We may therefore write 
    \begin{equation}
        \calI = \calI_0+\calI_1+\calI_2+\calI_3, 
    \end{equation}
    where 
    \begin{align}
    \begin{split} 
        \calI_0 &= \frac{k-1}{ix^{j+1}} \int_0^\infty  e^{\frac{ix}{\sqrt{1+\rho^2}}}  (1+\rho^2)^{3/2}\beta \left(\rho,\frac{\sqrt{1+\rho^2}}{\rho^2 x}  \right) \rho^{k-2} \dd \rho, \\
        \calI_1 &=- \frac{i}{x^{j+1}} \int_0^\infty  e^{\frac{ix}{\sqrt{1+\rho^2}}}  \rho (\partial_1 \check{\beta} )\left(\rho, \frac{\sqrt{1+\rho^2}}{\rho^2 x} \right) \rho^{k-2} \dd \rho, \\
        \calI_2 &= \frac{2i}{x^{j+2}} \int_0^\infty  e^{\frac{ix}{\sqrt{1+\rho^2}}}   (1+\rho^2)^2  (\partial_2 \beta)\left(\rho, \frac{\sqrt{1+\rho^2}}{\rho^2 x} \right) \rho^{k-4} \dd \rho, \\
        \calI_3 &= -\frac{i}{x^{j+2}} \int_0^\infty  e^{\frac{ix}{\sqrt{1+\rho^2}}}   \rho^2 (1+\rho^2) (\partial_2 \beta) \left(\rho, \frac{\sqrt{1+\rho^2}}{\rho^2 x} \right) \rho^{k-4} \dd \rho,
        \end{split} 
    \end{align}
    where $\check{\beta}= (1+\rho^2)^{3/2} \beta$.
    Each of these integrals has the same form as our original integral. The numerology is that we trade two factors of $\rho$ in the integrand for every factor of $x^{-1}$ out front. So $k$ decreases by two and $j$ increases by one: 
    \begin{equation}
        (j,k)\to (j+1,k-2)\text{ or }(j+2,k-4)
    \end{equation}
    preserving the sum $j+k/2$, which means the desired estimate for $\calI$ follows from those for $\calI_\bullet$.  
    This sets up the induction. Among the $\calI_\bullet$ above, we have decreased $k$ by at most $4$, so the base cases $k=-4,-3,-2,-1$ suffice to reduce all $k\geq 0$. 
\end{proof}

In \cite[\S14]{MZ} appears a microlocal argument fulfilling an analogous role to that above, and covering the result above as a special case.

\section{Theorem with black-box}
\label{sec:black-box}
Let $P$ be as in \S\ref{subsec:asymptotically_Euclidean}. Suppose we wish to modify $P$ in some compact subset in a manner leaving the class of operators considered thus far. This could mean allowing the coefficients to have singularities, or it could mean introducing some boundary conditions on some hypersurface. It could mean something else entirely. To state and prove a theorem in this generality, we utilize a custom ``black-box'' formalism.

Let $U\subset \bbR^d$ denote a bounded open set and $A\subset U$ an open subset, such that $U\backslash A \Subset U$. We think of $B=U\backslash A$ as the black-box. Then, we study the PDE  
\begin{equation}\label{eq:bbsteup}
    Pu=f 
\end{equation}
on $U^\complement \cup A = \bbR^d\backslash B$, that is outside of the black-box. We do not modify $P$; it does not take into account whatever we envision as taking place within the black-box. Instead, we formalize the restriction the latter places on $u|_A$, that is on the behavior of $u$ in the ``matching region'' $A$. 
We require that $u|_A$ lies within a specified set of functions which we consider \emph{admissible}.
If $u|_A$ is admissible, then we say that $u$ satisfies the \emph{black-box condition}.

One should imagine $A$ as an annulus around the black-box (see \Cref{fig:black-box_sets}) and the black-box condition as enforcing, based purely on behavior outside the black-box, that $u$ extend to be well-behaved within the black-box. This formalism is agnostic to where the restriction on $u|_A$ is coming from --- for example, this might entail solving some PDE $P_{\mathrm{true}}u=0$ on $U$, where $P_{\mathrm{true}}$ is a differential operator with $P|_{B^\complement} = P_{\mathrm{true}}|_{B^\complement}$. It could entail that $u$ extends to a solution of the Helmholtz operator on some manifold in which Euclidean space has been modified by introducing some non-trivial topology.

\begin{remark*}
    It should be emphasized that although $P$ is defined globally on $\bbR^d$, the values of its coefficients within $B$ are irrelevant and only stipulated to exist for convenience.
\end{remark*}

\begin{figure}[!htbp]
    \centering
\begin{tikzpicture}[scale=.8]

\tikzset{
  Ustyle/.style={thin, dashed},
  Astyle/.style={dashed, pattern=dots},
  Bstyle/.style={black!10},
  Outstyle/.style={fill=white},
}

\fill[Outstyle] (-4,-4) rectangle (4,4);
\draw[dashed] (-4,-4) rectangle (4,4);
\node at (4.5,4.5) {$\bbR^d$};

\begin{scope}
  \clip (-4,-3) rectangle (4,3);
  \fill[white]
    plot[smooth cycle, tension=0.9]
      coordinates {(-3,-0.5) (-2,2) (0.5,2.4) (2.8,1.2) (2.2,-1.8) (0,-2.5)};
\end{scope}

\draw[Ustyle]
  plot[smooth cycle, tension=0.9]
    coordinates {(-3,-0.8) (-2,2.4) (0.8,2.8) (2.8,1.6) (2.2,-2) (0,-2.8)};

\node at (3.2,2.0) {$U$};

\def\Bmin{-1.5}
\def\Bmax{1.5}

\fill[Astyle, even odd rule]
  plot[smooth cycle, tension=0.9]
    coordinates {(-3,-0.8) (-2,2.4) (0.8,2.8) (2.8,1.6) (2.2,-2) (0,-2.8)}
  (\Bmin,\Bmin) rectangle (\Bmax,\Bmax);

\fill[Bstyle] (\Bmin,\Bmin) rectangle (\Bmax,\Bmax);
\draw[dashed] (\Bmin,\Bmin) rectangle (\Bmax,\Bmax);
\node at (0,0) {\small $B$};

\filldraw[fill=white] (0,-0.8) circle (10pt);
\node at (0,-0.8) {$\Omega$};

\node at (-0.8,0.8) {\large $\times$};
\node at (-0.8,1.15) {$s$};

\node at (.7,.4) {$\mathfrak{g}$};
\node[inner sep=0pt]  at (.7,.9)
    {\includegraphics[scale = .04]{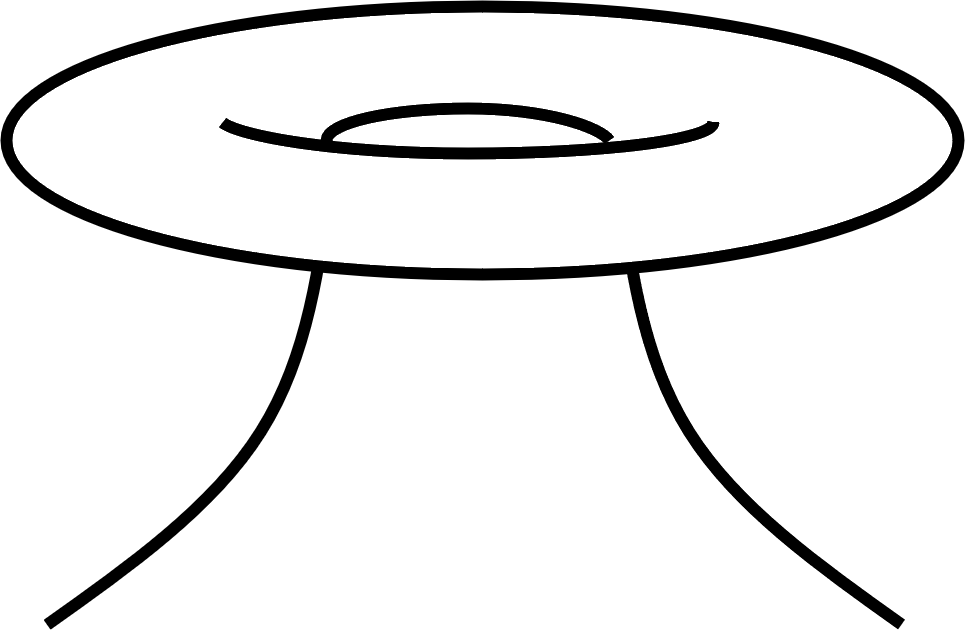}};

\end{tikzpicture}
\caption{Schematic of the black-box setup. The annulus $A=U\backslash B$ is dotted. The black-box region $B$ contains an obstacle $\Omega$, a singularity $s$, and non-Euclidean topology $\mathfrak{g}$.}
\label{fig:black-box_sets}
\end{figure}
In the same way that one can ask about the well-posedness of a PDE with certain boundary conditions imposed, one can ask about well-posedness with black-box conditions imposed. We make the following assumption about well-posedness:
\begin{itemize}
    \item[\hypertarget{WP_asump}{\textbf{(WP)}}] 
    \textit{For every Schwartz $f\in \calS(\bbR^d)$ supported outside the black-box region $U$, there exists a unique solution  
\begin{equation}
    u\in C^\infty( U^\complement \cup A)
\end{equation}
to the inhomogeneous Helmholtz equation $Pu=f$, depending linearly on $f$, such that 
\begin{enumerate}[label=(\roman*)] 
    \item near spatial infinity, $u\in e^{ir} r^{-(d-1)/2} C^\infty(\overline{\bbR^d})$, and $u$ depends smoothly on $f\in \calS$ with respect to this topology,
    \item $u$ satisfies the black-box condition, i.e., $u|_A$ is admissible.
\end{enumerate}
}

\end{itemize}

Here, we are only asking that $Pu=f$ holds in $U^\complement \cup A$ (as $u$ is not even defined on all of $\bbR^d$). So, we are assuming well-posedness of the ordinary Sommerfeld problem with the black-box condition imposed on $A$. 
\begin{remark*}
The requirement (i) is a particularly strong form of the Sommerfeld radiation condition, but Melrose's results in \cite{Me94} show that any weak form of the radiation condition implies the strong form. 
Indeed, choosing 
\begin{equation}
    \chi \in C^\infty(\overline{\bbR^d}), \quad \chi =1\text{ near }\infty\bbS^{d-1},\quad \chi=0\text{ near }\overline{U}, 
\end{equation}
we have 
\begin{equation}
  Pu=f\text{ on }U^\complement\cup A \longrightarrow  P(\chi u) = \chi f+\underbrace{[P,\chi]u}_{\mathclap{\in C_{\mathrm{c}}^\infty\text{ by elliptic regularity}}}\text{ on }\bbR^d,  
\end{equation}
and the radiation condition holds for $\chi u$ if and only if it does for $u$. We may apply the cited paper directly to $\chi u$. 
This means that, in the assumption above, the strong form of the radiation condition can be replaced by any weak form. 

See \Cref{prop:Sommerfeld_reduction} for a version of this. 
\end{remark*}

\begin{example}[Asymptotically Euclidean manifolds]
    Suppose that $(M,g)$ is an asymptotically Euclidean manifold, and consider the Helmholtz(--Beltrami) operator $\triangle_g-1$. 
    By definition, a neighborhood of $\partial M$ can be identified with a neighborhood 
    \begin{equation} 
        O \subset \overline{\bbR^d}
    \end{equation} 
    of $\infty \bbS^{d-1}$. There, $\triangle_g-1$ has the form $P$ above. 
    Take $U\Supset O^\complement$ to be a large open ball in $\bbR^d$ and $A= \smash{U\backslash O^\complement}$ to be an annular region. Note that $A$ is identified with a region in $M$. 
    Then, we say that $u$ is admissible if it can be extended to a Helmholtz solution on all of $M$.
    In \cite[Prop.\ 12]{Me94}, Melrose established the well-posedness of the Sommerfeld problem on $M$, so the well-posedness assumption holds.
\end{example}

\begin{example}[Obstacle + potential scattering]
    Physicists sometimes consider the Schr\"odinger equation with a potential $V$ that they say satisfies 
    \begin{equation}
        z\in \Omega\Longrightarrow V(z)=\infty
    \end{equation}
    for some $\Omega\Subset \bbR^d$ with smooth boundary. For example, in order to model the fact that an electron cannot penetrate the hydrogen nucleus, one considers the modified Coulomb potential 
    \begin{equation}
        V = \begin{cases}
            -1/r &  (r >r_{\mathrm{proton}}) \\ 
            \infty & (r<r_{\mathrm{proton}}).
        \end{cases}
    \end{equation}

    This can be made precise in the framework of obstacle + potential scattering. 
    Let $U\Supset \overline{\Omega}$ be a big ball containing $\overline{\Omega}$, and let $A=U\backslash \overline{\Omega}$. 
    We assume that $\bbR^d\backslash \overline{\Omega}$ is path connected. 
    Consider the ordinary Helmholtz operator 
    \begin{equation}
        P=\triangle-1+V \text{ for }V\in\smash{ \langle r \rangle^{-2} C^\infty(\overline{\bbR^d})}.
    \end{equation}
    Say that a Helmholtz solution $u$ on $A$ is admissible if it extends to a global solution of the Dirichlet problem 
    \begin{equation}\label{eq:1681}
        \begin{cases}
            u\in C^2(\bbR^d\backslash \overline{\Omega}) \cap C^0(\bbR^d\backslash \Omega), \\
            
            Pu=0 \text{ on }U \backslash \overline \Omega , \\
            u|_{\partial \Omega}=0. 
        \end{cases}
    \end{equation}
    Then, it is a classical result that the Helmholtz equation is well-posed with the provided black-box condition. See e.g.\ \cite[Chp.\ 9.1, Exc.\ 9.1.5]{TaylorPDE}. The cited text treats obstacle scattering and potential scattering individually, but the same arguments handle their combination.
    In this spirit, the reader may note that the inhomogeneous Dirichlet problem on $\bbR^d\backslash \overline{\Omega}$ can be reduced to the homogeneous problem with inhomogeneous Dirichlet boundary condition. 
    
    Alternatively, the reader may note that the arguments in \cite{Me94} go through in this setting. 
\end{example}

\begin{example}[Singularities]\label{examp:singularities} 
    Suppose that $V$ has a second-order pole at the origin:
    \begin{equation}\label{eq:sing_as_1}
        \exists \alpha,\mathsf{Z}\in \bbR\text{ s.t. }V-\frac{\alpha}{r^2}+\frac{\mathsf{Z}}{r\langle r \rangle} \in C^\infty(\overline{\bbR^d} ).
    \end{equation}
    Then, take $U=B_2(0)$ to be an open ball of radius $2$ and $A=B_2(0)\backslash \overline{B_1(0)}$ to be an annulus. 
    Suppose that 
    \begin{equation}\label{eq:sing_as_2}
        \alpha > - \frac{(d-2)^2}{4}
    \end{equation}
    so that the Helmholtz operator $P=\triangle-1+V$, defined initially on $C_{\mathrm{c}}^\infty(\bbR^d\backslash \{0\})$, is semibounded; \cite[Chapter X]{ReedSimon2}\footnote{If $d\geq 3$, this is proven via Hardy's inequality (and the lower-order nature of the $\mathsf{Z}/r$ term  with respect to the b-calculus $\operatorname{Diff}_{\mathrm{b}}[0,1)_r$, relevant to behavior near $r=0$). If $d=2$, semiboundedness still holds as a consequence of a fractional version of Hardy's inequality and therefore admits a Friedrichs extension; see \cite[(2)]{dimtwocol}}. 
    We define a Helmholtz solution $u\in C^\infty(A)$ to be admissible if it extends to a Helmholtz solution on $U\backslash \{0\}$ lying in the Friedrichs domain. 
    The limiting absorption principle is known in this setting; cf.\ \cite[Thm.\ 3]{MechanicalBurq}\cite{BoucletMizutani2018}\cite{Zubeldia2014}, under some assumptions on $V$.
    Actually, the assumptions listed at the beginning of this paragraph 
    suffice for Melrose's arguments in \cite{Me94} to go through. Unfortunately, we could not find a citation for this fact in the literature, so see \S\ref{sec:black-box_LAP}. 
\end{example} 

The appendix \S\ref{sec:black-box_LAP} is devoted to an explanation of how Melrose's results can be generalized to the setting of the previous two examples.

Now we state the theorem:
\begin{theorem}
	\label{thm:main_black-box} 
    Consider the setup of \cref{eq:bbsteup} with a given admissible set of functions satisfying the  well-posedness assumption above \hyperlink{WP_asump}{(WP)}. Then, there exists a solution $u\in C^\infty(U^\complement \cup A)$ of the Schr\"odinger--Helmholtz equation 
    \begin{equation} 
        Pu=0
    \end{equation} 
    satisfying the black-box condition and such that, outside of the black-box, $u$ has the form \begin{equation}
        u = e^{ix} \Big( 1 + \frac{a}{r^{L-1}} \Big) + e^{ir} \frac{w}{r^{(d-1)/2}}
    \end{equation}
    for $a,w$ satisfying $a \in \rho_{\mathrm{ff}}^{-(L-1)} C^\infty(X)$,  $w\in  \calA^{(0,0),\calF}(X)$, for some index set $\calF\subset (\bbN\times \{0\})\cup (\bbZ^{\geq L-d}\times \bbN )$.
\end{theorem}
\begin{proof}
    Most of the construction above was near $\infty \overline{\bbR^d}$, so proceeds identically here. The only step in which the deep interior is relevant is the final step, when the true resolvent is applied to solve away a Schwartz error, which we can take to be supported arbitrarily close to $\infty \bbS^{d-1}$, hence away from the black-box. The black-box assumption above was formulated in such a way to replace that step. 
\end{proof}

Assuming that the set of admissible functions is preserved under integration in $\omega$, and that the solution in (\textbf{WP}) depends continuously on $\omega$, then
the Poisson map makes sense in this generality. The conclusion and proof of
\Cref{prop:Poisson} apply verbatim.

\section{Some more about the inverse-square example}
\label{sec:dipole_some_more}

This ancillary section provides the rigorous back-end of \S\ref{subsec:inverse_square_example}.

For the inverse-square potential $V=\alpha/r^2 \in C^\infty(\bbR^3\backslash \{0\})$ with $\alpha>-1/4$, let $P=\triangle-1+V$ denote its Friedrichs realization. Applying
\Cref{thm:main_black-box}, with the singularity at the origin included
in the black box, shows that the perturbed plane wave has the form
\begin{equation}
    u=e^{ix}\Bigl(1+\frac{a}{r}\Bigr)
      +\frac{e^{ir}w}{r},
\end{equation}
for 
\begin{equation}
    a\in\rho_{\mathrm{ff}}^{-1}C^\infty(X),
    \qquad
    w\in\calA^{(0,0),\calF}(X),
\end{equation}
where $\calF$ is the index set furnished by \Cref{thm:main_black-box} with $L=2$ and $d=3$.
We have already explained how to compute $a$ to arbitrarily high order.
Focusing on the situation away from $\mathrm{bf}$, we can write $u=e^{ix} +e^{ir} \hat{v}$ for  
\begin{equation} 
    \hat{v} = \frac{w}{r} + e^{-i(r-x)} \frac{a}{r} \in  \rho_{\mathrm{ff}} \calA^{*,0}(X\backslash \mathrm{bf}) \cap \rho_{\mathrm{bf}}\rho_{\mathrm{ff}} L^\infty.
\end{equation}
Our goal in this appendix is to compute $\hat{v}$ modulo $O(1/r)$.
So far, we have mostly been using $\hat{\bfy}=\bfy/x^{1/2}$ to parametrize $\mathrm{ff}^\circ$. Here it is more convenient to use $\bfk\coloneq \bfy (r+x)^{-1/2}$ (away from the forward axis, this agrees with $y^{-1}\bfy\sqrt{r-x}$, and the former formula gives its smooth extension across $y=0$), because this is what allows us to combine the leading parts of the plane and spherical wave parts at $\mathrm{ff}$ to get the leading part of $\hat{v}$. (The point is that the ratio $e^{-i(r-x)}$ depends on $k$ alone.)

Extracting the leading-order term at $\mathrm{ff}$, 
\begin{equation}
    \exists 
    \hat{v}_1 \in C^\infty(\bbR^{d-1})\cap \langle k \rangle ^{-2} L^\infty(\bbR^{d-1}_{\bfk } )\text{ s.t. }\hat{v}- \rho_{\mathrm{ff}} \hat{v}_1\Big( \frac{\bfy}{\sqrt{r+x}}\Big) \in  \rho_{\mathrm{ff}}^2 \calA^{*,0-}(X\backslash \mathrm{bf}) \cap \rho_{\mathrm{bf}} \rho_{\mathrm{ff}}^{2-} L^\infty(X)
\end{equation}
near $\mathrm{ff}$. Indeed, take $\hat{v}_1(\bfk) = w_1(\bfk)+e^{-ik^2 } a_1(\bfk)$, where $a_1,w_1\in \langle k\rangle ^{-2} C^\infty( (\overline{\bbR^{d-1}_{\bfk}})_2)$ are the leading terms of $a,w$ at $\mathrm{ff}$, respectively: 
\begin{equation}
    a - r\rho_{\mathrm{ff}} a_1\Big( \frac{\bfy}{\sqrt{r+x}}\Big)   \in C^\infty(X),\quad w-r\rho_{\mathrm{ff}}w_1\Big( \frac{\bfy}{\sqrt{r+x}} \Big)\in \calA^{(0,0),0-}(X).
\end{equation}

On $\mathrm{ff}^\circ$, 
\begin{equation}
    \bfk = \hat{\bfy}/\sqrt{2}, 
\end{equation}
but the two differ slightly in the interior. 
Below, when we write `$\hat{v}_1$,' it is implicit that the argument is $\bfk$ (and actually $k=|\bfk|$ alone, by cylindrical symmetry). We note, for the reader's convenience, that
\begin{align}
     k^2=|\bfk|^2
    =\frac{y^2}{r+x}
    =r-x,
    \qquad
    z=\frac{r-x}{2}=\frac{k^2}{2}.
\end{align}

So, $ u
    =
    e^{ix}
    +e^{ir}\rho_{\mathrm{ff}}\hat v_1
    +O(r^{-1}\rho_{\mathrm{ff}}^{0-})$.
Being precise with logs, the $\rho_{\mathrm{ff}}^{0-}$ can be replaced by $\log \rho_{\mathrm{ff}}$. 

To complement the choice of coordinate $\bfk$, we use 
\begin{equation} 
    \rho_{\mathrm{ff}} = \sqrt{\frac{1+2k^2}{x} }
\end{equation} 
as a boundary-defining-function of $\mathrm{ff}$. 
Writing $\rho_{\mathrm{ff}}^{(0)}\coloneqq
x^{-1}\sqrt{x+y^2}$ for our usual choice, we have
$\rho_{\mathrm{ff}}=\gamma\rho_{\mathrm{ff}}^{(0)}$ for some smooth
positive function $\gamma$ with $\gamma|_{\mathrm{ff}}=1$. Thus, the two choices define the same weighted spaces, and $\gamma$ can be absorbed into the smooth coefficients.

Letting $I(j)$ denote the differential operator $I(j):\bullet\mapsto \rho^{-j}_{\mathrm{ff}} x N_{\mathrm{ff}}(\hat{P}) (\rho_{\mathrm{ff}}^j \bullet)$:
\begin{lemma}\label{lem:dipole_lead_vanish}
    $ I(1) \hat{v}_1 = 0$.
\end{lemma}
\begin{proof} 
Using the PDE $Pu=0$, 
\begin{equation}
    0 = Pu = e^{ix} V + e^{ir} ( N_{\mathrm{ff}}(\hat{P}) (\rho_{\mathrm{ff}} \hat{v}_1) + N_{\mathrm{ff}}(\hat{P}) (\hat{v} - \rho_{\mathrm{ff}} \hat{v}_1)  + (\hat{P}-N_{\mathrm{ff}}(\hat{P}) ) \hat{v} ) .
\end{equation}
Noting that $\hat{P}$ is a b-operator inducing two orders of decay at $\mathrm{ff}$, all of the terms on the right-hand side above are $O(\rho_{\mathrm{ff}}^{4-})$ away from $\mathrm{bf}$, with the exception of 
\begin{equation} 
    e^{ir} N_{\mathrm{ff}}(\hat{P}) (\rho_{\mathrm{ff}} \hat{v}_1) \in e^{ir} \rho_{\mathrm{ff}}^3 C^\infty(\bbR^{d-1}_{\bfk })  .
\end{equation} 
This cannot cancel out with something which is $O(\rho_{\mathrm{ff}}^{4-})$, so we must have $N_{\mathrm{ff}}(\hat{P}) (\rho_{\mathrm{ff}} \hat{v}_1)=0$. 
\end{proof}

The kernel $\ker_{\calS'(\mathrm{ff})} I(1)$ we work out in \S\ref{subsec:QHO_separation}; up to multiplicative constants, the only solution with cylindrical symmetry and the allowed behavior at $\hat{y}=\infty$ is
\begin{equation}
    \frac{1}{\sqrt{1+\hat{y}^2}} \cdot {}_1 F_1\Big( \frac{1}{2}, 1, - \frac{i\hat{y}^2}{2} \Big) =  \frac{1}{\sqrt{1+\hat{y}^2}} e^{-i\hat{y}^2/4} J_0 \Big( \frac{\hat{y}^2}{4} \Big) , 
\end{equation}
where we used the identity
${}_1 F_1(2^{-1},1,z) = e^{z/2} I_0(z/2)$ \cite[\href{http://dlmf.nist.gov/13.6.iii}{\S13.6.iii}]{NIST}. 
Thus, 
\begin{equation}
    \hat{v}_1|_{\mathrm{ff}^\circ} =  \frac{a_0}{\sqrt{1+\hat{y}^2}}   e^{- i\hat{y}^2/4} J_0 \Big( \frac{\hat{y}^2}{4} \Big)
\end{equation}
for some $a_0\in \bbC$.  This means 
\begin{equation}
    \hat{v}_1 = \frac{a_0}{\sqrt{1+2k^2}} e^{-ik^2/2} J_0\Big(\frac{k^2}{2}\Big) = \frac{a_0}{\sqrt{1+2k^2 }} e^{-iz} J_0(z).  
\end{equation}
We conclude 
\begin{equation}
    u = e^{ix} +  \frac{e^{ir}}{\sqrt{x}} a_0  e^{- iz} J_0 (z) + O \Big(\frac{\log \rho_{\mathrm{ff}} }{r} \Big) =  e^{ix} +  \frac{e^{ir}}{\sqrt{r}} a_0  e^{- iz} J_0 (z) + O \Big(\frac{\log \rho_{\mathrm{ff}} }{r} \Big). 
\end{equation}

This justifies  the formula for $u^{[1]}$ presented in \cref{eq:u_brak1}.

Next, we determine the log term. 
Let $\rho_{\mathrm{ff}}^2 (\log \rho_{\mathrm{ff}})  \hat{v}_2$ denote the next correction to $\hat{v}$. Thus, 
\begin{equation}\label{eq:dipole_v2_term}
    \hat v_2
\in
C^\infty(\bbR^{d-1})
\cap
\langle k\rangle^{-2}
L^\infty(\bbR^{d-1}_{\bfk}) \text{ and } \hat{v}-\rho_{\mathrm{ff}} \hat{v}_1- \rho_{\mathrm{ff}}^2 (\log \rho_{\mathrm{ff}}) \hat{v}_2 \in \rho_{\mathrm{ff}}^2 \calA^{*,0}(X\backslash \mathrm{bf}) \cap \underbrace{\rho_{\mathrm{bf}} \rho_{\mathrm{ff}}^2 L^\infty(X)}_{=O(1/r)} 
\end{equation}
near $\mathrm{ff}$. 

\begin{lemma}
    $I(2) \hat{v}_2=0$. 
\end{lemma}
\begin{proof}
    Above, we showed that 
\begin{align}
\begin{split} 
    0 = Pu &= e^{ix} V+e^{ir} (N_{\mathrm{ff}}(\hat{P}) (\hat{v} - \rho_{\mathrm{ff}} \hat{v}_1)  + (\hat{P}-N_{\mathrm{ff}}(\hat{P}) ) \hat{v} ) \\ 
    &= e^{ix} V+e^{ir} (N_{\mathrm{ff}}(\hat{P}) (\rho_{\mathrm{ff}}^2 (\log \rho_{\mathrm{ff}})\hat{v}_2 ) + N_{\mathrm{ff}}(\hat{P}) (\hat{v} - \rho_{\mathrm{ff}} \hat{v}_1 - \rho_{\mathrm{ff}}^2 (\log \rho_{\mathrm{ff}}) \hat{v}_2 )  + (\hat{P}-N_{\mathrm{ff}}(\hat{P}) ) \hat{v} ).
    \end{split} 
\end{align}
When no derivative falls on $\log\rho_{\mathrm{ff}}$, the resulting
term is
\begin{align}
    x^{-1}\rho_{\mathrm{ff}}^2(\log\rho_{\mathrm{ff}})
I(2)\hat v_2
=
\frac{\rho_{\mathrm{ff}}^4\log\rho_{\mathrm{ff}}}
     {1+2k^2}
I(2)\hat v_2.
\end{align}
Apart from this displayed $\rho_{\mathrm{ff}}^4\log\rho_{\mathrm{ff}}$
term, all the remaining terms are smooth
$O(\rho_{\mathrm{ff}}^4)$. Uniqueness of polyhomogeneous coefficients
therefore implies $I(2)\hat v_2=0$.
\end{proof}

Referring again to \S\ref{subsec:obstruction} for the calculation of $\ker I(2)$: up to multiplicative constants, the only solution with the requisite cylindrical symmetry is 
\begin{equation}
    \frac{1}{1+ \hat{y}^2} \cdot {}_1 F_1 \Big( 0,1, -  \frac{i\hat{y}^2}{2}\Big)=\frac{1}{1+ \hat{y}^2}. 
\end{equation}
It so happens that the ${}_1F_1(0,1,z)=1$, hence our logarithmic correction $\hat{v}_2$ will lack an interesting $\mathrm{ff}$ profile: $\hat{v}_2 = a_1 /(1+\hat{y}^2)$ for some $a_1\in \bbC$.  We conclude that 
\begin{align}
    \begin{split} 
    u&=e^{ix} + \frac{e^{ir}}{\sqrt{x}} a_0 e^{-iz} J_0(z) + \frac{e^{ir}}{x} a_1 \log \rho_{\mathrm{ff}} + O \Big(\frac{1}{r} \Big)  \\
    &=e^{ix} + \frac{e^{ir}}{\sqrt{x}} a_0 e^{-iz} J_0(z) + \frac{e^{ir}}{r} a_1 \log \Big(\theta^2+\frac{1}{r} \Big) + O \Big(\frac{1}{r} \Big),
    \end{split} 
\end{align}
using the estimate in \cref{eq:dipole_v2_term} (here the second equality is understood modulo a smooth
$O(r^{-1})$ term, which is absorbed into the remainder, and the
constant normalization relating the two logarithms has been absorbed
into $a_1$). We have therefore recovered the correction $u^{[2,0]}$ presented in \cref{eq:u_brak20}, except that we have yet to determine the constant. 

The constant can be read off the $\log \theta$ singularity of $f(\theta)=w|_{\infty \bbS^{2}}$, as seen in \cref{eq:f_ap}. The resulting $a_1$ is recorded in \cref{eq:dipole_constants}.

\section*{Notation}

\begin{itemize}
    \item $\bbR^d=\bbR^d_z$ denotes Euclidean space, with coordinates split as $\bbR^d = \bbR_x\times\bbR^{d-1}_{\bfy}$. We let $y=\lvert\bfy\rvert$, and $r=\sqrt{x^2+y^2}$.
     \item $\bbN=\{0,1,2,3,\dots\}$ denotes the natural numbers (which includes $0$).
     \item $\langle u  \rangle \coloneqq \sqrt{1+\lvert u \rvert^2}$ is the Japanese bracket of an element $u\in \bbR^d$.
    \item $\overline{\bbR^d} = \bbR^d\sqcup\infty\bbS^{d-1}$ denotes the \textit{radial compactification} of $\bbR^d$, viewed as the union of $\bbR^d$ with the sphere at infinity. This is a manifold with boundary, with $1/r$ as a boundary defining function (at least near the boundary). 
    \item $X = [\overline{\bbR^d};\rightarrow]_{\mathrm{par}}$ is the manifold-with-corners obtained by beginning with $\overline{\bbR^d}$ and performing a parabolic blowup of  the forward direction $\rightarrow\, \in \infty \bbS^{d-1}$. 
    This manifold-with-corners has two faces (boundary hypersurfaces), the \textit{boundary face}, $\mathrm{bf}$, representing non-forward directions at infinity, and the \textit{forward face}, $\mathrm{ff}$, representing the endpoints of forward parabolic trajectories at infinity. See Figure \ref{fig:X} for coordinates.
    \item $\rho_{\mathrm{f}} $ denotes a boundary-defining-function of the face $\mathrm{f}\in \{\mathrm{bf},\mathrm{ff}\}$. 
    Depending on the context, this may be a local boundary-defining-function, valid only in some local coordinate chart, or it may be globally defined.

      \item $\hat \bfy \coloneqq \bfy x^{-1/2} \in \bbR^{d-1}$  (in the region $x>0$) is a useful coordinate used throughout the text.
    
    \item $Y=[\overline{\bbR^d};\rightarrow]$ is the manifold-with-corners obtained by performing an ordinary blowup of the forward direction $\to$ (instead of the parabolic blowup). 
    See Figure \ref{fig:Y} for coordinates.
    \item For any manifold-with-corners $M$ with $n\in \bbN^+$ boundary hypersurfaces, $\mathcal{A}^{\alpha_{1},\dots ,\alpha_n}(M)\subset C^\infty(M^\circ)$ is the space of all \textit{conormal functions} with $\alpha_1,\dots,\alpha_n\in \bbR$ orders of decay at the respective boundary hypersurfaces of $M$. 
    
    See \cref{eq:1349} and \cref{eq:1352} for the definition of conormal functions on our main manifold $X$.
    \item $\mathcal{A}^{\mathcal{E}_1,\dots \mathcal{E}_n}(M)$ is the space of all \textit{polyhomogeneous functions} with \textit{index sets} $\mathcal{E}_1,\dots,\mathcal E_n$ 
    at the various boundary hypersurfaces.
    
    Polyhomogeneous functions and index sets are defined in \S \ref{s:function_spaces}.
	\item $\operatorname{Diff}^{m,s,\ell}_{\mathrm{b}}(X)$ consists of \textit{b-operators} on $X$ with at most $m$ derivatives and with $s$ and $\ell$ orders of growth at bf and ff respectively.

    Specifically, if $\mathcal{V}_{\mathrm{b}}(X)\subset \calV(X) $ denotes the space of smooth vector fields tangent to the boundary of $X$, then $\operatorname{Diff}^{m,0,0}_{\mathrm{b}}(X)$ is the vector space spanned by products of at most $m$ elements of $\mathcal{V}_{\mathrm{b}}(X)$, and
    \begin{align}
       \operatorname{Diff}^{m,s,\ell}_{\mathrm{b}}(X) \coloneqq \rho_{\rm{bf}} ^{-s}\rho_{\rm{ff}}^{-\ell}\operatorname{Diff}^{m,0,0}_{\mathrm{b}}(X).
    \end{align}
    \item $P$ is the Helmholtz operator, as specified in \cref{eq:P}.
    \item $\tilde{P} = e^{-ix}Pe^{ix}$ is the plane-wave conjugated Helmholtz operator (first defined in \cref{eq:883}).
    \item $\widehat{P} = e^{-ir}Pe^{ir}$ is the spherical-wave conjugated Helmholtz operator (first defined in \cref{eq:883}).
    \item $N_{\mathrm{bf}}(\bullet)$, $N_{\mathrm{ff}}(\bullet)$ denote the b-normal operators of an operator $\bullet$ at $\mathrm{bf}$ and $\mathrm{ff}$, respectively. See \cref{eq:Ns} for explicit formulas.
	\item  $N^{-1}$ denotes the inversion of the relevant normal operator (depending on the step of the proof).
    In step 1, $N^{-1} \coloneqq (N_{\rm{bf}}(\tilde P))^{-1}= (-2i\p_x)^{-1}$ which (after choosing a constant of integration) we define as integration from $x=-\infty$, along level sets of $\bfy\in \bbR^{d-1}$.  That is, 
	\begin{equation}
		(N^{-1} f)(x,\bfy) = -\frac{1}{2i}\int_{-\infty}^x f(s,\bfy) \dd s \label{eq:1826}
	\end{equation}
	for any function $f\colon \bbR^{d-1}_{\bfy}\to L^1(\bbR_x)$.

\end{itemize}

We use tildes to signal that we are working with $\tilde{P}$ and hats to signal that we are working with $\hat{P}$.

\appendix

\section{The Born heuristic}
\label{sec:Born}

We recall the Born approximation, as presented by Landau--Lifshitz \cite[(126.12)]{LL} for a central short-range potential $V(r)$ on $\bbR^3$. It reads 
\begin{equation}
		f(\theta) \approx - \int_0^\infty \frac{\sin (q r)}{q} r V(r) \dd r  ,\quad q=2\sin(\theta/2) .
		\label{eq:Born_0}
\end{equation}
Here $0\leq\theta\leq\pi$, so $q\geq0$, with $q>0$ whenever
$\theta\neq0$.
Because this is merely an approximation, there is no guarantee about the accuracy of any qualitative predictions about the structure of the perturbed plane wave made on its basis.
Regardless, we can press ahead with the $\theta\to 0$ limit of the right-hand side.

Under the classical-symbol hypothesis
$V\in\langle r\rangle^{-L}C^\infty(\overline{\bbR^3})$,
with $L\geq 2$, the integral converges for each $\theta\neq 0$;
for $L=2$ the convergence is generally only conditional. 
If $\theta\approx 0$, then na\"ively plugging in $\theta=0$ yields 
\begin{equation}
    -f(\theta)  = \int_0^\infty \frac{\sin (qr)}{qr} r^2 V(r) \dd r \approx  \int_0^\infty r^2 V(r) \dd r .
\end{equation}
If $V(r)=O(r^{-4})$ as $r\to\infty$, then this substitution is
justified. On the other hand, if
$V(r)\sim c r^{-L}$ for some $c\neq 0$ and $L\leq 3$, then the final
integral diverges.
This suggests that $f(\theta)$ diverges as $\theta\to 0^+$. 
A more careful analysis reveals that it diverges at the same asymptotic rate as 
\begin{equation}
    \int_0^R r^2 V(r) \dd r 
    \label{eq:Born_numerology} 
\end{equation}
diverges as $R\to\infty$. 

More precisely, and more generally, we have polyhomogeneity: 
\begin{proposition}
    For any $L\in\bbN$ with $L\geq 2$ and any central
    $V\in \langle r\rangle^{-L} C^\infty(\overline{\bbR^3})$,
    \begin{equation}
        \int_0^\infty \frac{\sin (qr)}{qr} r^2 V(r) \dd r \in \calA^{\calE}([0,1)_q)  
    \end{equation}
    for some index set $\calE$. For one choice of admissible $\calE$, the first permitted element not of the form $(j,0)$ for $j\in \bbN$ is 
    \begin{equation}
        \begin{cases}
            (L-3,0) & (L= 2), \\ 
            (L-3,1) & (\text{odd }L\geq 3), \\
            (L-2,1) & (\text{even }L\geq 4)
        \end{cases}
    \end{equation}
\end{proposition}
This agrees with our discussion of the inverse-square potential.

\begin{proof}
    First note that 
    \begin{equation}
        \int_0^1 \frac{\sin (qr)}{qr} r^2 V(r) \dd r = \sum_{j=0}^\infty \frac{(-1)^j}{(2j+1)!} q^{2j} \int_0^1 r^{2j+2} V(r) \dd r     
    \end{equation}
    is smooth (in fact, analytic) as a function of $q$, so we focus on the remainder of the integral.
For each $K\in\bbN$, expand
\begin{equation}
    V
    =
    V_{\mathrm{rem},K}
    +\sum_{j=0}^K r^{-2-j}V_j,
\end{equation}
where $V_j=0$ for $j<L-2$ and
$V_{\mathrm{rem},K}\in
\langle r\rangle^{-3-K}C^\infty(\overline{\bbR^3})$.
Thus, throughout the proof, $j$ records the power relative to $r^{-2}$; the decay order $r^{-L}$ corresponds to $j=L-2$.
Then, 
    \begin{equation}
        \int_1^\infty \frac{\sin (qr)}{qr} r^2 V(r) \dd r = \sum_{j=0}^K  \frac{V_j}{q} \int_1^\infty \frac{\sin (qr) }{r^{1+j}}   \dd r +  \int_1^\infty \frac{\sin (qr)}{qr} r^2 V_{\mathrm{rem},K}(r) \dd r .
    \end{equation}

    The first batch of integrals can be written 
    \begin{equation}
        \int_1^\infty \frac{\sin (qr) }{r^{1+j}}   \dd r =  q^j \int_q^\infty \frac{\sin t }{t^{1+j}}   \dd t . 
    \end{equation}
    Note that
    \begin{equation}
        \int_q^\infty \frac{\sin t }{t^{1+j}}   \dd t = \underbrace{\int_2^\infty \frac{\sin t }{t^{1+j}}   \dd t}_{<\infty} +\int_q^2 \frac{\sin t }{t^{1+j}}   \dd t. 
    \end{equation}
    The first term on the right-hand side is a (finite) constant. The second term can be calculated by Taylor expanding 
    \begin{equation}
        \int_q^2 \frac{\sin t }{t^{1+j}}   \dd t = \sum_{k=0}^{j} \frac{(-1)^{k} }{(2k+1)!} \int_q^2 t^{2k-j}   \dd t + \int_q^2 \underbrace{\frac{1}{t^{1+j}}\bigg[ \sin t -  \sum_{k=0}^j \frac{(-1)^k}{(2k+1)!}  t^{2k+1}\bigg]}_{\in t^{j+1} C^\infty([0,3)_t) } \dd t     . 
    \end{equation}
    By the fundamental theorem of calculus, the last term is a smooth function of $q$. The other terms on the right-hand side are all explicit: 
    \begin{equation}
        \int_q^2 t^{2k-j}   \dd t  = 
        \begin{cases}
            \log(2/q)  & (2k=j-1), \\ 
            \frac{1}{2k-j+1} (2^{2k-j+1} - q^{2k-j+1} ) & (\text{otherwise}),
        \end{cases}
    \end{equation}
    which is polyhomogeneous. In summary, 
    \begin{equation}\label{eq:Born_explicit_indices}
        \frac{1}{q}\int_1^\infty \frac{\sin (qr)}{r^{1+j}} \dd r \in  q^{\min\{0,j-1\}} C^\infty([0,1)_q) + q^{j-1} (\log q )C^\infty([0,1)_q),
    \end{equation}
    where the log term is absent if $j$ is even.

    On the other hand, given any fixed $k\in \bbN$, if we take $K\gg 1$ large enough, then  
    \begin{equation} 
        \int_1^\infty \frac{\sin (qr)}{qr} r^2 V_{\mathrm{rem},K}(r) \dd r\in C^k([0,1)_q), 
    \end{equation} 
    as differentiating under the integral sign shows. 
    Thus, what we have shown is that there exists a real index set $\calE\supset \bbN\times \{0\}$ such that 
    \begin{equation}
        \int_0^\infty \frac{\sin (qr)}{qr} r^2 V(r) \dd r \in \bigcap_{k\in\bbN}
\left(
    \calA^{\calE}([0,1)_q)
    +C^k([0,1)_q)
\right).
    \end{equation}
    This is sufficient to conclude polyhomogeneity, as \Cref{lem:funny_phg_alt_criterion} shows.
    The claim about the first nonsmooth index is read off \cref{eq:Born_explicit_indices}.
\end{proof}

\begin{lemma}\label{lem:funny_phg_alt_criterion}
    For any index set $\calE \subset \bbR\times \bbN$ containing $\bbN\times \{0\}$, 
   \begin{equation}
    \calA^{\calE}([0,1))
    =
    \bigcap_{k\in\bbN}
    \left(
        \calA^{\calE}([0,1))
        +C^k([0,1))
    \right).
\end{equation}
\end{lemma}
\begin{proof}
    The Taylor remainder theorem implies $C^{2k}[0,1)\subset C^\infty[0,1) + q^k C^k[0,1)_q$.
    So, $\calA^\calE + C^{2k} \subset \calA^\calE + q^k C^k[0,1)_q$, because $\calE\supset \bbN\times \{0\}$. This gives  
    \begin{equation}
        \bigcap_{k\in \bbN} ( \calA^{\calE}[0,1) + C^k[0,1) ) \subseteq \bigcap_{k\in \bbN} ( \calA^{\calE}[0,1) +  q^k C^k[0,1)_q ).
    \end{equation}
    That is, for each $k\in \bbN$, there exist $f_k \in \calA^\calE, g_k\in q^k C^k[0,1)_q$ such that $f=f_k+g_k$. 
    Each $f_k$ admits a polyhomogeneous expansion, and while the terms may depend on $k$, each term must stabilize as $k\to\infty$. 
    Indeed, if $K>k$, then
$f_K-f_k=g_k-g_K\in q^kC^k[0,1)_q$, so uniqueness of
polyhomogeneous coefficients implies agreement of all terms of order
strictly less than $k$.
    Consequently, for any $\alpha\in \bbR$, the terms in $f_k$'s expansion with index $\in (-\infty, \alpha] \times \bbN$ do not change with $k$ for $k>\alpha$.  
    By Borel's lemma (\Cref{lem:Borel}), $\exists$ an $f_\infty \in \calA^\calE$ whose polyhomogeneous expansion is the stabilized sequence. Thus, for any $k\in \bbN$, if $K\gg K_0\gg k$, then
    \begin{equation}
        f_\infty- f=(f_\infty - f_K)-g_K \in \calA^{K_0} + q^K C^{K}[0,1) \subset q^k C^k[0,1)_q,
    \end{equation}
    using that $\calA^{2k+1}[0,1) \subset q^{k} C^k[0,1)_q$. 
    Finally, note that $\bigcap_{k\in \bbN} q^k C^k[0,1)_q$ consists of functions that are rapidly decaying as $q\to 0^+$ along with all  of their derivatives. So, the equation above shows that $f\in \calA^\calE$. 
\end{proof}

\begin{example}[Inverse-square potential, cont.]
    When $V(r)=\alpha/r^2$, then 
    \begin{equation}
        \int_0^\infty  \frac{\sin (qr)}{q} r V(r) \dd r=
        \alpha \int_0^\infty \frac{\sin(qr)}{qr}\dd r
= \frac{\pi\alpha}{2q},
\qquad q>0  ,
    \end{equation}
    so the Born approximation says 
    \begin{equation}
       f(\theta) \approx -\frac{\pi\alpha}{2q}
= -\frac{\pi\alpha}{2\theta}
\bmod C^\infty([0,\pi/2)_\theta).
    \end{equation}
    The Born approximation has successfully  reproduced the leading singularity reported in \cref{eq:dipole_scat_sing}, but not the logarithmic singularity. 
\end{example}
Technically, the previous example is not a special case of the proposition, because $1/r^2$ is singular at $r=0$, but the proof still applies.

\section{Coulomb plane waves}
\label{sec:Coulomb} 
\subsection{Exact Rutherford scattering}\label{subsec:Coulomb_exact}
Let $d \geq 3$ and
\begin{equation} 
V=-\mathsf{Z}/r
\end{equation} 
be the  Coulomb potential generated by a point-particle with nonzero electric charge $\mathsf{Z}\in \bbR\backslash \{0\}$; if $\mathsf{Z}>0$, then the potential is attractive, and repulsive otherwise. 
The Coulomb potential is the archetypical example of a long-range interaction.
When $d=3$, both the classical and quantum scattering cross sections coincide and are given by the famous Rutherford formula
\begin{equation}
|f(\theta)|^2=\frac{\mathsf{Z}^2}{4}\frac{1}{(2\sin^2(\theta/2))^2},
\label{eq:Rutherford}
\end{equation}
see \cite[\S1]{YafCol} for definitions and references on Coulomb cross sections. As $\theta \to 0$, the amplitude $|f|$ diverges as $|\mathsf{Z}|/\theta^2$. 
Thus, $f$ is smooth away from the forward direction, with a genuine singularity in the forward direction. 

In this setting, an exact formula for the perturbed plane waves is known. 

When $d=3$, the observation that an exact Coulomb potential leads to a separable problem in parabolic coordinates goes back to Schr\"odinger himself \cite{Schrodinger}, leading to the following formula (which we state for general $d \geq 3$)\setcounter{footnote}{0}\footnote{ \cite[(2.1)]{Temple}\cite[(21)]{Mott}\cite[(7')]{Gordon}}:
\begin{equation}
u = e^{i x} F(r- x), 
\label{eq:Coulomb_plane_exact}
\end{equation}
where $F$ is a certain solution of the confluent hypergeometric equation,
\begin{equation}\label{eq:colexp}
    F(s) =  e^{\pi \mathsf{Z}/4} \frac{\Gamma(\frac{d-1}{2}- \frac{i\mathsf{Z}}{2} )}{\Gamma(\frac{d-1}{2})}\cdot {}_1F_1\Big( \frac{i\mathsf{Z}}{2},\frac{d-1}{2},is  \Big),
\end{equation}
where ${}_1F_1(a,b,-)$ is Kummer's confluent hypergeometric function \cite[\href{http://dlmf.nist.gov/13.2.E2}{Eq.\ 13.2.2}]{NIST}. The real part $\Re u$ of $u$ is plotted in \Cref{fig:Coulomb}. The imaginary parts look similar.

\begin{remark*}
    The normalization constant in \cref{eq:colexp} is a matter of convention. The choice presented above is standard and leads to the simplest asymptotic at $\mathrm{bf}$.
\end{remark*}

\begin{figure}[!htbp]
    \begin{center}
        \includegraphics[scale=1]{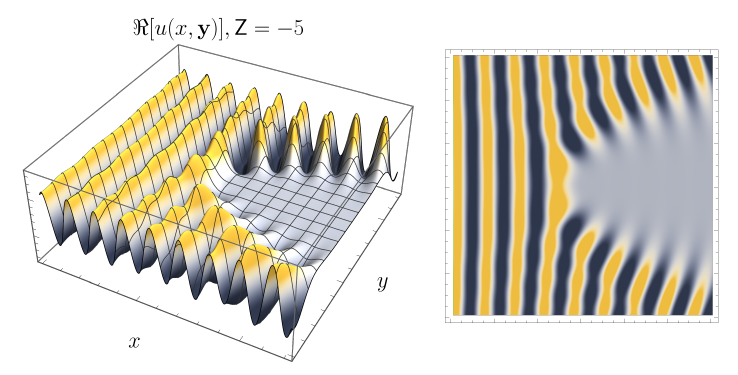}
        \includegraphics[scale=1]{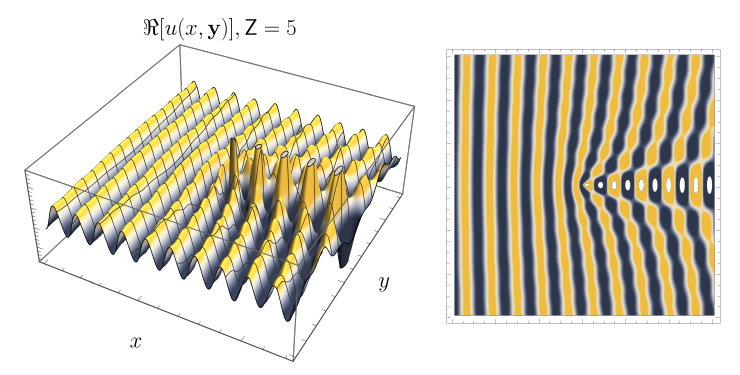}
    \end{center}
    \caption{The (real parts of the) Coulomb plane waves $u$, with $\mathsf{Z}=\pm 5$ and $d=3$. The horizontal scales are the same between the plots, but the vertical scales differ.}
    \label{fig:Coulomb}
\end{figure}

Dramatically apparent is the parabolic wake.  
For a repulsive Coulomb potential, the incoming plane 
wave is deflected, leaving a large parabolic shadow, something rediscovered several times \cite[Problem I.12]{Sommerfeld}\cite{CoulombShadow1}\cite{CoulombShadow2}\cite{CoulombShadow3}\cite{Hoffmann}. For an attractive Coulomb potential, the incoming plane wave is focused, generating an outgoing beam. 

Fixing $\theta\in (0,\pi]$ and sending $r\to\infty$, we have $r-x\to \infty$. So, we can approximate $f(r-x)$ using the large-argument expansions of the confluent hypergeometric functions \cite[\href{http://dlmf.nist.gov/13.7}{\S13.7}]{NIST}: 
\begin{equation}
 F(s) =  s^{-i\mathsf{Z}/2}\Big(1+\frac{D}{s}\Big)+C e^{is } s^{-\frac{d-1}{2}+i\mathsf{Z}/2}
 +O\Big(\frac{1}{s^2}\Big)
\end{equation}
as $s\to\infty$, where $C=e^{-i\pi(d-1)/4}  \Gamma(\frac{d-1}{2}-i\mathsf{Z}/2)\Gamma(i\mathsf{Z}/2)^{-1}$ and $D= \mathsf{Z}(d-3-\mathsf{Z}i)/4$ are some constants. 
Consequently, $u$ has the following form:
\begin{equation}
u=  \bigg[  \overbrace{(r-x)^{-i\mathsf{Z}/2}}^{\mathclap{\text{Dollard correction}}} \underbrace{e^{i  x}\Big(1+\frac{D}{r-x}\Big)}_{\text{perturbed plane wave}} + \overbrace{(r-x)^{i\mathsf{Z}/2}}^{\mathclap{\text{Dollard correction}}} \underbrace{\frac{C e^{i  r}}{(r-x)^{\frac{d-1}{2}}}}_{\mathclap{\text{spherical wave}}} \bigg]
+\underbrace{O\bigg(\frac{1}{(r-x)^2}\bigg)}_{\text{lower order at bf}} .
\end{equation} 
So, 
\begin{equation}
    |f(\theta)|^2 =\Bigg|\frac{Cr^{\frac{d-1}{2}}}{(r-x)^{\frac{d-1}{2}-i\mathsf{Z}/2}}\Bigg|^2 = \frac{|C|^2}{(1-\cos \theta)^{d-1}} =\frac{|\Gamma(\frac{d-1}{2}- \frac{i\mathsf{Z}}{2} )|^2}{|\Gamma(\frac{i\mathsf{Z}}{2})|^2}\frac{1}{ (2\sin^2(\theta/2))^{d-1}},
\end{equation}
recovering the Rutherford formula \cref{eq:Rutherford} when $d=3$.
The ``Dollard correction'' $(r-x)^{\pm i\mathsf{Z}/2}$ is a logarithmically oscillating term whose presence is owed to the long-range of the Coulomb potential.

To summarize, the Coulomb plane wave $u$ has the same structure as the perturbed plane waves for short-range potentials, except with the Dollard correction. 

\subsection{General Coulombic potentials}
Potentials of the form Coulomb + short-range can be handled using a variant of the method we apply in the main theorem. We state the main result in the general black-box setting similar to \S\ref{sec:black-box}.

Let $U\subset \bbR^d$ denote a bounded open set and $A\subset U$ an open subset, such that $U\backslash A \Subset U$, where we again think of $B=U\backslash A$ as the ``black-box.'' We study the PDE 
\begin{equation}\label{eq:bbsteup2}
    Pu=f \text{ on $U^\complement \cup A = \bbR^d\backslash B$,}
\end{equation}
where 
\begin{equation}
    P=\bigtriangleup-1-\frac{\mathsf{Z}}{ r}+Q
\end{equation}
for some $\mathsf{Z} \in \bbR \setminus \{0\}$ (positive $\mathsf{Z}$ is attractive) and $Q$ satisfies the assumptions in \S\ref{subsec:asymptotically_Euclidean}. Let $L \in \bbN^{\geq 2}$ be as in \cref{eq:L_def}. 

We modify the well-posedness assumption to handle long-range problems:
\begin{itemize}
    \item[\hypertarget{LRWP_asump}{\textbf{(LRWP)}}] 
    \textit{For every Schwartz $f\in \calS(\bbR^d)$ supported outside the black-box region $U$, there exists a unique solution  
\begin{equation}
    u\in C^\infty( U^\complement \cup A)
\end{equation}
to the inhomogeneous Helmholtz equation $Pu=f$ such that 
\begin{enumerate}[label=(\roman*)] 
    \item near spatial infinity, $u\in e^{ir} r^{-(d-1)/2+i\mathsf{Z}/2} C^\infty(\overline{\bbR^d})$, where $\mathsf Z$ is the fixed constant
in the definition of $P$,
    \item $u$ satisfies the black-box condition, i.e., $u|_A$ is admissible.
\end{enumerate}
}
\end{itemize}
This differs from the usual well-posedness assumption \hyperlink{WP_asump}{(WP)} in \S\ref{sec:black-box} in the form of $P$ and the extra factor $r^{i\mathsf{Z}/2}$ in (i). When $\mathsf{Z}=0$, we recover the old well-posedness assumption \hyperlink{WP_asump}{(WP)}.
\begin{theorem}
	\label{thm:main_black-box-coul} 
    Consider the setup of \cref{eq:bbsteup2} with a given admissible set of functions satisfying the long-range well-posedness assumption \hyperlink{LRWP_asump}{(LRWP)}. Then, there exists a solution $u\in C^\infty(U^\complement \cup A)$ of the Schr\"odinger--Helmholtz equation 
    \begin{equation} 
        Pu=0
    \end{equation} 
    satisfying the black-box condition and such that, outside of the black-box, $u$ has the form \begin{equation}\label{eq:Coul-bb}
        u =  u_{\mathsf{Z}}(x,\bfy)+ \rho_{\mathrm{bf}}^{i\mathsf{Z}/2} e^{ix}   \frac{a}{r^{L-1}}  + \rho_{\mathrm{bf}}^{-i\mathsf{Z}/2}e^{ir} \frac{w}{r^{(d-1)/2}} 
    \end{equation}
    for $a,w$ satisfying $a \in \rho_{\mathrm{ff}}^{-(L-1)} C^\infty(X)$,  $w\in  \calA^{(0,0),\calF}(X)$, for some index set $\calF\subset \bbC \times \bbN$ such that every $(j,k)\in\calF$ not of the form $(n,0)$ or $(n-i\mathsf Z,0)$,
$n\in\bbN$, satisfies $\Re j\ge L-d$, and $u_{\mathsf{Z}}$ is the exact Coulomb plane wave
    \begin{equation}\label{eq:Coulomb_plane_exact2}
        u_{\mathsf{Z}}(x,\bfy) \coloneqq e^{i x} F(r- x),\qquad  F(s) =  e^{\pi \mathsf{Z}/4} \frac{\Gamma(\frac{d-1}{2}- \frac{i\mathsf{Z}}{2} )}{\Gamma(\frac{d-1}{2})}\cdot {}_1F_1\Big( \frac{i\mathsf{Z}}{2},\frac{d-1}{2},is  \Big),
    \end{equation}
    as in \cref{eq:Coulomb_plane_exact}.
\end{theorem}
The conclusion of this theorem \cref{eq:Coul-bb} differs from that of \Cref{thm:main_black-box} in two aspects: the appearance of the Coulomb plane wave $u_{\mathsf{Z}}$ and the Dollard corrections $\rho_{\mathrm{bf}}^{\pm i \mathsf{Z}/2}$, and the ff index set of $w$ is possibly complex. Setting $\mathsf{Z}=0$ completely recovers \Cref{thm:main_black-box}. 
Theorem \ref{thm:main_black-box-coul} applies to the example of multiple Coulomb singularities:
\begin{example}[Multiple Coulomb Singularities]\label{examp:multsingularities}
Let $P \coloneqq \bigtriangleup-1+V$ where $V$ satisfies
    \begin{equation}\label{eq:coulsings}
        \exists \mathsf{Z}_1,\ldots,\mathsf{Z}_k\in \bbR,z_1,\ldots,z_k \in \bbR^d \text{ s.t. }V+\frac{\mathsf{Z}_1}{|z-z_1|}+\cdots+\frac{\mathsf{Z}_k}{|z-z_k|} \in \langle r \rangle^{-2}C^\infty(\overline{\bbR^d} ).
    \end{equation}
    Then, take $U$ to be an open ball centered at the origin containing all of the $z_k$'s, and let $A$ be an annulus constructed from $U$ by subtracting a closed ball containing all of the $z_k$'s with a slightly smaller radius than that of $U$.
   
    The Helmholtz operator $P=\triangle-1+V$, defined initially on $C_{\mathrm{c}}^\infty(\bbR^d\backslash \{z_1,\ldots,z_k\})$, is semibounded by localizing to each singularity individually and using Hardy's inequality (or the fractional version thereof in $d=2$, see \cite{dimtwocol}). We define a Helmholtz solution $u\in C^\infty(A)$ to be admissible if it extends to a Helmholtz solution on $U\backslash \{z_1\ldots,z_k\}$ lying in the Friedrichs domain.

    The limiting absorption principle in this setting can be shown by arguments in \S\ref{sec:black-box_LAP} with minor modifications, see \cite[\S16]{Me94} for the attractive $\mathsf{Z}>0$ case, and \cite[Theorem 1.2]{ACL} for a stronger, uniform result down to low energy. If $\mathsf Z\coloneqq\mathsf Z_1+\cdots+\mathsf Z_k\ne0$, this proves
\hyperlink{LRWP_asump}{(LRWP)} with that value of $\mathsf Z$, and
\Cref{thm:main_black-box-coul} applies. If $\mathsf Z=0$, the leading
Coulomb term cancels and the short-range \Cref{thm:main_black-box}
applies instead, with $L=2$.
\end{example} 
Using the boundary-defining $\rho_{\mathrm{bf}}=(r-x+1)^{-1}$ (valid away from the origin), the potential-scattering case (i.e.\ the Coulombic version of \Cref{thm:main}) becomes:
\begin{itemize}
    \item[] Suppose $V=-\mathsf{Z}/r+V_{\mathrm{sr}}$ where $V_{\mathrm{sr}} \in \langle r \rangle^{-L}C^{\infty}(\overline{\bbR^d};\bbR),L \geq 2$. Away from the origin, the perturbed plane wave has the form  
    \begin{equation}
        u(x,\bfy) = u_{\mathsf{Z}}(x,\bfy)+ (r-x+1)^{-i\mathsf{Z}/2} e^{ix}   \frac{a}{r^{L-1}}  + (r-x+1)^{i\mathsf{Z}/2}e^{ir} \frac{w}{r^{(d-1)/2}} 
    \end{equation}
    for some $a,w \in C^\infty(\bbR^{d})$ as above. 
\end{itemize}

We sketch the proof of \Cref{thm:main_black-box-coul} in the context of the potential-scattering case above since the reduction to this case is similar to that of  \Cref{thm:main_black-box}. The proof requires slight modifications, as explained below.  Unfortunately, the Coulomb potential enters some normal operators $N$. Indeed, 
\begin{equation}
    \frac{1}{r} , N \in \operatorname{Diff}_{\mathrm{b}}^{2,-1,-2}(X) 
\end{equation}
have the same decay orders at $\mathrm{bf},\mathrm{ff}$.
Thus, \cref{eq:Ns} gets modified to 
    \begin{equation}
    \label{eq:normalops}
    \begin{split}
        N_{\mathrm{bf}}(\tilde{P}) &=  -2i\partial_x -\mathsf{Z}/r\\
        N_{\mathrm{bf}}(\hat{P}) &= -2i\partial_r -  i(d-1) /r -\mathsf{Z}/r,\\
        N_{\mathrm{ff}}(\tilde{P}) &= \triangle_{\bfy}-2i\partial_x -\mathsf{Z}/r ,\\
        N_{\mathrm{ff}}(\hat{P}) &=   \triangle_{\bfy}-2i\partial_r-i(d-1)/r-\mathsf{Z}/r.
        \end{split} 
    \end{equation}
The normal operator solvability theory is virtually unchanged, except 
for the appearance of the Dollard factor at $\mathrm{bf}$. In the analysis of $N_{\mathrm{bf}}(\tilde{P})$, this is owed to the calculation 
\begin{equation}
    \partial_x \log (r-x+1)  = -\frac{1}{r}\Big(1-\frac{1}{r-x+1} \Big), 
\end{equation}
hence 
\begin{equation}\label{eq:nickbf}
    N_{\mathrm{bf}}\bigg(\exp \Big(- \frac{\mathsf{Z}}{2i} \log (r-x+1) \Big)   \tilde{P} \bigg[  \exp \Big( \frac{\mathsf{Z}}{2i}  \log(r-x+1) \Big) \bullet \bigg]\bigg) = -2i\partial_x.   
\end{equation}
So the solvability theory of $N_{\mathrm{bf}}(\tilde{P})$ is simply ``conjugated'' by 
\begin{equation}
    \exp \Big( \frac{\mathsf{Z}}{2i} \log (r-x+1) \Big)     = (r-x+1)^{\mathsf{Z}/2i} ,
\end{equation}
the Dollard correction. In the following, we give a sketch of the proof similar to that of \S\ref{s:complete_proof}, highlighting the modifications.
\begin{proof}[Sketch of the modified proof]
Like in \S\ref{s:complete_proof}, fix $\chi\in C^\infty(\bbR^d)$ such that $\chi=0$ identically in some large open ball and $\chi=1$ identically outside of some larger open ball. The very first step of the proof is changing our initial ``guess'' from $e^{ix}$ to the Coulomb plane wave $u_{\mathsf{Z}}.$ Letting $f_1\coloneqq -P[\chi u_{\mathsf{Z}}]$ and $\tilde{f}_1 = e^{-ix} f_1$, we perform a further conjugation by the Dollard correction. More specifically, define
\begin{equation}\label{eq:nicknotation}
    \dbtilde{P} \coloneqq (r-x+1)^{i\mathsf{Z}/2}\tilde{P}(r-x+1)^{-i\mathsf{Z}/2},\qquad 
    \begin{aligned}
        \dbtilde{f}_1 \coloneqq&\ (r-x+1)^{i\mathsf{Z}/2}\tilde{f}_1,\\
        =&\ (r-x+1)^{i\mathsf{Z}/2}e^{-ix} f_1.
    \end{aligned}
\end{equation}
\Cref{eq:nickbf} says that $\dbtilde{P}$ (unlike $\tilde{P}$!) satisfies 
\begin{equation} 
N_{\mathrm{bf}}(\dbtilde{P})=-2i\partial_x.
\end{equation} 
So, we can apply \Cref{prop:I_main}, \textit{mutatis mutandis}, with $\dbtilde{P}$ in place of the $\tilde{P}$ there.

Outside the big ball, $\dbtilde f_1=-(r-x+1)^{i\mathsf{Z}/2}e^{-ix}V_{\mathrm{sr}} u_{\mathsf{Z}}$.
By the large-argument expansions of the confluent hypergeometric functions \cite[\href{http://dlmf.nist.gov/13.7}{\S13.7}]{NIST}, we see that
\begin{equation}\label{eq:twopiece}
   \dbtilde{f}_1=\dbtilde{f}_{1,0}+\dbtilde{f}_{1,1},\quad
   \begin{array}{ll}
       \dbtilde{f}_{1,0} & \in  V_{\mathrm{sr}}\cdot C^{\infty}(X) \\
       \dbtilde{f}_{1,1}  & \in V_{\mathrm{sr}}\cdot e^{i(r-x)}(r-x+1)^{-\frac{d-1}{2}+i\mathsf{Z}}C^{\infty}(X)
   \end{array}
\end{equation}
outside some big ball. Instead of the entire error $\dbtilde{f}$ being solved away (up to a decaying term at bf), we instead only solve away $\dbtilde{f}_{1,0}$ at the first step, leaving $\dbtilde{f}_{1,1}$ for the second step.  \Cref{prop:I_main} gives
\begin{equation} 
    \dbtilde{v}_1 \in \rho_{\mathrm{bf}}^{L-1} \rho_{\mathrm{ff}}^{L-1}C^\infty(X)
\end{equation} 
 such that $\dbtilde{P} \dbtilde{v}_1 - \dbtilde{f}_{1,0} \in \rho_{\mathrm{bf}}^\infty \rho_{\mathrm{ff}}^{L+1} C^\infty(X)$. If $v_1 \coloneqq (r-x+1)^{-i\mathsf{Z}/2}e^{ix} \dbtilde{v}_1$, this means that $u_1\coloneqq u_{\mathsf{Z}}+v_1$ satisfies 
\begin{equation}\label{eq:onepiece}
    f_2 \coloneqq -f_{1,1}-Pu_1 \in e^{ix} \rho_{\mathrm{bf}}^\infty \rho_{\mathrm{ff}}^{L+1} C^\infty(X) \overset{\text{Lem.\ \ref{lem:gaussian_conversion} }}{=}  e^{ir} \rho_{\mathrm{bf}}^\infty \rho_{\mathrm{ff}}^{L+1} C^\infty(X) .
\end{equation}
Now let $\hat{f}_2+\hat{f}_{1,1} = e^{-ir} (f_2+f_{1,1})$ and define the further conjugations \begin{equation}\label{eq:nicknotation2}
    \dbhat{P} \coloneqq (r-x+1)^{-i\mathsf{Z}/2}\hat{P}(r-x+1)^{i\mathsf{Z}/2},\qquad 
    \begin{aligned}
        \dbhat{f}_2+\dbhat{f}_{1,1} \coloneqq&\ (r-x+1)^{-i\mathsf{Z}/2}(\hat{f}_2+\hat{f}_{1,1}),\\
        =&\ (r-x+1)^{-i\mathsf{Z}/2}e^{-ir}(f_2+f_{1,1}).
    \end{aligned}
\end{equation}
As with $\dbtilde{P}$, the conclusion of \Cref{prop:II_main} applies with $\dbhat{P}$ in place of our original $\hat{P}$. This is the purpose of the extra Dollard conjugation.

By \cref{eq:twopiece} and \cref{eq:onepiece}, we have
\begin{equation}
    \dbhat{f}_2+\dbhat{f}_{1,1}\in \rho_{\mathrm{bf}}^{L+\frac{d-1}{2}} \rho_{\mathrm{ff}}^{L+1} C^\infty(X)\subset \rho_{\mathrm{bf}}^{\frac{d+3}{2}} \rho_{\mathrm{ff}}^{L+1} C^\infty(X).
\end{equation}
\Cref{prop:II_main} applies with $\dbhat{P}$ in place of $\hat{P}$. This gives
\begin{equation} 
    \dbhat{v}_2 \in \rho_{\mathrm{bf}}^{(d-1)/2} \rho_{\mathrm{ff}}^{L-1} \calA^{(0,0),\calF}(X),
\end{equation} 
(supported outside some big ball), for some index set $\calF\subset \bbN^2$, such that $\dbhat{P} \dbhat{v}_2 - \dbhat{f}_2-\dbhat{f}_{1,1} \in \langle r \rangle^{-(d+3)/2} C^\infty(\overline{\bbR^d})$. Thus, setting $v_2 \coloneqq (r-x+1)^{i\mathsf{Z}/2}e^{ir} \dbhat{v}_2$ and $u_2 = u_1 + v_2$, 
\begin{equation}
    f_3\coloneqq -Pu_2  \in (r-x+1)^{i\mathsf{Z}/2}e^{ir} \langle r \rangle^{-(d+3)/2} \calA^{(0,0),\infty}(X)\subset e^{ir}r^{-(d+3)/2+i\mathsf{Z}/2}C^\infty(\overline{\bbR^d}). 
\end{equation}
Next, by a slight modification of \Cref{prop:III_main} (and the proof thereof), there exists some $v_3\in e^{ir}  r^{-(d+1)/2+i\mathsf{Z}/2}C^\infty(\overline{\bbR^d}) $, supported outside of some big ball, such that $Pv_3 - f_3\in \calS(\bbR^d)$. Then, $u_3 \coloneqq u_2+v_3$ satisfies 
\begin{equation}
    f_4\coloneqq -Pu_3  \in \calS(\bbR^d).   
\end{equation}
Finally, the limiting absorption principle in the long-range setting (again, see \cite[\S16]{Me94} for the idea and \cite[Theorem 1.2]{ACL} for the full, rigorous theorem) gives a $v_4 = \lim_{\varepsilon \to0^+}(P-i\varepsilon)^{-1} f_4 \in e^{i r} r^{-(d-1)/2+i\mathsf{Z}/2} C^\infty(\overline{\bbR^d})$ such that $Pv_4 = f_4$. Then, 
\begin{equation} 
    u_4\coloneqq u_3 + v_4 = u_{\mathsf{Z}} + v_1 +v_2 + v_3 + v_4 
\end{equation} 
satisfies $Pu_4 = 0$. 

Writing $a = e^{-ix} r^{L-1}(r-x+1)^{i\mathsf{Z}/2}v_1$ and $w=e^{-ir} r^{(d-1)/2}(r-x+1)^{-i\mathsf{Z}/2} (v_2+v_3+v_4)$, we have  
\begin{equation}
        u_{\mathsf{Z}}(x,\bfy)+ (r-x+1)^{-i\mathsf{Z}/2} e^{ix}   \frac{a}{r^{L-1}}  + (r-x+1)^{i\mathsf{Z}/2}e^{ir} \frac{w}{r^{(d-1)/2}} 
\end{equation}
away from the origin, and $a,w$ lie in the desired spaces since
\begin{equation}
     \frac{r}{r-x+1}=r \rho_{\mathrm{bf}}\sim \rho_{\mathrm{ff}}^{-2}
\end{equation} 
away from the origin. This completes the construction of the perturbed plane wave.

\end{proof}

We end this section by noting the related work of Yafaev \cite{YafaevDiagonalLong}, which investigated the forward singularity of the scattering matrix for arbitrary long-range potentials. 

\section{Perturbed plane waves for central potentials}
\label{sec:central}
In this section, we consider the case of a (real) short-range central potential $V \in \langle r \rangle^{-2} C^\infty(\overline{\bbR^d}\backslash \{0\})$, $d\geq 2$. Thus, 
\begin{equation} 
    \hspace{8em} V=V(r),\qquad r=\sqrt{x^2+\lVert \bfy \rVert^2}, 
\end{equation} 
depends only on the distance $r\geq 0$ from the origin $r=0$, at which $V(r)$ may have a singularity. We assume that the singularity at the origin is sufficiently mild; specifically, 
\begin{equation}
    \exists \alpha \geq  -\frac{(d-2)^2}{4} \text{ s.t. } V(r)  - \frac{\alpha}{r^2} \in r^{-1} \langle r \rangle^{-1} C^\infty(\overline{\bbR^d}) 
\end{equation}
suffices for the argument below. Then, the Helmholtz operator $P=\triangle-1+V$, acting on 
$C_{\mathrm{c}}^\infty(\bbR^d\backslash \{0\})$, is bounded below as a quadratic form on $L^2$ and therefore admits a canonical self-adjoint extension (the Friedrichs extension). 
This is not strictly necessary for the construction of perturbed plane waves, but it is necessary for their uniqueness. 

\subsection{Preliminaries on separation of variables}
Because the potential is central, we may leverage spherical symmetry by separating variables. The perturbed plane wave itself will only enjoy cylindrical symmetry, namely around the $x$-axis. Thus, the perturbed plane wave can be written as a pointwise convergent sum
\begin{equation}
    u = \sum_{\ell=0}^\infty Y_\ell(\theta) u_\ell(r), 
\end{equation}
where the sum is over 
the zonal (hyper)spherical harmonics $Y_\ell$, i.e.\ the (hyper)spherical harmonics invariant with respect to rotations/reflections about the $x$-axis.
We normalize $Y_\ell$ so that $\|Y_\ell\|_{L^2(\bbS^{d-1})}=1$.
Then, 
\begin{equation} 
\triangle_{\bbS^{d-1}} Y_\ell =\lambda_\ell Y_\ell, 
\end{equation}
where $\lambda_\ell = \ell(\ell+d-2)$. 
Since the zonal harmonics only depend on the variable $\theta = \operatorname{arccos}(x/r)$, we write their argument merely as `$\theta$.'

\begin{remark}
    When $d=2$, the zonal harmonics are those invariant under reflection across the $x$-axis. 
\end{remark}

Let
\begin{equation} 
    N(d,\ell) = 
    \begin{cases}
        1 & (d=2\text{ and }\ell=0), \\ 
         \frac{(\ell+d-3)!(2\ell+d-2)}{\ell!(d-2)!}  & (\text{otherwise}) 
    \end{cases}
\end{equation} 
denote the degeneracy of the $\ell$th eigenspace of (hyper)spherical harmonics. Then, in terms of  $\theta=\operatorname{arccos}(x/r)$:
\begin{itemize}
    \item if $d=2$, then $Y_0(\theta) = 1/\sqrt{2\pi}$ and $Y_\ell(\theta) = \cos(\ell \theta)/\sqrt{\pi}$ for $\ell\geq 1$,
    \item if $d\geq 3$, then 
    \begin{equation}\label{eq:Y_ell_m_form}
    Y_\ell = \overbrace{\frac{1}{C_\ell^\nu(1)} \sqrt{\frac{N(d,\ell)}{\operatorname{Area}(\bbS^{d-1})}}  }^{\text{Normalization factor}}  C_\ell^{\nu}(\cos \theta),
\end{equation}
where $\nu=\frac{d-2}{2}$, $\operatorname{Area}(\bbS^{d-1})=2\pi^{d/2}/\Gamma(d/2)$ is the surface area of the $(d-1)$-sphere, and
\begin{equation}
    C_\ell^{(\nu)}(z) = \sum_{k=0}^{\lfloor \ell/2 \rfloor} (-1)^k \frac{\Gamma(\ell-k+\nu)}{\Gamma(\nu) k! (\ell-2k)!} (2z)^{\ell-2k}
\end{equation}
is a Gegenbauer polynomial. 
\end{itemize}

When $d=3$, the formula \cref{eq:Y_ell_m_form} simplifies to the more familiar 
\begin{equation}
    Y_\ell = \sqrt{\frac{2\ell +1}{4\pi}}P_{\ell}(\cos \theta) .
\end{equation}

\subsection{Construction of perturbed plane waves}

Since $\triangle = -\partial_r^2 - \frac{d-1}{r}\partial_r + r^{-2}\triangle_{\bbS^{d-1}}$ in spherical coordinates, 
\begin{equation}
    P u = \sum_{\ell=0}^\infty Y_\ell(\theta) \bigg[ \underbrace{-\frac{\partial^2}{\partial r^2} - \frac{d-1}{r} \frac{\partial}{\partial r} -1+ V(r) + \frac{\lambda_\ell}{r^2} }_{P_\ell}\bigg] u_\ell(r). 
\end{equation}
So, $Pu=0\iff P_\ell u_\ell =0$ for all $\ell\geq 0$; 
thus, $u_\ell$ is an element of the two-dimensional kernel $\ker P_\ell \subset C^\infty(\bbR^+)$ of the ordinary differential operator $P_\ell\in \operatorname{Diff}^2(\bbR^+)$.   
Which element is determined by two constraints: 
\begin{enumerate}[label=(\roman*)]
    \item the ``recessive'' condition at $r=0$, which enforces $u \in L^2_{\mathrm{loc}}$ (or better), 
    \item a large-$r$ matching condition designed to ensure that $u$ has the desired form as $r\to\infty$, namely $u\approx e^{ix}$ modulo an outgoing term and $O(1/r)$ plane wave terms. 
\end{enumerate}

The $r\to 0$ behavior of elements of $\ker P_\ell$ is a classical application of the theory of regular singular ODE --- see e.g. \cite[Chp.\ 5.4, 5.5]{olver}. The behavior is governed by the terms in $P_\ell$ which have the worst behavior under dilations, namely 
\begin{equation}
    N_\ell = - \frac{\partial^2}{\partial r^2} - \frac{d-1}{r} \frac{\partial}{\partial r} + \frac{\alpha+\lambda_\ell}{r^2}.
\end{equation}
If $u_\ell(r) \sim r^c$ as $r\to 0^+$, then we should have $N_\ell r^c=0$. This is true if and only if $0= -c^2 - (d-2)c + \alpha+\lambda_\ell$. The solutions of this quadratic equation are 
\begin{equation}
    c_{\pm,\ell} =-\frac{d-2}{2} \pm \sqrt{\frac{(d-2)^2}{4}+\alpha+\lambda_\ell },
\end{equation}
the ``indicial roots.'' Note that, under our assumptions on $\alpha$, the indicial roots are real and satisfy $c_{-,\ell}<c_{+,\ell}$,
unless
\begin{equation}\label{eq:recessive_exception}
    \alpha= -\frac{(d-2)^2}{4} \text{ and }\ell=0. 
\end{equation}
This suggests that a generic element of $\ker P_\ell$ blows up like $r^{c_-}$ as $r\to 0^+$, while a one-dimensional subspace of \emph{recessive} solutions only go as $r^{c_+}$. This is indeed the case (with the one exception noted above), by the standard regular singular theory. 
In the exceptional case \cref{eq:recessive_exception}, the indicial roots $c_{\pm,\ell}$ are both equal to $c=-(d-2)/2$. In this case, a generic element of $\ker P_\ell$ blows up like $r^c \log r$ as $r\to 0^+$, whereas the recessive solutions only go as $r^c$. 
In either case, the condition (i) says that $u_\ell$ has to lie in the one-dimensional subspace of recessive solutions: 
\begin{equation}
    u_\ell = O(r^{c_{+,\ell}}) \text{ as }r\to 0^+.
\end{equation}

The large-$r$ asymptotics of general $v\in \ker P_\ell$ is a classic application of the Liouville--Green theory: 
\begin{equation}\label{eq:coeff_defs}
    \exists a_\pm \in \bbC\text{ s.t. }v- \frac{1}{r^{(d-1)/2}} (a_+ e^{ir} + a_- e^{-ir} ) = O\Big(\frac{1}{r^{(d+1)/2} } \Big) . 
\end{equation}
An \emph{outgoing} solution is one where $a_-=0$. An \emph{incoming} solution is one where $a_+=0$. 
The map $v\mapsto (a_-,a_+)\in \bbC^2$ is a linear isomorphism $ \ker P_\ell \cong \bbC^2$. So, solutions are uniquely characterized by their ``scattering data'' $a_\pm$, and any pair of scattering data come from some valid solution.

The expansion of the plane wave into spherical harmonics is
\begin{equation}
      e^{ix} = \sum_{\ell=0}^\infty Y_\ell(\theta) \int_{\bbS^{d-1}} Y_\ell(\omega)^* e^{ir\hat{\bfx}\cdot \omega} \dd \omega  
    =  \sum_{\ell=0}^\infty (2\pi)^{d/2}  i^\ell \sqrt{\frac{N(d,\ell)}{\operatorname{Area}(\bbS^{d-1})}}  \frac{J_{\ell+d/2-1}(r)}{r^{d/2-1}} Y_\ell(\theta),\label{eq:plane_Fourier_expansion}
\end{equation}
where $J_{\ell+d/2-1}(r)$ denotes the Bessel function of order $\ell+d/2-1$. Hankel's large-argument expansion of the Bessel function says 
\begin{equation}\label{eq:Hankel}
    J_{\nu}(r) = \sqrt{\frac{2}{\pi r}} \cos\Big( r  - \frac{\pi \nu}{2} - \frac{\pi}{4}  \Big) +O \Big( \frac{1}{r^{3/2}} \Big) 
\end{equation}
as $r\to\infty$ \cite[\href{http://dlmf.nist.gov/10.17.3}{\S10.17.3}, \href{http://dlmf.nist.gov/10.17.iii}{\S10.17(iii)}]{NIST}. 
The constant in the big-O is \emph{not} uniform as $\ell\to\infty$. This is crucial, because otherwise, plugging the Hankel asymptotic into \cref{eq:plane_Fourier_expansion}, we would reach the erroneous conclusion that $e^{ix}=O(1/r^{(d-1)/2})$ as $r\to\infty$.

Nevertheless, the Hankel asymptotic suggests the condition that should be imposed on the $r\to\infty$ behavior of $u_\ell$ 
to make $u$ a perturbed plane wave: \emph{the incoming part of $u_\ell$ should match that} (\cref{eq:coeff_defs}) \emph{in the $\ell$th mode of the plane wave}. 
This means that 
\begin{equation}
    \exists a_+\in \bbC\text{ s.t. }u_\ell - a_+ \frac{e^{ir}}{r^{(d-1)/2}}- (2\pi)^{d/2}  i^\ell \sqrt{\frac{N(d,\ell)}{\operatorname{Area}(\bbS^{d-1})}}  \frac{J_{\ell+d/2-1}(r)}{r^{(d-2)/2}}  = O\Big( \frac{1}{r^{(d+1)/2}} \Big) \text{ as }r\to\infty . 
\end{equation}
This is the precise statement of (ii). 
Via \cref{eq:Hankel}, we can replace the Bessel function here with its leading order incoming part, absorbing the leading outgoing part into a redefinition of $a_+$.

    \begin{proposition}
        There exists a unique solution satisfying (i) and (ii) above.
    \end{proposition}
    \begin{proof}
        Let $v_{0,\ell},v_{-,\ell},v_{+,\ell}\in \ker P_\ell$ denote nonzero recessive, incoming, and outgoing solutions, respectively. Because the potential $V$ is real, and because the recessive condition is closed under complex conjugation, we can choose $v_{0,\ell} $ to be real. We can choose $v_{\pm,\ell} $ such that 
        \begin{equation}
            v_{\pm,\ell} = \frac{e^{\pm i r}}{r^{(d-1)/2}}  + O\Big( \frac{1}{r^{(d+1)/2}} \Big)\text{ as }r\to\infty .
        \end{equation}

        \textit{Claim: $v_{\pm,\ell} $ is not recessive}. Indeed, since $v_{0,\ell} $ is real, no nonzero multiple thereof is purely incoming or purely outgoing. 

        The stated proposition follows immediately: the solution is $C v_{0,\ell}$ for $C$ chosen so as to satisfy (ii), which is possible because $v_{0,\ell}$ is not purely outgoing.
    \end{proof}

    Now, for each $\ell$, choose the solution $u_\ell$ which satisfies 
    \begin{enumerate}[label=(\roman*)] 
        \item $u_\ell = O(r^{c_{+,\ell}})$ as $r\to 0$, 
        \item $ \exists a_+\in \bbC$ such that 
        \begin{equation}\label{eq:bigrspec}
       r^{(d-1)/2} u_\ell - a_+ e^{ir} =(2\pi)^{d/2}  i^{2\ell} e^{\frac{i\pi}{4} (d-1)} \sqrt{\frac{N(d,\ell)}{2\pi \operatorname{Area}(\bbS^{d-1})}} e^{-ir}   +  O\Big( \frac{1}{r} \Big) 
    \end{equation}
    as $r\to\infty$.
    \end{enumerate}
    Then, the full perturbed plane wave $u$ is defined by 
    \begin{equation}
         u(z) = \sum_{\ell=0}^\infty Y_\ell(\theta) u_\ell(r) ,
    \end{equation}
    where $Y_\ell$ is the zonal (hyper)spherical harmonic, as above. 
	
	\vspace{.5em} 
	\noindent

\begin{example}[Inverse-square potential, cont.]
    Specializing to the case that $V=\alpha/r^2,\alpha>-(d-2)^2/4$, the ODE $P_{\ell}u_{\ell}=0$ simplifies to
\begin{equation}
    \bigg[ -\frac{\partial^2}{\partial r^2} - \frac{d-1}{r} \frac{\partial}{\partial r} -1 + \frac{\alpha+\lambda_\ell}{r^2} \bigg] u_\ell(r)=0, 
\end{equation}
which is a conjugated form of Bessel's ODE. The recessive solutions can be written in terms of the $J$-function. So, 
\begin{equation}\label{eq:dipole_plane_order_def}
    u_{\ell}(r)=a_{\ell}r^{-(d-2)/2}J_{\mu_{\ell}}(r), \quad \mu_{\ell}\coloneqq \sqrt{\alpha+\Big(\ell+\frac{d-2}{2}\Big)^2}, 
\end{equation}
for some normalization constant $a_\ell$. 
To find $a_{\ell}$, we use the large argument asymptotics of $J$ and compare with \cref{eq:bigrspec}. Indeed, using \cref{eq:Hankel}, we get 
\begin{equation}
    a_{\ell}=(2\pi)^{d/2}  i^{2\ell+d/2-1-\mu_{\ell}} \sqrt{\frac{N(d,\ell)}{ \operatorname{Area}(\bbS^{d-1})}},
\end{equation}
so overall 
\begin{equation}
u=\sum_{\ell=0}^\infty (2\pi)^{d/2}  i^{2\ell+d/2-1-\mu_\ell} \sqrt{\frac{N(d,\ell)}{\operatorname{Area}(\bbS^{d-1})}}  \frac{J_{\mu_\ell}(r)}{r^{(d-2)/2}} Y_\ell(\theta).
\end{equation}
Note that this reduces to $u=e^{ix}$ when $\alpha=0$ by \cref{eq:plane_Fourier_expansion}.

For dimensions $d=2,3$, the above simplifies to
\begin{equation}\label{eq:special_dipole_plane_wave}
    u(r,\theta)=\begin{cases}
        \displaystyle i^{-\sqrt{\alpha}}J_{\sqrt{\alpha}}(r)+2\sum_{\ell=1}^{\infty} i^{2\ell-\sqrt{\alpha+\ell^2}} J_{\sqrt{\alpha+\ell^2}}(r)\cos (\ell \theta), & \text{if }d=2,\\
        \displaystyle \sqrt{\frac{\pi } {2r}}\sum_{\ell=0}^{\infty} i^{2\ell-\mu_{\ell}+\frac{1}{2}} (2\ell+1)J_{\mu_{\ell}}(r)P_{\ell}(\cos \theta), &\text{if }d=3.
    \end{cases}
\end{equation}
\end{example}

\subsection{Phase shifts}

The scattering data $(a_{-,\ell},a_{+,\ell})\in \bbC^2$ of 
\begin{equation}
    u_\ell = \frac{1}{r^{(d-1)/2}} ( a_{-,\ell}e^{-ir} + a_{+,\ell}e^{ir} ) + O\Big( \frac{1}{r^{(d+1)/2}} \Big),\qquad r\gg 1   
\end{equation}
is not completely arbitrary. The condition (ii) above fixes 
\begin{equation}
    a_{-,\ell}= (2\pi)^{d/2}  i^{2\ell} e^{\frac{i\pi}{4} (d-1)} \sqrt{\frac{N(d,\ell)}{2\pi \operatorname{Area}(\bbS^{d-1})}} , 
\end{equation}
thus eliminating two real degrees-of-freedom. The recessivity condition (i) fixes the remaining two degrees-of-freedom. One of these can be solved for explicitly, using the fact that the recessivity condition is real. Letting $v_{0,\ell}$ denote a nonzero real recessive solution, as above, 
\begin{equation}
    u_{\ell}\propto v_{0,\ell} \iff\frac{a_{+,\ell}}{a_{-,\ell}} = \frac{b_{+,\ell}}{b_{-,\ell}},
\end{equation}
where $b_{\pm,\ell}\in \bbC$ is the scattering data of $v_{0,\ell}$, so that 
\begin{equation}
    v_{0,\ell} = \frac{1}{r^{(d-1)/2}} ( b_{-,\ell}e^{-ir} + b_{+,\ell}e^{ir} ) + O\Big( \frac{1}{r^{(d+1)/2}} \Big)
\end{equation}
as $r\to\infty$. The reality of $v_{0,\ell}$ implies $b_{+,\ell}=b_{-,\ell}^*$. Thus, the ratio $b_{+,\ell}/b_{-,\ell}$ is a phase. In summary, 
\begin{equation}
    \text{condition (i)}\Longrightarrow \exists \tilde{\delta}_\ell \in \bbR/\pi \bbZ\text{ s.t. }a_{+,\ell} =a_{-,\ell} e^{2i\tilde{\delta}_\ell} .
\end{equation}
The scattering data $(a_{-,\ell},a_{+,\ell})$ is therefore completely determined once we know the phase $\tilde{\delta}_\ell$. This remaining degree-of-freedom can be read off the scattering data of a recessive solution of the ODE, but this constitutes a ``connection formula,'' relating the small-$r$ behavior of solutions to large-$r$ behavior, and consequently depends on the specifics of the potential. 

Since $e^{-ir} + e^{2i\tilde{\delta}_\ell} e^{ir} = 2 e^{i\tilde{\delta}_\ell} \cos(r+\tilde{\delta}_\ell)$, we can express the large-$r$ asymptotics of $u_\ell$ entirely in terms of $\tilde{\delta}_\ell$ as 
\begin{equation}\label{eq:central_u_scat_form}
    u_\ell = (2\pi)^{d/2}  i^{2\ell} \sqrt{\frac{2N(d,\ell)}{\pi \operatorname{Area}(\bbS^{d-1})}} e^{\frac{i\pi}{4} (d-1)} e^{i\tilde{\delta}_\ell}\frac{\cos(r+\tilde{\delta}_\ell)}{r^{(d-1)/2}}  +  O\Big( \frac{1}{r^{(d+1)/2}} \Big).
\end{equation}
Consider 
\begin{equation}
    \delta_\ell = \tilde{\delta}_\ell - \frac{\pi}{4}(1-d-2\ell) . 
\end{equation}
Then, the asymptotic can be rewritten 
\begin{equation}
    u_\ell = (2\pi)^{d/2}  i^{\ell} \sqrt{\frac{2N(d,\ell)}{\pi \operatorname{Area}(\bbS^{d-1})}} e^{i\delta_\ell} \frac{1}{r^{(d-1)/2}} \cos\Big(r+\frac{\pi}{4} (1-d-2\ell)+\delta_\ell\Big)   +  O\Big( \frac{1}{r^{(d+1)/2}} \Big).
\end{equation}
Comparing with \cref{eq:Hankel}, we see that, if $u=e^{ix}$, then $\delta_\ell=0$.
Thus, $\delta_\ell$ measures the phase shift of $u_\ell$ relative to the exact plane wave. 

For general short-range potentials,  the representatives may be chosen so that $\lim_{\ell\to\infty}\delta_\ell=0$

The large-$r$ structure of the plane wave is encoded in the sequence $\delta_0,\delta_1,\delta_2,\cdots \in \bbR/\pi \bbZ$.

\begin{example}[Inverse-square potential, cont.]
    Whereas the $\ell$th mode of the plane wave involves $J_{\ell+d/2-1}$, the inverse-square perturbed plane wave involves $J_{\mu_\ell}$, with $\mu_\ell$ defined in \cref{eq:dipole_plane_order_def}. 
    Owing to the Hankel asymptotic
    \begin{equation} 
        J_{\nu}(r)\sim \sqrt{\frac{2}{\pi r}}\cos\Big(r-\frac{\pi \nu}{2}-\frac{\pi}{4}\Big),
    \end{equation} 
    the phase shift $\delta_\ell$  is exactly the change in the cosine argument resulting from changing the Bessel order from $\nu=\ell+d/2-1$ to $\mu_\ell$:
    \begin{equation}
        \delta_\ell = \frac{\pi}{4} \Big( 2\ell+d-2  -\sqrt{4\alpha+(2\ell+d-2)^2}\, \Big) .
    \end{equation}

    This is $O(1/\ell)$ as $\ell\to\infty$, \emph{not} $O(1/\ell^\infty)$, which signals the singularity in the forward direction. However, extracting any rigorous statement to this effect is difficult.
\end{example}

\subsection{The S-matrix}
The phase shifts provide the eigenvalues of the matrix. Indeed, the S-matrix $S:L^2(\bbS^{d-1})\to L^2(\bbS^{d-1})$ must be diagonal in the basis of (hyper)spherical harmonics, because the potential is spherically symmetric. The eigenvalue of the $\ell$th eigenspace of harmonics is $(-1)^\ell e^{2i\delta_\ell}$.
Thus, the S-matrix can be written 
\begin{equation}
    S(\omega,\omega') = \sum_{\ell=0}^\infty \sum_{Y\in \calY_\ell} (-1)^\ell  e^{2i\delta_\ell} Y(\omega) Y (\omega')^*,  
\end{equation}
where $\calY_\ell$ is an orthonormal basis for the $\ell$th eigenspace of harmonics. When we plug in $\omega'=\leftarrow$, only the zonal harmonics $Y_\ell$ survive, because (since we defined the azimuthal axis with respect to which the harmonics are defined to be the $x$-axis) $Y(\leftarrow)=0$ except when $Y$ is zonal. Thus:
\begin{equation}
    S(\omega,\leftarrow) = \sum_{\ell=0}^\infty (-1)^\ell e^{2i\delta_\ell} Y_\ell(\omega) Y_\ell(\leftarrow )^* =\sum_{\ell=0}^\infty \sqrt{\frac{N(d,\ell)}{\operatorname{Area}(\bbS^{d-1})}}  e^{2i\delta_\ell} Y_\ell(\omega). 
\end{equation}

\subsection{The scattered wave}

On the other hand, 
\begin{equation}
    \delta_{\rightarrow} (\omega) = \sum_{\ell=0}^\infty Y_\ell(\omega)  Y_\ell(\rightarrow)^* = \sum_{\ell=0}^\infty \sqrt{\frac{N(d,\ell)}{\operatorname{Area}(\bbS^{d-1})}} Y_\ell(\omega),  
\end{equation}
so the scattered wave $f(\omega) = e^{-\pi i (d-1)/4}(2\pi)^{(d-1)/2}(S(\omega,\leftarrow) - \delta_{\rightarrow} (\omega))$ is given by 
\begin{equation}
    f(\omega)=e^{-\pi i (d-1)/4}(2\pi)^{(d-1)/2}\sum_{\ell=0}^\infty   \sqrt{\frac{N(d,\ell)}{\operatorname{Area}(\bbS^{d-1})}}  \underbrace{(e^{2i\delta_\ell} - 1)}_{2i e^{i\delta_\ell} \sin(\delta_\ell ) }Y_\ell(\omega).
\end{equation}
For example, in the $d=3$ case, this simplifies to 
\begin{equation}
        f(\theta) = \sum_{\ell=0}^\infty   (2\ell+1) e^{i\delta_\ell} \sin(\delta_\ell )  P_\ell(\cos \theta)   .
\end{equation}

\begin{example}[Inverse-square potential, cont.]
    This sum is slowly convergent, owing to the slow decay as $\ell\to\infty$ of  
    \begin{equation}
        (2\ell+1)e^{i\delta_\ell} \sin(\delta_\ell) =-\frac{\pi \alpha}{2} +\frac{i\pi^2\alpha^2}{8(\ell+1/2)} + O\Big( \frac{1}{\ell^2} \Big),
    \end{equation}
    which implies
    \begin{equation}
f(\theta)=  - \frac{\pi\alpha}{2} \sum_{\ell=0}^\infty P_\ell (\cos \theta) + \frac{i\pi^2 \alpha^2}{8} \sum_{\ell=0}^\infty \frac{P_\ell (\cos\theta)}{\ell+1/2}+O(1) . 
    \end{equation}
    The two infinite sums in this expression can be summed explicitly using the generating function 
    \begin{equation}\label{eq:lengen}
        (1-2t \cos \theta +t^2)^{-1/2}=\sum_{\ell=0}^{\infty}P_{\ell}(\cos \theta)t^{\ell},\qquad \theta \neq 0,\pi,\ |t|\leq 1.    
    \end{equation}
Indeed, it follows that
    \begin{equation}
        \sum_{\ell=0}^\infty P_\ell (\cos \theta) =  \frac{1}{2|\sin(\theta/2)|},\qquad  
        \sum_{\ell=0}^\infty \frac{P_\ell (\cos\theta)}{\ell+1/2}= K \Big( \cos \frac{\theta}{2} \Big), \label{eq:infseries}
    \end{equation}
    for $\theta\ne 0$, where $K$ is the elliptic integral as in \cref{eq:K}.
    Note that the first equation in \cref{eq:infseries} follows from \cref{eq:lengen} by taking the Abel limit $t\to 1^-$, while the second follows by writing $(\ell+1/2)^{-1}=\int_0^1t^{\ell-1/2}dt$, using Fubini--Tonelli, and changing variables $t=1-u^2$. Now using $K(k)=-\frac{1}{2}\log(1-k^2)+O(1)$ as $k \to 1^-$, we have
    \begin{equation}\label{eq:f_ap}
        f(\theta) = -\frac{\pi \alpha}{2|\theta|} -\frac{i\pi^2\alpha^2}{8}\log|\theta|+O(1),
    \end{equation}
    as $\theta \to 0$.
\end{example}

\section{Mapping properties of forward integration on the polar forward blowup}
\label{sec:forward}

We carry on with the notation in \S\ref{sec:bf}. 
Recall that $Y=[\overline{\bbR^d};\rightarrow]$, and $N^{-1}$ is forward integration (up to a constant of proportionality).

Let $\mathrm{fd}\subset Y$ denote the singleton consisting of the midpoint of the front face, $\mathrm{fd} = (\partial Y)\cap  \operatorname{cl}_Y\{ x>1, \bfy=0\}$. 
In this section we prove:
\begin{proposition}\label{prop:intmain}
    Let $\alpha>1$. Then, 
	\begin{equation}
		N^{-1} \colon \mathcal{A}^{(\alpha,0),\mathcal{F}}(Y\backslash \mathrm{fd}) \to \mathcal{A}^{(\alpha-1,0),\mathcal{F}^+}( Y\backslash \mathrm{fd}),
	\end{equation}
    where 
    \begin{equation}\label{eq:F+}
        \calF^+\coloneqq  (\calF-1) \cup \{(n,0) : n\in \bbN \} \cup  \{ (n,j+1) : (1,j)\in \calF,n\in \bbN  \},
    \end{equation}
    and $\calF-1 \coloneqq \{(j-1,k) \,:\, (j,k) \in \calF\}.$
\end{proposition}
This mapping property is natural given that $\partial_x \in \operatorname{Diff}_{\mathrm{b}}^{1,-1,-1}(Y)$. Taking an $x$-derivative yields one order of decay at both of $\mathrm{bf},\mathrm{ff}$, so integrating removes the same amount of decay (possibly with a logarithmic loss).

In order to prove this, it suffices to prove 
\begin{equation}\label{eq:n}
		N^{-1} \colon \mathcal{A}^{(\alpha,0),\mathcal{E}}(Y\backslash \operatorname{cl}_Y\{(x,{\bf0}):x>0\} ) \to \mathcal{A}^{(\alpha-1,0),\mathcal{E}^+}( Y\backslash \operatorname{cl}_Y\{(x,{\bf0}):x>0\}),
\end{equation}
because translations lift to diffeomorphisms of $Y$ and commute with $N^{-1}$.

In the proof, we use the following partial atlas on $Y$.
\begin{itemize}
\item On the set $\{(x,\bfy)\in \bbR^d : x<0\}$, define the coordinates 
\begin{equation}\label{eq:backbdfs}
\rho_{\mathrm{bf}} \coloneqq -\frac{1}{x},\quad \tilde{\bfy} \coloneqq- \frac{\bfy}{x}. 
\end{equation}
These define a diffeomorphism between some open subset of $\overline{\bbR^d}$ (and hence of $Y$, since $Y,\overline{\bbR^d}$ agree in $x<-1$) and $[0,\infty)_{\rho_{\mathrm{bf}}}
\times\bbR^{d-1}_{\tilde{\bfy}}.$. We call this the \textit{back chart}.

\item On the set $\{(x,\bfy)\in \bbR^d:\bfy\neq 0\}$, define the coordinates  
\begin{equation}\label{eq:sidebdfs}
\rho_{\mathrm{bf}} \coloneqq \frac{1}{y},\quad s
\coloneqq \frac{x}{y},\quad \phi\coloneqq \frac{\bfy}{y}. 
\end{equation}
These define a diffeomorphism between some open subset of $Y$ and $[0,\infty)_{\rho_{\mathrm{bf}}} \times (-\infty,\infty)_s\times \bbS^{d-2}_\phi$. 
We call this the \textit{side chart}.

\item On the set $\{(x,\bfy)\in\bbR^d: x>0,y>0\}$, define the coordinates 
\begin{equation}\label{eq:cornerbdfs}
\rho_{\mathrm{bf}} \coloneqq \frac{1}{y},\quad \rho_{\mathrm{ff}} \coloneqq \frac{y}{x},\quad \phi \coloneqq \frac{\bfy}{y}.
\end{equation}
These define a diffeomorphism between some open subset of $Y$ (containing the corner $\mathrm{bf}\cap \mathrm{ff}$) and $[0,\infty)_{\rho_{\mathrm{bf}}}\times [0,\infty)_{\rho_{\mathrm{ff}}}\times \bbS^{d-2}_\phi$.
We call this the \textit{corner chart}.
\end{itemize}
The back and side charts are part of the usual atlas on $\overline{\bbR^d}$ via projective coordinates.  

\begin{remark*}
The charts above do not give a complete atlas for $Y$, because they exclude the forward axis $\operatorname{cl}_Y\{(x,0): x>0\}$. 
\end{remark*}

On these charts, the b-differential vector fields $\mathcal{V}_{\mathrm{b}}(Y)$ are spanned over $C^\infty(Y)$ by the following vector fields:
\begin{equation}\label{eq:bvecs}
\begin{array}{l}
x\partial_x,x\partial_{y_1},\ldots,x\partial_{y_{d-1}}
\text{ span }\mathcal{V}_{\mathrm{b}}(Y)
\text{ on compact subsets of the back chart},\\
y\partial_x,y\partial_{y_1},\ldots,y\partial_{y_{d-1}}
\text{ span }\mathcal{V}_{\mathrm{b}}(Y)
\text{ on compact subsets of the side chart},\\
x\partial_x,y\partial_{y_1},\ldots,y\partial_{y_{d-1}}
\text{ span }\mathcal{V}_{\mathrm{b}}(Y)
\text{ on compact subsets of the corner chart}.
\end{array}
\end{equation}
This is a standard result that can be seen in several ways; for example, $\partial_x,\partial_{y_j}$ span $r^{-1} \calV_{\mathrm{b}}(\overline{\bbR^d})$ over $C^\infty$ \cite{Me94}, and $-1/x,1/y$ serve as local boundary-defining-functions in the back and side charts, respectively. 
This yields the claim about the back and side charts. The claim about the corner chart is one about the b-vector fields on the mwcs resulting from blowing up an interior submanifold of the boundary; see e.g.\ performing a polar blowup of the origin in the half-plane $\bbR\times [0,\infty)$. 
Alternatively: in each chart, the relevant claim amounts to an elementary computation in local coordinates. The $d\geq 3$ case follows from the $d=2$ case together with the observation that the elements of $\calV(\bbS^{d-2}_\phi)$ are in the module generated by the $y\partial_{y_j}$ over $C^\infty(\bbS^{d-2}_\phi)$. In the $d=2$ case, within the corner chart, 
\begin{align}
\begin{split} 
    x\partial_x &= x\bigg[ \frac{\partial \rho_{\mathrm{bf}}}{\partial x}  \partial_{\rho_{\mathrm{bf}}} +   \frac{\partial \rho_{\mathrm{ff}}}{\partial x}  \partial_{\rho_{\mathrm{ff}}} \bigg] =- \rho_{\mathrm{ff}} \partial_{\rho_{\mathrm{ff}}} \\
    y\partial_y &= y\bigg[ \frac{\partial \rho_{\mathrm{bf}}}{\partial y}  \partial_{\rho_{\mathrm{bf}}} +   \frac{\partial \rho_{\mathrm{ff}}}{\partial y}  \partial_{\rho_{\mathrm{ff}}} \bigg] = - \rho_{\mathrm{bf}}\partial_{\rho_{\mathrm{bf}}} + \rho_{\mathrm{ff}} \partial_{\rho_{\mathrm{ff}}} 
    \end{split} 
\end{align}
and the right-hand sides are in $\calV_{\mathrm{b}}(Y)$. Conversely, 
\begin{equation}
    y\partial_y = - \rho_{\mathrm{bf}}\partial_{\rho_{\mathrm{bf}}} + \rho_{\mathrm{ff}} \partial_{\rho_{\mathrm{ff}}}  \Longrightarrow \rho_{\mathrm{bf}} \partial_{\rho_{\mathrm{bf}}}  = -x\partial_x -y\partial_y .
\end{equation}
So the local generators of $\calV_{\mathrm{b}}(Y)$ are linear combinations of $x\partial_x$ and $y\partial_y$, which therefore suffice to generate all of $\calV_{\mathrm{b}}(Y)$ (locally). 

\subsection{Conormality}\label{ss:fowardIntProof}

Before proving \Cref{prop:intmain}, we prove the corresponding mapping property between the function spaces of conormal functions:
\begin{lemma}\label{lem:conorm} For $\alpha,\beta>1$,
    \begin{equation}
		N^{-1} \colon \mathcal{A}^{\alpha,\beta}(Y\backslash \mathrm{fd}) \to \mathcal{A}^{\alpha-1,\beta-1}( Y\backslash \mathrm{fd})+\mathcal{A}^{\alpha-1,(0,0)}(Y\backslash \mathrm{fd}).
	\end{equation}
\end{lemma}
\begin{proof} 
As noted above, it suffices to forget about the entire forward axis $\operatorname{cl}_Y\{(x,\bf0):x>0\}$.

Suppose $f \in \mathcal{A}^{\alpha,\beta}(Y\backslash \mathrm{fd})$. Since $N^{-1}$ is linear, it suffices to prove the lemma when the support of $f$ lies in exactly one of the coordinate charts above (via a partition of unity). We need to check $N^{-1} f$ in each of the three coordinate charts above. Taking into account both the source chart and target chart, we have nine possibilities to consider. 
However, the charts have large overlap; when studying $N^{-1} f$ in a chart besides that which $f$ is supported in, it suffices to restrict attention to a small enlargement of the part of that chart not covered by the chart in which $f$ is supported. 
Consequently, because $N^{-1}$ propagates support forward, it is not necessary to consider the possibilities labeled `0' in the following table:
\begin{table}[!htbp]
    \centering
    \begin{tabular}{|c|c|c|c|c|}
        \hline
        \multirow{2}{*}{} & & \multicolumn{3}{c|}{$f$ support} \\ \cline{3-5} 
                          &  & Back Chart & Side Chart & Corner Chart \\ \hline
        \multirow{3}{*}{\rotatebox{90}{$N^{-1}f$}} 
                          & Back Chart  & \S\ref{subsec:same}   & $0$   & $0$     \\ \cline{2-5} 
                          & Side Chart  & \S\ref{subsec:offset}   & \S\ref{subsec:same2}  & $0$     \\ \cline{2-5} 
                          & Corner Chart & \S\ref{subsec:offset} & \S\ref{subsec:offset}  & \S\ref{subsec:same2} \\ \cline{1-5} 
    \end{tabular}
    \caption{Locations where the various cases are handled.}
    \label{tab:1939}
\end{table}
\par 
\noindent For example, if $f$ is supported in the side chart, then the side and corner charts suffice to cover the support of $N^{-1} f$ (besides the forward axis).

\subsubsection{back-back and side-side (conormality)}\label{subsec:same}
For $f \in \mathcal{A}^{\alpha,*}(Y)$ supported within the back chart, we prove that for all (unweighted) $\mathrm{b}$-differential operators $A$ that
\begin{equation} 
|A(N^{-1}f)| \lesssim \rho_{\mathrm{bf}}^{\alpha-1}
\end{equation} 
within the back chart. The `*' in $\calA^{\alpha,*}(Y)$ means that the order is arbitrary and unimportant, since here we are focused on the back chart.

We prove this by strong induction on the order of the
$\mathrm{b}$-differential operator $A$. For $0$th order
$\mathrm{b}$-differential operators (i.e.\ multiplication operators), the claim follows immediately from the bound
\begin{equation*}
|N^{-1}f(x,\bfy)|\leq \int_{-\infty}^x|f(t,\bfy)|\dd t \lesssim \int_{-\infty}^x |t|^{-\alpha}\dd t \lesssim \rho_{\mathrm{bf}}^{\alpha-1}.
\end{equation*}
We now induct on $\operatorname{deg} A=k$ (and all allowed $\alpha$). Assume that we have proven the estimate for operators of degree $j\leq k$. 
Any $(k+1)$th order
$\mathrm{b}$-operator can be written as a sum of operators of the form $AV$ where $A$ is
a $k$th order $\mathrm{b}$-operator and $V$ is a first-order operator (i.e.\ a 
$\mathrm{b}$-vector field). 
By \cref{eq:bvecs}, $V$ is in the span of $x \p_x , x \p_{\bfy_1},\dots, x\p_{\bfy_{d-1}}$.

Note that
\begin{equation}\label{eq:redman}
\begin{gathered}
x \partial_xN^{-1}f=-\frac{x}{2i} f \in \mathcal{A}^{\alpha-1,*}(Y)\\
\quad
x \partial_{y_j} N^{-1}f = -\frac{x}{2i} \int_{-\infty}^x \partial_{y_j}[f](t,\bfy)\dd t=  -\frac{x}{2i} \int_{-\infty}^x \frac{1}{t}\cdot t\partial_{y_j}[f](t,\bfy)\dd t.
\end{gathered}
\end{equation}
Since $f\in\mathcal{A}^{\alpha,*}(Y)$,
$t\partial_{y_j}f\in\mathcal{A}^{\alpha,*}(Y)$.
Multiplication by $t^{-1}$ gains one order at $\mathrm{bf}$, so the
integrand in \cref{eq:redman} belongs to
$\mathcal{A}^{\alpha+1,*}(Y)$. 
Therefore, we have
\begin{align}
    V(N^{-1}f)\in g+\sum_ja_j x\label{eq:603}
\int_{-\infty}^{x}g_j(t,\bfy)\dd t,
\end{align}
where $g\in \mathcal{A}^{\alpha-1,*}(Y)$, $ g_j \in
\mathcal{A}^{\alpha+1,*}(Y)$, and $a_j \in C^{\infty}(Y)$. Thus, 
\begin{equation}
    AV(N^{-1}f)\in Ag+\sum_j A\bigg[ \underbrace{a_jx}_{\in\mathcal A ^{-1,0} \text{ locally}}
\int_{-\infty}^{x}g_j(t,\bfy)\dd t\bigg],
\end{equation}
By the definition of conormal spaces, $Ag \in \mathcal{A}^{\alpha-1,*}$.
The terms in the sum are controlled using the product rule and the
inductive hypothesis with $\alpha+1$ in place of $\alpha$. Indeed,
forward integration loses one order, so
\begin{align}
     \int_{-\infty}^x g_j(t,\bfy)\dd t
    \in\mathcal{A}^{\alpha,*}(Y).
\end{align}
Moreover, $a_jx\in\mathcal{A}^{-1,*}(Y)$ in the back chart.
Because b-differentiation preserves these conormal orders, the product
rule gives
\begin{align}
      A\left[
        a_jx\int_{-\infty}^x g_j(t,\bfy)\dd t
    \right]
    \in\mathcal{A}^{\alpha-1,*}(Y).
\end{align}

So, altogether, we get  $|AV(N^{-1} f)| \lesssim \rho_{\mathrm{bf}}^{\alpha-1}$, as desired. 

When $f$ is supported within the side chart, the analysis of $N^{-1} f$ in the side chart proceeds identically. 

\subsubsection{corner-corner (conormality)}\label{subsec:same2}
The case where $f$ is supported within the corner chart and $N^{-1} f$ is analyzed there is similar to the two cases just discussed, except in that we need to keep track of behavior at $\mathrm{ff}$ in addition to $\mathrm{bf}$. 
Split $F\coloneqq -2i N^{-1} f= \int_{-\infty}^x f(t,\bfy)\dd t$ into two terms:
\begin{align}
 F (x,y) = \underbrace{\int_{-\infty}^\infty f(t,\bfy) \dd t}_{\coloneq F_1(\bfy)} - \underbrace{\int_{x}^\infty f(t,\bfy)\dd t}_{\coloneq F_2(x,\bfy)}.\label{eq:560}
\end{align} 
The first term is well-defined on account of $\beta>1$.
We will show that $F_1\in \calA^{\alpha-1,(0,0)}(Y)$ and $F_2\in \calA^{\alpha-1,\beta-1}(Y)$, near $\mathrm{bf}\cap\mathrm{ff}$. 
First consider $F_1$. In the coordinate system $(x,\bfy)$, this depends only on $\bfy$, so in the corner coordinate system $(\rho_{\mathrm{bf}},\rho_{\mathrm{ff}},\phi)$, it depends only on $\rho_{\mathrm{bf}},\phi$. This means that once $F_1$, considered as a function of only $\rho_{\mathrm{bf}},\phi$, is known to lie in $\calA^{\alpha-1}([0, \infty)_{\rho_{\mathrm{bf}}}\times \bbS^{d-2}_\phi)$, then it, when considered as a function on the higher dimensional space $Y$,  satisfies 
\begin{equation}
    F_1 \in 
    \calA^{\alpha-1,(0,0)}(Y ) \text{ locally}
\end{equation}
\emph{automatically}, since $\calA^{\alpha-1,(0,0)}(Y ) = C^\infty([0,\infty)_{\rho_{\mathrm{ff}}} ; \calA^{\alpha-1}([0, \infty)_{\rho_{\mathrm{bf}}}\times \bbS^{d-2}_\phi))$ locally. And since we would rather work in the coordinate system $x,\bfy$ than $\rho_{\mathrm{bf}},\phi$, note that 
\begin{equation}
    F_1 \in \calA^{\alpha-1,0}(Y )\text{ near }\mathrm{bf}\cap\mathrm{ff} \Longrightarrow F_1 \in \calA^{\alpha-1}([0, \infty)_{\rho_{\mathrm{bf}}}\times \bbS^{d-2}_\phi).  
\end{equation}
So, our goal for both $F_1,F_2$ is to establish conormal estimates 
\begin{equation}
    |A F_1 | \lesssim \rho_{\mathrm{bf}}^{\alpha-1},\quad |A F_2| \lesssim \rho_{\mathrm{bf}}^{\alpha-1} \rho_{\mathrm{ff}}^{\beta-1}   
\end{equation}
near $\mathrm{bf}\cap\mathrm{ff}$, for any polynomial $A$ in $x\partial_x,y \partial_{y_j}$. 

Beginning with $F_1$, note that $\partial_x F_1=0$, so it suffices to consider the case where $A$ is a polynomial in the $y\partial_{y_j}$. These derivatives fall under the integral sign:
\begin{equation}
    AF_1 = \int_{-\infty}^\infty A f(t,\bfy) \dd t = \int_{\varepsilon y}^\infty A f(t,\bfy) \dd t  ,  
\end{equation}
where to insert $\varepsilon y>0$ as the lower bound of the integral we used the assumption that $f$ is supported within the corner chart (and therefore a compact subset thereof). 
Because $f\in \calA^{\alpha,\beta}(Y)$, we have $Af\in \calA^{\alpha,\beta}(Y)\subset \rho_{\mathrm{bf}}^{\alpha}\rho_{\mathrm{ff}}^{\beta} L^\infty$ as well, giving 
\begin{equation}
    |AF_1| \lesssim \int_{\varepsilon y}^\infty \rho_{\mathrm{bf}}(t,\bfy)^{\alpha}\rho_{\mathrm{ff}}(t,\bfy)^{\beta} \dd t =  y^{\beta-\alpha}\int_{\varepsilon y}^\infty t^{-\beta} \dd t = \frac{y^{1-\alpha}}{\beta-1} = \frac{\rho_{\mathrm{bf}}^{\alpha-1}}{\beta-1}.
\end{equation}
This completes our discussion of $F_1$. 

Now consider $F_2$. It suffices to consider $A$ to be an individual monomial of the form $A=L (x\partial_x)^j$ for a polynomial $L$ in the $y\partial_{y_j}$ and $j\in \bbN$. When $j=0$, then $A$ consists only of $y$-derivatives, which fall under the integral sign:
\begin{equation}
    A F_2 = \int_x^\infty A f(t,\bfy) \dd t . 
\end{equation}
The integrand is in $\calA^{\alpha,\beta}(Y)\subset \rho_{\mathrm{bf}}^\alpha\rho_{\mathrm{ff}}^\beta L^\infty$, so 
\begin{equation}
    |A F_2| \lesssim \int_{x}^\infty \rho_{\mathrm{bf}}(t,\bfy)^{\alpha}\rho_{\mathrm{ff}}(t,\bfy)^{\beta} \dd t =  y^{\beta-\alpha}\int_{x}^\infty t^{-\beta} \dd t = \frac{y^{\beta-\alpha} x^{1-\beta}}{\beta-1} = \frac{\rho_{\mathrm{bf}}^{\alpha-1}
\rho_{\mathrm{ff}}^{\beta-1}}{\beta-1}.
\end{equation}
When $j\geq 1$, we write $A=Q (x\partial_x)$ for $Q\in \operatorname{Diff}_{\mathrm{b}}(Y)$, giving 
\begin{equation} 
    A F_2 = - Q (x f),
\end{equation} 
by the fundamental theorem of calculus. Since $x \in \calA^{-1,-1}(Y)$ near $\mathrm{bf}\cap\mathrm{ff}$, we have $xf\in \calA^{\alpha-1,\beta-1}(Y)$ locally, hence the same holds for $AF_2$. This completes the estimates for $F_2$. 

\subsubsection{$f$ and $N^{-1}f$ at offset charts}\label{subsec:offset}
The analysis handling $N^{-1} f$ in a chart \emph{ahead} of where $f$ is supported is straightforward, because $N^{-1} f$ is constant along lines of constant $\bfy$ ahead of the support of $f$.

For example, consider how $N^{-1} f$ looks in the side chart, when $f$ is supported in the back chart. In the part of the side chart not covered by the back chart, $N^{-1} f$ depends only on $\bfy$ with respect to the coordinate system $(x,\bfy)$. The relevant b-vector fields are $y\partial_x,y \partial_{y_j}$. The former annihilates $N^{-1} f$ past the support of $f$. The latter is a b-vector field on the overlap of the back and side charts, so hitting $N^{-1} f$ with any polynomial therein results in a conormal function $\in \calA^{\alpha-1,\beta-1}$. 
This completes the proof of \Cref{lem:conorm}.

\end{proof} 
\subsection{Polyhomogeneity (Proof of \Cref{prop:intmain})}

Let 
\begin{equation}
    F(x,\bfy) = \int_{-\infty}^x f(t,\bfy) \dd t \propto N^{-1} f, \label{eq:2370}
\end{equation}
for short.

The structure of the proof is the same as that of the previous proposition. That is, we break the support of $f$ into the back, side, and corner charts and then study $N^{-1} f$ in each of those charts. The nontrivial cases are the same as before. Additionally, the ``propagation'' argument used to handle the case where we analyze $N^{-1} f$ ahead of the support of $f$ proceeds verbatim. 
This means that it suffices to analyze $N^{-1} f$ in the same chart where $f$ is supported.

Suppose $f \in \rho_{\mathrm{bf}}^{\alpha}C^{\infty}(Y)$ (where $\alpha>1$) is supported in the back chart. 
By Taylor's theorem, for any $J\in \bbN$, we have 
\begin{equation}\label{eq:polybackexp}
f-\rho_{\mathrm{bf}}^\alpha \sum_{j=0}^{J-1}\rho_{\mathrm{bf}}^ja_j(\tilde{\bfy}) \in \mathcal{A}^{\alpha+J}(Y), \quad \tilde{\bfy}=-\bfy/x
\end{equation}
for some $a_j \in C_c^{\infty}(\bbR^{d-1})$. 
Thus, 
\begin{equation}
    F(x,\bfy) = \sum_{j=0}^{J-1} \int_{-\infty}^x  \rho_{\mathrm{bf}}^{j+\alpha}(t,\bfy) a_j(-\bfy/t ) \dd t + \int_{-\infty}^x \mathcal{A}^{\alpha+J}(Y).
\end{equation}
The final term is controlled using \Cref{lem:conorm}, which says that it is in $\calA^{\alpha+J-1}(Y)$ locally.
The other terms on the right-hand side are
\begin{equation}
\int_{-\infty}^x
\rho_{\mathrm{bf}}^{j+\alpha}(t,\bfy)
a_j(-\bfy/t)\dd t
=
\int_{-\infty}^x
(-t)^{-(j+\alpha)}
a_j(-\bfy/t)\dd t,
\end{equation}
where $\rho_{\mathrm{bf}}(x,\bfy)=-1/x$ in the back chart.
Letting $u=x/t$, the right-hand side becomes
\begin{equation}
\rho_{\mathrm{bf}}^{j+\alpha-1}
\int_0^1
u^{j+\alpha-2}
a_j(u\tilde{\bfy})\dd u
\in
\rho_{\mathrm{bf}}^{j+\alpha-1}
C^\infty(\bbR^{d-1}_{\tilde{\bfy}}).
\end{equation}
In summary, we have shown that 
\begin{equation}
    F \sim \rho_{\mathrm{bf}}^{\alpha-1} \sum_{j=0}^\infty  \rho_{\mathrm{bf}}^j  \underbrace{\int_{0}^1 u^{j+\alpha-2} a_j(u\tilde{\bfy}) \dd u}_{\in C^\infty(\bbR^{d-1}_{\tilde{\bfy}})}, 
\end{equation}
hence $F=N^{-1} f$ is in $\rho_{\mathrm{bf}}^{\alpha-1}C^\infty(Y)$ in the back chart.

Having analyzed the back-back case, regarding the side-side case we state only that it is strictly easier.

For the corner-corner case, the proof is slightly more involved. Suppose $f \in \mathcal{A}^{(\alpha,0),\calF}(Y)$ is supported in the corner chart, in $\{\varepsilon y< x\}$, for some fixed $\varepsilon>1$, so that 
\begin{equation}
    F(x,\bfy) = \int_{\varepsilon y}^x f(t,\bfy)\dd t. 
\end{equation}
The polyhomogeneous expansion at $\mathrm{ff}$ gives $a_{z,j}\in \rho_{\mathrm{bf}}^{\alpha} C^\infty([0,\infty)_{\rho_{\mathrm{bf}}}\times \bbS^{d-2}_\phi )$ such that 
\begin{equation}\label{eq:polycornerexp}
f-\sum_{\substack{ \Re z \leq \beta\\ (z,j) \in \mathcal{F}}}\rho_{\mathrm{ff}}^z\log(\rho_{\mathrm{ff}})^ja_{z,j}(\rho_{\mathrm{bf}},\phi) \coloneq f_\beta \in \mathcal{A}^{(\alpha,0),\beta}.
\end{equation}
The error $f_\beta$ can be expanded at $\mathrm{bf}$: $\exists b_{j,\beta} \in \calA^{\beta}([0,\infty)_{\rho_{\mathrm{ff}}}\times \bbS^{d-2}_\phi )$, supported in a small neighborhood of $\{\rho_{\mathrm{ff}}=0\}$, such that 
\begin{equation}
    f_\beta- \rho_{\mathrm{bf}}^{\alpha} \sum_{j=0}^{J-1} \rho_{\mathrm{bf}}^j b_{j,\beta}(\rho_{\mathrm{ff}},\phi) \in \calA^{\alpha+J,\beta}(Y),  
\end{equation}
for any $J\in \bbN$. 
Thus, 
\begin{multline}\label{eq:misc_218}
  F=\overbrace{\sum_{\substack{\Re z \leq \beta\\ (z,j) \in \mathcal{F}}}\int_{\varepsilon y }^x \rho_{\mathrm{ff}}(t,\bfy)^z\log(\rho_{\mathrm{ff}}(t,\bfy))^ja_{z,j}(\rho_{\mathrm{bf}}(t,\bfy),\phi) \dd t}^{\mathrm{I}}  \\  + \underbrace{ \rho_{\mathrm{bf}}^\alpha \sum_{j=0}^{J-1} \rho_{\mathrm{bf}}^j  \int_{\varepsilon y}^x b_{j,\beta}(\rho_{\mathrm{ff}}(t,\bfy),\phi  )   \dd t }_{\mathrm{II}}
  + \underbrace{\int_{\varepsilon y}^{x} \mathcal{A}^{\alpha+J,\beta}(Y)}_{\mathrm{III}},  
\end{multline}
where, in writing the sum labeled `II', we have pulled $\rho_{\mathrm{bf}}$ out of the integral, which can be done because $\rho_{\mathrm{bf}}(t,\bfy)=1/y$ is independent of $t$.

For $\beta>1$, \Cref{lem:conorm} tells us that 
\begin{equation} 
\int_{-\infty}^{x} \mathcal{A}^{\alpha+J,\beta}(Y) \in 
\calA^{\alpha+J-1,(0,0) }(Y)+\calA^{\alpha+J-1,\beta-1 }(Y).
\end{equation} 
Term III above is the same integral, except with the lower bound $\varepsilon y$ instead of $-\infty$. The difference between the two integrals is of the form  $\int_{-\infty}^{\varepsilon y} \calA^{\alpha+J}(\overline{\bbR^d})$.
This is in $\calA^{\alpha+J-1,(0,0)}(Y)$ near the corner $\mathrm{bf}\cap\mathrm{ff}$ (cf.\ the analogous argument in the proof of \Cref{lem:conorm}). So, term III satisfies 
\begin{equation} 
\mathrm{III} \in \calA^{\alpha+J-1,(0,0) }(Y)+\calA^{\alpha+J-1,\beta-1 }(Y).
\end{equation} 

Consider now the terms in $\mathrm{II}$. We have
\begin{multline}
    \int_{\varepsilon y}^x b_{j,\beta}(\rho_{\mathrm{ff}}(t,\bfy),\phi  )   \dd t = 
    \int_{\varepsilon y}^x b_{j,\beta}\left( \frac{y}{t} ,\phi  \right)   \dd t  =\frac{1}{\rho_{\mathrm{bf}}\rho_{\mathrm{ff}} } \int_{\varepsilon \rho_{\mathrm{ff}} }^1 b_{j,\beta}\left( \frac{\rho_{\mathrm{ff}} }{u } ,\phi  \right)   \dd u   \\ \in   \calA^{(-1,0),\beta-1 }(Y) + \calA^{(-1,0),(0,0)}(Y)  , 
\end{multline}
where the final inclusion follows from
\begin{multline}
    \int_{\varepsilon \rho_{\mathrm{ff}} }^1 b_{j,\beta}\left( \frac{\rho_{\mathrm{ff}} }{u } ,\phi  \right)   \dd u = \rho_{\mathrm{ff}}\int_{\rho_{\mathrm{ff}} }^{\varepsilon^{-1} } \frac{ b_{j,\beta}(s,\phi)}{s^2}   \dd s = \rho_{\mathrm{ff}}\int_{\rho_{\mathrm{ff}}}^{\varepsilon^{-1}}  \calA^{\beta-2} ([0,\infty)_s\times \bbS^{d-2}_\phi ) \\ \in (\calA^{\beta}+\calA^{(1,0)})([0,\infty)_{\rho_{\mathrm{ff}}}\times \bbS^{d-2}_\phi ). 
\end{multline}
Thus, term II is in $\calA^{(\alpha-1,0),\beta-1}(Y) + \calA^{(\alpha-1,0),(0,0) }(Y)$.

Finally, we turn to the terms in $\mathrm{I}$, which we integrate explicitly: because $\rho_{\mathrm{bf}}(t,\bfy)$ does not depend on $t$, 
\begin{align} \label{eq:misc_221}
    \begin{split} 
    \int_{\varepsilon y }^x \rho_{\mathrm{ff}}(t,\bfy)^z\log(\rho_{\mathrm{ff}}(t,\bfy))^ja_{z,j}(\rho_{\mathrm{bf}}(t,\bfy),\phi) \dd t &= a_{z,j} \int_{\varepsilon y }^x \rho_{\mathrm{ff}}(t,\bfy)^z\log(\rho_{\mathrm{ff}}(t,\bfy))^j\dd t \\
    &= a_{z,j} \int_{\varepsilon y }^x \Big( \frac{y}{t}\Big)^z\log\Big( \frac{y}{t} \Big)^j\dd t \\
    &= a_{z,j} y \int_{y/x}^{\varepsilon^{-1} } u^{z-2}(\log u)^j  \dd u.
    \end{split} 
\end{align}
The lower bound on the integral is therefore $\rho_{\mathrm{ff}}=y/x$. 
If $z=1$, then the integral is 
\begin{equation}\label{eq:piece2}
    \int_{\rho_{\mathrm{ff}}}^{\varepsilon^{-1}}u^{z-2}(\log u)^j\dd u=-\frac{1}{j+1}(\log \rho_{\mathrm{ff}})^{j+1}+C_{\varepsilon}.
\end{equation}
If $z \neq 1$, integration by parts yields
\begin{align}
\int_{\rho_{\mathrm{ff}}}^{\varepsilon^{-1}}u^{z-2}(\log u)^jdu &=\frac{1}{z-1}\int_{\rho_{\mathrm{ff}}}^{\varepsilon^{-1}}\partial_u[u^{z-1}](\log u)^j\dd u\nonumber \\
&=\frac{1}{z-1}\bigg[u^{z-1}(\log u)^j \bigg]_{u=\rho_{\mathrm{ff}}}^{u=\varepsilon^{-1}}-\frac{j}{z-1}\int_{\rho_{\mathrm{ff}}}^{\varepsilon^{-1}}u^{z-2}(\log u)^{j-1}\dd u\nonumber \\
&=\rho_{\mathrm{ff}}^{z-1}\sum_{k=0}^jc_{z,k}(\log\rho_{\mathrm{ff}})^k+C_{z,\varepsilon},\quad c_{z,k}\in \bbC \label{eq:finalpiece}
\end{align}
where the last line follows from repeated integration by parts. 
Thus, we have produced the polyhomogeneous expansion for the terms in $\mathrm{I}$. The indices found above are in $\calF^+$. Most are in $\calF-1$. The constant terms are covered by the $(0,0)$ part of $\calF^+$. The extra log terms in the $z=1$ case are covered by the final part of \cref{eq:F+}. 
The factor of $a_{z,j}y$ in \cref{eq:misc_221} is $a_{z,j}\rho_{\mathrm{bf}}^{-1} \in \calA^{(\alpha-1,0)}([0,\infty)_{\rho_{\mathrm{bf}}}\times \bbS^{d-2} )$. Thus, the sum I is in $\calA^{(\alpha-1,0),\calF^+}$. 

Thus, we have shown that the terms above lie in 
\begin{equation} 
\underbrace{\calA^{(\alpha-1,0),\calF^+}(Y)}_{\mathrm{I}}+ \underbrace{\calA^{(\alpha-1,0),\beta-1}(Y) + \calA^{(\alpha-1,0),(0,0)}}_{\mathrm{II}}+
\underbrace{\calA^{\alpha+J-1,\beta-1 }(Y)}_{\mathrm{III}} .
\end{equation} 
Hence, $F=\mathrm{I}+\mathrm{II}+\mathrm{III}$ is in a partially polyhomogeneous space: 
\begin{equation}
    F \in \calA^{((\alpha-1,0), \alpha+J-1),(\calF^+,\beta-1)}(Y) 
\end{equation}
(since $(0,0)\in \calF^+$). Taking the intersection over all $\beta,J$, 
\begin{equation}
  \bigcap_{\beta,J} \calA^{((\alpha-1,0), \alpha+J-1),(\calF^+,\beta-1)}(Y)  = \calA^{(\alpha-1,0),\calF^+}(Y),
\end{equation}
yields the final claim.

\section{Absence of logs in the plane wave contribution}
\label{sec:logs}

This appendix proves the absence of logarithmic terms in the conclusion
of \S\ref{sec:bf}. We continue using the notation introduced in that
section. The functions studied there are, up to alternating signs, the
iterates of $N^{-1}E$ applied to $1$. Here,
$N=-2i\partial_x$ and $E=\tilde P-N$.

In this appendix, we generalize the setup slightly by allowing more general $E$. (While $N$ will always be $-2i\partial_x$.) Fix 
\begin{equation} 
    E\in \operatorname{Diff}_{\mathrm{b}}^{2,-2}(\overline{\bbR^d})
\end{equation} 
such that $E$ differs from a (second-order) constant coefficient operator $E_0$ by an element of $ \operatorname{Diff}_{\mathrm{sc}}^{2,-2}(\overline{\bbR^d}) $. 
That is, 
\begin{equation}
    E_0 \in \operatorname{span}_{\bbC}\{\partial_{x}^2 ,\partial_x \partial_{y_j},\partial_{y_j}\partial_{y_k}\},
\end{equation}
\begin{equation}\label{eq:gen_E_form}
		E-E_0 \in \frac{1}{1+x^2+\lVert \bfy \rVert^2} \operatorname{span}_{C^\infty(\overline{\bbR^d}) }\{1, \partial_x, \partial_{y_j}, \partial_x^2,\partial_x \partial_{y_j}, \partial_{y_j}\partial_{y_k} \} =  \operatorname{Diff}_{\mathrm{sc}}^{2,-2}(\overline{\bbR^d}).
\end{equation}
For example, when $E=\tilde{P}-N$ (for any $\tilde{P}$ as in \S\ref{subsec:asymptotically_Euclidean}), then 
\begin{equation}
    E= \tilde{P}_0 - N+ (\tilde{P}-\tilde{P}_0) = \underbrace{- \partial_x^2  +\triangle_{\bfy}}_{\coloneq E_0}+ \underbrace{ (\tilde{P}-\tilde{P}_0)}_{\in \operatorname{Diff}_{\mathrm{sc}}^{2,-2}(\overline{\bbR^d} ) }.
\end{equation}
Here, $E_0 = \tilde{P}_0-N$. (Because $E$ has two orders of b-decay, $E_0$ is not allowed to have $1,\partial_x,\partial_{y_j}$ terms.)

Now define 
\begin{equation} 
\label{eq:uj_def}
u_j := (N^{-1}E)^j[1].
\end{equation}
\Cref{prop:mainI_init0} says that this is polyhomogeneous on the polar forward blowup $Y$, when $E=\tilde{P}-N$. The proof of that proposition only used the fact that $E\in \operatorname{Diff}_{\mathrm{b}}^{2,-2}(X)$, and any $E$ of the form above lies in this same space. Consequently, the proof goes through in the present generality, and we conclude that $u_j$ is polyhomogeneous on $Y$.

The goal of this appendix is to prove:

\begin{proposition}\label{prop:aut} 
The function $u_j$ has no log terms in its polyhomogeneous expansion at the front face $\mathrm{ff}\subset Y$. 
In other words, for each fixed $\bfy\in \bbR^{d-1}$, 
\begin{equation}
    u_j(x,\bfy) \in \bbC[ x] + C^\infty([0,1)_{1/x}).
\end{equation}
\end{proposition}
The proposition is more subtle than might appear. For example: 
	\begin{itemize}
		\item Even though the proposition is naturally regarded as about the behavior of $u_j$ on $Y$, it can fail if $E$ is only well-behaved on $Y$ instead of $\overline{\bbR^d}$. 

        Indeed, consider $E=\triangle_{\bfy}+V$ for 
        \begin{equation} \label{eq:V_y_ce}
            V = \chi\left(\frac{x}{\langle y\rangle}\right)\frac{1}{x^2\langle y\rangle},
            \end{equation} 
            where $\chi \in C^\infty(\bbR)$ is a smooth function such that $\chi=0$ on $(-\infty,\epsilon)$ for some $\epsilon>0$ and such that $\chi=1$ on $(2\epsilon,\infty)$. 
            
            Then $V$ is well-behaved on $Y$ and decays well, but it is not smooth on $\overline{\bbR^d}$, because $1/\langle y \rangle$ is not constant on $\mathrm{ff}\subset Y$.

            Then, 
            \begin{equation} 
                u_1 = N^{-1}E[1]= N^{-1}[V] = \frac{i}{2\langle y\rangle^2}F\left(\frac{x}{\langle y\rangle}\right),\quad F(t) = \int_{-\infty}^t \frac{\chi(s)}{s^2}\dd s. 
            \end{equation} 
            In particular, for $x\gg \langle y\rangle$, 
            \begin{equation} 
                u_1 = \frac{iC}{2\langle y\rangle^2}-\frac{i}{2\langle y\rangle x},
            \end{equation}
            where $C = \int_{-\infty}^\infty \frac{\chi(u)}{u^2}\dd u$.
            No log terms are present.

            Next consider $u_2$. We have
            \begin{equation} u_2 = N^{-1}Eu_1 = N^{-1}\triangle_{\bfy} u_1 + N^{-1}V u_1.
            \end{equation}
            Since $V = -2i\partial_xu_1$, the latter term is proportional to $ N^{-1}\partial_x (u_1^2)$ and hence to  $u_1^2$. This has no log terms at $\mathrm{ff}$. In contrast,
    \begin{equation} 
        \triangle_{\bfy} u_1 = \triangle_{\bfy}\left(i\frac{C}{2\langle y\rangle^2}\right) - \triangle_{\bfy}\left(\frac{i}{2\langle y\rangle}\right)\cdot\frac{1}{x}.
    \end{equation} 
    Hence integrating produces a nontrivial logarithmic term proportional to $\triangle_{\bfy} 1/\langle y \rangle$.
		\item The proposition's truth hinges on the presence of certain cancellations. If instead we consider
        \begin{equation}\label{eq:misc_888}
    u_j
    =
    N^{-1}E^{(j)}
    N^{-1}E^{(j-1)}
    \cdots
    N^{-1}E^{(1)}[1].
\end{equation}
		for \emph{different} $E^{(\ell)}$, then the proposition can fail.

        We illustrate this with $d=3$, $j=3$, with $E^{(1)} = V$, $E^{(2)} = \triangle_{\bfy}$, and $E^{(3)} = V$, where 
        \begin{equation} 
            V(x,\bfy) =\frac{1}{1+r^2}= \frac{1}{x^2+\langle\bfy\rangle^2}\in\mathcal{A}^2(\overline{\bbR^d}).
        \end{equation} 
        Then, $u_3 = N^{-1} (V N^{-1} \triangle_{\bfy} N^{-1} V  )$.
        Since $V$ has the expansion $V = x^{-2}+O_y(x^{-4})$ as $x\to\infty$ with $y$-fixed,
        \begin{equation}
            N^{-1}V \propto C_1(\bfy) - \frac{1}{x}+O_y(x^{-3}),\quad C_1(\bfy) = \int_{\bbR}\frac{1}{t^2+\langle\bfy\rangle^2}\,dt = \frac{\pi}{\langle\bfy\rangle}.
        \end{equation}
        The big-$O$ term is conormal at $\mathrm{ff}\backslash \mathrm{bf}$, so $\triangle_{\bfy}N^{-1}V = \triangle_{\bfy}C_1(\bfy) + O_y(x^{-3})$ as $x\to\infty$, so
        \begin{equation} 
            N^{-1}\triangle_{\bfy}N^{-1}V = (\triangle_{\bfy}C_1(\bfy))x + C_2(\bfy) + O_y(x^{-1})
            \end{equation} 
        as $x\to\infty$, for some $C_2(\bfy)$. Multiplying by $V$ thus yields
        \begin{equation} 
            VN^{-1}\triangle_{\bfy}N^{-1}V = (\triangle_{\bfy}C_1(\bfy))xV + C_2(\bfy)V + O(x^{-3}) = \frac{(\triangle_{\bfy}C_1(\bfy))}{x}+O(x^{-2}).
        \end{equation} 
        The coefficient of $x^{-1}$ is nonzero, so integrating again then yields a log term in $u_3$.
	\end{itemize}

The counter-examples above show that \Cref{prop:aut} cannot follow solely from computations regarding index sets on $Y$. 

Now we prove the proposition, outsourcing the main combinatorial lemma:
\begin{proof}
Set $\kappa=-1/(2i)$. Then $ N^{-1}=\kappa(I+C),$
where
\begin{align}
        Iu(x,\bfy)=\int_0^x u(t,\bfy)\dd t,
    \qquad
    Cu(x,\bfy)=\int_{-\infty}^0u(t,\bfy)\dd t.
\end{align}
Consequently,
\begin{align}
    u_j=\kappa^j\big((I+C)E\big)^j[1].
\end{align}
Since the nonzero scalar factor $\kappa^j$ has no effect on the
presence or absence of logarithmic terms, for the remainder of this
proof we use $u_j$ to denote the normalized iterate $\big((I+C)E\big)^j[1]$.

Note that $I,C$ are \emph{operators} taking functions of $x,\bfy$ to functions. In $C$'s case, the output is independent of $x$. Consequently, we write 
\begin{equation} 
    u_j = IEu_{j-1} + c_j,\quad c_j = CEu_{j-1}.
\end{equation} 
By induction, we have
\begin{equation}
\label{eq:(IE)^k-sum}
u_j= \sum_{k=0}^j (IE)^kc_{j-k},
\end{equation}
where $c_0 = 1$. Because the $c_\bullet$ depend only on $\bfy$, it suffices to study the mapping properties of $(IE)^k$ on smooth functions on $\bbR^{d-1}_\bfy$. 

Note that 
\begin{equation} 
    c_j\in C^\infty(\bbR^{d-1}_{\bfy})
\end{equation} 
(e.g.\ by the mapping properties of forward integration proven in \S\ref{subsec:same}).
The operator $(IE)^k$, acting on $C^\infty(\bbR^{d-1}_{\bfy})$, is a $x$-dependent \emph{differential operator}. That is, for each fixed $x\in \bbR$, 
\begin{equation}
    [ C^\infty(\bbR^{d-1})\ni f\mapsto 
    ((I E)^k f ) (x,-) ] \in  \operatorname{Diff}(\bbR^{d-1}_{\bfy}).
\end{equation}
The coefficients of this differential operator are smooth functions of $x$, so we write 
\begin{equation}
    (I E)^k \in C^\infty(\bbR_x; \operatorname{Diff}(\bbR^{d-1}_{\bfy})  ), 
\end{equation}
identifying $(I E)^k$ with the map $x\mapsto [f\mapsto ((I E)^k f)(x,-)]$.\footnote{In order to clarify this harmless identification, consider the case where $E= \Upsilon(x,y)\partial_{y_1}+\partial_x^2$, for some function $\Upsilon$. Then, if $f\in C^\infty(\bbR^{d-1}_\bfy)$,   
\begin{align}
    I E f  = \Big(\int_0^x \Upsilon(s,\bfy) \dd s \Big) \partial_{y_1} f.
\end{align}
So, $IE$ is identified with 
\begin{align}
     \Big(\int_0^x \Upsilon(s,\bfy) \dd s \Big) \partial_{y_1} \in C^\infty(\bbR_x; \operatorname{Diff}^1(\bbR^{d-1}_{\bfy})).
\end{align} }
The coefficients of $(I E)^k$ are some smooth functions of $x$ valued in $C^\infty(\bbR^{d-1})$. These amount to functions on $\bbR^d_{x,\bfy}$: we can write 
\begin{equation}
    (I E)^k = \sum_{\alpha} a_\alpha(x,\bfy) \partial_{y}^\alpha,\quad a_\alpha\in C^\infty(\bbR^d_{x,\bfy}). 
\end{equation}
Our key claim is that, for each $\bfy\in \bbR^{d-1}$, we have $a_{\alpha}(x,\bfy)\in \bbC[x]+C^\infty[0,1)_{1/x}$. That (combined with \cref{eq:(IE)^k-sum}) will complete the proof. 

Break up $E=\sum_\bullet  E_\bullet$ into finitely many building blocks $E_\bullet$ (satisfying the same assumptions as $E$ itself). A convenient choice of building blocks are:
\begin{equation}
    E_\bullet \in \bbC\{ \partial_x^2,\partial_x \partial_{y_j},\partial_{y_j}\partial_{y_k}\} \cup \Big\{ \Upsilon \partial_x^\nu \partial_y^\beta : \Upsilon \in  \frac{1}{1+x^2+y^2} C^\infty(\overline{\bbR^d}) \Big\}
\end{equation}
Expanding the power $(IE)^k = \left(\sum_\bullet IE_\bullet\right)^k$, and noting that the operators involve do not necessarily commute, we get
\begin{equation}
    (IE)^k = \sum_{S}  \Big[ c(S) \underbrace{\Big( \prod_{E_\bullet \in S } I E_\bullet\Big)+\text{permutations}}_{ \calS[S]} \Big].
\end{equation}
Here, $c(S)\in \bbR$ is a combinatorial normalization factor, and we use the ``$+\text{permutations}$'' notation to mean a symmetrized sum of the (ordered) products of the $IE_1,\dots,IE_k$, i.e.\ 
\begin{equation}
\calS[S] = \Big(\prod_{E_\bullet \in S } I E_\bullet\Big)+\text{permutations}= \sum_{\sigma \in \frakS_k} \prod_{\kappa=1}^k (I E_{\sigma(\kappa)}).
\end{equation}

We will prove that each one of these symmetrized sums has the desired property:
\begin{equation}
    \calS[S]= \sum_{\alpha} a_{\alpha,S}(x,\bfy) \partial_y^\alpha
\end{equation}
for 
\begin{equation}
    a_{\alpha,S}(x,\bfy)\in \bbC[x]+C^\infty[0,1)_{1/x}\text{ for each }\bfy\in \bbR^{d-1}. 
\end{equation}
This combinatorial lemma we outsource to \Cref{lem:logs}.
When these coefficients are used in Step I, the resulting log-free formal expansion is realized by the Borel summation and localization in \S\ref{subsec:stepI}; the present proposition concerns the individual coefficients $u_j$.
\end{proof}

\begin{lemma} 
	\label{lem:logs}
	Fix $k\in \bbN$, and suppose that 
	\begin{equation} 
		E_1,\dots,E_k \in \{\partial_x ,\partial_{y_j} , \partial_x^2,\partial_{y_j}\partial_{y_k}, \partial_x \partial_{y_j},\partial_x^3,\cdots \} \cup  \Big\{ \Upsilon \partial_x^\nu \partial_y^\beta : \Upsilon \in  \frac{1}{1+x^2+y^2} C^\infty(\overline{\bbR^d}) \Big\}.
	\end{equation} 
	Now consider the symmetrized product $\calS$. Then, the coefficients of this operator (see above) are, for each fixed $U\subset \bbR^{d-1}_{\bfy}$, of the form $C^\infty(U)[x]+C^\infty([0,1)_{1/x}\times U)$, i.e.\ a polynomial in $x$ plus a smooth function of $1/x$ (uniformly in $\bfy$, locally)
\end{lemma}
Proving this lemma is the main remaining task of this appendix.

\begin{remark*}
    This lemma allows $E_\bullet=\partial_x,\partial_{y_j},\partial_x^3,\cdots$ (any constant-coefficient monomial except $E_\bullet=1$), even though these terms are not allowed in \Cref{prop:aut}. 
\end{remark*}

As mentioned above, this argument utilizes algebraic cancellations --- it is essential that we are work with the symmetrized product (hence the ``$+\text{permutations}$''). As the counter-example above shows, given a short-range classical potential $V$, the operator 
\begin{align} 
    (IV) (I\triangle_{\bfy})(IV)   &\in C^\infty(\bbR_x; \operatorname{Diff}(\bbR^{d-1}_{\bfy}))\\
    &: f(\bfy) \mapsto I (V I \triangle_{\bfy} IV )(x,\bfy)
\end{align} 
can have log terms in its $x\to\infty$ expansion, because $(IV) (I\triangle_{\bfy})(IV)1$ does.  The key fact expressed by the lemma above is that the symmetrized product 
\begin{equation}\label{eq:symmetrized_example}
    \calS=(I\triangle_{\bfy})(IV)^2+(IV) (I\triangle_{\bfy})(IV)+(IV)^2 (I\triangle_{\bfy})
\end{equation}
does not. See the next subsection.
(Note that here $V$ denotes the multiplication operator $M_V$, so $\triangle_{\bfy} I V$ should be read as the composition $f\mapsto \triangle_{\bfy} (IVf)$.)

\subsection{An example}
Consider the case 
\begin{equation}
    E = \triangle_{\bfy} + V , 
\end{equation}
for $V\in \langle r \rangle^{-2} C^\infty(\overline{\bbR^d})$ a short-range classical potential. Instead of taking $V$ as a whole building block $E_\bullet$, Taylor expand 
\begin{equation}
    V \sim \sum_{\kappa=2}^\infty \frac{V_\kappa(\bfy)}{(1+x)^\kappa}
\end{equation}
at $(\infty,\bfy)\in \mathrm{ff}^\circ \subset Y$ to some high-order $K$, and take 
the building blocks $E_\bullet$ to be (in addition to the constant-coefficient building blocks $\partial_x,\partial_{y_j},\partial_x^2,\cdots$.) those $K$ terms, with another $E_\bullet$ as the Taylor remainder. If $K$ is large enough, then the remainder has many orders of decay, so the absence of log terms in the symmetrized sum 
\begin{equation}
    \calS = \Big(\prod_{E_\bullet \in S } I E_\bullet\Big)+\text{permutations}
\end{equation}
can be deduced from the corresponding statement for sums with one fewer term. 
The main task is therefore to prove that $\calS$ lacks log terms when the $E_\bullet$ are each either a single term in the Taylor series or the constant-coefficient building blocks.

\textit{Claim: $\calS$ lacks log terms when the $V_\kappa(\bfy)$ are polynomials of degree $\leq \kappa-2$.} Note that this claim is purely algebraic. (And it does not hold for the individual $\prod I E_\bullet$, without summing over permutations.)

The assumption on $V_\kappa$ is true when $V\in \langle r \rangle^{-2} C^\infty(\overline{\bbR^d})$. For example,
\begin{equation}
    V = \frac{V_{2,\circ}(\omega)}{r^2}+O\Big(\frac{1}{r^3}\Big)\text{ as }r\to\infty ,\quad \omega=z/r,\quad z=(x,\bfy) 
\end{equation}
for $V_{2,\circ} \in C^\infty(\bbS^{d-1})$. Consequently, 
\begin{equation} 
    V_2=V_{2,\circ}(\rightarrow)
\end{equation} 
is the value of $V_{2,\circ}$ at the forward direction, hence constant.  More generally, $V_\kappa$ is a polynomial whose terms of degree $\nu\in \bbN$ come from the $(\mu-\nu)$th terms in the Taylor expansion of $V_{\mu,\circ}$ at the forward direction. It is because $V_{\mu,\circ}=0$ for $\mu\leq 1$ that $\deg V_\kappa \leq \kappa-2$.  In contrast, given a general $V\in \langle r \rangle^{-2} C^\infty(Y)$, the $V_\kappa$ are not polynomials in $\bfy$, let alone ones of bounded degree.

Thus, if we can prove the italicized claim above, we will have a complete outline for a proof of \Cref{lem:logs}. Our actual proof will proceed via a different route, but the argument above isolates the algebraic core of the phenomenon and allows it to be tested systematically.

The case where $\calS$ consists of $k=1$ factors is  essentially trivial. The $k=2$ case is non-trivial, with the main case being that where one $E_\bullet$ is $\partial_y^\alpha$, for some multi-index $\alpha$ (which is not allowed to be 0), and the other is $V_\kappa/(1+x)^\kappa$. 
The symmetrized product is therefore 
\begin{equation}
    \calS = I \partial_y^\alpha I \frac{V_\kappa}{(1+x)^\kappa} + I   \frac{V_\kappa}{(1+x)^\kappa} I \partial_y^\alpha.
\end{equation}
One must remember that $V_\kappa$ is to be read as the multiplication operator $f\mapsto fV_\kappa$. 
Explicitly, 
\begin{align}
    I   \frac{V_\kappa}{(1+x)^\kappa} I \partial_y^\alpha &= \Big( \int_0^x   \frac{t}{(1+t)^\kappa} \dd t \Big)  (V_\kappa\partial_y^\alpha) \\
     I \partial_y^\alpha I \frac{V_\kappa}{(1+x)^\kappa} &= \Big( \int_0^x \int_0^t \frac{1}{(1+s)^\kappa} \dd s \dd t \Big) \sum_{\beta\leq \alpha} \binom{\alpha}{\beta} (\partial_y^\beta V_\kappa) \partial_y^{\alpha-\beta}.
\end{align}
These are both log-free unless $\kappa=2$. When $\kappa=2$, the two integrals above are 
\begin{align}
    \int_0^x   \frac{t}{(1+t)^2} \dd t &= \log (1+x) +\text{rational} \\ 
    \int_0^x \int_0^t \frac{1}{(1+s)^2} \dd s \dd t  &= -\log (1+x)+\text{rational}, 
\end{align}
where `rational' means a rational function of $x$.  Thus, unless $V_2=0$, the individual terms in $\calS$ have log terms, $\propto V_\kappa \partial_y^\alpha$. But $\calS$ itself is 
\begin{align}
    \begin{split} 
    \calS &= (\log (1+x) +\text{rational}) V_2 \partial_y^\alpha - (\log (1+x)+\text{rational}) ( \partial_y^\alpha V_2 + \text{intermediate terms} + V_2 \partial_y^\alpha) \\ &= (\partial_y^\alpha V_2+\text{intermediate terms}) \log (1+x)+\text{rational},
    \end{split}
\end{align}
where `rational' means a linear combination of $1,\partial_y^\alpha$ with coefficients which, for each individual $\bfy$, are rational functions of $x$. For general $V_2$, the symmetrized sum $\calS$ still has a log term, but if $V_2$ is \emph{constant}, then (because $\alpha$ is nonzero) 
\begin{equation} 
    \partial_y^\alpha V_2=0
\end{equation}
and therefore $\calS$ is log-free. Note that our assumption that $V_\kappa$ is a polynomial of degree $\leq \kappa-2$ says that $V_2$ is constant. 
Thus, the claim above is true when $\calS$ involves $k=2$ factors, due to a cancellation of log terms between the terms in the symmetrization.

Now consider the case of $k=3$ factors. For simplicity, we discuss just the example
\begin{multline}\label{eq:S_log_example}
    \calS = (I\triangle_{\bfy})\Big(I\frac{V_\kappa}{(1+x)^\kappa}\Big)\Big(I\frac{V_\ell}{(1+x)^\ell}\Big) + (I\triangle_{\bfy})\Big(I\frac{V_\ell}{(1+x)^\ell}\Big)  \Big(I\frac{V_\kappa}{(1+x)^\kappa}\Big)\\ 
    +\Big(I\frac{V_\kappa}{(1+x)^\kappa}\Big)(I\triangle_{\bfy})\Big(I\frac{V_\ell}{(1+x)^\ell}\Big) + \Big(I\frac{V_\ell}{(1+x)^\ell}\Big)(I\triangle_{\bfy})\Big(I\frac{V_\kappa}{(1+x)^\kappa}\Big) \\ 
    +\Big(I\frac{V_\ell}{(1+x)^\ell}\Big)\Big(I\frac{V_\kappa}{(1+x)^\kappa}\Big) (I\triangle_{\bfy})+\Big(I\frac{V_\kappa}{(1+x)^\kappa}\Big) \Big(I\frac{V_\ell}{(1+x)^\ell}\Big) (I\triangle_{\bfy}) ,
\end{multline}
which suffices to convey the main points.
Explicitly, 
\begin{equation}\label{eq:S_ex}
    \calS = \calS 1 + \sum_{j=1}^{d-1} [\cdots ]\partial_{y_j} + V_\kappa V_\ell C^\infty(\bbR_x) \triangle_{\bfy}.
\end{equation}
The coefficients here are ordinary functions of $x,\bfy$. The coefficient of $\partial_{y_j}$ is abbreviated.

First consider the zeroth order terms, $\calS 1$.
The contributions from the terms in the bottom line vanish (since $\triangle_{\bfy} 1=0$), and we can simplify:
\begin{align}
    (I\triangle_{\bfy})\Big(I\frac{V_\kappa}{(1+x)^\kappa}\Big)\Big(I\frac{V_\ell}{(1+x)^\ell}\Big) &= (\triangle_{\bfy}( V_\kappa V_\ell)) I^2 \bigg[ \frac{1}{(1+x)^\kappa} I \frac{1}{(1+x)^\ell}  \bigg] \label{eq:annoying} \\
    \Big(I\frac{V_\kappa}{(1+x)^\kappa}\Big)(I\triangle_{\bfy})\Big(I\frac{V_\ell}{(1+x)^\ell}\Big) &= (V_\kappa \triangle_{\bfy}  V_\ell) I \bigg[ \frac{1}{(1+x)^\kappa} I^2 \frac{1}{(1+x)^\ell}  \bigg] .  \label{eq:annoying_2}
\end{align}
We also have two more contributions to $\calS 1$ arising from swapping $\kappa\leftrightarrow \ell$ in the two expressions above. We will show that the sum of the two terms above are log-free. Then, swapping $\kappa \leftrightarrow \ell$ shows that the sum of the remaining two terms is also log-free.

If $\kappa=2,\ell=2$, then both terms above vanish, because $V_\kappa,V_\ell$ are a constant and hence killed by $\triangle_{\bfy}$. 
When $\kappa,\ell \geq 3$, neither of the functions of $x$ above have logarithms, because the integrands have sufficient decay. 
If $\ell=2$ and $\kappa\geq 3$, then \cref{eq:annoying} is log-free, and \cref{eq:annoying_2} vanishes because $\triangle_{\bfy} V_\ell=0$. 
If $\kappa=2$ and $\ell\geq 3$, then we have
\begin{align}
    I^2 \bigg[ \frac{1}{(1+x)^2} I \frac{1}{(1+x)^\ell}  \bigg] &= - \frac{1}{\ell-1} \log (1+x)+\text{rational},   \\
    I \bigg[ \frac{1}{(1+x)^\kappa} I^2 \frac{1}{(1+x)^\ell}  \bigg] &= \frac{1}{\ell-1} \log (1+x)+\text{rational},
\end{align} 
where ``rational'' means rational in $x$ for each fixed $\bfy$.  Their logarithmic contributions to $\calS 1$ cancel (because $\triangle_{\bfy} (V_\kappa V_\ell) = V_\kappa \triangle_{\bfy} V_\ell$ when $V_\kappa=V_2$ is constant).

So, the individual terms in $\calS 1$ can contain logarithmic terms. But the fully symmetrized sum $\calS 1$ has its log term proportional to the \emph{difference} 
\begin{equation} 
    V_2 \triangle_{\bfy} V_\ell - \triangle_{\bfy} (V_2V_\ell) \in C^\infty(\bbR^{d-1}_{\bfy}),
\end{equation} 
which vanishes identically if $V_2$ is constant, regardless of $V_\ell$. Thus, the claim above holds true, in this particular example.

Now consider the remaining terms in \cref{eq:S_ex}. The first-order terms $\propto \partial_{y_j}$ can be handled by an argument virtually identical to that above, so we omit the details. 
For the second-order terms, $\propto \triangle_{\bfy}$, the cancellation is more intricate, when $\kappa=2$ or $\ell=2$. When $\kappa=2$ and $\ell \geq 3$, we have, modulo terms with fewer derivatives,
\begin{align} 
\begin{split}
    (I\triangle_{\bfy})\Big(I\frac{1}{(1+x)^\kappa}\Big)\Big(I\frac{1}{(1+x)^\ell}\Big) &= \Big(- \frac{1}{\ell-1} \log(1+x)+\text{rational} \Big)\triangle_{\bfy} \\
    \Big(I\frac{1}{(1+x)^\kappa}\Big)(I\triangle_{\bfy})\Big(I\frac{1}{(1+x)^\ell}\Big) &= \Big( \frac{1}{\ell-1} \log (1+x) + \text{rational} \Big) \triangle_{\bfy} \\
    \Big(I\frac{1}{(1+x)^\ell}\Big)(I\triangle_{\bfy})\Big(I\frac{1}{(1+x)^\kappa}\Big) &= \Big( \frac{(1+x)^{1-\ell}}{\ell-1}\log(1+x)+\text{rational} \Big)\triangle_{\bfy} \\
    (I\triangle_{\bfy})\Big(I\frac{1}{(1+x)^\ell}\Big)\Big(I\frac{1}{(1+x)^\kappa}\Big) &= (\text{rational})\triangle_{\bfy}\\
    \Big(I\frac{1}{(1+x)^\ell}\Big)\Big(I\frac{1}{(1+x)^\kappa}\Big) (I\triangle_{\bfy}) &= \Big( -\frac{(1+x)^{1-\ell}}{\ell-1}
\log(1+x)+\text{rational}\Big)\triangle_{\bfy} \\
    \Big(I\frac{1}{(1+x)^\kappa}\Big) \Big(I\frac{1}{(1+x)^\ell}\Big) (I\triangle_{\bfy}) &= (\text{rational})\triangle_{\bfy}.
\end{split}
\end{align} 
Adding these together, the log terms cancel. The $\ell=2$ case is similar. 
This completes our discussion of the case of $k=3$ factors.

Proceeding similarly, one can check the italicized claim above for small values of $k$. Each individual $k$ takes finite time. In this way, we algorithmically verified the claim above for all $k\leq 6$, with the assistance of computer algebra software. This served as convincing evidence of the claim, before we found the inductive argument presented in the next subsection. 
Unfortunately, the algebra becomes increasingly arduous for $k\geq 4$, so we cannot present it here. The argument in the next subsection is presented not as a proof of the italicized claim above but as a proof of \Cref{lem:logs}. The two are essentially equivalent, so a proof of one can be converted into a proof of the other.

\subsection{Proof of \Cref{lem:logs}}

We will say that an element of $C^\infty(\bbR_x; \operatorname{Diff}(\bbR^{d-1}))$ is ``log-free'' if the coefficients are of the form $C^\infty(U)[x]+C^\infty([0,1)_{1/x}\times U)$ for each open $U\subset \bbR^{d-1}_{\bfy}$. 

First, we reduce to the special case where the $E_\bullet$ have no $x$-derivatives:
\begin{lemma}
	Fix $k\in \bbN$, and suppose that
	\begin{align} 
	E_1,\dots,E_k \in &\{ \partial_y^\alpha:\text{nonzero multi-index }\alpha \}  \\ &\cup \frac{1}{1+x^2+\lVert \bfy \rVert^2}\operatorname{span}_{C^\infty(\overline{\bbR^d})} \{\partial_y^\alpha:\alpha\text{ a multi-index}\}  .
	\end{align} 
    Then, $\calS$ (see above) is log-free.
	\label{lem:logs_baby}
\end{lemma}

\begin{proof}[Reduction of \Cref{lem:logs} to \Cref{lem:logs_baby}]
    Consider the setup of \Cref{lem:logs}. We proceed inductively on the total number $\#$ of $x$-derivatives in the $E_\bullet$. An $E_\bullet$ containing an $x$-derivative has the form $\Upsilon \partial_x^\nu \partial_y^\alpha$ for $\nu\geq 1$, where 
    \begin{equation}
        \Upsilon=1\text{ or }\Upsilon \in \frac{1}{1+x^2+\lVert \bfy \rVert^2} C^\infty(\overline{\bbR^d}).
    \end{equation}
    The basic inductive idea is that a $\partial_x$ cancels with one $I$, reducing $\#$.
    
    If $\#=0$, then the desired conclusion is immediate from \Cref{lem:logs_baby}. 

    Suppose $\#\geq 1$, and that we have proven the lemma for all smaller values of $\#$. 
    Pick one particular $E_\kappa = \Upsilon \partial_x^\nu \partial_y^\alpha$. 
    Without loss of generality, we can assume $\kappa=1$. 
    Note that $E_1$ kills any function of $\bfy$ alone, so, an individual summand $\prod_\bullet (I E_\bullet)$ in which $E_1$ appears in the right-most slot is zero (considered as an element of $C^\infty(\bbR_x; \operatorname{Diff}(\bbR^{d-1}_{\bfy}))$). Otherwise, $E_1$ appears immediately to the left of an $I E_\ell$, for some $\ell\neq 1$, and the fundamental theorem of calculus gives  
    \begin{equation}
        E_1 I E_\ell = \Upsilon \partial_x^{\nu-1} \partial_y^\alpha E_\ell. 
    \end{equation}
    Thus, 
    \begin{equation} 
        \calS= \sum_{\ell=2}^k \calS[\ell],
    \end{equation} 
    where $\calS[\ell]$ is the symmetrized product with $k-1$ terms in which $E_1$ has been excluded and $E_\ell$ has been replaced by  $\Upsilon \partial_x^{\nu-1} \partial_y^\alpha E_\ell$. If $E_\ell\in \{\partial_x,\partial_{y_j},\cdots\}$ is a constant-coefficient building block, then  $\Upsilon \partial_x^{\nu-1} \partial_y^\alpha E_\ell$ has the form of an allowed $E_\bullet$. Thus, 
    $\calS[\ell]$ has the form covered by \Cref{lem:logs} for $\#-1$ in place of $\#$. Otherwise, if 
    \begin{equation} 
        E_\ell =\digamma \partial_x^\mu \partial_y^\beta 
    \end{equation} 
    for some $\digamma \in \langle r \rangle^{-2} C^\infty(\overline{\bbR^d})$, then, 
    \begin{equation}
        \Upsilon \partial_x^{\nu-1} \partial_y^\alpha E_\ell=\Upsilon \partial_x^{\nu-1} \partial_y^\alpha (\digamma \partial_x^\mu \partial_y^\beta) \in \operatorname{span}_\bbC\{ \Upsilon (\partial_x^\tau \partial_y^\gamma \digamma) \partial_x^{\mu+ \nu-\tau-1} \partial_y^{\beta+\alpha-\gamma} \}
    \end{equation}
    is sum of allowed building blocks. Consequently, $\calS[\ell]$ is a sum of symmetrized products of the form covered by \Cref{lem:logs} with $\#-1,\#-2,\cdots$ in place of $\#$. 
    The inductive hypothesis therefore yields that $\calS$ is log free.
\end{proof}

\begin{proof}[Proof of \Cref{lem:logs_baby}]
    We proceed by induction on the number
    $k$ of factors. If $k=1$, then the desired result is trivial. Suppose $k\geq 2$, and that we have proven the result for smaller values of $k$.

    If every $E_\bullet$ is a constant-coefficient building block, then $\calS$ has coefficients which are polynomials in $x$, so the desired result holds in this case. Thus, suppose among the $E_\bullet$ there is one of the form $\Upsilon \partial_y^\alpha$ for 
    \begin{equation}    
        \Upsilon \in \langle r \rangle^{-2} C^\infty(\overline{\bbR^d}).
    \end{equation} 
    Call this $Q$.

    Examining $\calS$, we integrate-by-parts in each occurrence of $IQ I$,  using
	\begin{equation}
	I Q I \bullet = -\int_0^x \frac{\mathrm{d}}{\mathrm{d} s} \Big( \int_s^{\infty} Q(t,-) \dd t \Big) I \bullet \dd s = - \bigg( \int_x^\infty Q(t,-)\dd t  \bigg) (I\bullet(x))   + \int_0^x \Big(\int_s^\infty Q(t,-) \dd t \Big) \bullet  \dd s ,   
	\end{equation}
	which holds because $Q$ only involves $y$-derivatives.
    Here, $\bullet$ stands for an arbitrary element of $C^\infty(\bbR_x; \operatorname{Diff}(\bbR^{d-1}_{\bfy}))$.
	(The reason why only one boundary term arises when integrating-by-parts is that $I\bullet$ vanishes at $0$, since whatever goes in the $\bullet$ is smooth at $0$). 
	Let $B$ denote integration in the backwards direction from $x=\infty$ (along level sets of $\bfy$). The formula above reads 
	\begin{equation}
	I Q I \bullet =  -  (B Q(x) ) I \bullet +I (B Q) \bullet = [I,BQ] \bullet.  
	\end{equation}
    Here $BQ\in C^\infty(\bbR_x; \operatorname{Diff}(\bbR^{d-1}_{\bfy}))$ just denotes the operator that we get from integrating the coefficients of $Q$ backwards from $x=\infty$. The integral converges because the coefficients of $Q$ are short-range.

    A similar identity handles the case where $IQ$ is last:
    \begin{equation}
        IQ = -\int_0^x \frac{\mathrm{d}}{\mathrm{d} s} \Big( \int_s^{\infty} Q(t,-) \dd t \Big) \dd s =  \underbrace{\int_0^\infty Q(t,-)\dd t  }_{\in \operatorname{Diff}(\bbR^{d-1})}-  \underbrace{\int_x^\infty Q(t,-)\dd t }_{BQ} 
    \end{equation}
    as an identity holding in $C^\infty(\bbR_x; \operatorname{Diff}(\bbR^{d-1}_{\bfy}))$. That is, 
    \begin{equation}
        IQf = \int_0^\infty Q(t,-)\dd t  f - (BQ) f 
    \end{equation}
    for any $f\in C^\infty(\bbR^{d-1}_{\bfy})$. 
    
    Note that $E_\kappa BQ I$, $I(BQ)E_\kappa$ have only one $I$, whereas $IQI$ has two. This is the basic strategy for the induction, reducing the problem to one about a symmetrized product with one fewer factors.
    However, $E_\kappa BQ,BQ E_\kappa$ are not linear combinations of admissible building blocks for two reasons: 
    \begin{enumerate}[label=(\roman*)]
        \item it is poorly behaved in the backwards direction, 
        \item if $Q$ has a $O(1/r^2)$ term, then $BQ$ is long-range in the forward direction.
    \end{enumerate}
    The second issue is fixed by noting that, in the full sum $\calS$, a telescoping phenomenon occurs, replacing $E_\kappa BQ$ with a commutator $[E_\kappa,BQ]$. This has an extra order of decay. The first issue is fixed by noting that $[E_\kappa,BQ]$ agrees with a linear combination of admissible building blocks in $\{x>-1\}$. Indeed, the coefficients of $BQ$ are smooth on $\overline{\bbR^d}$ away from the backwards direction (by \S\ref{subsec:same2}, switching the forward/backward direction), but lie in $\langle r \rangle^{-1} C^\infty(\overline{\bbR^d})$, not $\langle r \rangle^{-2} C^\infty(\overline{\bbR^d})$, because integration reduces the decay rate by one (as explained in \S\ref{subsec:same2}). 
    If $E_\kappa$ is short-range, then $E_\kappa BQ$, $BQE_\kappa$ both have coefficients in $\langle r \rangle^{-3} C^\infty(\overline{\bbR^d})$, picking up the extra decay from $E_\kappa$. If $E_\kappa=\partial_y^\alpha$, $\alpha\neq 0$, is a constant coefficient building block, then the commutator $[E_\kappa,BQ]$ is a linear combination of $B\tilde{Q}$ for $\tilde{Q}$ the result of differentiating the coefficients of $Q$ at least once. Hence, $\tilde{Q}$ has coefficients in $\langle r \rangle^{-3} C^\infty(\overline{\bbR^d})$, which means that $B\tilde{Q}$ is short-range, at least in $\{x>-1\}$. 
    The behavior in $\{x\leq -1\}$ is irrelevant to the truth of the lemma.
    More precisely, fix
$\chi_{\rightarrow}\in C^\infty(\overline{\bbR^d})$ which is identically one on $\{x\geq 0\}$ and vanishes in a neighborhood of the backward direction. 
Multiplying the coefficients of the locally defined commutators by $\chi_{\rightarrow}$ and extending across the backward direction produces globally admissible building blocks that agree with the original operators in the region used below. 
Terms in which derivatives fall on $\chi_{\rightarrow}$ are supported away from
the forward face and therefore do not affect the log-free expansion.

    Before presenting the details, we illustrate the telescoping phenomenon.
    Consider a term of the form $I Q I A I \bullet  + I A I Q I \bullet$.
	The summands individually evaluate to 
	\begin{align}
	\begin{split} 
	I Q I A  I\bullet  &=- (B Q) I A I \bullet   + I (BQ) A I \bullet \\
	I A I Q  I\bullet  &= -I A (B Q)  I\bullet + I A   I (B Q)\bullet. 
	\end{split} 
	\end{align}
	Summing these, we get: 
    \begin{equation}
        I Q I A I \bullet  + I A I Q I \bullet  = I [(BQ),A]  I\bullet -  (B Q) I A I \bullet + I A   I (B Q)\bullet.
    \end{equation}
    The key point is the appearance of the commutator $[(BQ),A]$.

    Now we present the details.
    Concretely, in $\calS$, we look at the terms in which all the other factors besides $IQ$ appear in a fixed relative order. There are $k$ such terms,  differing in the location where $IQ$ is inserted.
    Assuming for notational simplicity that $Q=E_k$, and relabeling the $E_\bullet$'s if necessary, the sum in question is 
    \begin{align}
    \begin{split}
        \sum_{\kappa=1}^k \Big(\prod_{\varkappa=1}^{\kappa-1} I E_\varkappa \Big)  (IQ) \Big(\prod_{\varkappa=\kappa}^{k-1} IE_\varkappa\Big) &= \Big(\prod_{\varkappa=1}^{k-1} I E_\varkappa \Big)  (IQ)+  \sum_{\kappa=1}^{k-1} \Big(\prod_{\varkappa=1}^{\kappa-1} I E_\varkappa \Big)  (IQ) \Big(\prod_{\varkappa=\kappa}^{k-1} IE_\varkappa\Big)  \\
        &= \Big(\prod_{\varkappa=1}^{k-1} I E_\varkappa \Big)  (IQ)+  \sum_{\kappa=1}^{k-1} \Big(\prod_{\varkappa=1}^{\kappa-1} I E_\varkappa \Big)  (IQ I) \Big(\prod_{\varkappa=\kappa}^{k-2} E_\varkappa I \Big) E_{k-1} .
        \end{split} 
    \end{align}
    The full sum $\calS$ is the result of symmetrizing this with respect to all permutations of $E_1,\cdots,E_{k-1}$. 
    So, 
    \begin{multline}\label{eq:misc_E5}
        \hspace{5em}\calS = \Big(\prod_{\varkappa=1}^{k-1} I E_\varkappa +\text{permutations}\Big)  (IQ)\\ 
        + \sum_\sigma  \sum_{\kappa=1}^{k-1} \Big(\prod_{\varkappa=1}^{\kappa-1} I E_\varkappa^\sigma \Big)  (IQ I) \Big(\prod_{\varkappa=\kappa}^{k-2} E_\varkappa^\sigma I \Big) E_{k-1}^\sigma , \hspace{5em}
    \end{multline}
    where $E_\nu^\sigma = E_{\sigma(\nu)}$. The sum is over all $\sigma \in \frakS_{k-1}$.

    Using our endpoint integration-by-parts identity, the first term is 
    \begin{equation}\label{eq:misc_E6}
        \Big(\prod_{\varkappa=1}^{k-1} I E_\varkappa +\text{permutations}\Big)  \int_0^\infty Q(t,-)\dd t  - \Big(\prod_{\varkappa=1}^{k-1} I E_\varkappa +\text{permutations}\Big)  BQ. 
    \end{equation}
    On the other hand, using our main 
   integration-by-parts identity, 
    \begin{multline}
        \Big(\prod_{\varkappa=1}^{\kappa-1} I E_\varkappa^\sigma \Big)  (IQ I) \Big(\prod_{\varkappa=\kappa}^{k-2} E_\varkappa^\sigma I \Big) E_{k-1}^\sigma = \Big(\prod_{\varkappa=1}^{\kappa-1} I E_\varkappa^\sigma \Big)  I BQ \Big(\prod_{\varkappa=\kappa}^{k-2} E_\varkappa^\sigma I \Big) E_{k-1}^\sigma \\ - \Big(\prod_{\varkappa=1}^{\kappa-1} I E_\varkappa^\sigma \Big)  BQ\underbrace{I \Big(\prod_{\varkappa=\kappa}^{k-2} E_\varkappa^\sigma I \Big) E_{k-1}^\sigma }_{ =\prod_{\varkappa=\kappa}^{k-1} I E_\varkappa^\sigma } . 
    \end{multline}
    Now when we sum over the insertion point $\kappa$, we get 
    \begin{multline}
        \sum_{\kappa=1}^{k-1} \Big(\prod_{\varkappa=1}^{\kappa-1} I E_\varkappa^\sigma \Big)  (IQ I) \Big(\prod_{\varkappa=\kappa}^{k-2} E_\varkappa^\sigma I \Big) E_{k-1}^\sigma  = - BQ  \Big( \prod_{\varkappa=1}^{k-1} IE_\varkappa^\sigma  \Big) 
        \\ 
        + \sum_{\kappa=1}^{k-2}  \Big( \prod_{\varkappa=1}^{\kappa-1} IE_\varkappa^\sigma \Big) I[(BQ),E_\kappa^\sigma] \Big( \prod_{\varkappa=\kappa+1}^{k-1} I E_{\varkappa}^\sigma  \Big)\\ 
         + \Big( \prod_{\varkappa=1}^{k-2} I E_\varkappa^\sigma \Big) I(BQ) E_{k-1}^\sigma 
    \end{multline}
    In total, 
    \begin{multline}
        \calS =  -\overbrace{ BQ  \Big( \prod_{\varkappa=1}^{k-1} IE_\varkappa^\sigma +\text{permutations} \Big)}^{\mathrm{I}} +\overbrace{\Big(\prod_{\varkappa=1}^{k-1} I E_\varkappa +\text{permutations}\Big)  \int_0^\infty Q(t,-)\dd t}^{\mathrm{II}} \\
        +  \underbrace{\sum_\sigma \sum_{\kappa=1}^{k-1}  \Big( \prod_{\varkappa=1}^{\kappa-1} IE_\varkappa^\sigma \Big) I[(BQ),E_\kappa^\sigma] \Big( \prod_{\varkappa=\kappa+1}^{k-1} I E_{\varkappa}^\sigma  \Big)}_{\mathrm{III}}.
    \end{multline}
    Let us now discuss each of the terms labeled $\mathrm{I}$, $\mathrm{II}$, $\mathrm{III}$,  above. 

    The easiest is term $\mathrm{I}$, because the parenthetical factor to the right of $BQ$ is covered by inductive hypothesis. It is therefore log-free. Since $BQ$ is log-free, the same applies to the composition $\mathrm{I}$. 
    Likewise, the parenthetical factor in term $\mathrm{II}$, to the left of $\int_0^\infty Q(t,-)\dd t$, is log-free by the inductive hypothesis. Since $\int_0^\infty Q(t,-) \dd t$ is just a differential operator on $\bbR^{d-1}_{\bfy}$ (with no $x$ dependence), it follows that term $\mathrm{II}$ is log-free.

    We break up $\mathrm{III}$ into $k-1$ components based on the value of the index $\nu$ labeling which $E_\nu$ the operator $BQ$ is commuted with.
    Then, $\mathrm{III}$ is the sum over $\nu$ of the symmetrized product of the factors 
    \begin{equation}
        \{ I [(BQ),E_\nu] \} \cup \{ I E_1,\dots,I E_{k-1} \} \backslash \{I E_\nu\} 
    \end{equation}
    Now we use that $[(BQ),E_\nu]$ is a linear combination of admissible building blocks in $\{x>-1\}$. Thus, term $\mathrm{III}$ is, in the region $\{x>0\}$, a linear combination of symmetrized products covered by the lemma, with $Q=E_k$ removed and $E_\nu$ replaced by some other admissible building block. Thus, the inductive hypothesis applies to tell us that these individual terms are all log-free, hence $\mathrm{III}$ is as well. 
    \end{proof}

\section{Some theory of the quantum inverted harmonic oscillator (QIHO)}
\label{sec:QHO}
In this appendix we discuss some theory about the spectral family of the QIHO: 
\begin{equation}
	L(j) = \triangle_{\hat{\bfy}} - \frac{\hat{y}^2}{4} + i \Big(j-\frac{d-1}{2} \Big) \in \operatorname{Diff}^2(\bbR^{d-1}_{\hat{\bfy}}),
\end{equation}
where $j \in \bbC$ (the other notational conventions are carried from \S\ref{sec:ff}.) 
Our main task is to construct the two partial right-inverses $R_\pm(j)$ described in \S\ref{subsec:indicial}, each satisfying 
\begin{equation} 
    L(j) R_\pm(j) = 1
\end{equation} 
on their domain,
but differing in terms of the sort of oscillations produced at large-$\hat{y}$ when fed $C_{\mathrm{c}}^\infty$ input $f$: 
\begin{equation} 
	R_\pm(j) f\sim e^{\mp i\hat{y}^2/4} \times \operatorname{polyhomogeneous}.
\end{equation}

This appendix has four subsections. In \S\ref{subsec:QHO_separation}, we solve $L(j)u=f$ via separation of variables, utilizing the spherical symmetry of $L(j)$. This produces an explicit description of $R_\pm(j)$ in terms of special functions, making it straightforward to verify the mapping property 
\begin{equation}
R_-(j): e^{i\hat{y}^2/4} \calA^{( (d+1-j)/2,0 ) }(\mathrm{ff})\supseteq  \calD(R_-(j)) \to e^{i\hat{y}^2/4} \calA^{ ( (d-1-j)/2,0) } (\mathrm{ff}) 
\label{eq:misc_163}
\end{equation} 
(\Cref{prop:QHO_mapping}) when only finitely many angular modes are allowed. See \Cref{prop:separated_mapping}. For simplicity, we discuss only generic $j$. 
Since this subsection is incidental to our main line of argumentation, we omit full details.
Nevertheless, this subsection serves as an elementary crosscheck for the following subsections. 

In the first of these, \S\ref{subsec:QHO_rig}, we relate $L(j)$ to the spectral family of a Schr\"odinger operator (with complex potential) whose potential is \emph{decaying} at infinity. This enables us to cite general facts about such operators from the literature on geometric scattering theory; see \S\ref{subsec:QIHO_Fredholm}. This is how we prove \Cref{prop:QHO_mapping} in full generality --- see \S\ref{subsec:QIHO_mapping}, which upgrades the results in the preceding subsection using a standard argument.

\subsection{Separation of variables}
\label{subsec:QHO_separation}
Let $Y\in C^\infty(\bbS^{d-2})$ denote a spherical harmonic with 
\begin{equation} 
    \triangle_{\bbS^{d-2}} Y = \lambda Y.
\end{equation} 
If $d=2$, then there are two harmonics, $Y_+$ (the even one), and $Y_-$ (the odd one).
The eigenvalues $\lambda\geq 0$ all have the form 
\begin{equation} 
\lambda = \ell(\ell+d-3)\label{eq:azimuthal}
\end{equation}
\cite{SimonHarmonic}, where $\ell\in \bbN$ is the ``azimuthal quantum number:''  
\begin{itemize}
    \item when $d\geq 3$, the azimuthal quantum number can be any nonnegative integer.
    \item  when $d=2$, then the azimuthal quantum number of $Y_+$ is $\ell=0$ and the azimuthal quantum number of $Y_-$ is $\ell=1$.
\end{itemize} 
For example, when $d=4$, then $Y\in C^\infty(\bbS^{2})$ is an ordinary spherical harmonic, and the possible eigenvalues are given by the familiar formula $\lambda=\ell(\ell+1)$.

Then, if
\begin{equation} 
	u(\hat{\bfy})= \hat{y}^{-(d-2)/2}v(\hat{y}) Y(\hat{\bfy}/\hat{y}) \quad \text{ and }\quad 
	f(\hat{\bfy}) = \hat{y}^{-(d-2)/2} g(\hat{y})Y(\hat{\bfy}/\hat{y}),
	\label{eq:misc_1021103}
\end{equation} 
the PDE $L(j)u=f $, on $\bbR^{d-1}_{\hat{\bfy}}\backslash \{0\}$, is equivalent to the inhomogeneous ODE 
\begin{equation}
	v''(s)+ \Big[ \frac{s^2}{4}+i\Big( \frac{d-1}{2}-j \Big) -\Big(\lambda+\frac{(d-2)(d-4)}{4}\Big)\frac{1}{s^2}\Big] v(s) = -g(s).
	\label{eq:QHO_inh}
\end{equation}
(The $\hat{y}^{-(d-2)/2}$ factors in \cref{eq:misc_1021103} are for later convenience.)

The homogeneous problem is when $g=0$. The inhomogeneous problem can be solved using one of this ODE's Green functions, which are built out of the solutions of the homogeneous problem.
Indeed, whenever $v_0,v_\infty$ are linearly independent solutions of the homogeneous problem, then, letting $\frakW=v_\infty(s)v_0'(s)-v_\infty'(s)v_0(s)\in \bbC^\times$ denote their Wronskian, 
\begin{equation}
	G[\lambda](s,s') = \frac{1}{\frakW} \begin{cases}
		v_\infty(s) v_0(s') & (s>s'), \\ 
		v_\infty(s') v_0(s) & (s<s'), 
	\end{cases}
	\label{eq:Green}
\end{equation}
satisfies the inhomogeneous ODE with forcing  $g(s)=\delta(s-s')$, hence is a Green function.
When $G(s,s')$ is integrated against reasonable $g(s')$ (say $g\in C_{\mathrm{c}}^\infty(\bbR^+)$), the result 
\begin{equation}
	v=\int_0^\infty G[\lambda](s,s') g(s')\dd s'
\end{equation}
is a solution to the inhomogeneous ODE \cref{eq:QHO_inh} with forcing $g(s)$. It is the unique solution which behaves like $v_\infty(s)$ as $s\to\infty$ and behaves like $v_0(s)$ as $s\to 0^+$. When $g\in C_{\mathrm{c}}^\infty(\bbR^+)$, this means $v\propto v_\bullet$ in the specified regions. 

For each choice of independent $v_\infty,v_0$, we get a right inverse $R[\lambda]$ to $L(j)$ acting on $Y$-modes; the Schwartz kernel is
\begin{equation}
	R[\lambda](\hat{y},\hat{y}') = \frac{ (\hat{y}'/\hat{y})^{(d-2)/2}}{\frakW} \begin{cases}
		v_{\infty}(\hat{y}) v_0(\hat{y}') & (\hat{y}>\hat{y}'),\\
		v_{\infty}(\hat{y}') v_0(\hat{y}) & (\hat{y}<\hat{y}').
	\end{cases}
\end{equation}

\subsubsection{Solution in terms of the confluent hypergeometric equations}

The homogeneous ODE 
\begin{equation}
	v''(s)+ \Big[ \frac{s^2}{4}+i\Big( \frac{d-1}{2}-j \Big) -\Big(\lambda+\frac{(d-2)(d-4)}{4}\Big)\frac{1}{s^2}\Big] v(s) = 0
	\tag{$\star$}
\end{equation}
is a conjugated form of the confluent hypergeometric equation. Letting
\begin{align}
	\begin{split} 
	\alpha &= \ell + \frac{d-3}{2}, \\
	a&=(d+1)/4-j/2+\alpha/2 = (d-1)/2-j/2+\ell/2 ,
	\end{split} 
\end{align}
the solution space $\calS=\{v\in C^\infty(\bbR^+_s) :\text{($\star$) holds}\}$ of the ODE ($\star$) is given by:
\begin{equation}
	\calS = \Big\{ e^{is^2/4} s^{\alpha+1/2} w\Big(-\frac{is^2}{2}\Big):\underbrace{zw''(z)+(b-z) w'(z)-aw(z)=0}_{\text{confluent hypergeometric equation with }b=1+\alpha} \Big\}.
\end{equation}

One special solution of the confluent hypergeometric equation is \emph{Tricomi's confluent hypergeometric function} $U(a,b,z)$ \cite[\href{http://dlmf.nist.gov/13.2.i}{\S13.2.1}]{NIST}. Two linearly independent solutions are  $U(a,b,z)$ and $e^z U(b-a,b,e^{-\pi i}z)$ \cite[\href{http://dlmf.nist.gov/13.2.E24}{13.2.24}]{NIST}. So, 
\begin{equation}
	\calS = s^{\alpha+1/2} \operatorname{span}_\bbC\Big\{  e^{is^2/4}U\Big(a,1+\alpha,-\frac{is^2}{2}\Big),  e^{-is^2/4}U\Big(1+\alpha-a,1+\alpha,\frac{is^2}{2}\Big)\Big\}
	\label{eq:misc_1023113}
\end{equation}

Another useful solution of the confluent hypergeometric equation is Kummer's ${}_1F_1(a,b,z)$. This is also commonly denoted $M(a,b,z)$ \cite{NIST}; the weighted version \begin{equation} 
	\bfM(a,b,z)=\Gamma(b)^{-1} M(a,b,z)
\end{equation} 
is also used \cite[\href{http://dlmf.nist.gov/13.2.E4}{13.2.4}]{NIST}. Unlike $U(a,b,z)$, the $M$-function is only defined for $b\neq 0,-1,-2,\cdots$, but since we have $b=1+\alpha \geq 1/2$, this issue does not arise here. The regularized function $\bfM(a,b,z)$, however, is defined for all $b$.

So, 
\begin{equation}
	e^{is^2/4} s^{\alpha+1/2} {}_1 F_1\Big(a,1+\alpha,-\frac{is^2}{2}\Big) \in \calS. 
	\label{eq:misc_1023115}
\end{equation}
It turns out that $z^{1-b} \bfM(a-b+1,2-b,z)$ satisfies the same confluent hypergeometric equation as ${}_1F_1(a,b,z)$, and consequently 
\begin{equation}
	e^{is^2/4} s^{-\alpha+1/2} \bfM\Big(a-\alpha,1-\alpha,-\frac{is^2}{2} \Big) \in \calS
	\label{eq:misc_1023116}
\end{equation}
as well; note that $1-\alpha$ might lie in $-\bbN$, necessitating the use of $\bfM(\bullet,1-\alpha,z)$ instead of ${}_1F_1(\bullet,1-\alpha,z)$, which might not be defined.

Since $\calS$ is two-dimensional, 
there must be linear dependencies among the solutions written above. 
We will discuss this more below.

\subsubsection{Recessive, incoming, and outgoing solutions}\label{subsec:odesolutions}
 
Three one-dimensional subspaces of $\calS$ are of interest. These are the spans of:

\begin{itemize}
	\item The \emph{outgoing solution}, which is formed from the Tricomi $U$-function:
	\begin{equation}
		v_{\infty,-}(s) = e^{is^2/4} s^{\alpha+1/2}  U\Big(   a,1+\alpha, - \frac{is^2}{2} \Big).
	\end{equation}
	So, $v_{\infty,-}$ oscillates like $\sim e^{is^2/4} s^{-d/2 +j}$ as $s\to\infty$.  The defining feature of $U(a,b,z)$ is that $U(a,b,z)\sim z^{-a}$ as $z\to\infty$ \cite[\href{http://dlmf.nist.gov/13.2.E6}{13.2.6}]{NIST}. (This includes the $z\to \pm i\infty$ limit.) 
	\item  
	The \emph{incoming solution}, which is the other solution in \cref{eq:misc_1023113}: 
	\begin{align}
		\begin{split} 
		v_{\infty,+} &= e^{-is^2/4}s^{\alpha+1/2}U\Big(1+\alpha-a,1+\alpha,\frac{is^2}{2}\Big) 
		\end{split} 
	\end{align} 
	So, $v_{\infty,+}$ oscillates like: 
    \begin{equation} 
        v_{\infty,+}\sim e^{-is^2/4} s^{(d-2)/2 -j}
        \label{eq:outgoing_osc}
    \end{equation} 
    as $s\to\infty$.

	\Cref{eq:misc_1023113} makes manifest that solutions $v(s)$ of the ODE can have two sorts of oscillations present in their  large-$s$ asymptotics: 
	\begin{equation} 
		e^{ is^2/4} s^{-d/2+j},\quad e^{-is^2/4} s^{(d-2)/2-j}.
	\end{equation} 
	What makes $v_{\infty,\pm}$ (and their multiples) special is that they have only one sort of oscillation present.
	\item The \emph{recessive} solution is the solution that is better behaved as $\hat{y}\to 0^+$. It is built out of ${}_1F_1$: 
	\begin{equation}
		v_0(s) = e^{is^2/4} s^{\alpha+1/2} {}_1F_1\Big(a,1+\alpha ,-\frac{is^2}{2}\Big) .
	\end{equation}
	The feature of ${}_1F_1$ distinguishing it from the other solutions of the confluent hypergeometric equation is that it is analytic at $z=0$, and ${}_1F_1(-,-,0)=1$ \cite[\href{http://dlmf.nist.gov/13.2.E2}{13.2.2}]{NIST}. So, 
	\begin{equation} 
	v_0(s)= e^{is^2/4} s^{\alpha+1/2} (1 +O(s^2))
	\label{eq:misc_1020109}
	\end{equation} 
	as $s\to 0^+$. In fact, $v_0(s) \in e^{is^2/4} s^{\alpha+1/2} C^\infty([0,\infty)_{s^2})$, via the analyticity of the ${}_1F_1$ function in its last argument.
	
	Let us contrast this with the other solutions in $\calS$. First consider $e^{is^2/4}s^{-\alpha+1/2} \bfM(a-\alpha,1-\alpha,-is^2/2)$. If $\alpha=0$, this is just $v_0(s)$, so suppose $\alpha>0$. (We will return to $\alpha=0$ momentarily. The case when $d=2$ and $\alpha=-1/2$ requires separate discussion.) $\bfM(a,b,0) = \Gamma(b)^{-1}$ (by definition). So, as long as $1-\alpha\notin -\bbN$, 
	\begin{equation}
		e^{is^2/4}s^{-\alpha+1/2} \bfM(a-\alpha,1-\alpha,-is^2/2) = \Omega(s^{-\alpha+1/2})
	\end{equation}
	as $s\to 0^+$. This is $s^{-2\alpha}$ worse than $v_0(s)$.
	
	On the other hand, for $a \notin -\bbN$, the Tricomi-$U$ function satisfies 
	\begin{equation} 
		U(a,1+\alpha,z)= (1+o(1))\Gamma(a)^{-1}\begin{cases}
			\Gamma(\alpha) z^{-\alpha} & (\alpha> 0) \\
			-\log z & (\alpha=0)
		\end{cases} 
	\end{equation} 
	as $z\to 0$ in $\bbC \setminus (-\infty,0]$ \cite[\href{http://dlmf.nist.gov/13.2.iii}{\S13.2.iii}]{NIST}. So, unless $a\in -\bbN$, then  
	\begin{equation} 
		v_{\infty,-}= \Omega(s^{-\alpha+1/2} )
	\end{equation} 
	if $\alpha>0$ and  $\Omega(s^{1/2}\log |z|)$ if $\alpha=0$. 
	In either case, $v_{\infty,-}$ is worse as $s\to 0^+$ than $v_0(s)$, by at least a $\log z$. 
	The same statements apply to $v_{\infty,+}$ under the condition $1+\alpha-a\notin -\bbN$, instead of $a\notin -\bbN$. Note that $a\in -\bbN\Rightarrow 1+\alpha-a\notin -\bbN$. So, we can conclude that, in all cases relevant to us, $v_0(s)$ is less singular at $s=0$ than a generic element of $\calS$.

	The same conclusion can be reached with less casework by noting that the indicial roots of ($\star$) are the $c$ such that  
	\begin{equation}
		c(c-1) - \Big(\lambda + \frac{(d-2)(d-4)}{4}\Big) = 0,
	\end{equation}
	i.e.\ 
	\begin{equation}
		c = \frac{1\pm \sqrt{(d-3)^2+4\lambda}}{2} = \frac{1}{2} \pm| \alpha|.  
	\end{equation}
\end{itemize}

\emph{The recessive solution $v_0$ is not always linearly independent of $v_{\infty,-}$.} We saw that $v_{\infty,-}(s)$ was more singular than $v_0(s)$ as $s\to 0^+$ when $a\notin -\bbN$; this implies linear independence. But, when $a\in -\bbN$, 
\begin{equation}
	U(a,b,z) = (-1)^{-a} \frac{\Gamma(b-a)}{\Gamma(b)} \cdot {}_1F_1(a,b,z)
\end{equation}
\cite[\href{http://dlmf.nist.gov/13.2.E7}{13.2.7}]{NIST}.
So, $v_{\infty,-},v_0$ are proportional in this degenerate case. In fact, $U(a,b,z),{}_1F_1(a,b,z)$ are polynomials of degree $m=-a$:
\begin{equation}
		U(-m,b,z) = (-1)^m \sum_{s=0}^m \binom{m}{s} \frac{\Gamma(b+m)}{\Gamma(b+s)} (-z)^s.
\end{equation}
The right-hand side is the Laguerre polynomial $L_m^{(b-1)}(z)$, up to a constant of proportionality \cite[\href{http://dlmf.nist.gov/13.6.19}{\S13.6.19}]{NIST}\cite[\href{http://dlmf.nist.gov/18.8}{\S18.8}]{NIST}: 
\begin{equation} \label{eq:Laguerre_relation}
	L_m^{(\alpha)}(z) = \binom{m+\alpha}{m}\cdot  {}_1F_1(-m,1+\alpha,z).
\end{equation}

\subsubsection{$R_-$ for generic $j$}
Excepting the cases $a\in -\bbN$ noted above, the recessive solution $v_0(s)$ is linearly independent from the outgoing solution $v_{\infty,-}(s)$. Consequently, the Wronskian $\frakW_-$ of $v_0$ and $v_{\infty,-}$ is nonzero. This can be calculated using the Wronskian of the hypergeometric functions, which can be looked up \cite[\href{http://dlmf.nist.gov/13.2.vi}{\S13.2.vi}]{NIST}. Alternatively, we can appeal to the small-argument asymptotics of the hypergeometric functions (and the fact that these can be differentiated) to get 
\begin{align}
v_{\infty,-}(s) v_0'(s)-v_{\infty,-}'(s) v_0(s)  =o(1)+2 \Gamma(a)^{-1}\Gamma(\alpha+1)(-i/2)^{-\alpha}.  
\end{align}
And since the Wronskian is constant, the $o(1)$ term must actually vanish identically:
\begin{equation}
\frakW_- = 2 \Gamma(a)^{-1}\Gamma(\alpha+1)(-i/2)^{-\alpha}.
\label{eq:key_Wronskian}
\end{equation}
As expected, this is nonzero if (and only if) $a\notin -\bbN$. 

So, we now have a formula for the Schwartz kernel $R_\pm[\lambda](\hat{y},\hat{y}')$ of $R_\pm$ acting on $\lambda$-modes: 
\begin{equation}
	R_\pm[\lambda](\hat{y},\hat{y}') = \frac{ (\hat{y}'/\hat{y})^{(d-2)/2}}{\frakW_\pm} \begin{cases}
		v_{\infty,\pm}(\hat{y}) v_0(\hat{y}') & (\hat{y}>\hat{y}')\\
		v_{\infty,\pm}(\hat{y}') v_0(\hat{y}) & (\hat{y}<\hat{y}'). 
	\end{cases}
\end{equation}
In the case of $R_-$, we have been totally explicit about what the various terms above are.

Now we state the main proposition of this subsection:
\begin{proposition}\label{prop:separated_mapping}
For $a=(d-1)/2+\ell/2-j/2 \not\in -\bbN$,
    \begin{equation}
        R_-[\lambda]: e^{i\hat{y}^2/4} \calA^{ ((d+1-j)/2,0),(\ell/2,0)} ([0,\infty]_{\hat{y}^2} ) \to e^{i\hat{y}^2/4} \calA^{ ((d-1-j)/2,0),(\ell/2,0)} ([0,\infty]_{\hat{y}^2} ) 
    \end{equation}
    where the `$(\ell/2,0)$' denotes the index set at $\hat{y}=0$, and $\ell$ is the azimuthal quantum number, as in \cref{eq:azimuthal}. 
\end{proposition}
\begin{proof}[Proof sketch]
    We check that 
    \begin{equation} 
        g\in C_{\mathrm{c}}^\infty(\bbR^+_{\hat{y}}) \Longrightarrow R_-[\lambda]g\in  e^{i\hat{y}^2/4} \calA^{( (d-1-j)/2,0),(\ell/2,0)} ([0,\infty]_{\hat{y}^2} ).
    \end{equation} 
    Certainly, $R_-[\lambda]g$ is smooth as a function of $\hat{y}\in \bbR^+$, so the only claim to check is that it has appropriate behavior as $\hat{y}\to 0^+,\infty$. Outside of $\operatorname{supp} g\Subset \bbR^+$, 
    \begin{equation}
        R_-[\lambda] g(\hat{y}) = 
        \frac{1}{\frakW_-} 
        \begin{cases}
        \hat{y}^{-(d-2)/2} v_{\infty,-}(\hat{y})
        \int_0^\infty v_0(\hat{y}') g(\hat{y}') \hat{y}'^{(d-2)/2} \dd \hat{y}'  &   (\hat{y}\gg 1), \\
        \hat{y}^{-(d-2)/2} v_{0}(\hat{y})
        \int_0^\infty v_{\infty,-}(\hat{y}') g(\hat{y}') \hat{y}'^{(d-2)/2} \dd \hat{y}' &    (\hat{y} \ll 1). 
        \end{cases}
    \end{equation}
    Note that the integrals above are constant. Thus, near $\hat{y}=0$, we have 
    \begin{equation} 
        R_-[\lambda](g) \propto \hat{y}^{-(d-2)/2} v_0, 
        \end{equation} 
        and near $\hat{y}=\infty$, we have $R_-[\lambda] g \propto \hat{y}^{-(d-2)/2} v_{\infty,-}$. So, the desired $e^{i\hat{y}^2/4}\times \text{polyhomogeneity}$ follows directly from that of $v_0,v_{\infty,-}$.

    The case of general $g\in  e^{i\hat{y}^2/4} \calA^{( (d+1-j)/2,0),(\ell/2,0)} ([0,\infty]_{\hat{y}^2} )$ proceeds similarly. We have 
    \begin{multline}\label{eq:res_form}
        R_-[\lambda] g (\hat{y}) = \frac{\hat{y}^{-(d-2)/2}}{\frakW_-} \bigg[ v_{\infty,-}(\hat{y}) \int_0^{\hat{y}} v_0(\hat{y}') g(\hat{y}') \hat{y}'^{(d-2)/2} \dd \hat{y}'  \\ + v_{0}(\hat{y}) \int_{\hat{y}}^{\infty} v_{\infty,-}(\hat{y}') g(\hat{y}') \hat{y}'^{(d-2)/2} \dd \hat{y}'   \bigg].
    \end{multline}
    Now it is necessary to analyze the $\hat{y}\to 0^+,\infty$ behavior of the integrals. This can be done using the asymptotics of $v_0,v_{\infty,-}$. 
    We omit the details of the analysis. 
\end{proof}

Also, we sketch the following result relating $R_\pm(j)$ to the $L^2$-bounded inverse $L(j)^{-1}$ (i.e.\ \cref{eq:resolvent_relations}):
\begin{proposition}
If $\Re j-(d-1)/2$ has a definite sign $\pm$, then:
    \begin{enumerate}[label=(\alph*)]
        \item $R_\pm(j)$, initially defined only on terms involving finitely many spherical harmonics with $C_{\mathrm{c}}^\infty(\bbR^+)$-coefficients, extends boundedly to $L^2$,  
        \item the extension is the $L^2$-bounded inverse $L(j)^{-1}$.
    \end{enumerate}
\end{proposition}
\begin{proof}[Proof sketch]
	It suffices to check that $L(j)^{-1}f$ agrees with $R_\pm(j) f$ when $f$ has the form  $Yg$ with $g\in C_{\mathrm{c}}^\infty(\bbR^+)$. Then, 
	\begin{align}
	\begin{split} 
		R_\pm f &= Y R_\pm[\lambda] g \\
		L(j)^{-1} f &=  Y \tilde{u}
		\end{split} 
	\end{align}
	for some $\tilde{u}$. Beyond the support of $g$, the function $\tilde{u}$ must (by virtue of the ODE it satisfies) be a linear combination of the two $\hat{y}^{-(d-2)/2}v_{\infty,\pm}$. These satisfy 
	\begin{align}
	\begin{split} 
	\hat{y}^{-(d-2)/2}v_{\infty,-} &\sim e^{i\hat{y}^2/4} \hat{y}^{-(d-1)+j} \\
	\hat{y}^{-(d-2)/2}v_{\infty,+} &\overset{\text{using \cref{eq:outgoing_osc}}}{\sim} e^{-i\hat{y}^2/4} \hat{y}^{-j} 
	\end{split} 
	\end{align}
	as $\hat{y}\to\infty$. (Note that the decay rates agree when $\Re j=(d-1)/2$, as expected.)

    For these to lie in $L^2([1,\infty)_{\hat{y}}, \hat{y}^{d-2} \dd \hat{y} )$ (equivalently $L^2(\{\hat{y}>1\}; \bbR^{d-1}_{\hat{\bfy}} )$), we have the following conditions:
	\begin{align}
    \begin{split} 
	\hat{y}^{-(d-1)+j} \in L^2([1,\infty)_{\hat{y}}, \hat{y}^{d-2} \dd \hat{y} ) &\iff \Re j < (d-1)/2 \\ 
    \hat{y}^{-j} \in L^2([1,\infty)_{\hat{y}}, \hat{y}^{d-2} \dd \hat{y} ) &\iff\Re j>(d-1)/2.
    \end{split} 
	\end{align}
	So, $\tilde{u}$, being in $L^2$, must be proportional to $\hat{y}^{-(d-2)/2}v_{\infty,\pm}$ for $\hat{y}\gg 1$.  

	On the other hand, $\tilde{u}$ must be recessive for small $\hat{y}$. This is because, if the non-recessive solutions $\hat{y}^{-(d-2)/2} v$ managed to be in $L^2([0,1); \hat{y}^{d-2} \dd \hat{y})$ and $Y \hat{y}^{-(d-2)/2}v $ solved the homogeneous PDE near $\hat{y}=0$, then essential self-adjointness would fail. (By separation of variables, the PDE is automatically satisfied away from the origin. The only reason it might not satisfy the PDE globally is because of a $\delta$-function or derivative thereof at the origin.) Alternatively, use the small-$z$ expansions of $U(a,b,z)$ to show that $\hat{y}^{-(d-2)/2} v_{\infty,\pm}$ both fail to be in $L^2([0,1); \hat{y}^{d-2} \dd \hat{y})$ for most values of $\alpha$. Indeed, this failure obtains if $\alpha\geq 1$. This is automatic if $d\geq 5$, leaving only $d=2,3,4$, and $\lambda$ sufficiently small; in fact, plugging in the formula for the eigenvalues of $\bbS^{d-2}$, only $\lambda=0$ (i.e.\ the s-wave case) is a problem, if $d\geq 3$. When $d=2$, there are two $Y_\pm$ to consider. (The existence of these special cases is linked to the fact that the Laplacian is not essentially self-adjoint on $C_{\mathrm{c}}^\infty(\bbR^n\backslash \{0\})$ if $n\leq 3$.)  For these latter cases, it still must be true that $\tilde{u}$ is proportional to the recessive solution, but the reason is that the non-recessive solution fails to satisfy the PDE on the unpunctured $\bbR^{d-1}$. (For example, $1/\hat{y}$ solves $\triangle (1/\hat{y})\propto \delta$ on $\bbR^3_{\hat{\bfy}}$. When separating variables, you do not see the $\delta$ function.) 

    The above shows that $w\coloneqq \tilde{u} - R_\pm[\lambda] g$, which solves the homogeneous equation $L(j)w=0$, is both recessive at $0$ and incoming/outgoing as $\hat{y}\to\infty$. This forces $w=0$ (otherwise the recessive solution would not be linearly independent from the incoming/outgoing solution).
\end{proof}

\subsubsection{Cokernel}

If $a\in -\bbN$ when $j=j_0$, the construction above fails, because the Wronskian $\frakW_-$ vanishes, because $v_{\infty,-} = C v_0$ for some $C\neq 0$. Instead, we can try to define $R_-[\lambda]g$ by taking a limit as $j\to j_0$. 
The explicit formula for $\frakW_-$ shows that it vanishes simply at $j=j_0$.

Using the formula \cref{eq:res_form},
\begin{equation}
    R_-[\lambda]g = \frac{C\hat{y}^{-(d-2)/2} }{\frakW_-} v_{0}(\hat{y}) \int_0^\infty v_0(\hat{y}')g(\hat{y}') {\hat{y}'}^{ (d-2)/2} \dd \hat{y}' \bigg|_{j=j_0} +\text{ finite} 
\end{equation}
as $j\to j_0$. Thus, $R_-[\lambda]g$ is well-defined at $j=j_0$ if 
\begin{equation}\label{eq:coker_form}
     \int_0^\infty v_0(\hat{y}')g(\hat{y}') {\hat{y}'}^{ (d-2)/2} \dd \hat{y}' =0. 
\end{equation}
Then, it can be shown that $R_-[\lambda]g$ satisfies $L(j_0) R_-[\lambda]g=g.$ 
 
\subsection{Reduction to the short-range case}
\label{subsec:QHO_rig}
We are going to rewrite $L(j)$ in terms of a new coordinate system on $\mathrm{ff}\cong \bbR^{d-1}$. 
Consider a coordinate $z=z(\hat{y})\in C^\infty(\bbR^+)$ given by $\hat{y}$ in $\{\hat{y}<1\}$ and by $\hat{y}^2$ in $\{\hat{y}>2\}$, constructed such that $z$ is strictly increasing with $\hat{y}$. Then, the map \begin{equation} \iota: \bbR^{d-1} \ni \hat{\bfy} \mapsto \underbrace{z \theta}_{\coloneqq \bfz} = z(\hat{y}) \hat{y}^{-1} \hat{\bfy} \in \bbR^{d-1} \end{equation} is a diffeomorphism of $\bbR^{d-1}$. Then, we can rewrite $L(j)$ in terms of the new coordinate system $\bfz$.  

We will show that (up to a polynomial weight which can be factored out) $L(j)$ is a shorter-range Schr\"odinger operator on  $(\bbR^{d-1}_{\bfz},g_{\mathrm{QHO}})$ for some asymptotically conic metric $g_{\mathrm{QHO}}$ on $\bbR^{d-1}$.

\subsubsection{$d=2$ case} 

First, we illustrate the idea in the $d=2$ case. For simplicity, we only discuss $z\gg 1$, where $z=\hat{y}^2$. The remaining range of $z$ is easily handled. What we want to emphasize is how the change of coordinates $\hat{\bfy} \to \bfz$ affects the form of the PDE at large-$\hat{y}$, i.e.\ large-$z$.

Concretely:
\begin{equation}
	L(j)  = - \frac{\partial^2}{\partial \hat{y}^2} - \frac{\hat{y}^2}{4} + i\Big(j-\frac{1}{2}\Big) =  -4 z \frac{\partial^2}{\partial z^2} - 2\frac{\partial}{\partial z}  - \frac{z}{4}+i\Big(j-\frac{1}{2}\Big),
	\label{eq:misc_1023127}
\end{equation}
since 
\begin{equation}
	\frac{\partial}{\partial \hat{y}} = \Big(\frac{\partial z}{\partial \hat{y}}\Big) \frac{\partial}{\partial z} = 2 \hat{y}\frac{\partial}{\partial z}
\end{equation}
for $z\gg 1$. We may factor a $z$ from the right-hand side of \cref{eq:misc_1023127}, getting 
\begin{equation}
	\frac{1}{4z}  L(j) = -\frac{\partial^2}{\partial z^2} - \frac{1}{2z} \frac{\partial}{\partial z} - \frac{1}{16} + \frac{i}{4z} \Big(j-\frac{1}{2} \Big). 
\end{equation}
Recall the form of the radial part of the Coulomb--Helmholtz operator in $D$ dimensions:
\begin{equation}
	-\frac{\partial^2}{\partial z^2} - \frac{D-1}{z} \frac{\partial}{\partial z} - E + \frac{\mathsf{Z}}{z} + \frac{\lambda}{z^2}, 
\end{equation}
where $E>0$ is the spectral parameter, $\mathsf{Z}$ the Coulomb charge, and $\lambda$ the angular eigenvalue. The operator $(4z)^{-1}L(j)$ has, in $z\gg 1$, the same form, with 
\begin{itemize}
	\item a fractional $D=3/2$ ``number of dimensions,''
	\item $E=1/16$
	\item a complex Coulomb charge $\mathsf{Z} = 4^{-1} i(j-1/2)$, and 
	\item $\lambda =0$. 
\end{itemize}
Note that the Hooke term $\hat{y}^2/4$ has become the spectral term after the change of coordinates.

\subsubsection{General case}
Now let us discuss the general case. 
The aforementioned metric $g_{\mathrm{QHO}}$ on $\bbR^{d-1}_{\bfz}$ is defined by 
\begin{equation} 
	g_{\mathrm{QHO}} = 4 \langle z \rangle g_{\mathrm{Euclid},*},
\end{equation} 
where $g_{\mathrm{Euclid}}=\sum_{j=1}^{d-1}\dd \hat{y}_j^2$ on $\bbR^{d-1}$ is the Euclidean metric on $\bbR^{d-1}_{\hat{\bfy}}$, and the $*$ subscript denotes pullback via the coordinate change $\hat{\bfy}\mapsto \bfz$.

\begin{proposition} The metric 
	$g_{\mathrm{QHO}}$ is asymptotically conic.
\end{proposition}
\begin{proof}
	In $\{z>4\}$, 
	\begin{equation}
		g_{\mathrm{Euclid}} = \mathrm{d} \hat{y}^2 + \hat{y}^2 g_{\bbS^{d-2}} = \Big(\frac{\mathrm{d} \hat{y}}{\mathrm{d} z} \Big)^2\dd z^2  + z g_{\bbS^{d-2}} = \frac{1}{4z} (\mathrm{d} z^2 + 4z^2g_{\bbS^{d-2}} ).
	\end{equation}
	So, $g_{\mathrm{QHO}} = \sqrt{1+z^{-2}} ( \mathrm{d} z^2 +4 z^2 g_{\bbS^{d-2}})$ for large $z$. 
\end{proof}

So, letting $\triangle_{\mathrm{QHO}}=\triangle_{g_{\mathrm{QHO}}}$, 
\begin{equation}
	\frac{1}{4\langle z \rangle} L(j)= \triangle_{\mathrm{QHO}} + \Gamma + \frac{i}{4 \langle z \rangle} \Big(j-\frac{d-1}{2}\Big) - \frac{1}{16} + W(\bfz),  
\end{equation}
where 
\begin{align}
	W(\bfz) &= - \frac{ (\hat{y}^2 - \langle z \rangle) }{16\langle z \rangle}  \in \langle z \rangle^{-2} C^\infty(\overline{\bbR^{d-1}_{\bfz}} ),  \\
	\begin{split} 
		\Gamma &= \frac{1}{4\langle z \rangle} \triangle  - \triangle_{\mathrm{QHO}}  
		\overset{z\gg 1}{=}  \frac{(d-3)z^2}{2\langle z \rangle^3} \frac{\partial}{\partial z}. 
	\end{split} 
\end{align}
The computation of $\Gamma$ follows from: since 
\begin{equation}
	\triangle_g = -\frac{1}{\sqrt{|g|}} \sum_{i,j=1}^{d-1} \partial_i \big( \sqrt{|g|} g^{ij} \partial_j\big), 
\end{equation}
for any $\Omega\in C^\infty(\bbR^{d-1};\bbR^+)$,
\begin{multline}
	-\triangle_{\Omega g} = \frac{1}{\sqrt{\Omega^{d-1} |g|}} \sum_{i,j=1}^{d-1} \partial_i \big( \sqrt{\Omega^{d-1} |g|} \Omega^{-1} g^{ij} \partial_j\big) = -\frac{1}{\sqrt{\Omega^{d-1} |g|}} \sum_{i,j=1}^{d-1} \partial_i ( |g|^{1/2} \Omega^{(d-3)/2} g^{ij} \partial_j)\\ = -\Omega^{-1} \triangle_g + \Omega^{-(d-1)/2} \sum_{i,j=1}^{d-1} (\partial_i \Omega^{(d-3)/2}) g^{ij} \partial_j \\ = -\Omega^{-1} \triangle_g + \frac{d-3}{2} \Omega^{-2}\sum_{i,j=1}^{d-1} (\partial_i \Omega) g^{ij} \partial_j .
\end{multline}
This is applied with $g=g_{\mathrm{Euclid}}$ and $\Omega = 4\langle z \rangle$ to get 
\begin{multline}
	\triangle_{\mathrm{QHO}} \overset{z\gg 1}{=}\frac{1}{4 \langle z \rangle} \triangle  - \frac{d-3}{8 \langle z \rangle^2}  (\partial_{\hat{y}} (1+\hat{y}^4)^{1/2}) \partial_{\hat{y}} \\ =\frac{1}{4 \langle z \rangle} \triangle  - \frac{d-3}{4 \langle z \rangle^2}  \frac{\hat{y}^3}{(1+\hat{y}^4)^{1/2}} \partial_{\hat{y}} 
	= \frac{1}{4\langle z \rangle} \triangle - \underbrace{\frac{(d-3)z^2}{{2}\langle z \rangle^3}}_{\Gamma} \partial_z. 
\end{multline}

The radial part of $\triangle$ is $-\partial_{\hat{y}}^2 - \frac{d-2}{\hat{y}} \partial_{\hat{y}}$, so, when $z>4$, 
\begin{align}
    \begin{split} 
    \triangle_{\mathrm{QHO}} &= -\frac{1}{4\langle z \rangle}  \Big( \frac{\partial^2}{\partial \hat{y}^2} + \frac{d-2}{\hat{y}}\frac{\partial}{\partial \hat{y}}  \Big) + \frac{1}{4\langle z \rangle z} \triangle_{\bbS^{d-2}} - \frac{(d-3) z^2}{2 \langle z \rangle^3} \frac{\partial}{\partial z}\\ 
    &=  - \frac{z}{\langle z \rangle}\frac{\partial^2}{\partial z^2} + \frac{1}{4\langle z \rangle z} \triangle_{\bbS^{d-2}} - \frac{1}{2\langle z \rangle} \frac{ (d-1) + 2(d-2) z^2}{1+z^2} \frac{\partial}{\partial z}.
    \end{split} 
\end{align}
Note the appearance of a complex Coulomb term 
\begin{equation} 
    \mathsf{Z}=\frac{i}{4} \Big(j-\frac{d-1}{2}\Big) .
\end{equation}

\subsection{Fredholm setup}\label{subsec:QIHO_Fredholm}
Above, we showed that, using the coordinate $\bfz\in \bbR^{d-1}$,  the operator $K(j) \coloneqq  \frac{1}{4\langle z \rangle} L(j)$  can be written 
\begin{equation}
K(j)= \triangle_{\mathrm{QHO}} + \Gamma + \frac{i}{4\langle z \rangle } \Big( j - \frac{d-1}{2} \Big) - \frac{1}{16} +W,
\end{equation}
for $\Gamma,W$ as above. 
The term $W$ is $O(1/z^2)$ as $z\to\infty$. 
Moreover, 
\begin{equation}
    \Gamma = \frac{(d-3)}{2z} \frac{\partial}{\partial z} \text{ in }z\gg1 \mod \operatorname{Diff}^{1,-2}_{\mathrm{sc}}(\overline{\bbR^{d-1}_{\bfz}}).
\end{equation}
Elements of $\operatorname{Diff}^{1,-2}_{\mathrm{sc}}$ are negligible vis-\`a-vis the microlocal solvability theory of $ K(j)$. The main term in $\Gamma$, as well as the Coulomb term, is exactly subprincipal. 
The notion of a subprincipal part is not invariantly defined, but 
\begin{equation} 
    \Im  K(j) = (2i)^{-1} (K(j)-K(j)^*),
\end{equation} 
defined with respect to the $L^2(\bbR^{d-1},g_{\mathrm{QHO}})$-inner product, satisfies 
\begin{equation} 
\Im K(j) \in 
\operatorname{Diff}^{1,-1}_{\mathrm{sc}}(\overline{\bbR^{d-1}_{\bfz}}).
\end{equation} 
Its principal part is exactly 
\begin{equation}
-i\Gamma + \frac{1}{4 z} \Big(\Re j-\frac{d-1}{2} \Big)  \in \operatorname{Diff}_{\mathrm{sc}}^{1,-1}(\overline{\dot{\bbR}^{d-1}_{\bfz}}).
\end{equation}
The leading order part of $K(j)$, which has the largest effect on its solvability theory, is the Helmholtz operator 
\begin{equation}
\triangle_{\mathrm{QHO}} -\frac{1}{16} \in \operatorname{Diff}_{\mathrm{sc}}^{2,0}(\overline{\bbR^{d-1}_{\bfz}}).
\end{equation}

We can now cite the existing literature on  
Schr\"odinger operators of this form. The estimates that $K(j)$ satisfy are almost identical to those for the Helmholtz operator. 
The effect of the subprincipal term is to merely shift the absorption threshold. The particular form of the limiting absorption we use can be found in Vasy's \cite{VasyLA}. 
Note that the full operator $K(j) \in \operatorname{Diff}^2(\bbR^{d-1}_{\bfz})$, when written in terms of $\bfz$, satisfies the hypotheses in \cite[\S3]{VasyLA}, under which Vasy proves his main theorems. Vasy's frequency parameter is $\sigma=1/4$. 

\begin{remark*}
  What Vasy calls `$x$' is just $\langle z \rangle^{-1}$ here.   
\end{remark*}

Vasy makes use of particular Sobolev spaces, 
\begin{equation} 
H_{\mathrm{b}}^{m,\ell} = \Psi_{\mathrm{b}}^{-m,-\ell}( \overline{\bbR^{d-1}_{\bfz}} ) L^2(\bbR^{d-1}_{\bfz}) \subset \calS',\quad m,\ell\in \bbR,
\end{equation} 
the b-Sobolev spaces, capturing a finite amount of $L^2$-regularity under the application of b-vector fields. The exact definition need not concern us, because we will quickly transition to working with conormal function spaces. Vasy proves:
\begin{propositionp}[\cite{VasyLA}, Thm.\ 1.1]
    Let $\alpha_\pm$ be as defined in \cite[\S3]{VasyLA}. Suppose that $r,\ell\in \bbR$ satisfy 
    \begin{equation}\label{eq:vasy_threshold_cond}
        r+\ell+1/2 - \Im \alpha_- >0,\quad \ell+1/2  - \Im \alpha_+ <0. 
    \end{equation}
    Then, the conjugated operator $\check{K}(j)=e^{-i \langle z \rangle /4} K(j) e^{i\langle z \rangle /4}$, which maps 
    \begin{equation}
        \calX_{r,\ell}\coloneqq \{u\in H_{\mathrm{b}}^{r,\ell} : \check{K}(j) u \in H_{\mathrm{b}}^{r,\ell+1} \}\to H_{\mathrm{b}}^{r,\ell+1 } \coloneq \calY_{r,\ell},
    \end{equation}
    is Fredholm as a map between these spaces (with a suitable graph norm on the domain). In particular, it has finite-dimensional kernel and cokernel. 
\end{propositionp}
\begin{remark*}
    Let $s = \sigma_{\mathrm{sc}}^{1,0}(z \Im K(j)) =  (d-3)\xi/2 + \Im \mathsf{Z}$ denote the principal symbol of $\Im K(j)$, weighted to make it order-0 in the sense of decay. Here $\xi$ is the frequency coordinate dual to $z$. Then, the $\alpha_\pm$ in the previous theorem are (modulo an unimportant real term) defined by
    \begin{equation}
        \Im \alpha_\pm = \mp \frac{1}{2\sigma} s|_{ \operatorname{graph} \pm \sigma \dd z  }   = \mp \frac{1}{2\sigma} s|_{ \xi = \pm \sigma }   ;  
    \end{equation}
    Plugging in,   
    \begin{equation}
          s|_{ \xi = \pm \sigma } = \frac{d-3}{2} (\pm \sigma)  + \Im \mathsf{Z}  \Longrightarrow  \Im \alpha_\pm = - \frac{d-3}{4}  \mp 2 \Im \mathsf{Z} = -\frac{d-3}{4} \mp \frac{1}{2} \Big(\Re j-\frac{d-1}{2} \Big) .
    \end{equation} 
    In particular, $\Im \alpha_+ = -(\Re j-1)/2$. 
\end{remark*}
\begin{remark}\label{rem:j-independence}
    Because $j$ enters the operator $\check{K}(j)$ only through a Coulomb term, if $u\in H_{\mathrm{b}}^{r,\ell}$, then 
    \begin{equation}
        \check{K}(j) u \in H_{\mathrm{b}}^{r,\ell+1} \iff \check{K}(j') u \in H_{\mathrm{b}}^{r,\ell+1} ,
    \end{equation}
    for any $j'$. Thus, $\calX_{r,\ell}$ does not depend on $j$. However, the threshold conditions on $r,\ell$ depend on $j$. 
\end{remark}

Thus, we have a right-inverse 
\begin{equation} 
    \check{K}(j)^{-1}:H_{\mathrm{b}}^{r,\ell+1 } \supset \calD\to H_{\mathrm{b}}^{r,\ell}
\end{equation} 
to $\check{K}(j)$, 
defined on the closed subspace $\calD=\{ \check{K}(j) u : u \in \calX_{r,\ell} \}$. 
Part of \cite[Thm. 1.1]{VasyLA} is that this agrees with the resolvent defined via the limiting absorption principle, defined via adding a small imaginary part to the spectral parameter $\sigma$ of $K(j)$.\footnote{That part is only proven under the symmetry of the operator. However the proof only uses bijectivity of the limiting problem, which we have proven for generic $j$ via separation-of-variables. For non-generic $j$, the same proof works once we restrict the codomain to the range and the domain to the orthogonal complement of the kernel. In both cases, this is a standard argument where uniform estimates are parlayed into strong continuity.} In particular, $\check{K}(j)^{-1}$ does not depend on the parameters $r,\ell\in \bbR$ (assuming these satisfy the threshold assumptions above), in the sense that if $f$ lies in the intersection of multiple $H_{\mathrm{b}}^{r,\ell+1}$, then $\check{K}(j)^{-1} f$ is unambiguous.

This yields a right-inverse 
\begin{equation} 
    K(j)^{-1} : e^{i\langle z \rangle/4} \calD \to e^{i\langle z \rangle/4} H_{\mathrm{b}}^{r,\ell}
\end{equation} 
to the \emph{unconjugated} operator $K(j)$. Since $L(j) = 4\langle z \rangle K(j)$, 
\begin{equation}\label{eq:Vasy_res}
    R_-(j)\coloneqq K(j)^{-1} \circ \frac{1}{4\langle z \rangle} :  e^{i\langle z \rangle/4} \langle z \rangle \calD \to e^{i\langle z \rangle/4} H_{\mathrm{b}}^{r,\ell} 
\end{equation}
is a right-inverse to $L(j)$. 
Suppose that $f\in C_{\mathrm{c}}^\infty(\bbR^{d-1}_{\bfz})$ consists of only a single angular mode. In \S\ref{subsec:QHO_separation} we constructed (for most $j$) a solution $w$ of 
\begin{equation} \label{eq:misc_462}
L(j)w=f 
\end{equation}
regular at the origin and satisfying the outgoing Sommerfeld condition; hence $w\in \calX_{r,\ell}$.  We previously called this `$R_-(j)f$.' The notation is consistent, because if $R_-(j)$ is defined by \cref{eq:Vasy_res}, then $R_-(j)f$ is regular at the origin and satisfies the outgoing Sommerfeld condition, and these properties uniquely characterize solutions of the inhomogeneous problem. 
This shows that, for such $j$, the subspace $\calD\subset H_{\mathrm{b}}^{r,\ell+1}$ is dense therein. Therefore, it cannot have positive codimension, which means 
\begin{equation} 
    \calD=H_{\mathrm{b}}^{r,\ell+1}
\end{equation} 
is all of $H_{\mathrm{b}}^{r,\ell+1}$. 

This construction failed for those modes characterized by azimuthal quantum number $\ell\in \bbN$ such that the parameter $a=(d-1)/2-j/2+\ell/2$ was a nonpositive integer, $a\in -\bbN$. For each $j\in \bbC$, there are at most finitely many $\ell\in \bbN$ such that $a\in -\bbN$. The elements of $\calD$ in the problematic modes are characterized by \cref{eq:coker_form}. A more careful proof of the same fact is via a duality argument, finding the kernel of $\check{K}(j)^*$ in 
\begin{equation} 
    (H_{\mathrm{b}}^{r,\ell})^*\cong H_{\mathrm{b}}^{-r,-\ell}.
\end{equation} 
The formal adjoint $\check{K}(j)^*$ is a partial differential operator, understood distributionally on its domain.

\subsection{Refined mapping properties} 
\label{subsec:QIHO_mapping}

Consequently, taking an intersection over $r\gg 1$, 
\begin{equation}
    \check{K}(j)^{-1}: H_{\mathrm{b}}^{\infty,\ell+1 } \supset \calD\to H_{\mathrm{b}}^{\infty,\ell};
\end{equation}
here and below, $\calD$ is a finite-codimension subspace, changing line-by-line. 

The $L^2(\bbR^{d-1}_{\bfz})$-based spaces $H_{\mathrm{b}}^{\infty,\ell}$ can be related to the $L^\infty$-based spaces $\calA^{\bullet}(\mathrm{ff})=\calA^{\bullet}(\overline{\bbR^{d-1}_{\bfz}})$ of conormal functions by Sobolev embedding: 
\begin{equation}
    \calA^{(d-1)/2+\ell+} (\mathrm{ff}) \subseteq 
    H^{\infty,\ell}_{\mathrm{b}} (\overline{\bbR^{d-1}_{\bfz}} ) \subset  \calA^{(d-1)/2+\ell-} (\mathrm{ff}). 
\end{equation}
Thus, 
\begin{equation}
    \check{K}(j)^{-1} :  \calA^{(d-1)/2+\ell+1+} (\mathrm{ff}) \supseteq \calD \to \calA^{(d-1)/2+\ell-} (\mathrm{ff}).  
\end{equation} 
This holds for all $\ell< -1/2+\Im \alpha_+ = -\Re j/2$. 
In particular, $\check{K}(j)^{-1}:\calA^{ (d+1)/2-\Re j/2+ } (\mathrm{ff}) \supseteq \calD \to  \calA^{(d-1)/2-\Re j/2-} (\mathrm{ff})$.

Conormality is upgraded to polyhomogeneity via a standard argument; see \cite{Me94}\cite[\S6]{ACL}\cite[\S3]{phgFull}. Indeed,  
\begin{equation}
    \check{K}(j) u = f \Longrightarrow 
    N_{\mathrm{ff}\cap\mathrm{bf}}(\check{K}(j))   u=f - \underbrace{(\check{K}(j) - N_{\mathrm{ff}\cap\mathrm{bf}}(\check{K}(j)))}_{\in \operatorname{Diff}_{\mathrm{b}}^{2,-2}([0,\infty)_{1/z} \times \bbS^{d-2}_{\bfz/z}) } u  .
\end{equation}
Explicitly, 
\begin{equation}
    N_{\mathrm{ff}\cap\mathrm{bf}}(\check{K}(j)) = -\frac{i}{2} \frac{\partial}{\partial z}+\frac{i}{4z} (j-d+1) \in \operatorname{Diff}_{\mathrm{b}}^{1,-1}([0,\infty)_{1/z} \times \bbS^{d-2}_{\bfz/z}) . 
\end{equation}
This is inverted by simply integrating: 
\begin{equation}
    N_{\mathrm{ff}\cap\mathrm{bf}}(\check{K}(j)) u =g \Longrightarrow \exists c\in C^\infty(\bbS^{d-2}_{\bfz/z}) \text{ s.t. }u= z^{j/2-(d-1)/2} \bigg[ c  +2i  \int_1^z g(s\bfz/z)  s^{-j/2+(d-1)/2} \dd s  \bigg] .
\end{equation}
Now apply this with $g=f - (\check{K}(j) - N_{\mathrm{ff}\cap\mathrm{bf}}(\check{K}(j)))u$: 
\begin{multline}
    u = z^{j/2-(d-1)/2} \bigg[ c  +2i  \int_1^z f(s\bfz/z)  s^{-j/2+(d-1)/2} \dd s \\  -2i  \int_1^z ((\check{K}(j) - N_{\mathrm{ff}\cap\mathrm{bf}}(\check{K}(j)))u) (s\bfz/z)  s^{-j/2+(d-1)/2} \dd s  \bigg] 
\end{multline}
If $f$ is polyhomogeneous on $\mathrm{ff}$, then the terms in the first line on the right-hand side are also polyhomogeneous. If $u\in \calA^{\calE}+\calA^\alpha$ is known to be partially polyhomogeneous, with a conormal error of order $\alpha\in \bbR$, 
then the contribution to the right-hand side from the polyhomogeneous $\calA^\calE$ term is also polyhomogeneous. On the other hand, the contribution of the conormal term to the right-hand side is $\calA^{\alpha+1-}$, because applying $\check{K}(j) - N_{\mathrm{ff}\cap\mathrm{bf}}(\check{K}(j))$ gains two orders of decay, while integrating loses only one. Thus, $u$ is partially polyhomogeneous, with polyhomogeneity holding up to one order higher than we assumed. Proceeding inductively, the entire polyhomogeneous expansion of $u$ is uncovered. So, if $f$ is polyhomogeneous, we conclude that 
\begin{equation}
    u\in z^{j/2-(d-1)/2}C^\infty(\overline{\bbR^{d-1}_{\bfz}}) + \calA^\calE(\mathrm{ff})  
\end{equation}
for some index set $\calE$ depending on that of $f$. 

The simplest case is when $f$ is Schwartz. Then, the $f$-term above is also in $\langle z\rangle ^{j/2-(d-1)/2}C^\infty(\overline{\bbR^{d-1}_{\bfz}})$, so 
\begin{equation} 
u\in \langle z\rangle^{j/2-(d-1)/2}C^\infty(\overline{\bbR^{d-1}_{\bfz}}).
\end{equation} 
The same conclusion holds as long as 
\begin{equation} 
    f\in  \langle z \rangle^{j/2-(d+3)/2}C^\infty(\overline{\bbR^{d-1}_{\bfz}}).
\end{equation} 
Indeed, if $z>1/2\Longrightarrow f(\bfz) = z^{j/2-(d+3)/2} f_\circ(\bfz)$ for some $f_\circ \in C^\infty(\overline{\bbR^{d-1}_{\bfz}})$, then the corresponding contribution to $u$ above is 
\begin{equation}
    2i z^{j/2-(d-1)/2} \int_1^z f_\circ(s\bfz/z) \frac{\dd s }{s^2}  \in z^{j/2-(d-1)/2} C^\infty([0,1)_{1/z}\times \bbS^{d-2}_{\bfz/z})
\end{equation}
when $z>1$. 
Thus, we conclude 
\begin{equation}
    \check{K}(j)^{-1} :   \calD \cap  \langle z\rangle^{j/2-(d+3)/2}C^\infty(\overline{\bbR^{d-1}_{\bfz}}) \to \langle z\rangle^{j/2-(d-1)/2}C^\infty(\overline{\bbR^{d-1}_{\bfz}}).
\end{equation}

This yields a right-inverse 
\begin{equation} 
    K(j)^{-1} : e^{i\langle z \rangle/4} (\calD \cap  \langle z \rangle^{j/2-(d+3)/2}C^\infty(\overline{\bbR^{d-1}_{\bfz}})) \to e^{i\langle z \rangle/4} \langle z\rangle^{j/2-(d-1)/2}C^\infty(\overline{\bbR^{d-1}_{\bfz}})
\end{equation} 
to the \emph{unconjugated} operator $K(j)$. Since $L(j) = 4\langle z \rangle K(j)$, 
\begin{equation}
    R_-(j)\coloneqq K(j)^{-1} \circ \frac{1}{4\langle z \rangle} : \langle z \rangle e^{i\langle z \rangle/4} (\calD \cap  \langle z \rangle^{j/2-(d+3)/2}C^\infty(\overline{\bbR^{d-1}_{\bfz}}) ) \to e^{i\langle z \rangle/4} \langle z\rangle^{j/2-(d-1)/2}C^\infty(\overline{\bbR^{d-1}_{\bfz}})
\end{equation}
is a right-inverse to $L(j)$. The previous equation is a restatement of the desired mapping property \cref{eq:misc_163}.

\section{\texorpdfstring{The large-$s$ conormality of $\partial_a U(a,b,-is)$}
{The large-s conormality of the a-derivative of U(a,b,-is)}}
\label{sec:tech}
This appendix contains the proof that $U(a,b,-is)$, which is conormal at $s=\infty$ for each individual $a$, has a derivative in $a$ which is also conormal for $b>0$. This claim is highly plausible --- it follows by formally differentiating in $a$ the standard large-$z$ expansion of $U(a,b,z)$. Nevertheless, a rigorous argument is required as input to the special function-theoretic proof of \Cref{lem:coker_fund}. 
\begin{remark*}
The Fredholm-theoretic method in \S\ref{sec:Grushin} provides one avenue of proof, making use of the ODE that $U$ satisfies. 
Here we pursue a more direct, essentially self-contained proof.    
\end{remark*}

We begin with the integral formula 
\begin{equation}
    U(a,b,z)=\frac{1}{\Gamma(a)}\int_0^{\infty}e^{-zt}t^{a-1}(1+t)^{b-a-1} \dd t,
\end{equation}
which holds for all $(a,b,z)\in \bbC^3$ such that $\Re a>0$ and $\Re z>0$
by \cite[\href{http://dlmf.nist.gov/13.4.4}{\S13.4.4}]{NIST}. 
Unfortunately, we care about $z\in -i\bbR^+$, for which $\Re z=0$. To handle this borderline case, note that, as long as $z$ is in the lower-right quadrant, we can rotate the contour over which the integral is taken from $\bbR^+_t$ to $i\bbR^+_t$ (using principal branches for the power terms in the integrand):
\begin{equation}
    U(a,b,z)=\frac{1}{\Gamma(a)}\int_0^{i\infty}e^{-zt}t^{a-1}(1+t)^{b-a-1}\dd t =\frac{e^{\pi i a/2}}{\Gamma(a)}\int_0^{\infty }e^{-izt}t^{a-1}(1+it)^{b-a-1} \dd t  ,
\end{equation}
this holding for all $z$ in the lower-right (open) quadrant of the complex plane, and $a,b$ as above, where we use the principal branch of the log in the integrand. The final integral extends continuously to $z\in -i \bbR^+$. 
The $U$-function itself is continuous on $\bbC\backslash (-\infty,0]$, so, plugging in $z=-is$,   
\begin{equation}
    U(a,b,-is)= \frac{e^{\pi i a/2}}{\Gamma(a)}\int_0^{\infty }e^{-st}t^{a-1}(1+it)^{b-a-1} \dd t
\end{equation}
for $s>0$. 

To ease our analysis of the integral on the right-hand side, apply the substitution $u=st$. Then, the formula becomes 
\begin{equation}\label{eq:U_RHS}
    U(a,b,-is)= \frac{e^{\pi i a/2}}{s^a \Gamma(a)}\int_0^{\infty }e^{-u}u^{a-1}\Big(1+\frac{iu}{s}\Big)^{b-a-1} \dd u. 
\end{equation}

\begin{lemma}
    For $\Re a>0$, $\partial_a U(a,b,-is)$ is conormal at $s=\infty$. 
\end{lemma}
\begin{proof}
    Explicitly, 
    \begin{multline}\label{eq:U_RHS_II}
        \partial_a U(a,b,-is) = \Big( \frac{\partial}{\partial a} \frac{e^{\pi i a/2}}{s^a \Gamma(a)} \Big) \int_0^{\infty }e^{-u}u^{a-1}\Big(1+\frac{iu}{s}\Big)^{b-a-1} \dd u \\ 
        + \frac{e^{\pi i a/2}}{s^a \Gamma(a)} \int_0^{\infty }e^{-u}u^{a-1}\Big(1+\frac{iu}{s}\Big)^{b-a-1} \log u\dd u   \\ - \frac{e^{\pi i a/2}}{s^a \Gamma(a)}\int_0^{\infty }e^{-u}u^{a-1}\Big(1+\frac{iu}{s}\Big)^{b-a-1} \log \Big( 1+\frac{iu}{s} \Big) \dd u
    \end{multline}

    The first term is 
    \begin{equation}
        \Big(\frac{\pi i}{2} - \log s  - \frac{\Gamma'(a)}{\Gamma(a)} \Big) U(a,b,-is)
    \end{equation}
    which is conormal because $U(a,b,-is)$ is. 
    
    The second integral is seen to be conormal by an analysis similar to that showing that $U(a,b,-is)$ is conormal. Indeed, it suffices to note that 
    \begin{equation}
        (s\partial_s)^\nu \int_0^{\infty }e^{-u}u^{a-1}\Big(1+\frac{iu}{s}\Big)^{b-a-1} \log u\dd u \in L^\infty[1,\infty)_s 
    \end{equation}
    for any $\nu\geq 0$. This is a linear combination of integrals of the same form, with the real part of $a$ increased to some (real) $a'\geq \Re a$ and extra factors of $s^{-1}$ out front. Explicitly,
    \begin{equation}
        \Big| \int_0^{\infty }e^{-u}u^{a'-1}\Big(1+\frac{iu}{s}\Big)^{b-a'-1}  (\log u)\dd u \Big| \leq  
        \begin{cases}
             \int_0^{\infty }e^{-u}u^{a'-1}(1+u)^{b-a'-1} |\log u| \dd u  & (b-a'\geq 1), \\ 
             \int_0^{\infty }e^{-u}u^{a'-1} |\log u|\dd u  & (b-a'<1),
        \end{cases}
    \end{equation}
    for $s\geq 1$. The integrals on the right-hand side are finite, so we are done.

    The third integral in \cref{eq:U_RHS_II} is handled almost identically. 
    \begin{equation}
        (s\partial_s)^\nu \int_0^{\infty }e^{-u}u^{a-1}\Big(1+\frac{iu}{s}\Big)^{b-a-1} \log \Big( 1+\frac{iu}{s} \Big) \dd u \in L^\infty [1,\infty)_s. 
    \end{equation}
    The only additional complication to the analysis is that some derivatives may fall on the log. Since 
    \begin{equation}
        \int_0^{\infty }e^{-u}u^{a-1}\Big(1+\frac{iu}{s}\Big)^{b-a-1} \partial_s \log \Big( 1+\frac{iu}{s} \Big) \dd u = -\frac{i}{s^2}\int_0^{\infty }e^{-u}u^{a'-1}\Big(1+\frac{iu}{s}\Big)^{b-a'-1}  \dd u 
    \end{equation}
    for $a'=a+1$, 
    these terms are estimated using what we have already observed.
\end{proof}

It remains to remove the requirement $\Re a>0$, because we need $a\approx -j$ for $j\in \bbN$. For negative $a$, the integral on the right-hand side of \cref{eq:U_RHS} fails to converge absolutely near the end $u=0$. 

We can make use of a recurrence relation for $U$, namely:
\begin{equation}\label{eq:U_rec}
    U(a,b,-is) = -isU(a+1,b+1,-is) - (b-a-1) U(a+1,b,-is)
\end{equation}
\cite[\href{https://dlmf.nist.gov/13.3.10}{13.3.10}]{NIST}. This relates $U(a,b,-is)$ to $U(a+1,-,-is)$. 
\begin{remark*}
    If $\Re a>0$, this identity is a simple consequence of integration-by-parts in the integral formula above. Indeed, 
    \begin{equation}\label{eq:ufull}
    U(a,b,-is) = \frac{e^{\pi i a/2}}{s^a }\frac{1}{ \Gamma(a+1)} \int_0^{\infty }e^{-u} \Big(\frac{\mathrm{d}}{\mathrm{d} u}  u^{a} \Big) \Big(1+\frac{iu}{s}\Big)^{b-a-1} \dd u. 
\end{equation}
Integrating-by-parts, noting that the boundary terms vanish:
\begin{align}
\begin{split}
    &\int_0^{\infty }e^{-u} \Big(\frac{\mathrm{d}}{\mathrm{d} u}  u^{a} \Big) \Big(1+\frac{iu}{s}\Big)^{b-a-1} \dd u =  (-1)  \int_0^{\infty } u^{a} \frac{\mathrm{d}}{\mathrm{d} u} \Big(  e^{-u}\Big(1+\frac{iu}{s}\Big)^{b-a-1} \Big) \dd u,\\
    &=\int_0^{\infty } u^{a}   e^{-u}\Big(1+\frac{iu}{s}\Big)^{b-a-1}  \dd u- (b-a-1)\frac{i}{s} \int_0^{\infty } u^{a}   e^{-u}\Big(1+\frac{iu}{s}\Big)^{b-a-2} \dd u.
    \end{split}
\end{align}
This is \cref{eq:U_rec}.

Once we know the recursion identity for $\Re a>0$, it follows for all $a$ via analyticity.
\end{remark*}

Differentiating \cref{eq:U_rec} in $a$ yields
\begin{equation}
     \partial_a U(a,b,-is) = -is\partial_a U(a+1,b+1,-is) + U(a+1,b,-is) - (b-a-1) \partial_a U(a+1,b,-is).
\end{equation}
If the functions on the right-hand side are known to be conormal at $s=\infty$, then we conclude the same of  $\partial_a U(a,b,-is)$.
So, we can conclude the desired result for $a$ if we know it for $a+1$. Proceeding inductively, we deduce the result for any $a\in \bbC$ from the already handled $\Re a>0$ case. In summary:
\begin{propositionp}
    For any $a\in \bbC$, $\partial_a U(a,b,-is)$ is conormal at $s=\infty$. 
\end{propositionp}

\section{Borel's lemma for Fr\'echet-valued functions}
\label{sec:Borel}

Polyhomogeneous series admit a variant of Borel's lemma, stating that any formal polyhomogeneous series is the expansion of some actual polyhomogeneous function. For scalar-valued functions, an exposition of this well-known fact can be found in \cite{HintzMicrolocal}. 
For \emph{Fr\'echet-valued} functions, the statement is:

\begin{lemma}\label{lem:Borel}
    Suppose $\mathcal{E}$ is an index set and $\mathscr{F}$ is a Fr\'echet space. Given $a_{z,j}\in \mathscr{F}$ for any $(z,j) \in \mathcal{E}$, there exists $u \in \mathcal{A}_{\mathrm{phg}}^{\mathcal{E}}([0,1)_x;\mathscr{F})$ such that
        \begin{equation}
    u(x)-\sum_{\substack{(z,j) \in \mathcal{E}\\\Re z \leq \alpha}}a_{z,j}x^z(\log x)^j \in \mathcal{A}^{\alpha}([0,1);\mathscr{F})
    \label{eq:u_borel_exp}
\end{equation}
for all $\alpha\in \bbR.$
\end{lemma}

This has the status of a folk theorem. Here we record the proof. 
\begin{proof}
    The proof is a modification of that for scalar-valued functions.
    Let $(z_n,j_n),n=1,2,\ldots$ enumerate $\mathcal{E}$ (where $\Re z_n$ is nondecreasing). For any sequence of $\epsilon_{n}>0$ such that $\epsilon_n \to 0$ as $n \to \infty$, define \begin{equation}\label{eq:ethlazy}  
        u(x) \coloneqq  \sum_{n=1}^{\infty}\underbrace{a_{z_n,j_n}\chi(x/\epsilon_{n})x^{z_n}(\log x)^{j_n}}_{\in \mathcal{A}^{\Re z_n-}},
    \end{equation}
    where $\chi \in C_{\mathrm{c}}^{\infty}(\bbR;[0,1])$ is some cutoff that is identically $1$ in a neighborhood of the origin. 
    Note that, in any compact subset $K\Subset (0,1)$, the sum above consists of only finitely many terms, so $u$ is well-defined and satisfies 
    \begin{equation} 
        u\in C^\infty((0,1);\mathscr{F}).
    \end{equation}
    What we want to show is that if the $\epsilon_n$ decrease to zero sufficiently fast, then \cref{eq:u_borel_exp} holds. That would mean that $u$ lies in the desired $\mathcal{A}^{\mathcal{E}}_{\mathrm{phg}}$, with the desired expansion. 

    The proof that we can choose such $\epsilon_n$ proceeds similarly to how it does in the case of scalar-valued functions, the only difference being that now we have a countable number of Fr\'echet seminorms $\lVert - \rVert_m : \mathscr{F}\to [0,\infty)$, $m\in \bbN$, that need to be controlled. 
    The fact that there could be infinitely many might be a cause for concern, but this can be handled using a diagonal-type argument, in which, when choosing $\epsilon_n$, we only consider the first $n$ seminorms $\lVert - \rVert_1,\cdots,\lVert - \rVert_n$. (Because the $a_{z,j}$ are assumed to lie in $\mathscr{F}$, this will suffice.)
    This is analogous to the argument used to handle the fact that infinitely many seminorms are required to define the conormal function spaces. 

    Indeed, the left-hand side of \cref{eq:u_borel_exp} is a sum of two terms:
    \begin{enumerate}
        \item $\sum_{\Re z_n \leq \alpha}a_{z_n,j_n}(\chi(x/\epsilon_n)-1)x^{z_n}(\log x)^{j_n}$, 
        \item  $v\coloneq \sum_{\Re z_n > \alpha}a_{z_n,j_n}\chi(x/\epsilon_n)x^{z_n}(\log x)^{j_n}$
    \end{enumerate}
    The first of these two terms is supported away from zero and therefore in $C_{\mathrm{c}}^\infty((0,1];\mathscr{F})\subset \mathcal{A}^\alpha$. 
    
    The second term is in $C^\infty((0,1);\mathscr{F})$ for the same reason that $u$ is. In order for this to lie in $\mathcal{A}^\alpha$, there are countably many seminorms which we need to be finite, namely
    \begin{equation}
        \lVert v \rVert_{\alpha,m,k}\coloneqq \sup_{x \in (0,1)}x^{-\alpha} \big\lVert(x\partial_x)^{k}v\big\rVert_m, 
    \end{equation}
    for $m,k\in\bbN$.
    We will choose the $\epsilon_n$ such that the series $v$ is geometrically convergent in each seminorm -- more specifically, we want to arrange that there exists a $N=N(\alpha,m,k)\in\bbN$ such that 
    \begin{equation}
        \lVert a_{z_n,j_n} \chi(x/\epsilon_n) x^{z_n} (\log x)^{j_n} \rVert_{\alpha,m,k} < 2^{-n} \text{ for all } n\geq 
        N(\alpha,m,k)\label{eq:Borel_gl}.
    \end{equation}
    This would give the desired geometric convergence.

    By making $N(\alpha,m,k)$ large enough, $n\geq N(\alpha,m,k) \Rightarrow  \Re z_n> \alpha+1$. Because the seminorms $\lVert - \rVert_{\alpha,m,k}$, in their dependence on $\alpha$, get \emph{worse} as $\alpha$ gets bigger, it suffices to arrange the bound above for $\alpha = \Re z_n-1$, and then it holds automatically for all smaller $\alpha$, for those same $n$.
    It therefore suffices to arrange
    \begin{equation}\label{eq:Borel_gl_0}
        \lVert a_{z_n,j_n} \chi(x/\epsilon_n) x^{z_n} (\log x)^{j_n} \rVert_{\Re z_n-1,m,k} < 2^{-n} \text{ for all }n\geq N(\alpha,m,k).
    \end{equation}
    Note that, for each $n$, we are considering only finitely many seminorms. Thus, we have finitely many restrictions on $\epsilon_n$. 
    The reason why it is possible to satisfy \cref{eq:Borel_gl} by taking $\epsilon_n$ small enough is that, for any $\beta< \Re z_n$, the term $x^{z_n} (\log x)^{j_n}$ lies in $\mathcal{A}^{\beta-}$ with $O(x^{\Re z_n-\beta})$ to spare. So, by making $\epsilon_n$ smaller, and thus shrinking the support of the cutoff $\chi(x/\epsilon_n)$ to where $x^{\Re z_n-\beta}$ has a suppressive effect,
    we can make 
    \begin{equation} 
    \lVert a_{z_n,j_n} \chi(x/\epsilon_n) x^{z_n} (\log x)^{j_n} \rVert_{\Re z_n-1,m,k}
    \end{equation} 
    as small as we like. More precisely, note that 
    \begin{align} 
    \begin{split}
    \lVert a_{z_n,j_n} \chi(x/\epsilon_n) x^{z_n} (\log x)^{j_n} \rVert_{\Re z_n-1,m,k} &\lesssim \epsilon^{1/2}_n \lVert a_{z_n,j_n} \chi(x/\epsilon_n) x^{z_n-1/2} (\log x)^{j_n} \rVert_{\Re z_n-1,m,k}  \\ 
    &\lesssim \epsilon_n^{1/2}  \underbrace{\lVert a_{z_n,j_n} x^{z_n - 1/2} (\log x)^{j_n} \rVert_{\Re z_n-1,m,k}}_{<\infty},
    \end{split}
    \end{align} 
    and the right-hand side $\to 0$ as $\epsilon_n\to 0$.\footnote{Here we used that multiplication by $\chi(x/\epsilon)$ is bounded with respect to the $\lVert-\rVert_{\alpha,m,k}$ seminorm uniformly as $\epsilon\to 0$: 
    \[
    \exists C_{\alpha,m,k}\text{ s.t. } \lVert \chi(x/\epsilon) v \rVert_{\alpha,m,k} \leq C_{\alpha,m,k} \lVert v \rVert_{\alpha,m,k} .
    \]
    This is obvious for $k=0$. For $k\geq 1$, this holds because $(x\partial_x \chi)^\nu(x/\epsilon) = (t\partial_t)^\nu  \chi(t)|_{t=x/\epsilon}$, which has $L^\infty$ norm $\lVert t\partial_t \chi \rVert_{L^\infty}$ independent of $\epsilon$.}
\end{proof}

\section{The limiting absorption principle in the black-box setting}
\label{sec:black-box_LAP}

Consider the setup of \S\ref{sec:black-box}. In particular, $P$ is a Helmholtz operator differing from the free case by a short-range perturbation.

By a \emph{tempered} distribution on $U^\complement \cup A \subset \bbR^d$, we mean one tempered in $\{r>R\}$ for some large $R\gg 1$. 

\begin{proposition}[LAP implies Sommerfeld]\label{prop:Sommerfeld_reduction}
    Fix $f\in \calS(\bbR^d)$ supported outside the black-box neighborhood $\overline{U}$. Let $\varepsilon_k\downarrow 0$. 
    Suppose that, 
    for each $k\in \bbN$,
    there exists a tempered classical solution $u[\varepsilon_k]$ (defined on $U^\complement \cup A$) to 
    \begin{equation}
        (P-i\varepsilon_k) u[\varepsilon_k]=f .
    \end{equation}
    Suppose moreover that there exists a tempered $u$ such that \begin{equation} 
        \lim_{k\to\infty} u[\varepsilon_k]=u,
    \end{equation} 
    in the topology of $\calS'(U^\complement \cup A)$. Then, $u$ satisfies the radiation condition: 
    \begin{equation} 
        u\in e^{ir} r^{-(d-1)/2} C^\infty(\overline{\bbR^d})
    \end{equation} 
    outside of some sufficiently large ball.  
\end{proposition}
\begin{proof}
Automatically, $Pu=f$. 
We can now reduce to a problem outside the black-box region: letting $\chi\in C^\infty(\overline{\bbR^d})$ vanish near the black-box neighborhood $\overline{U}$ and $=1$ identically outside of a larger ball,  
\begin{equation}
    (P-i\varepsilon)(\chi u[\varepsilon])= \chi f + \underbrace{[P,\chi] u[\varepsilon]}_{\in C_{\mathrm{c}}^\infty(\bbR^d) } .
\end{equation}
For the rest of the proof, we work with $P-i\varepsilon$ defined as an operator on all of $\bbR^d$. We do not need to worry about topological obstructions to extending the coefficients, because we assumed at the outset, before tossing out the black-box region, that $P$ was defined globally. One of the advantages to working globally is that $P$ is essentially self-adjoint acting on $C_{\mathrm{c}}^\infty(\bbR^d)$, with respect to the $L^2(\bbR^d,g)$-inner product. Therefore, we have an $L^2$-bounded resolvent 
\begin{equation} 
    (P-i\varepsilon)^{-1} : L^2(\bbR^d)\to L^2(\bbR^d),
\end{equation} 
defined using the functional calculus.

Note that $P-i\varepsilon$ is (totally) elliptic in the sc-calculus. 
Elliptic regularity in the sc-calculus (see e.g.\ \cite{HintzMicrolocal}) tells us that $\chi u[\varepsilon] \in \calS(\bbR^d)$ for  each $\varepsilon>0$. 
Consequently, 
\begin{equation}
    \chi u[\varepsilon] = (P-i\varepsilon)^{-1}(\chi f+[P,\chi] u[\varepsilon]).
\end{equation}

We can now apply the limiting absorption principle in the form \cite[Prop.\ 14]{Me94}. The cited proposition tells us that $\lim_{\varepsilon \to 0^+} (P-i\varepsilon)^{-1} f$ exists in a weighted Sobolev space $\calX$, and satisfies the radiation condition. By a quantitative version of that argument, the same holds for 
\begin{equation}
     u[0]\coloneqq  \lim_{k\to\infty}  (P-i\varepsilon_k)^{-1}(\chi f+[P,\chi] u[\varepsilon_k]), 
\end{equation}
because ordinary elliptic regularity tells us that $[P,\chi]u[\varepsilon_k] \to [P,\chi]u$ in $C_{\mathrm{c}}^\infty$. 
The topology of $\calX$ is finer than that of $\calS'$, so
\begin{equation}
    \chi u[\varepsilon_k]\longrightarrow u[0]
    \quad\text{in }\calS'.
\end{equation}
On the other hand, the assumed convergence
$u[\varepsilon_k]\to u$ in $\calS'$ implies
\begin{equation}
    \chi u[\varepsilon_k]\longrightarrow \chi u
    \quad\text{in }\calS'.
\end{equation}
Since $\calS'$ is Hausdorff, these limits agree:
\begin{equation}
    u[0]=\chi u.
\end{equation}
Thus, $\chi u$ satisfies the radiation condition. Since $\chi=1$
outside a sufficiently large ball, the same is true of $u$.
\end{proof}

In summary, the Sommerfeld radiation condition is automatic if our solution arises as a tempered limit of solutions $u[\varepsilon]$ of the equation with spectral parameter shifted by $+i\varepsilon$. 
The existence of $u[\varepsilon]$ typically comes about from a self-adjoint realization of $P$ on an appropriate domain capturing the black-box boundary condition. 
Formally, 
\begin{equation}
    u[\varepsilon] = (P-i\varepsilon)^{-1} f. 
\end{equation}
Since $P$ is symmetric with respect to an $L^2$-space, temperedness will be automatic. Our main task is therefore to prove the convergence (in the $\calS'$ topology) of $u[\varepsilon]$ to some admissible limit $u$. 

This can often be proven by means of an estimate. 
Suppose that $\calX$, $\calY$, and $\calZ$ are Hilbert spaces with
continuous embeddings
\begin{equation}
    \calX\hookrightarrow\calY\hookrightarrow\calS'(\bbR^d),
    \qquad
    \calX\hookrightarrow\calZ\hookrightarrow\calS'(\bbR^d),
\end{equation}
such that the embedding $\calX\hookrightarrow\calZ$ is compact.
Suppose in addition that
\begin{equation}
    P:\calX\longrightarrow\calY
\end{equation}
is bounded, with $(P-i \varepsilon)$ mapping $\calX\to \calY$ bijectively for each $\varepsilon>0$, and that the half-Fredholm estimate
\begin{equation}
    \lVert u\rVert_{\calX}
    \lesssim
    \lVert(P-i\varepsilon)u\rVert_{\calY}
    +\lVert u\rVert_{\calZ}
\end{equation}
holds for every $u\in\calX$, uniformly in
$\varepsilon\in[0,1]$.

If we know that $P-i\varepsilon$ is \emph{injective} on $\calX$ for each $\varepsilon \in [0,1]$, then a standard functional analytic argument allows us to remove the $\lVert u \rVert_{\calZ}$ term from the right-hand side of the estimate above, getting 
\begin{equation} \label{eq:LAP_improved}
    \lVert u \rVert_{\calX} \lesssim \lVert (P-i\varepsilon) u \rVert_{\calY}.
\end{equation} 
Indeed, suppose to the contrary that, for any $n\in \bbN^+$, there exist $\varepsilon_n \in [0,1]$ and $u_n \in \calX$ such that 
\begin{equation}\label{eq:LAP_contradictor}
    \lVert u_n \rVert_{\calX} \geq n \lVert (P-i\varepsilon_n) u_n \rVert_{\calY}
\end{equation} 
Because both sides of this equation scale linearly with $u_n$, we may assume without loss of generality that $\Vert u_n\rVert_{\calX}=1$ for all $n$.
By passing to a subsequence if necessary, we may assume without loss of generality that $\varepsilon_n\to \varepsilon_\infty$ for some $\varepsilon_\infty \in [0,1]$. 
Similarly, by Banach--Alaoglu, we may assume (after passing to a subsequence if necessary) that $u_n\to u_\infty$ weakly in $\calX$, for some $u_\infty\in \calX$. 
Now our goal is to establish a contradiction by proving (i) that $(P-i\varepsilon_\infty)u_\infty =0$, (ii) $u_\infty \neq 0$. Indeed:
\begin{enumerate}[label=(\roman*)]
    \item \Cref{eq:LAP_contradictor} (together with $\lVert u_n \rVert_{\calX}=1$) yields 
    \begin{equation}
        (P-i\varepsilon_n) u_n\to 0 \text{ in }\calY.
    \end{equation}
    But we also have $(P-i\varepsilon_n) u_n\to (P-i\varepsilon_\infty)u_\infty$ weakly in $\calY$. Since the weak topology on a Hilbert space is Hausdorff, limits in it are unique, so $(P-i\varepsilon_\infty)u_\infty=0$. 
    \item  The half-Fredholm estimate gives us a lower bound 
    \begin{equation} 
        1=\lVert u_n \rVert_{\calX} \lesssim \lVert u_n \rVert_{\calZ}.
    \end{equation} 
    Compact maps send weakly convergent sequences to strongly convergent sequences, so $u_n\to u_\infty$ strongly in $\calZ$. 
    This implies $1\lesssim \lVert u_\infty \rVert_{\calZ}$. Consequently, $u_\infty$ cannot be zero. 
\end{enumerate}
This completes the argument.

Now suppose, given $f\in \calS(\bbR^d)$ supported outside the black-box neighborhood $U$, we can find $u[\varepsilon]\in \calX$ such that 
\begin{equation} 
    (P-i\varepsilon)u[\varepsilon]=f.
\end{equation} 
By the estimate above, \cref{eq:LAP_improved}, the family $\{u[\varepsilon]\}_{0< \varepsilon\leq 1}$ is bounded in $\calX$, so 
\begin{itemize}
    \item $\varepsilon u[\varepsilon]\to 0$ strongly in $\calX$ and therefore in $\calY$, 
    \item Banach--Alaoglu implies that there is a weakly convergent subsequence, $u[\varepsilon_k]\to u$.
\end{itemize}
So, 
\begin{equation}
   Pu=\lim_{k\to\infty} (P-i\varepsilon_k) u[\varepsilon_k]=f ,   
\end{equation}
where the limit is with respect to the weak topology on $\calY$. Thus, we get a solution $u\in \calX$ of $Pu=f$, and $u[\varepsilon_k] \to u$ in $\calS'$.  
We have therefore proven: 

\begin{proposition}
    Suppose that we have Hilbertizable function spaces as above,
    such that 
    \begin{enumerate}[label=(\roman*)]
        \item any solution $u\in \calX$ of $Pu=f$ is admissible, 
        \item if $u\in \calX$ and $(P-i\varepsilon)u=0$ for some $\varepsilon \in [0,1]$, then $u=0$, 
        \item any admissible solution of $Pu=0$ lies in $\calX$. 
    \end{enumerate}
    Suppose moreover that we have a half-Fredholm estimate 
    \begin{equation}
        \lVert u \rVert_{\calX} \lesssim \lVert (P-i\varepsilon) u \rVert_{\calY} + \lVert u \rVert_{\calZ},
    \end{equation}
    uniformly in $\varepsilon \in [0,1]$. 
    Then, the black-box well-posedness assumption holds. 
\end{proposition}

Thus, a proof of the limiting absorption principle for a problem with a black-box consists of two components: a half-Fredholm estimate in suitable Sobolev spaces, and an injectivity guarantee. 

A half-Fredholm estimate can often be stitched together from two separate half-Fredholm estimates, one near the black-box and one away from the black-box.
Suppose we have (Hilbertizable) function spaces $\calX,\calY,\calZ$ as above such that, for any $\chi \in C^\infty(\bbR^d)$ supported near either the black-box or spatial infinity, a half-Fredholm estimate 
\begin{equation}
    \lVert \chi u \rVert_{\calX}\lesssim \lVert (P-i\varepsilon) \chi u \rVert_{\calY}+\lVert \chi u \rVert_{\calZ}
\end{equation}
holds (uniformly in $\varepsilon$).
We can apply this for a partition of unity $1=\chi_{\mathrm{B}}+\chi_\infty$, where $\chi_{\mathrm{B}}\in C_{\mathrm{c}}^\infty(\bbR^d)$ is supported near the black-box and $\chi_\infty$ is supported away from the black-box. Then, assuming that multiplication by $\chi_\bullet$ is a bounded operator on $\calZ$, 
\begin{equation}
    \lVert u \rVert_{\calX} \leq \lVert \chi_{\mathrm{B}} u \rVert_{\calX} + \lVert \chi_\infty u \rVert_{\calX} \lesssim \lVert (P-i\varepsilon) (\chi_{\mathrm{B}} u) \rVert_{\calY} +\lVert (P-i\varepsilon) (\chi_\infty u)  \rVert_{\calY}+\lVert u \rVert_{\calZ} .
\end{equation}
Thus, assuming also that multiplication by $\chi_\bullet$ is a bounded operator on $\calY$, 
\begin{equation}
    \lVert u \rVert_{\calX} \lesssim \lVert (P-i\varepsilon) u \rVert_{\calY} + \lVert [P,\chi_{\mathrm{B}}] u \rVert_{\calY} +  \lVert [P,\chi_{\infty} ] u \rVert_{\calY} +\lVert u \rVert_{\calZ}. 
\end{equation}
The commutators $[P,\chi_\bullet]$ are first order differential operators with coefficients supported away from the black-box and spatial infinity. Consequently, we should be able to combine them with $\lVert u \rVert_{\calZ}$ to get an overall compact error. Indeed, suppose that $\calX$ is equivalent, for some $m\in \bbR$, to $H^m$ away from the black-box and spatial infinity. Suppose also that $\calY$ is equivalent to $H^{m-2}$. This is a typical occurrence  (because $P$ is elliptic in the ordinary differential sense). Then, 
\begin{equation}
    \lVert [P,\chi_{\bullet }] u \rVert_{\calY}\lesssim \lVert [P,\chi_{\bullet }] u \rVert_{H^{m-2}}\lesssim \lVert \chi_0 u \rVert_{H^{m-1}} 
\end{equation}
for $\chi_0\in C_{\mathrm{c}}^\infty(\bbR^d)$ identically equal to one on a slightly larger neighborhood than $\operatorname{supp} \nabla \chi_\bullet$. The map $\calX \ni u\mapsto \chi_0 u\in H^{m-1}$ is compact, so by modifying $\calZ\rightsquigarrow \calZ_\circ$ slightly, we get a global half-Fredholm estimate 
\begin{equation}
    \lVert u \rVert_{\calX} \lesssim \lVert (P-i\varepsilon) u \rVert_{\calY} + \lVert u \rVert_{\calZ_\circ}. 
\end{equation}
Thus, in order to prove a Fredholm estimate for a black-box setup, it suffices to give a Fredholm setup \emph{near the black-box}. Then, we construct a global function space $\calX$ (using the usual Fredholm setup away from the black-box), and the remaining question is whether $P-i\varepsilon$ is injective on $\calX$.

As stated above, the injectivity of $P-i\varepsilon$ on $\calX$ for $\varepsilon>0$ typically follows from a self-adjoint realization of $P$, so the main issue is injectivity at $\varepsilon=0$.

The injectivity argument in \cite[\S14]{Me94}, for $\varepsilon=0$, combines two ingredients: 
\begin{itemize}
    \item a unique continuation lemma for elliptic second-order PDE, cited in \cite[\S10]{Me94} from \cite{Hormander}, 
    \item a boundary-pairing argument \cite[\S13]{Me94}, based on Green's second identity 
    \begin{equation}
        \int_U (\phi \triangle \psi - \psi \triangle \phi ) \dd V = - \int_{\partial U} (\phi \nabla \psi - \psi \nabla \phi  )\cdot  \dd A ,
    \end{equation}
    for reasonable domains $U$ and $\phi,\psi\in C^2(\bbR^d)$. 
\end{itemize}
The unique continuation lemma says that any Schwartz $u$ satisfying $(P-i\varepsilon) u=0$ must vanish near infinity. As long as the black-box exterior $\bbR^d\backslash B=B^\complement$ is path-connected, a unique continuation lemma (say that of Aronszajn) forces $u=0$ everywhere on the black-box exterior. So, given the half-Fredholm setup, the only remaining question is whether the boundary-pairing argument goes through on $\calX$, for elements of $\ker (P-i\varepsilon)$.

\begin{example}[Obstacle + potential scattering, cont.]
    In the presence of an obstacle, the usual local Fredholm setup is to take a function space $\calX$ that is $H^2\cap H^1_0$ near the obstacle.
    In the boundary-pairing argument, the enemy is the boundary term 
    \begin{equation}
        \int_{\partial \Omega} (\bar{v} \partial_\nu u - u \partial_\nu \bar{v} ) \dd A ,\quad \partial_\nu=\text{normal derivative}.
    \end{equation}
    But this vanishes for $u,v \in H_0^1$, by virtue of the Dirichlet boundary condition. So Melrose's argument goes through.
\end{example}

\begin{example}[Singularities, cont.]
    Consider again the setting of singular potential scattering:
    \begin{equation}
        \exists \alpha,\mathsf{Z}\in \bbR\text{ s.t. }V-\frac{\alpha}{r^2}+\frac{\mathsf{Z}}{r\langle r \rangle} \in C^\infty(\overline{\bbR^d} ),\quad 
        \alpha > - \frac{(d-2)^2}{4}.
    \end{equation}
    To modify the Fredholm analysis in \cite{Me94} for the present setting, one merely needs to supplement it with a Fredholm analysis at the regular singularity at $r=0$. (The two estimates are easily combined.)
    At $r=0$, the Helmholtz equation has the form of a regular singular PDE, so Melrose's b-tools can be applied. For $\alpha>-(d-2)^2/4$, the two indicial roots associated with an angular mode $Y\in C^\infty(\bbS^{d-1})$ are distinct. The Friedrichs extension selects for the recessive boundary condition -- hence the greater of the two indicial roots. So, via the elliptic parametrix construction in the b-calculus \cite{Me93}, one has a Fredholm setup between spaces $\calX,\calY,\calZ$ identical to those in \cite{Me94} (sc-Sobolev spaces) away from the origin, with $\calX$ the Friedrichs domain near the origin and $\calY,\calZ=L^2$ near the origin.

    Regarding the boundary-pairing argument: applying the Green identity on a punctured ball in which a small ball of radius $\varepsilon>0$ has been removed, we get an additional boundary term: 
    \begin{equation}
        \varepsilon^{d-1} \int_{\bbS^{d-1}} ( \bar{v} \partial_r u - u \partial_r \bar{v}) \dd \omega. 
    \end{equation}
    What we want to show is that, when $u,v$ are recessive solutions of $Pu,Pv=0$ (say, proportional to a single angular mode $Y$), the boundary term above is $o(1)$ as $\varepsilon\to 0^+$. 
    Letting $c\in \bbR$ denote the smallest recessive indicial root, the integrand is $O( r^{2c-1})$, using that $u,v$ are conormal (which follows from the assumptions above) and that 
    \begin{equation} 
        r\partial_r\in \operatorname{Diff}_{\mathrm{b}}^{1,0}[0,1)_r.
    \end{equation}
    So, the integral in question is 
    \begin{equation}
        \varepsilon^{d-1} \int_{\bbS^{d-1}} ( \bar{v} \partial_r u - u \partial_r \bar{v}) \dd \omega  = O(\varepsilon^{ d-2 +2c }). 
    \end{equation}
    We therefore win if $c>-(d-2)/2$. A quick computation of the indicial roots reveals that this is indeed the case ---- among all the angular modes $Y$, the one with the smallest larger indicial root is the s-wave $Y=1$. The b-normal operator controlling the corresponding radial profiles is 
    \begin{equation}
        -\frac{\mathrm{d}^2}{\mathrm{d} r^2} - \frac{d-1}{r} \frac{\mathrm{d}}{\mathrm{d}r} + \frac{\alpha}{r^2} \in \operatorname{Diff}_{\mathrm{b}}^{2,2}([0,1)_r).
    \end{equation}
    Thus, the two indicial roots are the zeros of $-c(c-1)-(d-1) c + \alpha=0$. That is, 
    \begin{equation}
        c = -\frac{d-2}{2} \pm \sqrt{\frac{(d-2)^2}{4}+\alpha}. 
    \end{equation} 
    The better indicial root is the one with the $+$ sign in front of the square root, and this indicial root is $>-(d-2)/2$. 
\end{example} 

\section{A Grushin lemma}
\label{sec:Grushin} 

Fix a point $j_0\in \bbC$ and an open neighborhood $U\ni j_0$. 
Fix Hilbert spaces $\calX,\calY$, and suppose that $L=\{L(j)\}_{j\in U}$ is an analytic family of Fredholm operators $L(j):\calX\to \calY$, such that $L(j)$ is invertible for each $j\in U\backslash \{j_0\}$. Now let $\calK = \ker L(j_0)\subset \calX$ and $\calX_1 = \calK^\perp$ denote the orthogonal complement in $\calX$, so that 
\begin{equation} 
    \calX=\calK\oplus \calX_1.
\end{equation} 
Similarly, let $\calY_1 =\operatorname{range} L(j_0) \subset \calY$, and let $\calC  = \operatorname{coker} L(j_0)= \calY_1^\perp$ denote the orthogonal complement in $\calY$, so that 
\begin{equation}
    \calY = \calC\oplus \calY_1. 
\end{equation}

We can decompose 
\begin{equation}
\label{eq:block-decomp}
    L(j) = 
    \begin{pmatrix}
        a(j) & \beta(j) \\ 
        \gamma(j) & D(j) 
    \end{pmatrix} : 
    \calK\oplus \calX_1
    \to 
    \begin{array}{c}
           \calC\\ \oplus \\ \calY_1
    \end{array}
\end{equation}
for $a(j):\calK\to \calC$, $\beta(j):\calX_1\to \calC$, $\gamma(j):\calK\to \calY_1$, and $D(j):\calX_1\to \calY_1$.

By construction, $D(j_0)$ is invertible, so the same applies for $D(j)$ for $j$ near $j_0$. Thus, we can define the \emph{Schur complement} 
\begin{equation}
\label{eq:schur}
    S(j)=a(j) - \beta(j) D(j)^{-1} \gamma(j) : \calK\to \calC,
\end{equation}
and this makes sense for $j$ close to $j_0$. 
The Schur complement formula says that the invertibility of $L(j)$ is equivalent to that of $S(j)$. Concretely, we have the formula 
\begin{equation}
    \begin{pmatrix}
        S(j)^{-1} & 0 \\ 
        0 & D(j)^{-1} 
    \end{pmatrix}
    = \begin{pmatrix}
        1 & 0 \\
         D(j)^{-1}\gamma & 1
    \end{pmatrix}
    L(j)^{-1} 
    \begin{pmatrix}
        1 & \beta D(j)^{-1} \\ 
        0 & 1
    \end{pmatrix}.
\end{equation}
Since $L(j)$ is invertible for $j$ in a punctured neighborhood of $j_0$, the Schur complement $S(j):\calK\to \calC$ is an \emph{isomorphism} for such $j$. In particular, $\calC,\calK$ have the same (finite) dimension. This same fact can be deduced from the continuity of the Fredholm index.

Because $L(j_0)\calK=0$, the operators $a(j),\gamma(j)$ vanish at $j=j_0$. Consequently, so does $S(j)$.

We now consider $L'(j)=\partial_j L(j)$ and $[L'(j)]:\calK\to\calC$, $[L'(j)]:k\mapsto L'(j)k \bmod\calY_1$, to prove the main result of this appendix:
\begin{lemma}\label{lem:Grushin_magic}
   Suppose that $S(j)$ vanishes simply at $j=j_0$, meaning that $S'(j_0):\calK\to \calC$ is an isomorphism. Then, $[L'(j_0)]$ is an isomorphism $\calK\to \calC$.  
\end{lemma}
\begin{proof}
    From the block decomposition \eqref{eq:block-decomp}, we have $[L'(j)] = a'(j)$. 
    By the definition of $\calC$, we have $\beta(j_0)=0$. Combining this with $\gamma(j_0)=0$, we see that the $\beta(j)D(j)^{-1}\gamma(j)$ term in \eqref{eq:schur} vanishes quadratically at $j=j_0$, and hence 
    \begin{equation}
        S'(j_0)= a'(j_0).
    \end{equation}
    The lemma follows from the assumption that $S'(j_0)$ is an isomorphism.
\end{proof}

\section*{Acknowledgements}

ES is supported by NSF grant \texttt{DMS-2401636}.  \Cref{lem:coker_fund} and its proof were suggested by LLM. The computational details in \S\ref{subsec:inverse_square_example} were derived with LLM assistance.

\printbibliography

\end{document}